\documentclass[11pt,a4paper,reqno]{amsart}

\numberwithin{equation}{section}        
\usepackage{amsfonts,amsmath,amssymb}
\usepackage{upgreek}

\usepackage{graphicx}
\usepackage{caption}
\usepackage{subcaption}

\usepackage{psfrag}
\usepackage{enumitem}

\usepackage{comment} 
\excludecomment{TOOLONG} 

\usepackage[linktocpage]{hyperref} 

\DeclareRobustCommand{\SkipTocEntry}[5]{}

\newtheorem{theorem}{Theorem}[section]
\newtheorem{lemma}[theorem]{Lemma}
\newtheorem{proposition}[theorem]{Proposition}
\newtheorem{corollary}[theorem]{Corollary}
\newtheorem{conjecture}[theorem]{Conjecture}
\newtheorem{remark}[theorem]{Remark}
\newtheorem{definition}[theorem]{Definition}

\newcommand{\be}{\begin{equation}}
\newcommand{\ee}{\end{equation}}
\newcommand{\ba}{\begin{array}}
\newcommand{\ea}{\end{array}}
\newcommand{\bea}{\begin{eqnarray}}
\newcommand{\eea}{\end{eqnarray}}
\newcommand{\bee}{\begin{eqnarray*}}
\newcommand{\eee}{\end{eqnarray*}}

\newcommand{\lab}{\label}

\newcommand{\ds}{\displaystyle}
\newcommand{\nn}{\nonumber}

\renewcommand{\c}{\cdot}
\newcommand{\les}{\lesssim}

\newcommand{\R}{\mathbb{R}}
\newcommand{\N}{\mathbb{N}}
\newcommand{\D}{{\bf D}}
\renewcommand{\gg}{{\bf g}}

\newcommand{\ep}{\varepsilon}
\newcommand{\pr}{\partial}

\renewcommand{\a}{\alpha}
\renewcommand{\b}{\beta}

\newcommand{\la}{\lambda}
\newcommand{\ub}{{\underline{u}}}
\newcommand{\ubb}{\underline{u}_b}

\newtheorem{thm}{Theorem}[section]
\newtheorem{cor}[thm]{Corollary}
\newtheorem{lem}[thm]{Lemma}

\def\eps{\epsilon}%
\def\vrho{\varrho}%
\def\tensor{\,\raise2pt\hbox{${}_{\otimes}$}\,}
\def\fdg{\,:\,}
\def\ptl{\partial}
\def\rest#1{\raise-2pt\hbox{${\lfloor_{#1}}$}}

\def\cal#1{\mathcal{#1}}
\def\mbo#1{\boldsymbol{#1}}

\def\hatt#1{\widehat #1{}}
\def\ulin#1{\underline{#1}{}}
\def\tild#1{\widetilde{#1}{}}

\def\grad{{\nabla}}
\newcommand{\leftexp}[2]{{\vphantom{#2}}^{#1}{#2}}

\def\halb{\frac{1}{2}}
\usepackage{color}

\newcommand{\green}[1]{{\color{green}#1}}

\newcounter{mnotecount}[section]

\renewcommand{\themnotecount}{\thesection.\arabic{mnotecount}}

\newcounter{mymnotecount}[section]
\renewcommand{\themymnotecount}{\thesection.\arabic{mymnotecount}}
\newcommand{\mymnote}[1]{\protect{\stepcounter{mymnotecount}}${\raisebox{0.5\baselineskip}[0pt]{\makebox[0pt][c]{\color{green}{\tiny\em$\bullet$\themnotecount}}}}$\marginpar{\raggedright\tiny\em$\!\bullet$\themymnotecount:

\green{#1}}\ignorespaces}

\renewcommand{\mymnote}[1]{}

\newcommand{\MM}{M}
\renewcommand{\Re}{\mathbb R}

\newcommand{\Orb}{\mathcal Q}

\newcommand{\half}{\frac{1}{2}}

\newcommand{\Om}{\Omega} 

\newcommand{\cp}{\kappa} 	

\newcommand{\RR}{\mathcal R} 
\newcommand{\TT}{\mathcal T}
\renewcommand{\AA}{\mathcal A}

\newcommand{\Xcal}{\mathcal{X}}
\newcommand{\DXp}{D^+(\Xcal)\setminus\{p\}} 

\title[]{Global regularity for the 2+1 dimensional equivariant Einstein-wave map system}
\author[L. Andersson]{Lars Andersson}
\address{Max Planck Institute for Gravitational Physics (Albert Einstein Institute), Germany}
\email{lars.andersson@aei.mpg.de}
\author[N. Gudapati]{Nishanth Gudapati}
\address{Department of Mathematics, Yale University, USA}
\email{nishanth.gudapati@yale.edu}
\author[J. Szeftel]{J\'er\'emie Szeftel}
\address{Laboratoire Jacques-Louis Lions, Universit\'e Pierre et Marie Curie, France}
\email{jeremie.szeftel@upmc.fr}

\begin{document}

\begin{abstract} In this paper we consider  the equivariant 2+1 dimensional  Einstein-wave map system and show that if the target satisfies the so called Grillakis condition, then global existence holds. In view of the fact that the 3+1 vacuum Einstein equations with a spacelike translational Killing field reduce to a 2+1 dimensional Einstein-wave map system with target the hyperbolic plane, which in particular satisfies the Grillakis condition, this work proves global existence for the equivariant class of such spacetimes. 
\end{abstract}

\maketitle

\tableofcontents

\section{Introduction}

In this paper we shall prove that global existence holds for the maximal Cauchy development of  asymptotically flat initial data for the equivariant 2+1 dimensional  Einstein-wave map system assuming that the target $(N, h)$ is a rotationally symmetric 2-manifold with metric satisfying the Grillakis condition, see \eqref{eq:grillakis-cond-first} below. The Grillakis condition holds in particular if $h$ has negative sectional curvature. Therefore, our result applies in the important special case obtained by considering the 3+1 vacuum Einstein equations with a spacelike translational Killing field which reduces to a 2+1 dimensional Einstein-wave map system with target the hyperbolic plane $\mathbb{H}^2$, see \cite{moncrief:MR955865} and also \cite{andersson:1999gr.qc....11032A,moncrief:MR3248565} and references therein. It follows that global existence holds for an equivariant solution of the 3+1 vacuum Einstein equations with a spacelike translational Killing field. \\

Before discussing the equivariant 2+1 dimensional  Einstein-wave map system, we provide some background on the equivariant wave map problem.  

\subsection{Equivariant critical wave maps} \label{sec:equivcrit}

Let $(\MM, \gg_{\mu\nu})$ be a Lorentzian spacetime and $(N, h_{AB})$ a Riemannian manifold. 
The action defined for a map $\Phi: \MM \to N$ by 
\begin{equation}\label{eq:wm-action-general} 
S_{\text{WM}} \fdg = -\halb\int_\MM \gg^{\mu\nu} \partial_\mu \Phi^A \partial_\nu \Phi^B h_{AB}\circ\Phi 
\end{equation} 
has Euler-Lagrange equation 
\begin{equation}\label{eq:wm-eqn-general}
\square_\gg \Phi^A + {}^{(h)} \Gamma_{BC}^A\circ\Phi\,  \partial_\mu \Phi^B \partial_\nu \Phi^C \gg^{\mu\nu} = 0
\end{equation} 
where, denoting by $\nabla$ the Levi-Civita covariant derivative of $\gg$, $\square_{\gg} = \nabla^\alpha \nabla_\alpha$ is the d'Alembertian, and  where ${}^{(h)} \Gamma_{BC}^A$ denote the Christoffel symbols of $h$. The action \eqref{eq:wm-action-general} is the Lorentzian analogue of the Dirichlet integral, or harmonic map energy, and if $\MM$ is static, time independent solutions of \eqref{eq:wm-eqn-general} are simply harmonic maps. 
In the particular case where the target is a compact Lie group, this system is known in the physics literature as a  $\sigma$-model, and in the mathematics literature (with general target), it is known as the wave map equation. 

Next, we restrict ourselves to the equivariant class. We assume $\MM$ is a globally hyperbolic $2+1$-dimensional spacetime with Cauchy surface diffeomorphic to $\Re^2$ and that $N$ is a complete Riemannian manifold of dimension $2$ with metric $h$ of the form 
$$
h = d\rho^2 + g^2(\rho) d\theta^2 
$$
for an odd function $g: \Re \to \Re$ with $g'(0) = 1$.  
Let $e^{i\theta}$, $\theta \in \mathbb{S}^1 = \Re/2\pi\mathbb{Z}$  denote a semifree circle action on $\MM$ and $N$. 
We assume that the $S^1$ action on $\MM$ is generated by a hypersurface orthogonal Killing field $\partial_\theta$, that it has a non-empty fixed point set\footnote{It follows that the fixed point set is a timelike line, see \cite{Carot:2000CQGra..17.2675C} and references therein.}, and that the non-trivial orbits of this action in $\MM$ are spatial.Then we may write $\gg$ in the form\footnote{As an example, consider $\Re^{2+1}$ with the metric 
$$
\gg = -dt^2 + dr^2 + r^2 d\theta^2
$$
In this case, the orbit space is $\Orb = \{(t,r), \ r \geq 0\}$ with metric $\check{\gg} = - dt^2 + dr^2$.} 
\begin{equation}\label{eq:gg-split}
\gg = \check{\gg} + r^2 d\theta^2
\end{equation} 
where $\check{\gg}$ is a metric on the orbit space $\Orb = \MM/\mathbb{S}^1$ and $r$ is the radius function, defined such that $2\pi r(p)$ is the length of  the $\mathbb{S}^1$ orbit through $p$. 
We assume that $\MM$ has Cauchy surface $\Sigma$ diffeomorphic to $\Re^2$, which we may, without loss of generality, assume to be symmetric\footnote{\label{foot:cauchytime}To see this, note that $\Orb$ is globally hyperbolic, and hence by \cite{bernal:sanchez:2006LMaPh..77..183B}, there is a Cauchy time function $\check{t}$ on $\Orb$ which may be lifted to a symmetric Cauchy time function $t$ on $\MM$.}. 

A map $\Phi: \MM \to N$ is equivariant, with rotation number $k \in \mathbb{Z}$ if 
$$
\Phi \circ e^{i\theta} = e^{ik\theta} \circ \Phi.
$$
Let the function $\phi$ be defined by 
\begin{equation}\label{eq:phidef}
\phi = \rho \circ \Phi
\end{equation} 
where $\rho: N \to \Re_+$ is the radial coordinate function on $N$. With the above definitions, the 
wave maps equation takes the form 
\begin{equation} 
\square_\gg \phi - \frac{k^2 g(\phi)g'(\phi)}{r^2} = 0.  \label{eq:wave-wm} 
\end{equation} 

The Cauchy problem for equivariant wave maps with base $M = \Re^{2+1}$ was studied by Shatah and Tahvildar-Zadeh \cite{jal_tah1} who proved that for targets satisfying\footnote{This condition is equivalent to the assumption that the target $(N, h)$ is geodesically convex.} 
$$g'(s)\geq 0,\quad \text{for } s \geq 0$$
global well-posedness holds for the equivariant wave map problem. It was then shown by Grillakis in \cite{Grillakis} that it suffices for the target to satisfy the Grillakis condition\footnote{For example, the Grillakis condition is satisfied in the important particular case $N = \mathbb{H}^2$, with $g(\rho) = \sinh(\rho)$.}
\begin{equation} \label{eq:grillakis-cond-first} 
sg'(s)+g(s) >0, \quad \text{for } s > 0. 
\end{equation}  
Let us also mention subsequent developments by Shatah and Struwe in \cite{shatah_struwe}, Shatah and Tahvildar-Zadeh in \cite{jal_tah}, and by Struwe in \cite{struwe:MR1990477}. Finally, let us mention the work of Christodoulou and Tahvildar-Zadeh for the case of spherically symmetric solutions \cite{chris_tah1}. Note that in these works, the proof consists of two main steps, a proof of energy non-concentration and a proof of global existence for small energy initial data.  

\begin{remark}\label{rem:severalremarks} 
\begin{enumerate} 
\item
In \cite{Berger}, it was established that vacuum Einstein's equations for $G_2$-symmetric 3+1 dimensional spacetimes reduce to spherically symmetric wave maps from $U \fdg  \Re^{2+1} \to \mathbb{H}^2.$ Consequently, the aforementioned work of Christodoulou and Tahvildar-Zadeh \cite{chris_tah1}
was applied in \cite {Berger} to prove global regularity for large data for these spacetimes. 
In the context of our problem, we would like to emphasize that the non-zero rotation number prevents a similar reduction to flat-space wave maps. Thus we are forced to consider the coupling with Einstein's equations. 
\item \label{point:generalwmlargedata}
More recent work shows that large data global existence holds for the wave map problem \eqref{eq:wm-eqn-general} with $\MM= \Re^{2+1}$ and target\footnote{Note that more general targets are considered in \cite{sterbenz:tataru:MR2657818}.} $N = \mathbb{H}^2$ even in the absence of  equivariant symmetry, see \cite{Tao}, \cite{sterbenz:tataru:MR2657818} and \cite{krieger:schlag:MR2895939}.   
\item 
It is known that for targets which are not geodesically convex, 
e.g. $N = \mathbb{S}^2$, singularities may form, see \cite{rodnianski:sterbenz:MR2680419,raphael:rodnianski:MR2929728}.
\end{enumerate}
\end{remark}

\subsection{The equivariant Einstein-wave map problem}

Let $R_\gg$ denote the scalar curvature of the Lorentzian metric $\gg$ on $M$, and let $\kappa>0$ a constant. Let
$$
S_{\text{grav}} \fdg= \frac{1}{2 \kappa} \int_\MM R_\gg 
$$ 
denote the Einstein-Hilbert action, then
the Euler-Lagrange equation for an Einstein-wave map with action 
$$
S_{\text{grav}} + S_{\text{WM}} 
$$
consists of \eqref{eq:wm-eqn-general} coupled to the Einstein equation  
\begin{equation}\label{eq:einstein-general} 
G_{\mu\nu} = \kappa S_{\mu\nu} 
\end{equation} 
where $G_{\mu\nu} = R_{\mu\nu} - \frac{1}{2} R \gg_{\mu\nu}$ is the Einstein tensor for the metric $\gg$ and 
\begin{equation}\label{eq:stress-form} 
S_{\mu\nu} = \partial_\mu \Phi^A \partial_\nu \Phi^B h_{AB} - \frac{1}{2} \partial_\alpha \Phi^C \partial_\beta \Phi^D \gg^{\alpha\beta} h_{CD} \gg_{\mu\nu} 
\end{equation} 
is the energy-momentum tensor for the wave map. 

\begin{remark}\lab{remarkrelevanceU(1)}
As emphasized above, the main motivation for considering the Einstein-wave map problem is that the 3+1 dimensional vacuum Einstein equations with a spacelike translational Killing field reduces to a 2+1 dimensional Einstein-wave map system with target the hyperbolic plane $\mathbb{H}^2$, see \cite{moncrief:MR955865} and also \cite{andersson:1999gr.qc....11032A,moncrief:MR3248565} and references therein. 
\end{remark}

In this paper, we restrict ourselves to the equivariant class and recall some of the notations already introduced in Section \ref{sec:equivcrit}. 

\begin{definition}[Equivariant 2+1 dimensional Einstein-wave map] \label{def:1.3}
Let $(M, \gg)$ be a globally hyperbolic spacetime with an $\mathbb{S}^1$ action by isometries $e^{i\theta}$, with  hypersurface orthogonal generator $\partial_\theta$ which is spacelike away from fixed points. Let the metric $h$ on $N$ be of the form $h = d\rho^2 + g^2(\rho) d\theta^2$.  
Assume that $M$ has Cauchy surface diffeomorphic to $\Re^2$. 
Let $\Phi: M \to N$ be an equivariant map, with rotation number $k \in \mathbb{Z}$, i.e. 
$\Phi \circ e^{i\theta} = e^{ik\theta} \circ \Phi $,
and let $\phi = \rho \circ \Phi$. 

An equivariant critical Einstein-wave map spacetime with target $N$ is a triple $(M, \gg, \Phi)$ 
solving 
\begin{subequations}\label{eq:ein-equiv-wm-first} 
\begin{align} 
G_{\mu\nu} &= \kappa S_{\mu\nu} \label{eq:ein-wm-first} \\ 
\square_\gg \phi - \frac{k^2 f(\phi)}{r^2} &= 0  \label{eq:wave-wm-first} 
\end{align} 
\end{subequations} 
where $G_{\mu\nu} = R_{\mu\nu} - \half R \gg_{\mu\nu}$ is the Einstein tensor for the metric $\gg_{\mu\nu}$, $r$ is the radius function, and $f(\phi) = g(\phi) g'(\phi)$. 
\end{definition}

\begin{remark} \label{rem:1.4}
\begin{enumerate} 
\item 
We shall throughout the paper restrict to the case when the generator $\partial_\theta$ of the $S^1$ action on $\MM$ is hypersurface orthogonal. 
\item  
See Section \ref{sec:equivcrit} for the technical conditions on $g(\rho)$ which will be assumed to hold throughout the paper. 
\item \label{point:k}
In this work we shall assume $k=1$, however the arguments easily extend to the general case $k \in \mathbb{Z}$.
\end{enumerate} 
\end{remark} 

For a Cauchy surface $\Sigma$, let $T^\mu$ be the future directed unit normal. Denote also by $R$ the induced scalar curvature and $K_{ab}$ the second fundamental form of the embedding, defined by $K(X,Y)  = \gg(\nabla_X T, Y)$ for vector fields $X,Y$ tangent to $\Sigma$. 
It is well known that the Cauchy data for the Einstein equations \eqref{eq:ein-wm-first} must satisfy some compatibility conditions known as the constraint equations\footnote{They correspond respectively to $G_{TT}=\kappa S_{TT}$ and $G_{Ta}=\kappa S_{Ta}$.} 
\begin{subequations}\label{eq:sgwm-constraints}  
\begin{align} 
R + (K^c{}_c)^2 - K_{ab} K^{ab} &= 2\cp S_{\mu\nu} T^\mu T^\nu, \\
D^c K_{ac} - D_a K^c{}_c &= \cp S_{a\nu} T^\nu,  
\end{align} 
\end{subequations} 
 where $D_a$ is the intrinsic covariant derivative on $\Sigma$. 
 
A smooth metric $\gg$ satisfying the assumptions in Definition \ref{def:1.3} can be put in the form 
\begin{equation}\label{eq:trmetric} 
\gg = - e^{2\alpha(t,r)} dt^2 + e^{2\beta(t,r)} dr^2 + r^2 d\theta^2, 
\end{equation} 
see Section \ref{sec:trcoord} for further discussion. For a smooth solution of the 2+1 dimensional equivariant Einstein-wave map system, it must hold that $\alpha(t,r), \beta(t,r)$ are even functions of $r$. In order to avoid a conical singularity at the axis, it must hold that $\beta(t,0) = 0$. See eg. \cite{2008GReGr..40..159R} for discussion of these points. Further, $\alpha(t,0)$ is determined only up to a choice of time parametrization. We shall choose a time coordinate such that $\alpha(t,0) = 0$. 

Recall that we are restricting our considerations to the case of rotation number $k=1$, cf. Remark \ref{rem:1.4}. In this case, the map $\Phi$ must be odd under reflection. Hence in view of \eqref{eq:phidef}, $\phi(t,r)$ is odd, as a function of $r$. Thus, smoothness of $\gg, \Phi$,  together with a time parametrization such that $\alpha(t,0) = 0$ gives the leading order asymptotic behavior for $\alpha,\beta,\phi$, 
\begin{subequations}\label{eq:leadingorder}
\begin{align} 
\alpha(t,r) &= 
\alpha_2(t) r^2 + O(r^4) , \\
\beta(t,r) &= \beta_2(t)r^2 + O(r^4) , \\ 
\phi(t,r) &= \phi_1(t)r + O(r^3) . 
\end{align}
\end{subequations}
A calculation shows that with $\gg$ of the form  \eqref{eq:trmetric}, the second fundamental form $K$ of a Cauchy $t$-level set $\Sigma$ is of the form 
$$
K = K_{rr} dr^2 ,
$$ 
with 
$K_{rr} =  e^{-\alpha+2\beta} \partial_t \beta$. 

\begin{definition}[Cauchy data set for the 2+1 dimensional equivariant Einstein-wave map system]
A Cauchy data set for the 2+1 dimensional equivariant Einstein-wave map system with target $(N,h)$ is a 5-tuple $(\Sigma, q, K, \phi_0, \phi_1)$ consisting of  
\begin{enumerate} 
\item a Riemannian 2-manifold  $(\Sigma,q)$ with an isometric action by $e^{i\theta}$ as above and a 2-tensor $K$ of the form $K_{rr} dr^2$ symmetric under the same action, 
\item rotationally symmetric functions $\phi_0: \Sigma \to \Re_+$, $\phi_1: \Sigma \to \Re$,
\end{enumerate} 
such that the constraint equations \eqref{eq:sgwm-constraints} hold.   
\label{def:CDS-equivariant}
\end{definition} 

The proof by Choquet-Bruhat and Geroch \cite{ChB:Geroch:MR0250640} of existence and uniqueness of maximal solutions to the Cauchy problem for the vacuum Einstein equations, together with the equivariance of the Cauchy data, is readily generalized to give the following result. 
\begin{theorem}[Maximal Cauchy development for the $2+1$ equivariant Einstein-wave map problem] 
Let $(\Sigma, q, K, \phi_0, \phi_1)$ be an equivariant Cauchy data set for the $2+1$ Einstein-wave map system. Then there is a unique, maximal Cauchy development $(M, \gg, \Phi)$ satisfying the equivariant Einstein-wave-map system \eqref{eq:ein-equiv-wm-first}. 
\label{thm:cauchy-sgwm-equiv}
\end{theorem}

\subsection{Asymptotic flatness}  \label{sec:asympt-flat} 

Let $H^s_\delta$ be the weighted $L^2$ Sobolev spaces\footnote{We here use the same  conventions as Huneau \cite{huneau:2013arXiv1302.1473H}. In particular on $\Re^n$, $u = o(r^{-\delta-1})$ if $u \in H^s_\delta$ for $s > n/2$.}  on $\Re^2$. 
A 2-dimensional rotationally symmetric Cauchy data set $(\Sigma, q, K)$ is asymptotically flat if 
$$
q = e^{2\beta} dr^2 + r^2 d\theta^2 
$$
with $\beta = \beta_\infty + \tilde \beta$ and $(\tilde\beta, K) \in H^{s+1}_{\delta} \times H^{s}_{\delta+1}$ for some $\delta \in (-1,0)$. This is compatible with the setup in \cite{huneau:2013arXiv1302.1473H}, specialized to the rotationally symmetric case. Note that the existence of such asymptotically flat solutions to the constraint equations without rotational symmetry is proved in \cite{huneau:2013arXiv1302.1473H} \cite{Huneauconstatint2} (and used in \cite{huneau:2014arXiv1410.6068H} to prove stability in exponential time of the Minkowski space-time in this setting).

\subsection{Global existence conjecture}

A major open problem in the field of general relativity is given by the Cosmic Censorship conjectures formulated for large data solutions of the Einstein equations by Penrose in 1969 \cite{penrose:1969NCimR...1..252P} (republished as \cite{penrose:2002:golden-oldie}, see also the discussion in \cite{krolak:penrose-2002-note}), see for example \cite{LA:MR2098914} for a precise statement. 

We recall the formulation of one version of these conjectures. The Strong Cosmic Censorship conjecture, which is most relevant for our purposes states that the maximal Cauchy development of generic vacuum Cauchy data is inextendible, while the Weak Cosmic Censorship conjecture states that any singularity in a generic, asymptotically flat, vacuum spacetime is hidden from observers at future null infinity by an event horizon. 

These fundamental conjectures are still widely open in general, but have been proved in some cases when assuming certain symmetries, see in particular the seminal proof of Christodoulou of the Cosmic Censorship conjectures for the Einstein equations  coupled to a scalar field in spherical symmetry (see \cite{christodoulou:MR1680551} and references therein). An intermediate goal toward the general case would be to assume the presence of only one Killing field, and prove global regularity for the 3+1 vacuum Einstein equations with a spacelike translational Killing field.

\begin{conjecture}[Global existence for the 3+1 vacuum Einstein equations with a spacelike translational Killing field]
Maximal Cauchy developments of `asymptotically flat' solutions to the 3+1 vacuum Einstein equations with a spacelike translational Killing field are regular and geodesically complete. 
\label{conj:WCC}
\end{conjecture}

Recall from point \ref{point:generalwmlargedata} of remark \ref{rem:severalremarks}that large data global existence holds for the corresponding semilinear analog, namely the wave map problem with $\MM= \Re^{2+1}$ and target $N = \mathbb{H}^2$. A proof of Conjecture \ref{conj:WCC} would likely require a local existence result at the critical level which seems currently out of reach\footnote{Note for instance that in the absence of symmetry, the  resolution of the bounded $L^2$-curvature conjecture in \cite{klainerman:etal:2012arXiv1204.1772K} for the 3+1 Einstein vacuum equations provides a local existence result which is 1/2 derivative above the scaling.} for quasilinear wave equations in dimensions higher than 1+1. As a first step towards Conjecture \ref{conj:WCC}, we prove in this paper the special case of equivariant symmetry (see Remark \ref{remark:WCCequivariant} below).

\subsection{Large data global regularity for the equivariant Einstein-wave map problem}

We are now ready to state our main result. 
\begin{theorem}[Global regularity of equivariant Einstein-wave maps]\label{thm:main-first}
Let $(M, \gg, \Phi)$ be the maximal Cauchy development of a smooth, asymptotically flat, finite energy Cauchy data set with finite initial energy, for the $2+1$ equivariant Einstein-wave map problem \eqref{eq:ein-equiv-wm-first} with target $(N, h)$. Assume that the metric $h$ has the form 
$$
h = d\rho^2 + g^2(\rho) d\theta^2 
$$
for an odd function $g: \Re \to \Re$ with $g'(0) = 1$. Assume that $g$ satisfies 
\begin{equation} \label{eq:nospherecond-first} 
\int_0^s g(s') ds' \to \infty \quad \text{when }  s \to \infty
\end{equation} 
and the Grillakis condition \eqref{eq:grillakis-cond-first}. Then, $(M,\gg, \Phi)$ is smooth, and $(M, \gg)$ is asymptotically flat and geodesically complete. \end{theorem} 

\begin{remark}\lab{remark:WCCequivariant}
As mentioned in Remark \ref{remarkrelevanceU(1)}, an important motivation for studying the Einstein-wave map system arises from the fact that this system with target $N = \mathbb{H}^2$ arises naturally as the reduction of the 3+1 vacuum Einstein equations with a spacelike translational Killing field. In particular, Theorem \ref{thm:main-first} proves Conjecture \ref{conj:WCC} in the special case of equivariant symmetry and should be seen as the analog of the proof of the Cosmic Censorship conjectures in this setting.
\end{remark}

The proof of Theorem \ref{thm:main-first} follows, as in the semilinear analog of free wave maps on Minkowski space, from non-concentration of energy and small energy global existence. The proof of
non-concentration of energy and the initial framework of the global existence problem is contained 
in the PhD Thesis of the second author \cite{gudapati:2013arXiv1311.4495G}. Let us emphasize in particular the following
\begin{itemize}
\item The proof of non-concentration of energy relies on the vectorfield method. Unlike the semilinear case where one relies on vectorfields of Minkowski, the vectorfields we use here have to be carefully constructed and controlled. In particular, we exhibit a vectorfield\footnote{the analog of $\frac{\pr}{\pr t}$ in Minkowski.} $T$, which is not Killing but leads nevertheless to a conserved current.

\item The proof of small energy global existence relies in a fundamental way on the null structure of the equations written in null coordinates. Indeed, derivatives along outgoing null cones of $\phi$ as well as the metric coefficients behave better, while the null structure allows to integrate by parts derivatives along ingoing null cones such that the new terms generated behave better. 
\end{itemize}

\vspace{0.3cm}

The structure of the paper is as follows. In Section \ref{sec:prel}, we introduce null coordinates $(u,\ub)$ and a notion of mass. In Section \ref{sec:firstsing}, we prove the absence of trapped surfaces and that the first singularity, if it exists, must lie on the axis of symmetry. In Section \ref{sec:tr-coord}, we introduce $(t,r)$ coordinates. In Section \ref{sec:non-conc}, we prove the non-concentration of the energy. ln Section \ref{sec:proofmaintheorem}, we state a result on small energy global existence and use it in conjunction with non-concentration of energy to prove Theorem \ref{thm:main-first}. The rest of the paper is then devoted to the proof of small energy global existence. In Section \ref{sec:improved-uniform-phi} we derive a uniform weighted upper bound for $\phi$. In Section \ref{sec:improved-uniform-dphi}, we rely on the upper bound of Section \ref{sec:improved-uniform-phi} to derive a uniform upper bound for $\pr\phi$. Finally, we rely on the upper bound of Section \ref{sec:improved-uniform-dphi} to prove small energy global existence in Section \ref{sec:small-energy-global}.

\section{Null coordinates} \label{sec:prel}

We assume that all objects are smooth\footnote{Throughout the paper, we use the terms smooth and regular interchangeably.}, unless otherwise stated. In this section we introduce a null coordinate system and a notion of mass in 2+1 dimension. This setup will be used in the next section to prove that the first singularity for the critical, equivariant Einstein-wave map system occurs on the axis of rotation, i.e. the set $\{r=0\}$ which we denote by $\Gamma$.

\subsection{Existence of null coordinates}

Recall from the discussion in Section \ref{sec:equivcrit} that the orbit space $(\Orb, \check{\gg})$ is a 2-dimensional globally hyperbolic Lorentzian space and in particular conformally flat. Hence we may introduce a global null coordinate system with respect to which $\check{\gg}$ takes the form 
$$
\check{\gg} = - \Omega^2(u, \ub) du d\ub,
$$
which shows that $(M, \gg)$ admits a coordinate system $(u, \ub, \theta)$ such that $\gg$ takes the form 
$$
\gg = - \Omega^2 du d\ub + r^2 (u, \ub) d\theta^2 ,
$$
where now $d\theta^2$ is the line element on the $\mathbb{S}^1$ symmetry orbit. In view of \eqref{eq:leadingorder} we may, by redefining the coordinates $u, \ub$, without loss of generality assume that the conditions 
\begin{equation}\label{eq:axis-cond-uub} 
r=0,\,\, \pr_{\ub} r=\frac{1}{2},\,\, \pr_u r=-\frac{1}{2},\,\, u=\ub,\textrm{ and }\Omega=1,
\end{equation} 
are valid on the axis $\Gamma$. Also, the volume element is $\mu_{\gg} = \Om^2 r/2$ and the wave operator on symmetric functions (i.e. $\ptl_\theta \phi = 0$) is
\begin{equation}\label{eq:wave-null} 
\square_{\gg} \phi = - \frac{2}{\Om^2 r} \Big(\ptl_u (r \ptl_\ub \phi) + \ptl_\ub (r
\ptl_u \phi)\Big).
\end{equation}

\subsection{The energy-momentum tensor in null coordinates} \label{sec:stress-null}

The non-vanishing components of $S_{\alpha\beta}$ in the coordinate system $(u, \ub, \theta)$ are 
\begin{subequations} 
\begin{align} 
S_{uu} &= \ptl_u \phi \ptl_u \phi,  \label{eq:Suu} \\
S_{\ub\ub} &= \ptl_\ub \phi \ptl_\ub \phi,  \label{eq:Subub} \\
S_{u\ub} &= \frac{\Om^2}{4} \frac{g^2(\phi)}{r^2}, \label{eq:Suub} \\ 
S_{\theta\theta} &= \frac{r^2}{2} \left ( \frac{4}{\Om^2} \ptl_u \phi \ptl_\ub
\phi + \frac{g^2(\phi)}{r^2} \right ). \label{eq:Sthetatheta} 
\end{align} 
\end{subequations} 
The components $S_{u\theta}$, $S_{\ub\theta}$ vanish identically. 
The energy-momentum tensor satisfies the dominant energy condition since 
\begin{equation}\label{eq:DEC} 
S_{uu} \geq 0, \quad S_{u\ub} \geq 0, \quad S_{\ub\ub} \geq 0.
\end{equation}

\subsection{The Einstein equation}\label{sec:Einsteineqinnullcoordinates}

The non-vanishing components in the coordinate system $(u, \ub, \theta)$ of the Einstein tensor $G_{\mu \nu} = R_{\mu\nu} - \half R \gg_{\mu\nu}$ are 
\begin{subequations}\label{eq:Gcompos}
\begin{align} 
G_{uu} &= - \Om^2 r^{-1} \ptl_u (\Om^{-2} \ptl_u r), \label{eq:Guu} \\
G_{\ub\ub} &= - \Om^2 r^{-1} \ptl_\ub (\Om^{-2} \ptl_\ub r), \label{eq:Gvv} \\	
G_{u\ub} &=  r^{-1} \ptl_u \ptl_\ub r, \label{eq:Guv} \\
G_{\theta\theta} &= 4 r^2 \Om^{-4} ( \ptl_u \Om \ptl_\ub \Om 
- \Om \ptl_u \ptl_\ub \Om ). \label{eq:Gthth} 
\end{align} 
\end{subequations} 
The components $G_{u\theta}$, $G_{\ub\theta}$ vanish identically. 

We can now write the $u,\ub$ components of the 
Einstein equations $G_{\a\b} = \cp S_{\a\b}$ in the form 
\begin{subequations} \label{eq:ein-uv} 
\begin{align} 
\ptl_u (\Om^{-2} \ptl_u r) &= - \Om^{-2} r \cp S_{uu}, \label{eq:Tuu}\\
\ptl_\ub (\Om^{-2} \ptl_\ub r) &= - \Om^{-2} r \cp S_{\ub\ub}, \label{eq:Tvv} \\
\ptl_u \ptl_\ub r &= r \cp S_{u\ub}, \label{eq:Tuv} \\
\Om^{-2} (\ptl_u \Om \ptl_\ub \Om - \Om \ptl_u \ptl_\ub \Om ) &= \frac{1}{4} r^{-2}
\Om^2 \cp S_{\theta\theta}. \label{eq:Tthth} 
\end{align} 
\end{subequations} 
Here, the equation (\ref{eq:Tuv}) is special to $2+1$ dimensions.

\subsection{The mass} \label{sec:mass} 

Define the quantity 
\begin{equation}\label{eq:mdef} 
m = 1 + 4\Om^{-2} \ptl_u r \ptl_\ub r.
\end{equation} 
It follows from \eqref{eq:axis-cond-uub} that $m=0$ on $\Gamma$. 
We can write $m$ in the form $m = 1-\gg^{\alpha\beta} \partial_\alpha r \partial_\beta r$.  In view of the form  \eqref{eq:trmetric} for $\gg$, we have 
\begin{equation}\label{eq:m-beta}
m = 1-e^{-2\beta}. 
\end{equation}

\begin{remark} The quantity $m$ defined by \eqref{eq:mdef} 
has a form related to the Hawking mass in 3+1 dimensional
spherically symmetric gravity. In 3+1 dimensions and spherical symmetry, the
Hawking mass $m_H$ is given by  
$$
m_H = \frac{r}{2}(1 + 4\Om^{-2} \ptl_u r \ptl_\ub r) .
$$ 
In $2+1$ dimensions, the quantity $m$ defined in (\ref{eq:mdef}), when
evaluated at infinity, is a
function of the mass defined by Ashetkar and Varadarajan \cite{ashtekar:varadarajan}.
\end{remark}
 
\begin{lemma} \label{lem:2.2}
The quantity $m$ admits a limit along any a space like asymptotically flat curve, which does not depend on the particular curve. We denote this limit by $m_\infty$. We have furthermore
$$m_\infty\in [0,1).$$  
\end{lemma}

\begin{proof}
See \cite{huneau:2013arXiv1302.1473H} for the proof of this lemma, where $m_\infty$ is called the deficit angle.
\end{proof}

The mass $m$ satisfies the following equations which are analogous to the
ones satisfied by the Hawking mass in the $3+1$ dimensional case,
\begin{subequations}\label{eq:meqs} 
\begin{align}
\ptl_u m &= 4 \cp \Om^{-2} r ( S_{u\ub} \ptl_u r - S_{uu} \ptl_\ub r) \label{eq:mu} \\
\ptl_\ub m &= 4 \cp \Om^{-2} r (S_{u\ub} \ptl_\ub r - S_{\ub\ub} \ptl_u r) \label{eq:mv} 
\end{align} 
\end{subequations}

\section{The first singularity occurs on the axis}  \label{sec:firstsing}

\subsection{Absence of trapped surfaces}

Following Dafermos \cite{dafermos:trap}, 
we define the regions 
\begin{align*} 
\RR &= \{ p \in \Orb \text{ such that } \ptl_\ub r > 0, \quad \ptl_u r < 0 \}, \\
\TT &= \{ p \in \Orb \text{ such that } \ptl_\ub r < 0, \quad \ptl_u r < 0 \}, \\
\AA &= \{ p \in \Orb \text{ such that } \ptl_\ub r = 0, \quad \ptl_u r < 0 \}. 
\end{align*} 
Then $\RR, \TT, \AA$ are the non-trapped (or regular), 
trapped and marginally trapped
regions, respectively. Due to work of Ida \cite{ida:2000PhRvL..85.3758I}, one expects that in a 2+1
dimensional spacetime satisfying the dominant energy condition, trapped
or marginally trapped surfaces occur only in exceptional cases. In fact, as shown by Galloway et al. 
\cite{galloway:etal:2012CMaPh.310..285G} a 2+1 dimensional spacetime satisfying the dominant energy condition and a mild asymptotic condition, weaker than asymptotic flatness, cannot contain any marginally trapped surfaces. 
We give below a direct proof that in the case under consideration, there are no trapped or marginally trapped
surfaces. 

\begin{theorem}[Absence of trapped surfaces]\label{thm:regular}
We have 
\begin{enumerate} 
\item $\Orb = \RR$
\item For $q \in \Orb$, 
\begin{equation}\label{eq:mineq}
0 \leq m(q) <m_\infty<1 .
\end{equation}  
\end{enumerate} 
In particular, the spacetimes under consideration contain no
trapped or marginally trapped surfaces, i.e. $\TT = \emptyset$, $\AA =
\emptyset$. 
\end{theorem} 

\begin{proof}
Let $\check{\Sigma}$ be a Cauchy curve in $\Orb$. Note that each $p \in \Orb$ is on such a
Cauchy curve. Let $s$ be a coordinate on $\check{\Sigma}$ and let $x(s)$ be the point in $\check{\Sigma}$ with
coordinate value $s$.
We may without loss of generality assume $\check{\Sigma}$ has one endpoint $x(0)$ 
on $\Gamma$ corresponding to $s=0$ and an ``asymptotically flat'' end
corresponding to $s \to \infty$ so that the coordinate $s$ takes values in
$[0,\infty)$. As discussed in Section \ref{sec:mass}, we have $m(x(0)) = 0$. 

Now $V = \ptl_s$ is a vectorfield tangent to $\check{\Sigma}$ and in 
particular is spatial. Therefore, since $V$ points towards increasing values of $s$, $V = a\ptl_\ub - b\ptl_u$ for positive 
functions $a, b$. 
It follows from the assumption of asymptotic flatness that $x(s)$ is contained in $\RR$ for $s$ large enough. Due to the
dominant energy condition, see (\ref{eq:DEC}), and equations (\ref{eq:meqs}),
we have 
\begin{equation}\label{eq:Vm} 
V m \geq 0
\end{equation} 
in the regular region $\RR$. Now consider a point $q \in
\check{\Sigma} \cap \partial \RR$, where $\partial \RR$ denotes the boundary of $\RR$. 
At such a point, one of the equations $\ptl_u r = 0$ or
$\ptl_\ub r = 0$ holds, and hence $m(q) = 1$.  Due to asymptotic flatness, $\lim_{s \to \infty} m(x(s))
= m_\infty \in [0,1)$. Hence due to the monotonicity of $m$, see
(\ref{eq:Vm}), we get a contradiction from $m(q) = 1$. Therefore $\check{\Sigma} \cap \partial
\RR = \emptyset$. This argument also shows that $\check{\Sigma} \subset \RR$. Also, since $m(x(0)) = 0$ with $x(0) = \check{\Sigma} \cap \Gamma$, then we have $0\leq m<1$ on $\check{\Sigma}$.

The properties of the mass discussed above, together with the 
fact that each point of $\Orb$ is on a Cauchy curve, and the maximality of $\Orb$, allow us to conclude the proof of the theorem.
\end{proof}

\subsection{First singularities}  \label{sec:first-sing-subsection}

We now restrict our consideration to the future $\Orb^+$ of $\check{\Sigma}$. Due to
Theorem \ref{thm:regular}, $\Orb^+ = \Orb \cap J^+(\check{\Sigma})$. We now introduce some
notions following Dafermos \cite{dafermos:trap,dafermos:naked}. Recall that $\Orb^+$ is conformally isometric to a domain in the 2-dimensional Minkowski space $\Re^{1,1}$. As in \cite[section 1.1]{dafermos:naked}, we shall consider the topological closure $\overline{\Orb^+}$ of $\Orb^+$ in the topology of $\Re^{1,1}$ and define its boundary $\overline{\Orb^+} \setminus \Orb^+$ accordingly. 

\begin{definition}
Let $p \in \overline{\Orb^+}$. The indecomposable past subset $J^-(p) \cap
\Orb^+$  is said to be {\bf eventually compactly
generated} if there is a compact subset $\Xcal \subset \Orb^+$ such that 
\begin{equation}\label{eq:ecg}
J^-(p) \cap \Orb^+ \subset D^+(\Xcal) \cup J^-(\Xcal) .
\end{equation} 
We will say that in this situation $\Xcal$ generates $J^-(p) \cap \Orb^+$. 
\end{definition}
Here, for convenience, we have followed the usage of \cite{dafermos:naked} rather than the standard definition, cf. \cite[Section 6.4]{MR1384756}. In particular, in the following, we take an indecomposable past subset to be a set of the form $J^{-}(p) \cap \Orb^+$ for some $p \in \overline{\Orb^+}$. 

\begin{definition}
A point $p \in \overline{\Orb^+} \setminus \Orb^+$ is said to be a {\bf first
singularity} if $J^-(p) \cap \Orb^+$ is eventually compactly generated and if
any eventually compactly generated indecomposable past subset of $J^-(p) \cap
\Orb^+$ is of the form $J^-(q) \cap \Orb^+$ for some $q \in \Orb^+$. 
\end{definition}

We will now state an extension criterion, which is a direct consequence of
the well posedness of the characteristic initial value problem, see
\cite[Prop. 1.1]{dafermos:naked}. To state this we need to introduce for a
subset $Y \subset \Orb^+\setminus \Gamma$, the quantity $N(Y)$,
\begin{equation} \label{eq:NYdef} 
N(Y) = \sup_Y \{ |\Omega|_1, |\Omega^{-1}|_0, |r|_2, |r^{-1}|_0, |\phi|_1\}
\end{equation}
where $|f|_k = \max(|f|_{C^k(u)},|f|_{C^k(v)})$. 

We can now state the extension criterion
\begin{proposition}[\protect{\cite[Property 1.1]{dafermos:naked}}]\label{prop:cont} 
Let $p \in \overline{\Orb^+} \setminus \overline{\Gamma}$ be a first
singularity. Then, for any compact $\Xcal \subset \Orb^+ \setminus \Gamma$,
generating $J^-(p)$, i.e. which satisfies \eqref{eq:ecg}, we have 
$$
N(\DXp) = \infty.
$$
\end{proposition}  

The following theorem states that the first singularity occurs on the axis. 
\begin{theorem}[The first singularity occurs on the axis]\label{thm:singax} 
Let $p \in \overline{\Orb^+} \setminus \Orb^+$ be a first singularity. Then 
$p \in \overline{\Gamma}\setminus \Gamma$. 
\end{theorem} 

\begin{remark} 
Theorem \ref{thm:singax} 
should be compared to \cite[Theorem 3.1]{dafermos:naked}, which
states that a first singularity occurs either on the axis or has a trapped
surface in its past. 
\end{remark} 

\begin{proof}
Assume for a contradiction that $p \in \overline{\Orb^+} \setminus \overline \Gamma$. Let us introduce the notations 
\begin{align} 
\nu &\fdg= \ptl_u r, \\
\lambda &\fdg= \ptl_\ub r, \\
\zeta &\fdg= r \ptl_u \phi, \\ 
\vartheta &\fdg= r \ptl_\ub \phi, \label{eq:thetadef} \\
\varkappa &\fdg= - \frac{1}{4} \Om^2 \nu^{-1}. \label{eq:varkappadef} 
\end{align} 
In the present, 2+1 dimensional case, $m$ is given by \eqref{eq:mdef}, which using the above notation takes the form 
$$
m = 1 + 4 \Omega^{-2} \nu\lambda.
$$
Note that we have by Theorem \ref{thm:regular} $m < 1$ and also, 
since $\Orb^+ \subset \RR$, it holds that $\nu < 0$, $\lambda > 0$, $\varkappa > 0$  
everywhere in $\Orb^+$. Further,
note that from the definitions $r > 0$ in $\Orb^+ \setminus \Gamma$. We may
assume without loss of generality that $\Xcal \subset \Orb^+ \setminus \Gamma$. 
If $p = (u_s, \ub_s)$ denotes first singularity, we may further assume that $\Xcal$ is given by
$$
\Xcal = \Big(\{u_0\} \times [\ub_0 , \ub_s]\Big) \cup \Big([u_0, u_s] \times \{ \ub_0\}\Big) 
$$
where $u_0<u_s$, $\ub_0<\ub_s$ and $u_s<\ub_0$ to ensure $\Xcal \subset \Orb^+ \setminus \Gamma$. Note that we have 
$$
[u_0, u_s] \times [\ub_0, \ub_s] = D^+(\Xcal) = J^-(p) \cap D^+ (\Xcal).
$$
 
In view of the compactness of $\Xcal$ the following bounds hold on $\Xcal$,
\begin{subequations} \label{eq:Xbounds} 
\begin{align} 
0 &< r_0 \leq r \leq R , \label{eq:rineq} \\
0 &< 
\lambda \leq \Lambda, 
\qquad  
0 > 
\nu \geq -N , \label{eq:lambdanuineq} \\ 
|\phi| &\leq P, \nonumber \qquad 
|\vartheta| \leq \Theta, \nonumber \qquad 
|\zeta| \leq Z ,  \nonumber \\
0 & < 
\varkappa \leq K , \label{eq:kappaineq} \\
|\ptl_u \Omega| &\leq H,  \nonumber \qquad 
|\ptl_\ub \Omega| \leq H,  \nonumber \\
|\ptl_u \nu| &\leq H,  \nonumber \qquad 
|\ptl_\ub \lambda| \leq H, \nonumber 
\end{align} 
\end{subequations} 
for some positive real numbers $r_0, R, \Lambda, N, P, \Theta, Z, K, H$. 
Equation (\ref{eq:Tuu}) yields 
\begin{equation}\label{eq:kappaeq}
\ptl_u \varkappa = \kappa \frac{1}{r} \left ( \frac{\zeta}{\nu} \right)^2 \nu\varkappa .
\end{equation} 
Due to (\ref{eq:kappaeq}) and $\nu < 0$, it follows that 
inequality
(\ref{eq:kappaineq}) holds in all of  
$\DXp$. Since $\nu < 0$,
$\lambda > 0$, it follows that inequality (\ref{eq:rineq}) holds throughout
$\DXp$. 

Now consider $p^* = (u^*, \ub^*) \in \DXp$. The past null curves
starting at $p^*$ intersect $\Xcal$ at $(u^*,\ub_0)$ and $(u_0,\ub^*)$, respectively, see figure \ref{fig:Xset}. 
\begin{figure}
\centering
\psfrag{MApoint}{$p$}
\psfrag{MApstar}{$p^*$}
\psfrag{MAusube}{$(u^*, \ub_0)$}
\psfrag{MAueubs}{$(u_0,\ub^*)$}
\psfrag{MAXset}{$\Xcal$}
\includegraphics[width=3in]{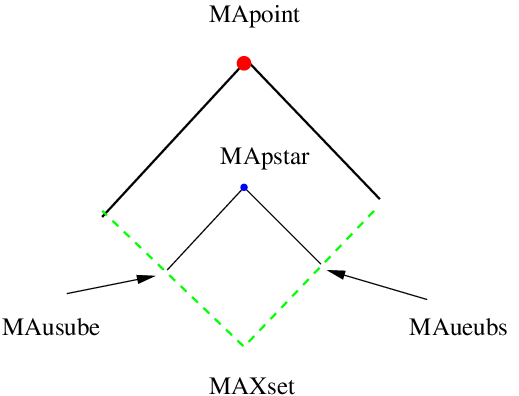}
\caption{$D^+(\Xcal)$}
\label{fig:Xset}
\end{figure}
Integrating (\ref{eq:thetadef}) yields 
\begin{align} 
|(\phi(u^*,\ub^*)| &\leq |\phi(u^*,\ub_0)| + \left| \int_{\ub_0}^{\ub^*}
\frac{\vartheta}{r}(u^*,\ub) d\ub  \right| \nonumber \\
&\leq P + \sqrt{\int_{\ub_0}^{\ub^*} \frac{\vartheta^2}{r\varkappa}  d\ub }\,\,
\sqrt{\int_{\ub_0}^{\ub^*} \frac{\varkappa}{r} d\ub }. \label{eq:phi-bound-firstsing} 
\end{align}
Equation \eqref{eq:mv} gives, using the present notation 
$$
\partial_\ub m = \kappa \left ( \frac{g^2(\phi)}{r} \lambda + \frac{\vartheta^2}{r\varkappa}  \right ). 
$$
In view of $\lambda > 0$, this gives, integrating along the same null curve as above,  
\begin{align} 
\int_{\ub_0}^{\ub^*}  \frac{\vartheta^2}{r\varkappa} d\ub &\leq 
\frac{1}{\kappa} \left ( m(u^*, \ub^*) - m(u^*, \ub_0) \right ) \nonumber \\
&\leq \frac{1}{\kappa} \label{eq:vartheta-bound} 
\end{align} 
where we used \eqref{eq:mineq}. We can now use the inequality \eqref{eq:vartheta-bound} together with the previous estimates of $\varkappa$ and $r$ and \eqref{eq:phi-bound-firstsing} to show that $\phi$ is uniformly bounded in $D^+(\Xcal) \setminus \{p\}$.

We next estimate $\lambda$ and $\nu$. First, use the relation 
$$
\varkappa ( 1 - m) = \lambda
$$
and the previous estimates for $m, \varkappa$, to get the inequality $0 < \lambda < K$ on $\DXp$. In order to estimate $\nu$, recall that $\nu < 0$ on $\Orb$ by theorem \ref{thm:regular}. Next, note that in view of \eqref{eq:Tuv} and \eqref{eq:DEC} we have 
$\partial_\ub \nu > 0$ and hence integrating as above gives  
$$
\nu(u^*, \ub_0) < \nu(u^*, \ub^*) < 0 .
$$
This means that the inequalities \eqref{eq:lambdanuineq} hold on $\DXp$. 

From the definition of $\varkappa$, cf. \eqref{eq:varkappadef}, we have 
$$
\Omega^2 = - 4 \nu \varkappa
$$
which in view of the above estimates gives  
\begin{equation}\label{eq:OmNK} 
\Omega^2 \leq 4 NK \quad \text{on $\DXp$.}
\end{equation}

To estimate the first derivative of $\phi$, we write 
\eqref{eq:ein-equiv-wm-first}
in
the form 
\begin{align} 
\ptl_u \vartheta &= \half r^{-1} \nu \vartheta  
- \half r^{-1} \lambda \zeta + \varkappa \nu \frac{f(\phi)}{r}   \\
\ptl_\ub \zeta &= \half r^{-1} \lambda \zeta - \half r^{-1} \nu \vartheta 
+ \varkappa \nu \frac{f(\phi)}{r} 
\end{align} 
Integrating these relations as above yields uniform bounds for $\vartheta,\zeta$ in
$\DXp$.

Next, observe that \eqref{eq:Tthth} takes the form 
\begin{equation}\label{eq:dudublnOm}
- \partial_u \partial_\ub \log(\Omega) = \frac{1}{8} r^{-2} \kappa \left ( 4 \vartheta \zeta + \Omega^2 g^2(\phi) \right ) 
\end{equation}
in the current notation. 
The right hand side of \eqref{eq:dudublnOm} is uniformly bounded on $\DXp$ by the above estimates. 
Integrating as above along the null curves $\{ (u, \ub^*), u_0 \leq u \leq u^*\}$ and 
$\{(u^*, \ub), \ub_0 < \ub < \ub^* \}$ yields uniform bounds on $\partial_u \log(\Omega)$ and $\partial_\ub \Om$ on $\DXp$, and hence in view of \eqref{eq:OmNK} also on $\partial_u \Om$ and $\partial_\ub \Om$. A second integration of $\partial_u \log(\Omega)$ or $\partial_\ub \log(\Omega)$ allows us to give a uniform bound on $|\log(\Omega)|$, and hence also on $|\Omega^{-1}|$, in $\DXp$. 

Now we have uniform bounds in $\DXp$ for the quantities 
$|r^{-1}|, |\Omega^{-1}|$, $|\ptl_u r|$, $|\ptl_\ub r|$, $|\phi|$, $|\ptl_u
\phi|$, $|\ptl_\ub \phi|$, $|\partial_u \Omega|$, $|\partial_\ub \Omega|$. A bound on 
$|\ptl_u \ptl_\ub r|$ 
follows in view of these estimates directly from (\ref{eq:Tuv}). It remains
only to estimate $\ptl_u \ptl_u r = \ptl_u \nu$ and $\ptl_\ub \ptl_\ub r = \ptl_\ub
\lambda$.  In order to do this, we can use equations (\ref{eq:Tuu}) and (\ref{eq:Tvv}) 
since all occuring
terms are bounded by our previous estimates. 

This completes the proof that if $p$ is a first singularity in
$\overline{\Orb^+}\setminus \overline{\Gamma}$, we have 
$N(\DXp) < \infty$ which by Proposition \ref{prop:cont} gives a contradiction. This shows that every first
singularity occurs in $\overline{\Gamma} \setminus \Gamma$, i.e. on the
axis, and hence concludes the proof of Theorem \ref{thm:singax}. 
\end{proof} 

\section{$(t,r)$ coordinates} \label{sec:tr-coord}

\subsection{Construction of $(t,r)$ coordinates} \label{sec:trcoord}

Let $(M, \gg, \phi)$ be the maximal Cauchy development of an asymptotically flat Cauchy data set for the equivariant Einstein-wave map problem. 
Let $\Gamma = \{r=0\}$ be the axis of rotation in $M$. If $\Gamma$ is incomplete to the future, we let $p_\Gamma$ be the first singularity. 

\newcommand{\cSigma}{\check{\Sigma}}
\newcommand{\cgg}{\check{\gg}}
\newcommand{\cq}{\check{q}}

\begin{lemma}
Let $t$ be the parameter  on $\Gamma$ such that $\dot \Gamma = d\Gamma/dt$ satisfies 
$$\gg_{\alpha\beta}\dot \Gamma^\alpha \dot \Gamma^\beta = -1\textrm{ for }t<0\textrm{ and }\lim_{t\nearrow 0} \Gamma(t) = p_\Gamma.$$
Extend $t$ to be constant on the maximal orbit $\cSigma_t$ of $\nabla r$ starting at $\Gamma(t) \in \Gamma \cap \RR$. Then,  $(t,r)$ is a regular coordinate system on $\cup_{t<0}\cSigma_t$ and 
$$
\cgg = -e^{2\alpha} dt^2 + e^{2\beta} dr^2 
$$
for some functions $\alpha = \alpha(t,r)$, $\beta = \beta(t,r)$. Furthermore, we have  
$$\a=\b=0\textrm{ on }\Gamma.$$ 
\end{lemma}

\begin{proof}
Recall that the radius function $r$ is well-defined and smooth on the regular part of $M$ and hence also on $\Orb$. Let $\nabla r$ be the gradient field of $r$ on $(\Orb, \check{\gg})$. We have 
$$
\nabla r = -2\Omega^{-2} (\partial_u r \partial_\ub + \partial_\ub r \partial_u) 
$$
and 
$$
\cgg(\nabla r , \nabla r) = -4 \Omega^{-2} \partial_u r \partial_\ub r.
$$
This means that 
$$
\cgg(\nabla r , \nabla r) = 1 - m
$$
where $m$ is the mass as defined in Section \ref{sec:mass}.  In view of \eqref{eq:mineq}, we have that $m \in [0,1)$ in $\Orb$. Thus we have $\check{\gg}(\nabla r, \nabla r ) > 0$ on $\Orb$.  

Consider a maximal orbit $\cSigma_t$ of $\nabla r$ starting at some point $\Gamma(t) \in \Gamma \cap \RR$. Since $\check{\gg}(\nabla r, \nabla r ) > 0$ on $\Orb$, the radius function $r$ is a parametrization of $\cSigma$. By Cauchy stability for the ODE 
\begin{equation}\label{eq:r-ODE}
\frac{dx}{dr} = \nabla r, 
\end{equation} 
we have that $\cSigma_t$ defines a foliation in $\Orb$. This foliation does not cover all of $\Orb$, but the domain of the foliation includes the past domain of influence of the first singularity, cf. figure \ref{fig:trfoli}. Let 
\newcommand{\tOrb}{\tilde \Orb} 
$\tOrb$ denote the domain of the foliation. 

\begin{figure}
\centering
\psfrag{MAfinalpt}{$p_\Gamma$} 
\psfrag{MAcSt}{$\cSigma_t$}
\includegraphics[width=3in]{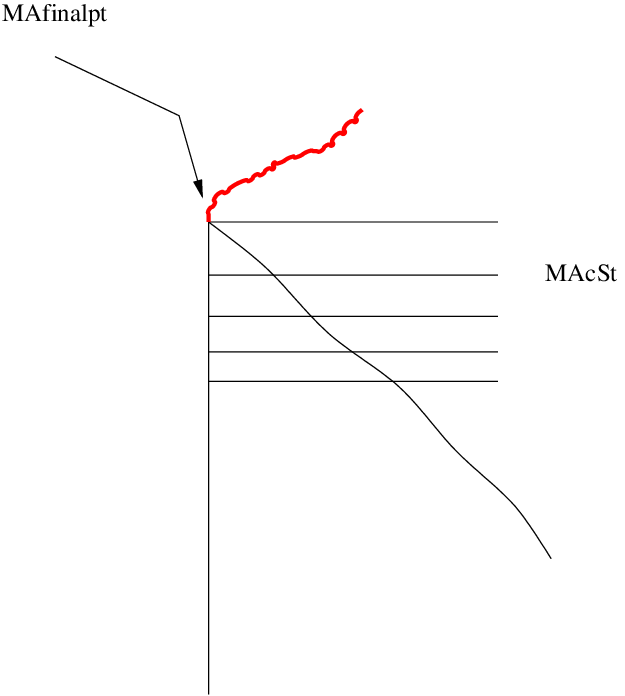}
\caption{The $\cSigma_t$ foliation}  
\label{fig:trfoli} 
\end{figure}

We can now extend the coordinate $t$ from the axis $\Gamma$ to the domain of the $\cSigma_t$ foliation. This defines a function $t$ on $\tOrb$. Recall that the $\cSigma_t$ are orbits of a vector field $\nabla r$ on $\Orb$. By uniqueness for \eqref{eq:r-ODE} we have that the function $t$ has non-vanishing gradient. Furthermore, we have by construction that $\cgg(\nabla t, \nabla r)=0$ on the domain of the time foliation. Together with the fact that $t$ has non-vanishing gradient and $\check{\gg}(\nabla r, \nabla r ) > 0$ on $\Orb$, we infer
$$\cgg(\nabla t, \nabla r)=0,\,\, \cgg(\nabla t, \nabla t)<0,\,\, \cgg(\nabla r, \nabla r)>0\textrm{ on }\cup_{t<0}\cSigma_t.$$
It follows that $(t,r)$ as coordinate functions on the domain of the foliation, and in this coordinate system, we have
$$
\cgg = -e^{2\alpha} dt^2 + e^{2\beta} dr^2 
$$
where $\alpha = -\log(-\cgg(\nabla t, \nabla t))/2$ and $\beta(t,r) = -\log(\cgg(\nabla r, \nabla r))/2$.  Furthermore, in view of our choice for $t$ in $\Gamma$ and the fact that $\cgg(\nabla t, \nabla r)=0$, we have $\cgg(\nabla t, \nabla t)=-1$ and hence also  $\alpha = 0$ on $\Gamma$. Further, as discussed in connection with \eqref{eq:leadingorder}, smoothness of $\gg$ implies $\beta = 0$ on $\Gamma$, in particular it holds that 
$\cgg(\nabla r, \nabla r)=1$ on $\Gamma$.
This concludes the proof of the lemma.  
\end{proof}

\newcommand{\tM}{\tilde M}
The above construction lifts to $(M, \gg)$ to give a foliation $\Sigma_t$. We denote the domain of this foliation $\tM$. On $\tM$ we have coordinates $(x^\alpha) = (t,r,\theta)$, and the metric $\gg$ takes the form 
\begin{equation}\label{eq:gg-tr-form} 
\gg = - e^{2\alpha} dt^2 + e^{2\beta} dr^2 + r^2 d\theta^2. 
\end{equation} 

\begin{remark} By \eqref{eq:m-beta} and the fact that $m$ is monotone increasing as a function of $r$ and satisfies $m < 1$ by Lemma \ref{lem:2.2}, we have that $\beta$ is monotone increasing
and bounded. See section \ref{sec:non-conc} below. 
\end{remark}

\subsection{Einstein Tensor}

The components  in the polar coordinates $(t,r,\theta)$ of the Einstein tensor  $G_{\mu \nu }= R_{\mu \nu} -\halb R \gg_{\mu \nu}$ are
\begin{align*}
 G_{tt} &= e^{2(\alpha -\beta)} \beta_r r^{-1}, \\
G_{tr} &= \beta_t r^{-1}, \\
G_{rr} &= \alpha_r r^{-1}, \\
G_{\theta \theta} &= r^2 \bigl(e^{-2\beta} (-\beta_r\alpha_r + \alpha^2_r + \alpha_{rr}) -e^{-2\alpha} (\beta^2_t -\beta_t\alpha_t + \beta_{tt}) \bigr), \\
G_{t \theta} & = 0, \\
G_{r \theta}&=0.
\end{align*}

\subsection{Energy-momentum Tensor}

Recall that the energy-momentum tensor $ S (\Phi)$ for a wave map $\Phi : (M,\gg) \to (N,h)$ is as follows
\begin{align}
 S_{\mu \nu} (\Phi)\fdg =  \langle \ptl_\mu \Phi , \ptl_\nu \Phi \rangle_{h(\Phi)} -\halb g_{\mu \nu} \langle \ptl^\sigma \Phi ,\ptl_\sigma \Phi \rangle_{h(\Phi)},
\end{align}
where $\mu$,$\nu$,$\sigma = 0,1,2$. 
In the following we will calculate each of the components of the energy-momentum tensor in $(t,r,\theta)$ coordinates.
Note,
\begin{align}
 \langle \ptl^\sigma \Phi ,\ptl_\sigma \Phi \rangle_{h(\Phi)} = -e^{-2\alpha}\phi_t^2 + e^{-2\beta}\phi_r^2 + \frac{g^2(\phi)}{r^2}.
\end{align}
Now we proceed to calculate $ S_{\mu \nu}$ 
\begin{align*}
 S_{tt} 
&= \halb e^{2\alpha}\left(e^{-2\alpha}\phi_t^2 +e^{-2\beta}\phi_r^2 + \frac{g^2(\phi)}{r^2} \right) , \\
S_{tr} 
&= \phi_t \phi_r, \\
S_{rr} 
&= \halb e^{2\beta}\left(e^{-2\alpha}\phi_t^2 + e^{-2\beta}\phi_r^2 - \frac{g^2(\phi)}{r^2}\right), \\
S_{\theta \theta} 
&=\halb r^2 \left(e^{-2\alpha} \phi_t^2 - e^{-2\beta}\phi_r^2 + \frac{g^2(\phi)}{r^2}\right), \\
S_{t \theta} &= 0,\\
S_{r \theta}&=0.
\end{align*}
 
Let  $T$ and $R$ be the normalization of $\pr_t$ and $\pr_r$ 
$$T \fdg = e^{-\alpha}\ptl_t\textrm{ and }R \fdg = e^{-\beta}\ptl_r .$$
We define the energy density $\mathbf{e} \fdg = S (T,T) $ and momentum density 
$\mathbf{m} \fdg = S(T, R)$
\begin{align*}
\mathbf{e}  & = \halb \left( e^{-2\alpha} \, \phi_t^2 + e^{-2\beta} \, \phi_r^2 + \frac{g^2(\phi)}{r^2} \right) \\
&= \halb \left( (T(\phi))^2 + (R(\phi))^2 + \frac{g^2(\phi)}{r^2} \right) \\
\mathbf{m} & = e^{-(\alpha + \beta)} \phi_t \, \phi_r  \\
 &= T(\phi) \, R(\phi).
\end{align*}
We further define 
$$\mathbf{e_0} \fdg = \left( T(\phi) \right)^2 + \left(R(\phi)\right)^2,\,\,\,\, \mathbf{f} \fdg= \frac{g^2(\phi)}{r^2}.$$

\subsection{Einstein equivariant wave map system of equations}
Using the above expressions for $G_{\mu \nu}$ and $S_{\mu \nu}$ we have the system of equations
\begin{subequations}\label{ewmequations}
\begin{align}
\beta_r & = \halb \, r \, \kappa \, e^{2\beta} \left( e^{-2\alpha} \, \phi_t^2 + e^{-2\beta}\, \phi_r^2 + \frac{g^2(\phi)}{r^2} \right)\label{gamma_r}, \\
\beta_t & = r \, \kappa \, \phi_t \,  \phi_r \label{gamma_t}, \\
\alpha_r & = \halb \, r \, \kappa \, e^{2\beta}   \left( e^{-2\alpha} \, \phi_t^2 + e^{-2\beta} \, \phi_r^2 - \frac{g^2(\phi)}{r^2} \right) \label{omega_r}, \\
\square_\gg \phi & = \frac{g'(\phi)g(\phi)}{r^2}, \label{wmequi}
 \end{align}
\end{subequations}
where
 \[\square_\gg \phi = -e^{-2\alpha}(\phi_{tt} + (\beta_t-\alpha_t)\phi_t) + e^{-2 \beta}\left(\phi_{rr} + \frac{\phi_r}{r} + (\alpha_r - \beta_r)\phi_r\right). \]

We remark that the full system \eqref{eq:ein-equiv-wm-first} yields some redundant equations. The system \eqref{ewmequations} is a subset containing the equations which are relevant for our purposes. 

\section{Non-concentration of energy} \label{sec:non-conc}

Let us define the energy on a Cauchy surface $\Sigma_t$
\begin{align*}
E(\Phi)(t)  \fdg =& \int_{\Sigma_t} \mathbf{e}\, \bar{\mu}_q  \\
=& 2\pi \int^{\infty}_{0} \mathbf{e}(t,r) re^{\beta(t,r)} \, d\, r\,, 
\intertext{the energy in a coordinate ball $B_r$} 
 E(\Phi)(t,r) \fdg =& \int_{B_r} \mathbf{e}\, \bar{\mu}_q\, , \\
 =& 2\pi \int^r_{0} \mathbf{e}(t,r') r'e^{\beta(t,r')} \, d\, r' 
\intertext{ the energy inside the causal past $J^-(O)$ of $O$} 
E^O(t) \fdg =& \int_{\Sigma_t \cap J^-(O)} \mathbf{e}\, \bar{\mu}_q \,.
\end{align*}
The goal of this section is to prove the following result. 
\begin{thm}[Non-concentration of energy] \label{theorem_ener_nonconc}
Let $(M,\gg,\Phi)$ be a smooth, globally hyperbolic, equivariant maximal development of smooth, asymptotically flat, equivariant 
initial data set $(\Sigma,q,K, \phi_0, \phi_1)$, satisfying the constraint equations \eqref{eq:sgwm-constraints}, with finite initial energy,
and let $(N,h)$ be a rotationally symmetric, complete, connected Riemannian manifold satisfying the Grillakis condition \eqref{eq:grillakis-cond-first} as well as \eqref{eq:nospherecond-first}. Then, the energy of the Einstein-wave map system \eqref{eq:ein-equiv-wm-first}  cannot concentrate,
i.e., $E^O (t) \to 0$, as $t \nearrow 0$, where $O$ is the first (hypothetical) singularity of $M$ and $t$ is as in Section \ref{sec:tr-coord}.
\end{thm}

\subsection{Energy conservation}

We start by proving the energy is conserved.
\begin{lem} \label{energy_cons}
The energy $E(\phi)(t)$ is conserved.
\end{lem}
\begin{proof}
 Consider two Cauchy surfaces $\Sigma_{s}$ and $\Sigma_{\uptau}$ at $t=s$ and $t=\uptau$ respectively, with $-1\leq \uptau \leq s < 0$. 
We shall now construct a divergence free vector field $P_{T}$ as follows.
Consider the Einstein's equations \eqref{gamma_r} and \eqref{gamma_t}. They can be rewritten as follows
\begin{align*}
 -\ptl_r\left(e^{-\beta} \right) &=  r\,\kappa  e^{\beta} \mathbf{e}, \\
 -\ptl_t \left(e^{-\beta} \right) &=  r\,\kappa  e^{\alpha} \mathbf{m}.
\end{align*}
From the smoothness of $\beta$ we have $ -\ptl^2_{rt}\left(e^{-\beta} \right)= -\ptl^2_{tr}\left(e^{-\beta} \right)$, which implies
\begin{align}\label{commute_partials}
 - \ptl_t\left(re^{\beta} \mathbf{e}\right) + \ptl_r(re^{\alpha}\mathbf{m}) =0.
\end{align}
Now define a vectorfield 
\[ P_{T} \fdg= -e^{-\alpha}\,\mathbf{e}\, \ptl_t + e^{-\beta}\, \mathbf{m}\, \ptl_r, \]
then the divergence of $P_{T}$ is given by
\begin{eqnarray}\label{eq:computationdivergencePT}
\nabla_\nu  P_T ^\nu &=& \frac{1}{\sqrt{|\gg|}}\, \ptl_\nu \left(\sqrt{|\gg|}\,P_T ^\nu \right) \\
\nn&=& \frac{1}{re^{\beta + \alpha}} \left(-\ptl_t \left( re^{\beta}\,\mathbf{e}\right)+ \ptl_r \left( re^{\alpha}\,\mathbf{m}\,\right) \right) \\
\nn&=& 0
\end{eqnarray}
from \eqref{commute_partials}.
Let us now apply Stokes' theorem in the space-time region whose boundary is $\Sigma_{s} \cup \Sigma_{\uptau}$. Due to asymptotic flatness, the boundary terms at $r=\infty$ do not contribute. We have
\begin{align}
0= \int_{\Sigma_s} e^{\alpha} P^t_T \bar{\mu}_q-
\int_{\Sigma_{\uptau}} e^{\alpha} P^t_T \bar{\mu}_q.
\end{align}
Therefore, it follows that 
\begin{align}
 E(\phi)(\uptau)= E (\phi)(s)
\end{align}
for any $\uptau$, $s$ such that  $-1\leq \uptau \leq s < 0$.
\end{proof}
In the following lemma we shall prove that the metric functions $\beta(t,r)$ and $\alpha(t,r)$ are uniformly bounded during the evolution of 
the Einstein-wave map system.
\begin {lem} \label{metric_uniform}
There exist constants $c^-_\beta,c^+_\beta,c^-_\alpha,c^+_\alpha$  depending only on the
initial data and the universal constants such that the following uniform bounds on the metric functions $\beta(t,r)$ and $\alpha(t,r)$ hold
\[c^-_\beta \leq \beta(t,r) \leq c^+_\beta,\]
\[c^-_\alpha \leq \alpha(t,r) \leq c^+_\alpha.\]
 \end {lem}
\begin{proof}
For simplicity of notation, we use a generic constant $c$ for the estimates on $\beta(t,r)$ and $\alpha(t,r)$.
The Einstein equation \eqref{gamma_r} for $\beta_r$ can be rewritten as 
\[ - (e^{-\beta} )_r  = \kappa\, r \,e^\beta \mathbf{e}.  \]
Integrating with respect to $r$ and recalling that $\beta_{|_\Gamma}=0$, we get
\[
 1-e^{-\beta} =  \kappa \int^r_0 \mathbf{e}\, r'\, e^{\beta} d\,r' 
 = \frac{\kappa}{2\pi}  E(\phi)(t,r)
\]
so,
\[e^{\beta} = \left( 1- \frac{\kappa}{2\pi}  E(\phi)(t,r)\right)^{-1}. \]
Let us introduce the notation $\beta_{\infty}(t)  = \lim_{r \to \infty } \beta (r,t)$. Then we have
\[ e^{\beta_\infty (t)} = \left( 1- \frac{\kappa}{2\pi}  E(\phi)(t)\right)^{-1}. \]
Since $E(\phi)(t,r)$ is a nondecreasing function of $r$, then so is $\beta(t,r)$
\[1 = e^{\beta(t,0)} \leq e^{\beta(t,r)} \leq e^{\beta_\infty(t)}. \]
Furthermore, since the energy is conserved $ E(\phi)(t) = E(\phi)(-1)$, 
$\beta_\infty (t) = \beta_\infty (-1)$ is also conserved during the evolution of the Einstein
wave map system and hence
\[0\leq \beta(t,r) \leq \beta_\infty(-1).\]
Similarly let us consider the Einstein's equation \eqref{omega_r} for $\alpha_r$
\[ \alpha_r  = r \, \kappa \, e^{2\beta} (\mathbf{e}-\mathbf{f}).\]
Integrating with respect to $r$ and recalling that $\alpha_{|_\Gamma}=0$, we get
\begin{align*}
 \alpha(t,r) & \leq c \int^r_0 (\mathbf{e}-\mathbf{f}) re^{\beta} d\, r \\
& \leq c \int^r_0 \mathbf{e}\,r\,e^{\beta} d\, r \\
& \leq c
\end{align*}
and
\begin{align*}
\alpha(t,r) & \geq - c \int^r_0 \frac{\mathbf{f}}{2} r e^{\beta} \, d\, r \\
& \geq - c \int^r_0 \mathbf{e}\, r\, e^{\beta} \, d\,r \\
& \geq - c.
\end{align*}
This concludes the proof of the lemma.
\end{proof}

\begin{remark} By Lemma \ref{lem:2.2} and \eqref{eq:m-beta} we have   
$$
0 \leq \frac{\kappa}{2\pi} E(\phi)(t) = 1-e^{-\beta_\infty(t)} < 1 .
$$
\end{remark}

\begin{lem} \label{lem:phiLinfty}
Assume that the target manifold $(N,h)$ satisfies
\begin{align}\label{nosphereinfinity}
\wp(\phi) := \int^\phi_0 g(s)\,ds \to \infty \,\,\, \text{as}\,\,\, \phi \to \infty,
\end{align}
 then there exists a constant $c$ dependent only on the initial
data and the universal constants such that
\[ \phi \in L^{\infty} \text{ with } \Vert \phi \Vert_{\infty} \leq c \] for every solution
$\phi$ of the equivariant wave map equation.
 \end{lem}
 
\begin{proof}
Extending the technique used in Lemma 8.1 in \cite{shatah_struwe} and noting that $\phi \big\vert_{\Gamma} =0$, we consider
\begin{align*}
\wp(\phi(t,r)) & = \int_0^r \ptl_r (\wp(\phi(t,r))) \, dr  \\
& = \int_0 ^r g(\phi) \ptl_r \phi \,dr \\
& = \int_0 ^r \left( g(\phi) (re^{-\beta})^{-1/2} \right) \left( \ptl_r \phi (re^{-\beta})^{1/2} \right) \, dr. 
\intertext{Consequently,}
|\wp(\phi(t,r))| & \leq \left( \int^r_0 (g(\phi))^2(re^{-\beta})^{-1}\, dr\right)^{1/2} \left(\int^r_0 (\ptl_r \phi)^2 re^{-\beta}\, dr \right)^{1/2} \\
& \leq \left( \int^\infty_0 \frac{g(\phi)^2}{r^2}re^{\beta}\, dr\right)^{1/2} \left(\int^\infty_0 e^{-2\beta}(\ptl_r \phi)^2 re^{\beta} \, dr\right)^{1/2} \\
& \leq c(E_0). 
 \end{align*}
Arguing via contradiction, the result follows.
\end{proof}

\subsection{The vectorfield method}

Let $X$ be a space-time vectorfield. The corresponding momentum $P_X$ is
given by the contraction of $S$ with $X$ i.e.,
\begin{align}\label{mom_def}
P_X^\mu = S^\mu_{\,\,\,\nu} X^{\nu}.
\end{align}
We have,
\begin{align}
\grad_\nu P_X^\nu =& \,X^\mu\, \grad_\nu S^\nu_{\,\,\mu} + S^\nu_{\,\,\mu}\,\grad_\nu X^\mu. \label{se_divfree}
\intertext{Since the energy-momentum tensor $S$ is divergence free, the first term in the right hand side of $\eqref{se_divfree}$ drops out, therefore} \notag 
\grad_\nu P_X^\nu =& S^{\mu\nu} \, \grad_{\mu}X_\nu \notag \\
=& \halb \, \leftexp{(X)} {\mbo{\pi}}_{\mu \nu} S^{\mu \nu},  \notag
\end{align}
where the deformation tensor $\leftexp{(X)} {\mbo{\pi}}_{\mu \nu}$ is given by
\begin{align*}
\leftexp{(X)} {\mbo{\pi}}_{\mu \nu} \fdg=& \nabla_\mu X_\nu + \nabla_\nu X_\mu \\
=& \gg_{\sigma \nu}\ptl_\mu X^\sigma + \gg_{\sigma \mu}\ptl_\nu X^\sigma + X^\sigma \ptl_\sigma \gg_{\mu \nu}.
\end{align*}
Construction of useful identities using suitably chosen multipliers $X$ and
Stokes' theorem is central to our method to prove non-concentration of energy of equivariant Einstein-wave maps.
In the following let us calculate the divergence of $P_X$ for various choices of $X$. \\
Consider  $T = e^{-\alpha} \ptl_t$. The corresponding momentum $P_T$ is
\begin{align}
 P_T =& - e^{-\alpha}\mathbf{e}\, \ptl_t + e^{-\beta}\mathbf{m}\,\ptl_r. 
\end{align}
Then, we have,
\begin{align}\label{div_pxi}
 \nabla_\nu  P_T ^\nu & = \halb e^{-2\alpha} \left( e^{\alpha}\beta_t \right) \phi_t^2  + 
 \halb e^{-2\beta} \left( e^{\alpha}\beta_t \right) \phi_r^2  \notag \\ 
& \quad - \halb \left( e^{\alpha}\beta_t \right) \frac{g^2(\phi)}{r^2} - \alpha_r e^{-\alpha-2\beta} \phi_t \phi_r \notag \\
&= e^{-\alpha} \left(\beta_t (\mathbf{e}-\mathbf{f})-\alpha_r e^{-\beta} \mathbf{m} \right) \notag \\
& = 0
\end{align}
after using Einstein's equations \eqref{gamma_t} and \eqref {omega_r}. Also, recall from \eqref{eq:computationdivergencePT} that 
\begin{align}\label{equiv_div_px1}
0=\grad_\nu P_T^\nu &=  \frac{1}{re^{\beta + \alpha}} \left(-\ptl_t (re^{\beta}\mathbf{e}) 
+ \ptl_r (re^{\alpha}\mathbf{m}) \right).
\end{align}
For $R = e^{-\beta}\ptl_r$ and 
\begin{align}\label{intro_px2}
P_R  =&- e^{-\alpha} \mathbf{m} \, \ptl_t + e^{-\beta} (\mathbf{e}-\mathbf{f}) \,\ptl_r,
\end{align}
the divergence $\grad_\nu P_R^\nu $ is
\begin{align}\label{div_px2}
\grad_\nu P_R^\nu & =  \halb \, \leftexp{(R)} {\mbo{\pi}}_{\mu \nu} S^{\mu \nu} \notag\\
& = - e^{-\beta} \alpha_r \,\mathbf{e} + \frac{1}{2r} e^{-\beta} (e^{-2\alpha} \phi_t^2 - e^{-2\beta}\phi_r^2 + \mathbf{f}) 
+ e^{-\alpha} \beta_t \mathbf{m}.
\end{align}
Equivalently,
\begin{align}\label{equiv_div_px2}
 \grad_\nu P_R^\nu &= \frac{1}{\sqrt {-\gg}} \ptl_\nu (\sqrt{-\gg}\, P_R^\nu) \notag\\
 & = \frac{1}{re^{\beta + \alpha}} \left(-\ptl_t (re^{\beta}\mathbf{m}) 
+ \ptl_r ((\mathbf{e}-\mathbf{f}) re^{\alpha}) \right) .
\end{align}
Similarly for the choice ${\mathcal{R}_a} \fdg=r^a\,\ptl_r$, we have
\begin{align}\label{intro_px4}
 P_{\mathcal{R}_a}  =& -e^{\beta -\alpha} r^a \, \mathbf{m} \ptl_t + r^a \,(\mathbf{e}-\mathbf{f}) \ptl_r, \notag\\
\nabla_\nu  P^\nu_{\mathcal{R}_a} =&\, \halb \left ( r^a (-\alpha_r + \beta_r) + (1+a) r^{a-1}  \right) e^{-2\alpha} \phi^2_t \notag\\
&\quad+ \halb \left ( r^a (-\alpha_r + \beta_r) + (a-1) r^{a-1}  \right) e^{-2\beta} \phi^2_r \notag\\
& \quad + \halb \left ( - r^a (\alpha_r + \beta_r) + (1-a) r^{a-1}  \right) \frac{g^2(\phi)}{r^2}\notag \\
=&\, \halb \left (  (1+a) r^{a-1}  \right) e^{-2\alpha} \phi^2_t
+ \halb \left ( (a-1) r^{a-1}  \right) e^{-2\beta} \phi^2_r \notag\\
& \quad+ \halb \left ( (1-a) r^{a-1}  \right) \frac{g^2(\phi)}{r^2}
 \end{align}
where we used Einstein equations \eqref{omega_r} and \eqref{gamma_r} for $ \alpha_r$ and $\beta_r$ respectively. In particular, we have $\mathcal{R}_1 \fdg = r\,\ptl_r$ and
\begin{align}\label{intro_px3}
P_{\mathcal{R}_1} =& -re^{\beta - \alpha} \,\mathbf{m}\, \ptl_t + r(\mathbf{e} -\mathbf{f})\ptl_r, \notag\\
\nabla_\nu  P^\nu_{\mathcal{R}_1} = & \,e^{- 2 \alpha} \phi_t^2.
\end{align}

Let $J^-(O)$ be the causal past of the the point $O$ and $ I^- (O)$ the chronological past of $O$.
We will need the following definitions 
\begin{align*}
 \Sigma^O_t \fdg &= \Sigma_t \cap J^- (O) \\
K(t) \fdg &= \cup_{ t \leq t' < 0}\, \Sigma_{t'} \cap J^-(O) \\
C(t) \fdg &= \cup_{t \leq t' < 0} \, \Sigma_{t'} \cap (J^-(O) \setminus I^-(O)) \\
K(t,s) \fdg &= \cup_{ t \leq t' < s}\, \Sigma_{t'} \cap J^-(O) \\
C(t,s) \fdg &= \cup_{t \leq t' < s} \, \Sigma_{t'} \cap (J^-(O) \setminus I^-(O))
\end{align*}
for $-1\leq t < s <0.$
In the following we will try to understand the behaviour of various quantities of the wave map as one approaches $O$ in a limiting sense.
For this we will use Stokes' theorem in the region $K(\uptau,s),-1\leq \uptau \leq s <0$ (as shown in figure \ref{fig:Stokes_picture}) for divergence of 
vector fields $P_X$ with suitable choices of the vector field $X$. 
\begin{figure}[!hbt]
\psfrag{O}{$O$}
\psfrag{Ktaus}{$K(\uptau,s)$}
\psfrag{teq0}{$t=0$}
\psfrag{teqs}{$t=s$}
\psfrag{teqtau}{$t=\uptau$}
\psfrag{teq-1}{$t=-1$}
\psfrag{Fluxpx}{$\text{Flux}(P_X)(\uptau,s)$}
\psfrag{E(s)}{$E^O(s)$}
\psfrag{E(tau)}{$E^O(\uptau)$}
\centerline{\includegraphics[height=2.5in]{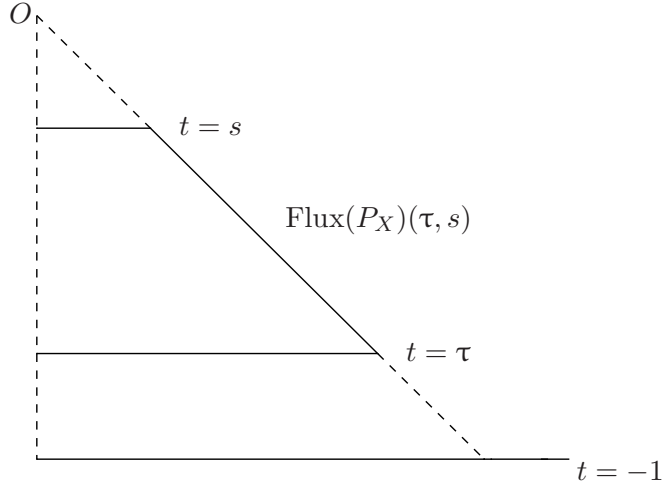}}
\caption{Application of the Stokes' theorem for the divergence of $P_X$ } 
\label{fig:Stokes_picture}
\end{figure}
The volume 3-form of $(M,g)$ is given by 
\begin{align*}
\bar{\mu}_g \fdg = re^{\beta + \alpha} d\,t \wedge d\,r \wedge d\,\theta
\end {align*}
and the area 2-form of $(\Sigma,q)$ by
\begin{align*}
 \bar{\mu}_q = r e^{\beta} d\,r \wedge d\,\theta.
\end{align*}
Let us define 1-forms $\tild{\ell}$, $\tild{n}$ and $\tild{m}$ as follows
\begin{align*}
 \tild{\ell} \fdg =& -e^\alpha d\,t + e^{\beta} d\,r, \\
 \tild{n}\fdg =& -e^\alpha d\,t - e^{\beta} d\,r, \\
 \tild{m} \fdg =&\,rd\,\theta,
\end{align*}
so we have,
\begin{align*}
 \bar{\mu}_g= \halb \left( \tild{\ell} \wedge \tild{n} \wedge \tild{m} \right).
\end{align*}
Let us also define the 2-forms $ \bar{\mu}_{\tild{\ell}}$ and $ \bar{\mu}_{\tild{n}}$ such that
\begin{align*}
 \bar{\mu}_{\tild{\ell}} \fdg=&\,-\halb \tild{n} \wedge \tild{m}, \\
\bar{\mu}_{\tild{n}} \fdg =&\, \halb \tild{\ell} \wedge \tild{m},
\end{align*}
so that
\begin{align*}
 \bar{\mu}_g =\, & -\tild{\ell} \wedge \bar{\mu}_{\tild{\ell}},\\
\bar{\mu}_g =\, & -\tild{n} \wedge \bar{\mu}_{\tild{n}}\,.
\end{align*}
We now apply the Stokes' theorem for the $\bar{\mu}_g$-divergence of $P_X$ in the region $K(\uptau,s)$ to get 

\begin{align}\label{stokes_gen} 
\int_{K(\uptau,s)} \nabla_\nu  P^\nu_{X} \, \bar{\mu}_g  = \int_{\Sigma^O_{s}} e^{\alpha}\, P^t_{X}\, \bar{\mu}_q
- \int_{\Sigma^O_{\uptau}}e^{\alpha}\,P^t_{X}\, \bar{\mu}_q
+ \text{Flux}(P_X) (\uptau,s) 
\end{align}
where
\begin{align*}
 \text{Flux} (P_X)(\uptau,s) = - \int_{C(\uptau,s)} \tild{n}(P_X)\, \bar{\mu}_{\tild{n}} .
 \end{align*}

\subsection{Monotonicity of energy} \label{sec:mon-energ}

\begin{lem}\label{mon}
We have  $ E^O(\uptau) \geq E^O(s)$ for $-1 \leq \uptau < s <0$.
\end{lem}
\begin{proof}
 
 Let us apply Stokes' theorem \eqref{stokes_gen} to the vector field $P_T$.  We have
\begin{align} \label{stokes_mono}
 0  = -\int_{\Sigma^O_{s}} \mathbf{e}\,\,\bar{\mu}_q
+ \int_{\Sigma^O_{\uptau}} \mathbf{e} \,\, \bar{\mu}_q
+ \text{Flux} (P_T)(\uptau,s) 
\end{align}
and
\begin{align*}
\text{Flux} (P_T)(\uptau,s) =&- \int_{C(\uptau,s)} \tild{n}\,(P_T)\,  \bar{\mu}_{\tild{n}}  \\
=&- \int_{C(\uptau,s)} (\mathbf{e-m})\,  \bar{\mu}_{\tild{n}}\,.  
\end{align*}
\begin{figure}[!hbt]
\psfrag{O}{$O$}
\psfrag{Ktaus}{$K(\uptau,s)$}
\psfrag{teq0}{$t=0$}
\psfrag{teqs}{$t=s$}
\psfrag{teqtau}{$t=\uptau$}
\psfrag{teq-1}{$t=-1$}
\psfrag{Fluxpx}{$\text{Flux}(P_T)(t,s)$}
\psfrag{E(s)}{$E^O(s)$}
\psfrag{E(tau)}{$E^O(\uptau)$}
\centerline{\includegraphics[height=2.5in]{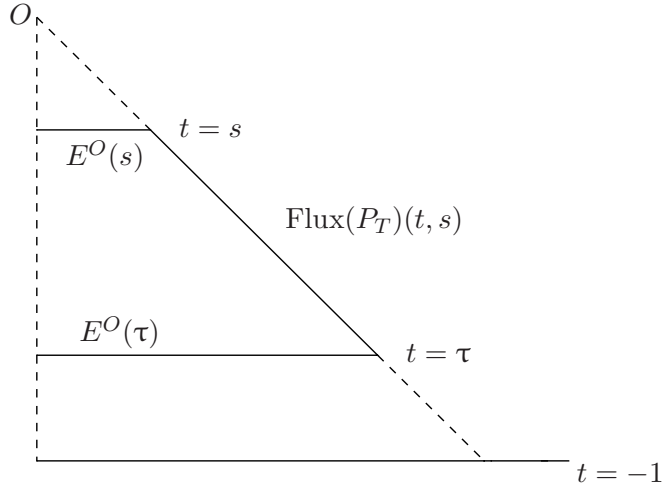}}
\caption{Monotonicity of Energy inside the past null cone of $O$ } 
\label{fig:Stokes_picture1} 
\end{figure}
Note that we have $\mathbf{e} \geq |\mathbf{m}|$. 
Hence, we have $\text{Flux} (P_T)(\uptau,s) \leq 0$ which implies  
\[ E^O(\uptau) - E^O(s) \geq 0 \,\,\forall \,\, -1\leq \uptau \leq s <0.\]
This concludes the proof of the lemma.
\end{proof}

Let us define
\begin{align} \label{E_conc}
 E^O_{\text{conc}} \fdg =  \inf_{\uptau \in [-1,0)} E^O (\uptau).
\end{align}
As a consequence of Lemma \ref{mon}, \eqref{E_conc} is equivalent to 
\begin{align}
 E^O_{\text{conc}} = \lim_{\uptau \to 0} E^O(\uptau).
\end{align}
We say that the energy of equivariant Cauchy problem 
concentrates  if  $ E^O_{\text{conc}} \neq 0$ and does not concentrate if
 $E^O_{\text{conc}} = 0.$
\begin{cor} \label{flux_px1_zero}
For a vectorfield $X$, let 
\begin{align*}
 \text{\textnormal{Flux}}(P_X)(\uptau) \fdg = \lim_{s\to 0} \text{\textnormal{Flux}}(P_X)(\uptau,s).
\end{align*}
Then, we have
\begin{align*}
  \text{\textnormal{Flux}}(P_T)(\uptau) \to 0 \,\,\text{as}\,\, \uptau \to 0.
\end{align*}
\end{cor}

\begin{proof}
Recall the equation \eqref{stokes_mono}. For $s \to 0$, we have
\begin{align} \label{flux_to_zero}
 0  = -E^O_{\text{conc}}
+ \int_{\Sigma^O_{\uptau}} \mathbf{e} \,\,\bar{\mu}_q
+ \text{Flux} (P_T)(\uptau).
\end{align}
Now by the definition \eqref{E_conc}, as $\uptau \to 0$ we get
\begin{align}
\lim_{\uptau \to 0} \int_{\Sigma^O_{\uptau}} \mathbf{e} \,\,\bar{\mu}_q \to E^O_{\text{conc}}.
\end{align}
Therefore, it follows from \eqref{flux_to_zero} that $\text{\textnormal{Flux}}(P_T)(\uptau)\to 0$ as $\uptau \to 0$.
\end{proof}

\subsection{$L^{\infty}$ estimate for the Jacobian}\label{sec:null-coord-construct}\label{sec:J-est} 

The goal of this section is to derive uniform bounds for the Jacobian transformation between $(t,r,\theta)$ and $(u,\ub,\theta)$ coordinates. 
Recall that we defined the 1-forms $\tild{\ell}$ and $\tild{n}$. Their corresponding vectors are null, given by
\begin{align*}
 \tild{\ell}   =&\, e^{-\alpha}\ptl_t + e^{-\beta}\ptl_r \\
\tild{n} =&\, e^{-\alpha} \ptl_t - e^{-\beta}\ptl_r \,.
\end{align*}

\begin{lem}
There exists two scalar functions $\mathcal{F}$ and $\mathcal{G}$ such that
\begin{align}\label{coonulltriad}
\pr_{\ub} = \frac{1}{2}e^\mathcal{F} \tild{\ell},\,\,\,\,\pr_u = \frac{1}{2}e^\mathcal{G} \tild{n}, 
\end{align}
with the normalization on $\Gamma$
$$\mathcal{F}=\mathcal{G}=0\textrm{ on }\Gamma.$$
Furthermore, $\cal{F}$ and $\cal{G}$ satisfy  
\begin{subequations} \label{nullode}
\begin{align}
 2\pr_{\ub}(\cal{G}) &= e^{\cal{F}}r \kappa e^{\beta} (\mathbf{e} + \mathbf{m} -\mathbf{f}), \\
2\pr_u(\cal{F}) &= -e^{\cal{G}}r \kappa e^{\beta} (\mathbf{e} - \mathbf{m} -\mathbf{f}). 
\end{align}
\end{subequations}
\end{lem}

\begin{proof}
In view of various normalizations on $\Gamma$, note that we have
$$\pr_{\ub}r=\frac{1}{2},\,\, \pr_ur=-\frac{1}{2},\,\, \tild{\ell}(r)=1,\,\, \tild{n}(r)=-1\textrm{ on }\Gamma.$$
Furthermore, $\pr_u$ and $\pr_{\ub}$ are also null, and $\pr_u$, $\pr_{\ub}$, $\tild{\ell}$ and $\tild{n}$ are all future directed. We infer that there exists two scalar functions $\mathcal{F}$ and $\mathcal{G}$ such that
\begin{align*}
\pr_{\ub} = \frac{1}{2}e^\mathcal{F} \tild{\ell},\,\,\,\,\pr_u = \frac{1}{2}e^\mathcal{G} \tild{n}, 
\end{align*}
with the normalization on $\Gamma$
$$\mathcal{F}=\mathcal{G}=0\textrm{ on }\Gamma.$$
 
Next, we derive equations for $\mathcal{F}$ and $\mathcal{G}$. We have 
\begin{align*}
[ \tild{\ell}, \tild{n}] &=  2e^{- (\beta + \alpha)} (-\alpha_r \ptl_t +\beta_t \ptl_r).
\end{align*}
We infer
\begin{eqnarray*}
[ \pr_{\ub} , \pr_u ] & =& \frac{e^{(\cal{F} + \cal {G})}}{4} \left( [\tild{\ell},\tild{n}] + \tild{\ell}(\cal{G}) \tild{n} - \tild{n}(\cal{F})\tild{\ell} \right) \\
& =& \frac{e^{(\cal{F} + \cal {G})}}{2}e^{- (\beta + \alpha)} (-\alpha_r \ptl_t +\beta_t \ptl_r)+\frac{e^{\cal{G}}}{2}\pr_{\ub}(\cal{G})(e^{-\alpha} \ptl_t - e^{-\beta}\ptl_r)\\
&&-\frac{e^{\cal{F}}}{2}\pr_u(\cal{F})(e^{-\alpha} \ptl_t + e^{-\beta}\ptl_r).
\end{eqnarray*}
Since $[ \pr_{\ub},\pr_u]=0$, $\cal{F}$ and $\cal{G}$ are such that 
\begin{align*}
 e^{-\cal{F}}\pr_{\ub}(\cal{G}) - e^{-\cal{G}}\pr_u(\cal{F}) &=  r \kappa e^{\beta} (\mathbf{e} -\mathbf{f}), \\
e^{-\cal{F}}\pr_{\ub}(\cal{G}) + e^{-\cal{G}}\pr_u(\cal{F}) &= r \kappa e^{\beta} \mathbf{m},
\end{align*}
and hence
\begin{subequations}
\begin{align*}
 2\pr_{\ub}(\cal{G}) &= e^{\cal{F}}r \kappa e^{\beta} (\mathbf{e} + \mathbf{m} -\mathbf{f}), \\
2\pr_u(\cal{F}) &= -e^{\cal{G}}r \kappa e^{\beta} (\mathbf{e} - \mathbf{m} -\mathbf{f}). 
\end{align*}
\end{subequations}
This concludes the proof of the lemma.
\end{proof}

Let us revisit Stokes' theorem for $\bar{\mu}_g$-divergence of $P_{X}$ in $K(\uptau,s)$. We have
$$d\,{\ulin{u}} = -e^{-\cal{F}} \tild{n},\,\,  d\,u = -e^{-\cal{G}} \tild{\ell}.$$
The volume 3-form of $(M,g)$ is
\begin{align*}
 \bar{\mu}_g  =& \, \halb r\,\Omega^2 d\,u \wedge d\,{\ulin{u}} \wedge d\,\theta.
\end{align*}
Let us introduce the 2-forms $\bar{\mu}_{\ulin{u}}$ and $\bar{\mu}_u$ as follows
\begin{align*}
 \bar{\mu}_g=  d\,{\ulin{u}} \wedge \bar{\mu}_{\ulin{u}},\,\,\,\,\bar{\mu}_g= d\,u \wedge \bar{\mu}_u
\end{align*}
so that
\begin{align*}
 \bar{\mu}_{\ulin{u}}  = - \halb r \,\Omega ^2 \left( d\,u \wedge d\,\theta  \right),\,\,\,\,\bar{\mu}_u = \, \halb r\, \Omega^2 \left( d\,{\ulin{u}} \wedge d\,\theta  \right).
\end{align*}
Now,
\begin{align*}
 \text{Flux}(P_{X}) (\uptau,s) =& \int_{C(\uptau,s)} d\,{\ulin{u}}(P_X)\, \bar{\mu}_{\ulin{u}}, \\
 \intertext{for instance,}
 \text{Flux}(P_{T}) (\uptau,s) =& \int_{C(\uptau,s)} d\,{\ulin{u}}(P_T) \,\bar{\mu}_{\ulin{u}},\\
 =& - \int_{C(\uptau,s)}e^{-\cal{F}}(\mathbf{e-m})  \bar{\mu}_{\ulin{u}}.
\end{align*}

\begin{lem} \label{null_uniform}
 There exist constants $c^{\,-}_{\,\cal{G}},\,c^{\,+}_{\,\cal{G}},\,c^{\,-}_{\,\cal{F}}\,\text{and}\,\,c^{\,+}_{\,\cal{F}} $  depending only on the initial data  and the universal constants such that
 the following uniform bounds hold
\begin{align*}
 c^{\,-}_{\,\cal{G}} \,\leq&\, \cal{G}\, \leq\, c^{\,+}_{\,\cal{G}} \\
 c^{\,-}_{\,\cal{F}}\, \leq&\, \cal{F}\, \leq \,c ^{\,+}_{\,\cal{F}}. 
 \end{align*}
\end{lem}

\begin {proof}
We integrate \eqref{nullode} using the fact that $\mathcal{F}=\mathcal{G}=0$ on $\Gamma$. We infer
\begin{subequations}
\begin{align*}
 \cal{G}(u,\ub) &= \int_u^{\ub}e^{\cal{F}}r \kappa e^{\beta} (\mathbf{e} + \mathbf{m} -\mathbf{f})(u,\ub')d\ub', \\
\cal{F}(u,\ub) &= \int_u^{\ub}e^{\cal{G}}r \kappa e^{\beta} (\mathbf{e} - \mathbf{m} -\mathbf{f})(u',\ub)du'. 
\end{align*}
\end{subequations}
Next, note that
$$-\frac{\Omega^2}{2}=\gg(\pr_u, \pr_{\ub})=\frac{e^{\cal{F}+\cal{G}}}{4}\gg(\tild{n}, \tild{\ell})=-\frac{e^{\cal{F}+\cal{G}}}{2}$$
and hence
\begin{equation}\label{relationOmFG}
 \Omega^2= e^{\cal{F}+ \cal{G}}.
\end{equation}
We infer
\begin{subequations}
\begin{align*}
 \cal{G}(u,\ub) &= \kappa\int_u^{\ub}e^{\beta} e^{-\cal{G}}(\mathbf{e} + \mathbf{m} -\mathbf{f})r\Omega^2d\ub', \\
\cal{F}(u,\ub) &= \kappa\int_u^{\ub}e^{\beta}e^{-\cal{F}}  (\mathbf{e} - \mathbf{m} -\mathbf{f})r\Omega^2du'. 
\end{align*}
\end{subequations}
Using the fact that $|\mathbf{e}\pm \mathbf{m}-\mathbf{f}|\leq \mathbf{e}\pm \mathbf{m}$ and $u\leq \ub$, we infer 
\begin{subequations}
\begin{align*}
|\cal{G}(u,\ub)| &\leq c\int_u^{\ub}e^{-\cal{G}}(\mathbf{e} + \mathbf{m})r\Omega^2d\ub', \\
|\cal{F}(u,\ub)| &\leq c\int_u^{\ub}e^{-\cal{F}}  (\mathbf{e} - \mathbf{m})r\Omega^2du'. 
\end{align*}
\end{subequations}
Since $d\,{\ulin{u}} = -e^{-\cal{F}} \tild{n}$ and $d\,u = -e^{-\cal{G}} \tild{\ell}$, we infer
\begin{subequations}
\begin{align*}
 |\cal{G}(u,\ub)| &\leq c\int_u^{\ub}d\,u(P_T)  r\Omega^2d\ub', \\
|\cal{F}(u,\ub)| &\leq c\int_u^{\ub}d\,{\ulin{u}}(P_T) r\Omega^2du'. 
\end{align*}
\end{subequations}
After integration in $\theta$, the right-hand sides are bounded by fluxes which in turn are bounded by the energy, and hence
$$|\cal{G}|\, \leq c,\,\, |\cal{F}|\, \leq \,c.$$ 
This concludes the proof of the lemma.
\end {proof}

Let us consider the Jacobian $\mathbf{J}$ of the transition functions between $(t,r,\theta)$ and $({\ulin{u}},u,\theta)$
\begin{align*}
 \mathbf{J} \fdg= &  \left( \begin{array}{ccc}
\ptl_{\ulin{u}} t & \ptl_u t & \ptl_\theta t \\
\ptl_{\ulin{u}} r& \ptl_u r & \ptl_\theta r\\
\ptl_{\ulin{u}} \theta & \ptl_u \theta & \ptl_\theta \theta \end{array} \right) \\[2mm] 
=& \frac{1}{2}\left( \begin{array}{ccc}
e^{\mathcal{F}-\alpha} & e^{\mathcal{G}-\alpha} &0  \\
e^{\mathcal{F}-\beta} & -e^{\mathcal{G}-\beta} &0 \\
 0 & 0 & 2 \end{array} \right)
\end{align*}
then the inverse Jacobian $\mathbf{J}^{-1}$ is given by
\begin{align*}
 \mathbf{J}^{-1} =& \left( \begin{array}{ccc}
e^{-\mathcal{F}+\alpha} & e^{-\mathcal{F}+ \beta} &0  \\
e^{-\mathcal{G}+\alpha} & -e^{-\mathcal{G}+\beta} &0 \\
 0 & 0 & 1 \end{array} \right).
\end{align*}
Therefore,
\begin{align}
 \ptl_t {\ulin{u}} =  e^{-\cal{F} + \alpha},\,\, \ptl_r {\ulin{u}} =  e^{-\cal{F} + \beta},\,\, \ptl_t u =  e^{-\cal{G} + \alpha},\,\, \ptl_r u = - e^{-\cal{G} + \beta},
\end{align}
and
\begin{align*}
 d\, {\ulin{u}} = e^{(\mathcal{-F} + \alpha)} d\,t + e^ {(\mathcal{-F} + \beta)} d\,r, \,\,\,\,d\, u =  e^{(\mathcal{-G} + \alpha)} d\,t - e^ {(\mathcal{-G} + \beta)}d \, r.
\end{align*}

\begin{cor}\label{Jac-uniform}
 There exist constants $c^{\,-}_{\,\mu\nu},c^{\,+}_{\,\mu \nu}$ and $C^{\,-}_{\,\mu \nu}, C^{\,+}_{\,\mu \nu}$ depending only on the inital data and the universal constants 
such that all the entries of the Jacobian $\mathbf{J}$ and its inverse 
$\mathbf{J}^{-1}$ are uniformly bounded 
\begin{align*}
 c^{\,-}_{\,\mu \nu} \,\leq&\, \mathbf{J}_{\,\mu \nu}\, \leq\, c^{\,+}_{\,\mu \nu} \\
C^{\,-}_{\,\mu \nu} \,\leq&\, \mathbf{J}^{-1}_{\,\mu \nu}\, \leq\, C^{\,+}_{\,\mu \nu}
\end{align*}
for $\mu,\nu = 0,1,2.$
\end{cor}
\begin{proof}
The proof follows from Lemmas \ref{metric_uniform} and \ref{null_uniform}.
\end{proof}

\begin{cor}
 There exist constants $c^-_{\Omega}$ and $c^+_{\Omega}$  depending only on the initial energy and the universal constants such that the following uniform bounds hold
 on the metric function $\Omega$ in null coordinates.

\begin{align}
 c^-_{\Omega} \leq \Omega \leq c^+_{\Omega}.
\end{align}
\end{cor}

\begin{proof}
This follows immediately from \eqref{relationOmFG} and Lemma \ref{null_uniform}.
\end{proof}

\begin{cor}\label{cor:comprttauvarho}
Let us introduce the notation
$$\tau=\frac{u+\ub}{2},\,\,\vrho=\frac{\ub-u}{2}.$$
Then, there exist constants $c_1,c_2,c_3,c_4$ such that the pointwise bounds 
\begin{align*}
 r \geq & \, c_1\, \vrho,\,\,\,\,\,\,\,\,\,\,\,\,\,\,\,\,\,  t \geq \,c_3\, \tau, \\
r \leq &\,c_2\, \vrho \,\,\,\,\,\text{and}\,\,\,\,\,  t \leq  \, c_4\, \tau 
\end{align*}
hold in $J^{-}(O)$ for the scalar functions $r,t,\vrho$ and $\tau$.
\end{cor}
\begin{proof}
 From Corollary \ref{Jac-uniform}, we have uniform bounds on the Jacobian and its inverse of the transition functions between 
 $(t, r, \theta)$ and $(\tau, \varrho, \theta)$. As a consequence, 
\begin{subequations} \label{pointwise_scalar}
\begin{align}
|\ptl_\vrho r| =& |\ptl_{\ulin{u}} r -\ptl_u r | = \frac{1}{2}| e^{\cal{F}-\beta} + e^{\cal{G}-\beta}| \leq c_1,  \\ \notag\\
|\ptl_r \vrho| =& \halb |\ptl_r {\ulin{u}} -\ptl_r u| = \frac{1}{2} |e^{-\cal{G}+ \beta} -e^{-\cal{F} + \beta}| \leq c_2,    \\ \notag \\
|\ptl_\tau t| =& | \ptl_{\ub} t + \ptl_u t|= \frac{1}{2}|e^{\cal{F}-\alpha}+ e^{\cal{G} -\alpha} | \leq c'_3,  \\ \notag \\
|\ptl_\varrho t| =& |\ptl_\ub t - \ptl_u t| = \frac{1}{2} |e^{\cal{F}- \alpha} + e^{\cal{G}- \alpha}| \leq c'_4, \\ \notag \\
|\ptl_t \tau| =& \halb |\ptl_t {\ulin{u}} + \ptl_t u| = \frac{1}{2} |e^{-\cal{F}+ \alpha} + e^{-\cal{G}+ \alpha}| \leq c_5, \\ \notag\\
|\ptl_r \tau| =& \halb |\ptl_r {\ulin{u}} + \ptl_r u| = \frac{1}{2} |e^{-\cal{F}+ \b} - e^{-\cal{G}+ \b}| \leq c_6.
\end{align}
\end{subequations}
The proof follows by applying the fundamental theorem of calculus in the region $J^-(O)$ and noting
that at $O$, $t=\tau=0$ and $r=\vrho=0$ on the axis. 
\end{proof}

\subsection{Non-concentration away from the axis} \label{sec:non-conc-away-axis}

In this section we shall use the vector fields method introduced previously to prove that
energy does not concentrate. We start with proving that
the energy does not concentrate away from the axis using the divergence free vector $P_{T}$.

\begin{lem}[Non-concentration away from the axis]\label{ann}
We have
\begin{align*}
 E^{\,O}_{\text{ext}} (\uptau) \fdg = \int_{B_{r_2(\uptau)}\setminus B_{r_1(\uptau)}} \mathbf{e} \,\,\bar{\mu}_q \to 0 \,\, \text{as} \,\, \uptau \to 0,
\end{align*} 
where $r = r_2(\uptau)$ is the radius where the $t = \uptau$ slice intersects the $\vrho = |\tau|$ curve i.e the mantel of the null cone $ J^-(O)$
 and  $r = r_1 (\uptau)$ is the radius where the $ \vrho = \lambda |\tau|$ curve intersects the $t = \uptau$ slice, for any real
$\lambda \in (0,1)$.
 Observe that both $r_1(\uptau)$
and $r_2(\uptau) \to 0$ as $\uptau \to 0$.
\end{lem}
\begin{proof}
Consider a tubular region $\cal{S}$ with triangular cross section (as shown in the figure \ref{fig:annular_disc_1} ) in  $\vrho > \lambda \tau, \, \lambda \in (0,1)$ 
of the spacetime i.e., the ``exterior'' part of the interior of the past null cone of $O$.
\begin{figure}[!hbt]
\psfrag{O}{$O$}
\psfrag{O}{$O$}
\psfrag{teqtau}{$t=\uptau$}
\psfrag{teq-1}{$t=-1$}
\psfrag{S}{$\cal{S}$}
\psfrag{1}{$\ptl \cal{S}_1$}
\psfrag{2}{$\ptl \cal{S}_2$}
\psfrag{3}{$\ptl \cal{S}_3$}
\psfrag{Req0}{$\varrho=0$}
\psfrag{ReqlT}{$\varrho=\lambda \vert\tau\vert$}
\psfrag{ReqT}{$\varrho=\vert \tau\vert$}
\centerline{\includegraphics[height=2.5in]{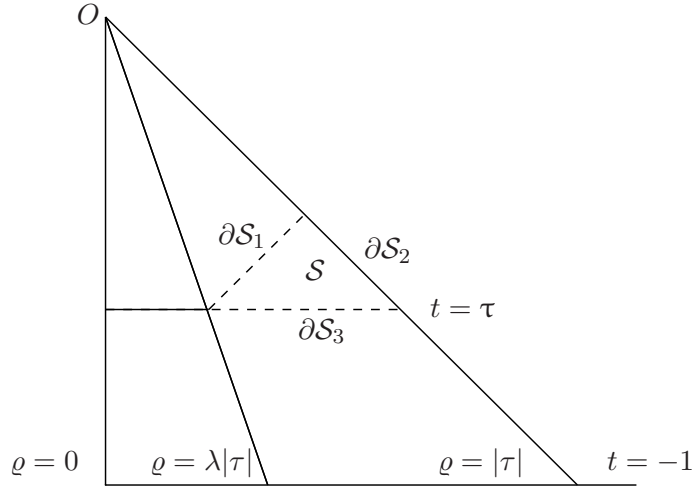}}
\caption{ Application of Stokes' theorem on the $\bar{\mu}_g$-divergence free $P_{T}$ to relate the fluxes through 
surfaces $\ptl \cal{S}_1$,\,$\ptl \cal{S}_2$\, and $\ptl \cal{S}_3$} 
\label{fig:annular_disc_1}
\end{figure} 
As shown in the figure \ref{fig:annular_disc_1}, let us use the divergence-free vector field 
$P_{T}$  and Stokes' theorem in the region $\cal{S}$ to relate the fluxes through the three boundary segments
$\ptl \cal{S}_1$,\,$\ptl \cal{S}_2$ and $\ptl \cal{S}_3$. We have
\begin{align}\label{div_tube}
0 =&  \int_{\ptl \cal{S}_1} d\,u(P_T) \bar{\mu}_u 
 + \int_{\ptl \cal{S}_2} d\,{\ulin{u}}(P_T)\bar{\mu}_{\ulin{u}}
 - \int_{\ptl \cal{S}_3} e^{\alpha} P^t_T \, \bar{\mu}_q \notag \\
 =& - \int_{\ptl \cal{S}_1} e^{-\mathcal{G}} (\mathbf{e+m})\,\bar{\mu}_u  
-\int_{\ptl \cal{S}_2} e^{-\mathcal{F}} (\mathbf{e-m})\,\bar{\mu}_{\ulin{u}} + \int_{\ptl \cal{S}_3} \mathbf{e} \, \bar{\mu}_q .
\end{align}
To analyze the behaviour of the flux terms $\int_{\ptl \cal{S}_1}$ and $\int_{\ptl \cal{S}_2}$ in \eqref{div_tube} close to $O$, let us define
\begin{align*}
 \hatt{l} &\fdg= e^{\beta + \alpha}\tild{\ell} = e^{\beta}\ptl_t + e^{\alpha}\ptl_r, \\
\hatt{n} &\fdg= e^{\beta + \alpha}\tild{n} = e^{\beta}\ptl_t - e^{\alpha}\ptl_r, \\
\mathcal{A}^2 &\fdg= r(\mathbf{e}-\mathbf{m}), \\
\mathcal{B}^2 &\fdg = r(\mathbf{e}+ \mathbf{m}). 
\end{align*}

From  \eqref{div_px2} and \eqref{equiv_div_px2}, we have
\begin{align}\label{equal_div_px2}
& \frac{1}{re^{\beta + \alpha}} \left(-\ptl_t (re^{\beta}\mathbf{m}) 
+ \ptl_r ((\mathbf{e}-\mathbf{f}) re^{\alpha}) \right) \notag\\
 =& - e^{-\beta} \alpha_r \,\mathbf{e} + \frac{1}{2r} e^{-\beta} (e^{-2\alpha} \phi_t^2 - e^{-2\beta}\phi_r^2 + \mathbf{f}) + e^{-\alpha} \beta_t \mathbf{m}.
\end{align}
We have the following identities from \eqref{equiv_div_px1} and \eqref{equal_div_px2}
\begin{subequations} \label{px1px2_identities}
\begin{align}
 \ptl_t(re^{\beta}\mathbf{e})- \ptl_r(re^\alpha \mathbf{m}) & = 0, \\
\ptl_t(r e^{\beta}\mathbf{m}) - \ptl_r({re^{\alpha} \mathbf{e}}) &= L, 
\end{align}
\end{subequations}
where 
\[L \fdg = \frac{re^{\alpha}\alpha_r}{2} \left((T\phi)^2 + (R\phi)^2 -\mathbf{f} \right) + e^{\alpha} L_0 - r\beta_t e^{\beta} \mathbf{m}  \]
for
\[ L_0 \fdg = \halb \left(-(Tu)^2 + (Ru)^2 + \mathbf{f} \right) -\frac{2 g(\phi)g'(\phi) \phi_r}{r} .\]
Furthermore, we can construct the following using the identities in \eqref{px1px2_identities}
\begin{subequations}\label{px1px2_newidentities}
\begin{align}
 \ptl_{\nu}\left(re^{\beta + \alpha}(\mathbf{e}-\mathbf{m})\tild{\ell} ^\nu \right) &= \ptl_\nu(\mathcal{A}^2 \hatt{\ell}^{\nu}) = -L, \\
\ptl_{\nu}\left(re^{\beta + \alpha}(\mathbf{e}+\mathbf{m})\tild{n}^\nu \right) &= \ptl_\nu(\mathcal{B}^2 \hatt{n}^{\nu}) = L. 
\end{align}
\end{subequations}
Let us try to express $L$ in terms of $\cal{A}^2$ $\cal{B}^2$ after using the Einstein equations
\begin{align}\label{simplifyL}
 L &= e^{\alpha}L_0 + \kappa r^2 e^{2\beta + \alpha} \left( \mathbf{e}-\mathbf{f} \right)^2 - \kappa r^2 e^{2\beta + \alpha} \mathbf{m}^2 \notag\\
&= e^{\alpha}L_0 + \kappa r^2 e^{2\beta + \alpha} \left( \mathbf{e}^2 -2 \,\mathbf{e} \,\mathbf{f} + \mathbf{f}^2 -\mathbf{m}^2 \right) \notag\\
&= e^{\alpha}L_0 + \kappa e^{2\beta + \alpha} \left( \mathcal{A}^2 \, \mathcal{B}^2 -2\, r^2 \,\mathbf{e}\,\mathbf{f} + r^2 \mathbf{f}^2 \right). 
\end{align}
We would like to set up a Gr\"onwall estimate for $\cal{A}$ and $\cal{B}$ using the identities in \eqref{px1px2_newidentities}. However, the quantity $L$ as shown
in \eqref{simplifyL} has nonlinear terms involving $\mathbf{e}$ and $\mathbf{f}$. Therefore, in what follows we use Einstein equations to estimate these terms. 

Firstly note that
\begin{align*}
  \hatt{\ell}^\mu \ptl_\mu e^{2 \beta} =& 2 e^{2 \beta} (e^{\beta}\beta_t + e^{\alpha}\beta_r)  &
\hatt{n}^\mu \ptl_\mu e^{2 \beta} =& 2 e^{2 \beta} (e^{\beta}\beta_t - e^{\alpha}\beta_r) \\ 
=& 2 e^{2 \beta} \kappa r e^{2\beta + \alpha}(\mathbf{m} + \mathbf{e})  &
 =& 2 e^{2 \beta} \kappa r e^{2\beta + \alpha}(\mathbf{m} - \mathbf{e})  \\
=& 2 \kappa e^{2 \beta}e^{2\beta + \alpha} \mathcal{B}^2, &
=& -2 \kappa e^{2 \beta}e^{2\beta + \alpha} \mathcal{A}^2, 
\end{align*}
and
\begin{align*}
\ptl_\mu \hatt{\ell}^\mu =& e^{\beta} \beta_t + e^{\alpha} \alpha_r & 
\ptl_\mu \hatt{n}^\mu =& e^{\beta} \beta_t - e^{\alpha} \alpha_r \\
=& r \kappa e^{2\beta + \alpha} \left(\mathbf{e} + \mathbf{m} - \mathbf{f} \right) &
=& r \kappa e^{2\beta + \alpha} \left(-\mathbf{e} + \mathbf{m} + \mathbf{f} \right)\\
=& \kappa e^{2\beta + \alpha} \left(\mathcal{B}^2 -r \,\mathbf{f} \right), &
=&  \kappa e^{2\beta + \alpha} \left(-\mathcal{A}^2 + r \, \mathbf{f} \right). 
\end{align*}
Now consider the quantities $ \ptl_{\mu} (e^{2\beta} \mathcal{A}^2 \hatt{\ell}^\mu)$ and $ \ptl_{\mu} (e^{2\beta} \mathcal{B}^2 \hatt{n}^\mu)$,
\begin{align*}
\hatt{\ell}^\mu \ptl_\mu (e^{2\beta}\mathcal{A}^2) &= \ptl_{\mu} (e^{2\beta} \mathcal{A}^2 \hatt{\ell}^\mu) - e^{2\beta}\mathcal{A}^2 \ptl_\mu \hatt{\ell}^\mu \\
& = e^{2 \beta} \ptl_\mu (\mathcal{A}^2 \hatt{\ell}^\mu) + \mathcal{A}^2 \hatt{\ell}^\mu \ptl_\mu e^{2\beta} -\kappa e^{2\beta}e^{2\beta + \alpha}\mathcal{A}^2 \mathcal{B}^2  
+ r \kappa e^{2\beta}e^{2\beta + \alpha}\mathcal{A}^2  \, \mathbf{f} \\
& = - e^{2 \beta} L +  \kappa e^{2\beta}e^{2\beta + \alpha}\mathcal{A}^2 \mathcal{B}^2 
+ r \kappa e^{2\beta}e^{2\beta + \alpha}\mathcal{A}^2 \, \mathbf{f} \\
&= e^{2 \beta}e^{\alpha}\left( -L_0  + 2\,r^2\, \mathbf{e}\,\mathbf{f}- r^2\mathbf{f}^2 
+ r \mathcal{A}^2 \mathbf{f}  \right) \\
& = e^{2 \beta}e^{\alpha}\left( -L_0 + \kappa r^2 e^{2\beta}\left( 3 \mathbf{e}\,\mathbf{f}  - \mathbf{f}^2
 - \mathbf{m}\, \mathbf{f} \right) \right), 
\end{align*}
\begin{align*}
\hatt{n}^\mu \ptl_\mu (e^{2\beta}\mathcal{B}^2) &= \ptl_{\mu} (e^{2\beta} \mathcal{B}^2 \hatt{n}^\mu) - e^{2\beta}\mathcal{B}^2 \ptl_\mu \hatt{n}^\mu \\
& = e^{2 \beta} \ptl_\mu (\mathcal{B}^2 \hat{n}^\mu) + \mathcal{B}^2 \hat{n}^\mu \ptl_\mu e^{2\beta} + \kappa e^{2\beta}e^{2\beta + \alpha}\mathcal{A}^2 \mathcal{B}^2  
- r \kappa e^{2\beta}e^{2\beta + \alpha}\mathcal{B}^2 \mathbf{f} \\
& = e^{2 \beta} L  - \kappa e^{2\beta}e^{2\beta + \alpha}\mathcal{A}^2 \mathcal{B}^2 
- r \kappa e^{2\beta}e^{2\beta + \alpha}\mathcal{B}^2 \, \mathbf{f} \\
&= e^{2 \beta}e^{\alpha}\left( L_0 + \kappa e^{2\beta}\left( - 2r^2 \mathbf{f}\mathbf{e} + r^2\mathbf{f}^2 
- r \mathcal{B}^2 \mathbf{f} \right) \right) \\ 
& = e^{2 \beta} e^{\alpha} \left(L_0 + \kappa r^2 e^{2\beta}\left(  - 3 \mathbf{e}\,\mathbf{f} + \mathbf{f}^2 
- \mathbf{m}\,\mathbf{f}  \right) \right). 
\end{align*}
Let us define
\begin{align*}
 S_1 \fdg &=  3 \mathbf{e} \, \mathbf{f} - \mathbf{f}^2 - \mathbf{m}\, \mathbf{f}  \\
& =   (\mathbf{e} - \mathbf{m}) \, \mathbf{f} + \mathbf{e_0} \, \mathbf{f}   \\
& \geq 0.
\end{align*}
Note that we have $\mathbf{e} \geq |\mathbf{m}|$. Similarly define
\begin{align*}
 S_{2} \fdg &=  - 3 \mathbf{e} \, \mathbf{f} + \mathbf{f}^2 - \mathbf{m}\, \mathbf{f}  \\
& =  - (\mathbf{e} + \mathbf{m}) \, \mathbf{f} - \mathbf{e_0} \, \mathbf{f} \\
& \leq 0.
\end{align*}
Let us now introduce the quantities $\widehat{\cal{A}}$ and $\widehat{\cal{B}}$
such that 
\[  \widehat{\cal{A}} \fdg = e^{\beta} \cal{A},\,\,\,\,\widehat{\cal{B}} \fdg = e^{\beta} \cal{B}. \]
In the following we will try to estimate $L_0^2$ by $\mathbf{e}^2 - \mathbf{m}^2$.
We will use the following identities which are valid for any real numbers $a,b,c$
\begin{align*}
  (a + b + c) ^2   \leq 3(a^2 + b^2 + c^2),\,\,\,\,\frac{1}{4}(-a^2 + b^2)^2 = \frac{1}{4}(a^2 + b^2)^2 -a^2 b^2 .
\end{align*}
Consider,
\begin{align*}
 L_0^2 & \leq 3\left( \frac{1}{4}\left(-(T\phi)^2 + (R\phi)^2\right)^2 +  4  g'(\phi)^2\phi^2_r\, \mathbf{f} + \frac{1}{4}\,\mathbf{f}^2 \right) \\
&=  3 \left( \frac{1}{4} \mathbf{e_0}^2 + 4  g'(\phi)^2\phi^2_r\, \mathbf{f} + \frac{1}{4}\,\mathbf{f}^2 -\mathbf{m}^2 \right) \\
& \leq 3\left(\frac{1}{4} \mathbf{e_0}^2 + \frac{c}{2} (R\phi)^2 \, \mathbf{f} + \frac{1}{4}\,\mathbf{f}^2 -\mathbf{m}^2  \right) \\
& \leq c\left(\frac{1}{4} \mathbf{e_0}^2 + \frac{1}{2} \, \mathbf{e_0}\, \mathbf{f} + \frac{1}{4}\,\mathbf{f}^2 -\mathbf{m}^2  \right)  \\
& = c(\mathbf{e}^2 -\mathbf{m}^2 ) \\
& \leq  c \frac{\widehat{\cal{A}}^2\,\widehat{\cal{B}}^2} {r^2}
\end{align*}
where we have used the fact that both $\Vert \phi \Vert_{L^\infty}$ and $\Vert \beta \Vert_{L^\infty} \leq c$. Consequently, 
\begin{align*}
 \ptl_{{\ulin{u}}} \widehat{\cal{A}}^2 =& \frac{1}{2}e^{\beta + \cal{F}} \left(-L_0 + \kappa\, r^2\, e^{2\beta}S_1\right)  \\
\geq  & -\frac{1}{2}e^{\beta + \cal{F}}L_0.
\end{align*}
So,
\begin{align*}
\widehat{\cal{A}}\, \ptl_{\ulin{u}} \widehat{\cal{A}} \geq -c|L_0| \geq -c \frac{\widehat{\cal{A}}\,\widehat{\cal{B}}}{r}
\end{align*}
that gives us
\[  \ptl_{\ulin{u}} \widehat{\cal{A}} \geq -c\frac{\widehat{\cal{B}}}{ r}  \]
and similarly,
\[  \ptl_u \widehat{\cal{B}} \leq c \frac{\widehat{\cal{A}}} { r}. \]
The rest of the proof is comparable to the case of wave maps on the Minkowski background as in \cite{jal_tah1} and \cite{chris_tah1}.
Consider the region of spacetime $[{\ulin{u}} , 0] \times [u_0, u]$ where ${\ulin{u}},  u \leq 0$ and $u_0<0$ will be chosen later.
\begin{figure}[!hbt]
\psfrag{O}{$O$}
\psfrag{etaeqeta0}{$u$}
\psfrag{etaeqeta}{$u_0$}
\psfrag{xieqxi}{${\ulin{u}}= {\ulin{u}}$}
\psfrag{Req0}{$\varrho=0$}
\psfrag{ReqlT}{$\varrho=\lambda \vert \tau\vert$}
\psfrag{ReqT}{$\varrho=\vert \tau\vert$}
\psfrag{reqReq0}{$r=\varrho=0$}
\psfrag{teq-1}{$t=-1$}
\centerline{\includegraphics[height=2.5in]{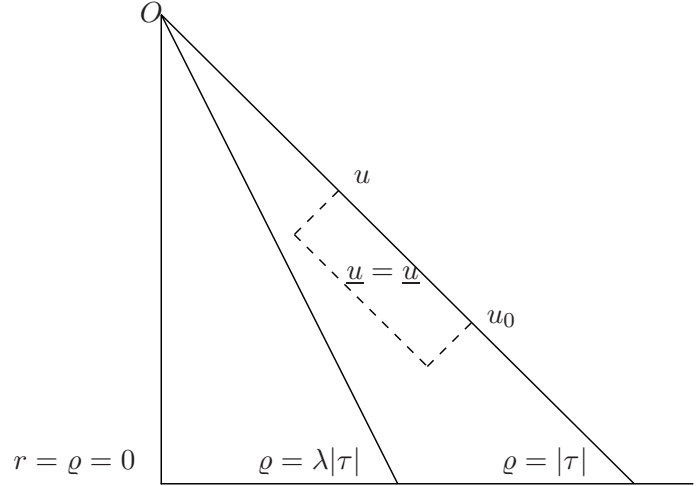}}
\caption{Application of the fundamental theorem of calculus for $\widehat{A}$ and $\widehat{B}$ in the region
$[{\ulin{u}} , 0] \times [u_0, u]$ } 
\label{fig:annular_disc_2}
\end{figure} 
Using the fundamental theorem of calculus,
\begin{align}
 \widehat{\cal{B}}({\ulin{u}},u) &= \widehat{\cal{B}}({\ulin{u}},u_0) + \int^u_{u_0} \ptl_{u'} \widehat{\cal{B}}({\ulin{u}},u')\,\, d\, u' \notag \\
&\leq \widehat{\cal{B}}({\ulin{u}},u_0) + c\int^u_{u_0} \frac{\widehat{\cal{A}}({\ulin{u}},u')}{r({\ulin{u}},u')}\,\, d\, u', \label{ftc_B} \\
\widehat{\cal{A}}({\ulin{u}},u') &= \widehat{\cal{A}}( 0,u') - \int^0 _{{\ulin{u}}} \ptl_{{\ulin{u}}'} \widehat{\cal{A}}({\ulin{u}}',u')\,\, d\, {\ulin{u}}' \notag  \\
&\leq \widehat{\cal{A}}(0,u') + c\int^0_{{\ulin{u}}} \frac{\widehat{\cal{B}}({\ulin{u}}',u')}{r({\ulin{u}}',u')}\,\, d\, {\ulin{u}}' \label{ftc_A}.
\end{align}
After plugging \eqref{ftc_A} in \eqref{ftc_B} we get
\begin{align} \label{original_B}
  \widehat{\cal{B}}({\ulin{u}},u) & \leq \widehat{\cal{B}}({\ulin{u}},u_0) + c \left(\int^u_{u_0} \frac{\widehat{\cal{A}}(0,u')}{r({\ulin{u}}, u')}\,\, d\, u'
 + \int^u_{u_0}\frac{1}{r({\ulin{u}},u')} \left(\int^0_{\ulin{u}} \frac{\widehat{\cal{B}}({\ulin{u}}',u')}{r({\ulin{u}}',u')}\,\, d\, {\ulin{u}}'\right) \,\,d\,u' \right) \notag\\
&= \widehat{\cal{B}}({\ulin{u}},u_0) + c \left(\int^u_{u_0} \frac{\widehat{\cal{A}}(0,u')}{r({\ulin{u}}, u')}\,\, d\, u' \right) 
 + c \left( \int^u_{u_0}\int^0_{\ulin{u}} \frac{\widehat{\cal{B}}({\ulin{u}}',u')}{ r({\ulin{u}}, u') r({\ulin{u}}',u')}\,\, d\, {\ulin{u}}' \,\,d\,u' \right).
\end{align}
Now consider the second term in the right hand side of \eqref{original_B}, firstly recall
\[ r({\ulin{u}},u')  \geq c \vrho ({\ulin{u}},u') = c \, \halb ({\ulin{u}}-u'), \]
\begin{align} \label{eq_before_pxi}
\int^u_{u_0} \frac{\widehat{\cal{A}}(0,u')}{r({\ulin{u}}, u')}\,\, d\, u' & \leq  \left(\int^u_{u_0} \widehat{\cal{A}}^2(0,u')\,\,d\,u' \right)^{\halb} 
\left(\int^u_{u_0} \frac{1}{({\ulin{u}}-u')^2}\,\,d\,u' \right)^{\halb} \notag \\
& \leq c \, \vert \text{Flux} (P_{T}) (u_0,u) \vert^{\halb} \left(\frac{1}{{\ulin{u}}-u} -  \frac{1}{{\ulin{u}}-u_0}\right)^{\halb} \notag \\
&\leq c\, \vert \text{Flux}(P_{T}) (u_0) \vert^{\halb}\left(\frac{1}{{\ulin{u}}-u}\right)^{\halb}.
\end{align}
We infer
\begin{align}
  \widehat{\cal{B}}({\ulin{u}},u) &\leq \widehat{\cal{B}}({\ulin{u}},u_0) + c\,\frac{\vert\text{Flux}(u_0) \vert^{\halb}}{({\ulin{u}}-u)^\halb} 
+ c  \left( \int^u_{u_0}\int^0_{\ulin{u}} \frac{\widehat{\cal{B}}({\ulin{u}}',u')}{ r({\ulin{u}}, u') r({\ulin{u}}',u')}\,\, d\, {\ulin{u}}' \,\,d\,u' \right)  .
\end{align}
Let us define the function 
$$\widehat{\cal{H}}({\ulin{u}},u) \fdg= \sup_{{\ulin{u}}\leq {\ulin{u}}' \leq 0}\sqrt{{\ulin{u}}'-u}\,\, \widehat{\cal{B}}({\ulin{u}}',u).$$ 
We have,
\[ \sqrt {{\ulin{u}}' -u'}\,\widehat{\cal{B}}({\ulin{u}}',u') \leq \sup_{{\ulin{u}} \leq {\ulin{u}}' \leq 0} \sqrt{{\ulin{u}}'-u'}\,\widehat{\cal{B}}({\ulin{u}}',u')  = \widehat{\cal{H}}({\ulin{u}},u'). \]
So,
\begin{align} \label{Hestimate}
 ({\ulin{u}}-u)^\halb \widehat{\cal{B}}({\ulin{u}},u) &\leq \left(\frac{{\ulin{u}}-u}{{\ulin{u}}-u_0}\right) ^\halb ({\ulin{u}}-u_0)^\halb \widehat{\cal{B}}({\ulin{u}},u_0) + c \vert \text{Flux}(u_0)\vert^{\halb} \notag \\
&\quad + c \left( \int^u_{u_0}\int^0_{\ulin{u}}  \widehat{\cal{H}}({\ulin{u}},u') \, \frac{({\ulin{u}}-u)^\halb}{ ({\ulin{u}}- u') ({\ulin{u}}' -u')^{3/2}}\,\, d\, {\ulin{u}}' \,\,d\,u' \right). 
\end{align}
Now consider the function $p({\ulin{u}})$ defined as follows
\[ p(\ub) \fdg = \frac{{\ulin{u}}-u}{{\ulin{u}}-u_0}.\] 
We have ${\ulin{u}}-u \leq {\ulin{u}}-u_0$ so $ p \leq 1$. Also, $p$ is increasing and hence
\begin{align} \label{pxi}
p({\ulin{u}})\leq p(0).
\end{align}
Let us go back to the inequality \eqref{Hestimate} and use \eqref{pxi}. We infer
\begin{align}
({\ulin{u}}-u)^\halb \widehat{\cal{B}}({\ulin{u}},u) & \leq \left(\frac{-u}{-u_0}\right) ^\halb ({\ulin{u}}-u_0)^\halb \widehat{\cal{B}}({\ulin{u}},u_0)+ c\vert \text{Flux}(u_0) \vert^{\halb}  \notag \\
& \quad + c \left( \int^u_{u_0}\widehat{\cal{H}}({\ulin{u}},u') \, \frac{({\ulin{u}}-u)^\halb}{ ({\ulin{u}}- u')}\,\left( \frac{1}{\sqrt{{\ulin{u}}-u'}}-\frac{1}{\sqrt{-u'}} \right)\,d\,u' \right) .
\end{align}
Consequently,
\begin{align}
\widehat{\cal{H}}({\ulin{u}},u) &\leq \left(\frac{-u}{-u_0}\right)^\halb \widehat{\cal{H}}({\ulin{u}},u_0) + c \vert \text{Flux}(u_0)\vert^{\halb} 
+ c \int^{u}_{u_0} \widehat{\cal{H}}({\ulin{u}},u') \frac{{\ulin{u}}}{u' ({\ulin{u}}-u')} \, d\, u'.
\end{align}
Also, we have
\begin{align}
 \widehat{\cal{H}}({\ulin{u}},u_0) & = \sup_{{\ulin{u}}\leq {\ulin{u}}' \leq 0}\sqrt{{\ulin{u}}'-u_0}\,\, \widehat{\cal{B}}({\ulin{u}}',u_0) \notag \\
& \leq \sup_{{\ulin{u}}\leq {\ulin{u}}' \leq 0}\sqrt{{\ulin{u}}'-u_0} \,\, \sup_{{\ulin{u}}\leq {\ulin{u}}' \leq 0} \widehat{\cal{B}}({\ulin{u}}',u_0) \notag \\
& \leq c(u_0) \sqrt{-u_0}, 
\end{align}
where we have used the fact that $u$ is regular away from the axis so that $\widehat{\cal{B}}({\ulin{u}},u_0)$ is finite. So,
\begin{align} \label{before_gronwall}
\widehat{\cal{H}}({\ulin{u}},u) &\leq c(u_0) \sqrt{-u} + c \vert\text{Flux}(u_0) \vert^{\halb}
+ c \int^{u}_{u_0} \widehat{\cal{H}}({\ulin{u}},u') \frac{{\ulin{u}}}{u' ({\ulin{u}}-u')} \, d\, u'.
\end{align}
Using Gronwall's lemma to obtain an estimate for $ \widehat{\cal{H}}({\ulin{u}},u)$, we infer 
\begin{align}
 \widehat{\cal{H}}({\ulin{u}},u) & \leq  \sqrt{-u} c(u_0) + c \, \vert\text{Flux}(u_0) \vert^{\halb} \notag \\
& \quad + c \int^{u}_{u_0} \left(  \sqrt{-u} c(u_0) + c \, \vert 
\text{Flux}(u_0) \vert^\halb \right) \left( \frac{{\ulin{u}}}{u'({\ulin{u}}-u')} \right) e^{\int^u_{u'}\frac{{\ulin{u}}}{u'' ({\ulin{u}} -u'')} \, d\, u''} \, d\,u'.
\end{align}
We have for $u_0 \leq u' \leq u $ and setting ${\ulin{u}} = \lambda' u$ where $ \lambda' \fdg = \frac{1-\lambda}{1+ \lambda} < 1$
\begin{align*}
 \int^u_{u'} \frac{{\ulin{u}}}{u'' ({\ulin{u}} -u'')} d\, u'' = & \log \frac{u(\lambda' u -u')}{u' (\lambda' u -u)} \\
\leq & \log \frac{1}{1-\lambda'}.
\end{align*}
For any $\epsilon > 0$ we can choose an $u_0$ small enough such that 
$$c\vert \text{Flux}(u_0) \vert^{\halb}< \frac{\epsilon}{2}.$$ 
Furthermore, one can choose $u \in (u_0, 0)$ small enough such that   
$$c(u_0)\sqrt{-u} < \frac{\epsilon}{2}.$$
So we have $\widehat{\cal{H}}({\ulin{u}},u) < c \epsilon$ for $u_0 <u<0$ small enough.
Then, 
$$\widehat{\cal{B}}({\ulin{u}},u) \leq \frac{\widehat{\cal{H}}({\ulin{u}},u)}{\sqrt{{\ulin{u}}-u}} \leq \frac{c\epsilon}{\sqrt{{\ulin{u}}-u}}.$$
Now going back to the flux integrals in \eqref{div_tube}, we have
\begin{align}
\int_{\ptl \cal{S}_1} e^{-\cal{G}}(\mathbf{e+m}) \bar{\mu}_u &\leq c \int^0_{\ulin{u}} \widehat{\cal{B}}({\ulin{u}}',u)^2 \,d\,{\ulin{u}}' \notag \\
  &\leq \epsilon \int^0_{\ulin{u}} \frac{1}{({\ulin{u}}'-u)}\,\, d\, {\ulin{u}}' \notag \\
 & =  \eps  \int^0_{\lambda' u} \frac{1}{({\ulin{u}}'-u)}\,\, d\, {\ulin{u}}' \notag \\
& = \eps \log \frac{1}{1-\lambda'} \notag \\
& < c \eps \label{eq:estimateonS1}
\end{align}
and
\begin{align}\label{eq:estimateonS2}
 \halb \int_{\ptl \cal{S}_2} r\, \Omega^2\,e^{-\mathcal{F}} (\mathbf{e-m})\,\,  d\, u \wedge d\, \theta &= \text{Flux} (P_T)(u_0,u)  < \eps
\end{align}
for $u_0, u$ small enough.
Finally, in view of \eqref{div_tube}, \eqref{eq:estimateonS1} and \eqref{eq:estimateonS2}, we conclude that $E^O_{\text{ext}} (\uptau) \to 0$ as $\uptau \to 0$. This concludes the proof of Lemma \ref{ann}.
\end{proof}

\subsection{Local spacetime integral estimates}

\begin{lem}[Non-concentration of integrated kinetic energy]\label{kin}
Let the kinetic energy be defined as 
\[\mathbf{e}_{\text{kin}} \fdg = \halb e^{-2\alpha} \phi_t^2 \]
then the spacetime integral of $\mathbf{e}_{\text{kin}}$ does not concentrate in the past null cone
of $O$, i.e.,
\begin{align*}
\frac{1}{r_2(\uptau)}\int_{K_\uptau}  \mathbf{e}_{\text{kin}} \,\, \bar{\mu}_g  \to 0 \,\, \text{as} \,\, \uptau \to 0
\end{align*}
where $r_2(\uptau)$ is the radial function defined as in Lemma \ref{ann}.
\end{lem}

\begin{proof}
Recall from \eqref{intro_px3} the computation of the vectorfield $P_{\mathcal{R}_1}$ and its divergence
\begin{align*}
P_{\mathcal{R}_1} =& -re^{(1-k)\beta - \alpha} \,\mathbf{m}\, \ptl_t + re^{-k\beta}(\mathbf{e} -\mathbf{f})\ptl_r,\\
\nabla_\nu  P^\nu_{\mathcal{R}_1} = & \,e^{- 2 \alpha} \phi_t^2.
\end{align*}
for the choice of $k=0.$
Using  Stokes' theorem as in \eqref{stokes_gen} for $P_{\mathcal{R}_1}$
\begin{align}\label{stokes_px3}
 \int_{K(\uptau,s)} \nabla_\nu  P^\nu_{\mathcal{R}_1}\, \bar{\mu}_g  & = \int_{\Sigma^O_{s}} e^{\alpha}  P^t_{\mathcal{R}_1} 
\, \bar{\mu}_q
- \int_{\Sigma^O_{\uptau}}e^{\alpha} P^t_{\mathcal{R}_1}  \, \bar{\mu}_q
+ \text{Flux}(P_{\mathcal{R}_1}) (\uptau,s) \notag 
\intertext{that is}
 \int_{K(\uptau,s)} e^{- 2 \alpha} \phi_t^2 \,\bar{\mu}_g  & = - \int_{\Sigma^O_{s}} r  e^{\beta} \mathbf{m} \, \bar{\mu}_q
+ \int_{\Sigma^O_{\uptau}} r  e^{\beta} \mathbf{m} \, \bar{\mu}_q
+ \text{Flux}(P_{\mathcal{R}_1}) (\uptau,s) 
\end{align}
where,
\begin{align*}
 \text{Flux}(P_{\mathcal{R}_1}) (\uptau,s) = & \int_{C(\uptau,s)} d\,{\ulin{u}}(P_{\mathcal{R}_1})\, \bar{\mu}_{\ulin{u}} \\
 =& \,\int_{C(\uptau,s)}r e^{\beta-\cal{F}} \mathbf{(e-m-f)}  \, \bar{\mu}_{\ulin{u}} \\
 \leq &\, c \, r_2(\uptau) \int_{C(\uptau,s)} (\mathbf{e}-\mathbf{m}) \, \bar{\mu}_{\ulin{u}} \\
= &\, -c \, r_2(\uptau) \text{Flux}(P_{T}) (\uptau,s). 
\end{align*}
We infer
\begin{align*}
 \int_{K(\uptau,s)} e^{- 2 \alpha} \phi_t^2 \,  \bar{\mu}_g  & \leq  \int_{\Sigma^O_{s}} r  e^{\beta} \mathbf{e} \,   \bar{\mu}_q
+ \int_{\Sigma^O_{\uptau}} r  e^{\beta} \mathbf{e} \,\, \bar{\mu}_q
- c \, r_2(\uptau) \text{Flux}(P_{T}) (\uptau,s) \\
& \leq c r_2(s)\int_{\Sigma^O_s} \mathbf{e} \, \bar{\mu}_q
+ \int_{\Sigma^O_{\uptau}} r  e^{\beta} \mathbf{e} \, \bar{\mu}_q
- c \, r_2(\uptau) \text{Flux}(P_{T}) (\uptau,s).
\end{align*}
Now let $s \to 0$. We get
\begin{align*}
\frac{1}{ r_2(\uptau) } \,\int_{K(\uptau)} e^{-2 \alpha} \phi_t^2 \,  \bar{\mu}_g  &\leq \, \frac{1}{ r_2(\uptau) }\, \int_{\Sigma^O_{\uptau}} r  e^{ \beta} \mathbf{e} \, \bar{\mu}_q
- c \,\text{Flux}(P_{T}) (\uptau).
\intertext{Therefore,}
\frac{1}{ r_2(\uptau) } \int_{K(\uptau)} e^{- 2 \alpha} \phi_t^2 \, \bar{\mu}_g & \leq  \frac{1}{ r_2(\uptau) } \int_{B_{r_2}(\uptau)} r  e^{ \beta} \mathbf{e} \,\bar{\mu}_q
- c \,\,  \text{Flux}(P_{T}) (\uptau) \\
& = \, \frac{1}{ r_2(\uptau) } \left( \int_{B_{r_1(\uptau)}} r  e^{ \beta} \mathbf{e} \, \bar{\mu}_q + 
\int_{B_{r_2(\uptau)}\setminus B_{r_1(\uptau)}} r  e^{ \beta} \mathbf{e} \, \bar{\mu}_{q} \right) \\
& \quad - c\,  \text{Flux}(P_{T}) (\uptau) \\
& \leq c \,\lambda\, E_0 +  c\,  E^O_{\text{ext}} (\uptau) - c\,  \text{Flux}(P_{T}) (\uptau).
\end{align*}
For any $\eps > 0$ we can choose $\lambda$ small enough so that the first term $ < \frac{\eps}{3}$, then we can make $\uptau$ small enough
so that $E^{O}_{\text{ext}} (\uptau) < \frac{\eps}{3} $ and $ |\text{Flux}(P_{T}) (\uptau)| < \frac{\eps}{3} $
as discussed previously. This concludes the proof of Lemma \ref{kin}.
\end{proof}

In Lemma \ref{kin} we proved that the spacetime integral of $e^{-2\alpha}\phi_t^2$ does not concentrate in the past null cone of $O$.
In the following lemma we shall prove that the spacetime integral of rotational potential energy i.e.,
\[  \int_{K_\uptau} \frac{g^2(\phi)}{r^2} \bar{\mu}_g = \int_{K_\uptau} \mathbf{f} \bar{\mu}_g \]
does not concentrate.

\begin{lem}[Non-concentration of integrated rotational potential energy]\label{fterm}
Let $(N,h)$ be the target manifold satisfying the Grillakis condition \eqref{eq:grillakis-cond-first}. 
Then the spacetime integral of rotational potential energy does not concentrate i.e.,
 \begin{align}
 \int_{K_\uptau} \mathbf{f} \,\bar{\mu}_g \to 0 \,\,\,\text{as}\,\,\, \uptau \to 0.
 \end{align}
\end{lem}

\begin{proof}
Recall from \eqref{intro_px4} the computation of the vectorfield $P_{\mathcal{R}_a}$ and its divergence
\begin{align*}
 P_{\mathcal{R}_a} & = -e^{\beta -\alpha} r^a \, \mathbf{m} \ptl_t + r^a \,(\mathbf{e}-\mathbf{f}) \ptl_r,\\
\nabla_\nu  P^\nu_{\mathcal{R}_a} &= \halb \left (  (1+a) r^{a-1}  \right) e^{-2\alpha} \phi^2_t
+ \halb \left ( (a-1) r^{a-1}  \right) e^{-2\beta} \phi^2_r \\
& \quad + \halb \left ( (1-a) r^{a-1}  \right) \frac{g^2(\phi)}{r^2}.
 \end{align*}
Let now us construct the vector $P_\zeta^\nu$ such that
\begin{align*}
 P_\zeta^\nu \fdg = \zeta \phi^\nu \phi - \zeta^\nu \frac{\phi^2}{2},
\end{align*}
where $\zeta \fdg = \frac{1-a}{2}r^{a-1}$ for $a \in (\halb, 1)$. 
Then the divergence is given by
\begin{align*}
 \nabla_\nu  P^\nu_\zeta &= \nabla_\nu (\zeta \phi^\nu \phi) - \nabla_\nu \left(\zeta^\nu \frac{\phi^2}{2}\right) \\
&= \zeta (\square \phi)\phi + \zeta \phi^\nu \phi_\nu + \phi^\nu \zeta_\nu \phi - ( \square \zeta) \frac{\phi^2}{2} - \zeta^\nu \phi \phi_\nu \\
&= \zeta \frac{g(\phi)g'(\phi)\phi}{r^2} + \zeta \phi^\nu \phi_\nu - ( \square \zeta) \frac{\phi^2}{2}
\end{align*}
and
\begin{align*}
 \square \zeta &= e^{-2\beta} \left(\zeta_{rr} + \frac{\zeta_r}{r} + (\alpha_r -\beta_r)\zeta_r \right) \\
&= e^{-2\beta}r^{a-3}\frac{(1-a)^2}{2} \left(1-a + r^2 \kappa e^{2\beta}\, \mathbf{f} \right). 
\end{align*}
Let us define a vector $P^\nu_{\text{tot}}$ such that
\begin{align*}
 P^\nu_{\text{tot}} \fdg =P^\nu_{\mathcal{R}_a} + P^\nu_\zeta. 
\end{align*}
We have 
\begin{align}
 \nabla_\nu  P^\nu_{\text{tot}} &= \nabla_\nu  P^\nu_{{\mathcal{R}_a}}+ \nabla_\nu  P^\nu_\zeta \nonumber \\
&= \zeta \frac{g(\phi)g'(\phi)\phi}{r^2} + a r^{a-1} e^{-2\alpha} \phi_t^2 + \zeta \,\mathbf{f} - e^{-2\beta}\frac{(1-a)^2}{4} r^{a-1} \left(1-a 
+ r^2 \kappa  e^{2\beta} \mathbf{f} \right) \frac{\phi^2}{r^2} \nonumber \\
&= \zeta \left [ \frac{1}{r^2} B_1 + B_2 \right ] , \label{eq:Ptotdiv} 
\end{align}
where 
\begin{subequations}\label{eq:B12}
\begin{align} 
B_1 &= g(\phi) g'(\phi) \phi + g^2(\phi) - e^{-2\beta} \frac{(1-a)^2}{2} \phi^2 - \frac{1-a}{2} \kappa g^2(\phi) \phi^2 \\ 
B_2 &= \frac{2a}{1-a} e^{-2\alpha} \phi_t^2 
\end{align} 
\end{subequations} 
Applying  Stokes' theorem on $K(\uptau,s),$
\begin{align}\label{stokes_px4}
 \int_{K(\uptau,s)} \nabla_\nu  P_{\text{tot}} ^\nu\,  \bar{\mu}_g  = \int_{\Sigma^O_{s}} e^{\alpha} \, P_{\text{tot}}^{\,t} \,\bar{\mu}_q
- \int_{\Sigma^O_{\uptau}}e^{\alpha} \,P_{\text{tot}} ^{\,t}  \,\bar{\mu}_q
+ \text{Flux}(P_{\text{tot}}) (\uptau,s). 
\end{align}
We have
\begin{align} \label{px4_s}
 \int_{\Sigma^O_{s}} e^{\alpha} \, P_{\text{tot}}^{\,t} \, \bar{\mu}_q &= 
 \int_{\Sigma^O_{s}}( -\mathbf{m}\, r^a e^{\beta} +e^{\alpha} \zeta\, \phi^t \, \phi) \,\bar{\mu}_q \notag\\
&\leq \int_{\Sigma^O_{s}} \mathbf{e}\,r^a e^{\beta} +|  \, e^{-\alpha} \, \phi_t|\, |\zeta\,\phi| \,\bar{\mu}_q, \notag 
\intertext{and applying the Cauchy-Schwarz inequality, we get}
\int_{\Sigma^O_{s}} e^{\alpha} \, P_{\text{tot}}^{\,t} \, \bar{\mu}_q
&\leq c\,r^a_2(s)\int_{\Sigma^O_{s}} \mathbf{e} \,\bar{\mu}_q   
+ \frac{1-a}{2}\left( \int_{\Sigma^O_{s}} e^{-2\alpha} \, \phi^2_t r^{2a}\, 
\bar{\mu}_q \right)^\halb \left(\int_{\Sigma^O_{s}} \frac{\phi^2}{r^2}\,   \bar{\mu}_q \right)^\halb \notag \\
& \leq  c\,r^a_2(s). 
\end{align}
Similarly, the second term in \eqref{stokes_px4} can be estimated as 
\begin{align}\label{px4_tau}
-\int_{\Sigma^O_{\uptau}} e^{\alpha} \, P_{\text{tot}}^{\,t} \,  \bar{\mu}_q \leq c\, \int_{\Sigma^O_{\uptau}} \mathbf{e}\,r^a\,  \bar{\mu}_q 
+ c \left( \int_{\Sigma^O_\uptau} \mathbf{e}\,r^{2a}\,\bar{\mu}_q \right)^{\halb}.
\end{align}
The flux of $P_{\text{tot}}$ though the null surface $C(\uptau,s)$ can be decomposed as
\begin{align}\label{ptot_genflux}
 \text{Flux}(P_{\text{tot}}) (t,s) & =  \text{Flux}(P_{\mathcal{R}_a}) (t,s)+ \text{Flux}(P_{\zeta}) (t,s).
\end{align}
It follows from our previous estimates that exists a real constant $c$ dependent only
on the initial energy $E_0$ such that 
\begin{align} \label{eq:phi-gphi}
 \frac{1}{c} \phi^2 \leq g^2(\phi) \leq c \phi^2
\end{align}
(see for example \cite[Eq. (2.11)]{jal_tah1}). 
Let us consider the terms in the right side of \eqref{ptot_genflux} individually.
We have
\begin{align} 
 \text{Flux}(P_{{\mathcal{R}_a}})(\uptau,s) &= \int_{C(\uptau,s)} d\, {\ulin{u}} (P_{{\mathcal{R}_a}}) 
 \bar{\mu}_{\ulin{u}} \notag\\
&= \int_{C(\uptau,s)} e^{\beta-\cal{F}} r^a (\mathbf{e}-\mathbf{m}-\mathbf{f}) \bar{\mu}_{\ulin{u}} \notag \\
& \leq -c r^a_2(\uptau)\,\text{Flux}(P_{T})(\uptau,s) \notag
\intertext{and}
\text{Flux}(P_{\zeta})(\uptau,s) &= \int_{C(\uptau,s)} d\, {\ulin{u}} (P_{\zeta}) \bar{\mu}_{\ulin{u}} \notag \\
& =  \int_{C(\uptau,s)}  \left( \phi \, \zeta\, e^{-\cal{F}}  \left(-T(\phi) + R(\phi)\right) 
+  \halb \zeta e^{-(\beta +\cal{F})} (1-a)r^{-1} \phi^2  ) \right)\,\bar{\mu}_{\ulin{u}} \notag\\
&=  \int_{C(\uptau,s)}  \left( \phi\,\zeta\,e^{-\cal{F}} \left(-T(\phi) + R(\phi)\right) 
+  e^{-(\beta +\cal{F})} \frac{(1-a)^2}{4} \,\frac{ \phi^2}{r^2}r^a   \right)\,\bar{\mu}_{\ulin{u}} \notag\\
&\leq  \int_{C(\uptau,s)}  \left( \phi\,\zeta\,e^{-\cal{F}} \left(-T(\phi) + R(\phi)\right) 
+  c\, \frac{(1-a)^2}{4} \mathbf{f}\, r^a \,e^{-\cal{F}} \right)\,\bar{\mu}_{\ulin{u}} \notag\\
&\leq \int_{C(\uptau,s)}  \left( \phi\,\zeta\,e^{-\cal{F}} \left(-T(\phi) + R(\phi)\right) 
+  c\, \frac{(1-a)^2}{2} (\mathbf{e-m})\, r^a \,e^{-\cal{F}} \right)\,\bar{\mu}_{\ulin{u}}. 
\intertext{Using the Cauchy-Schwarz inequality, \eqref{ptot_flux} can be estimated as}
\text{Flux}(P_{\zeta})(\uptau,s) & \leq  
c r^a_2(\uptau) \left(\int_{C\uptau,s)} (\mathbf{e} - \mathbf{m} ) \bar{\mu}_{\ulin{u}} \right)\notag\\
& \leq  - c r_2^a(\uptau) \text{Flux}(P_T)(\uptau,s).
\end{align}
Therefore,
\begin{align}\label{ptot_flux}
\text{Flux}(P_{\text{tot}}) (t,s) &\leq - c r_2^a(\uptau) \text{Flux}(P_T)(\uptau,s).
\end{align}
Using \eqref{eq:phi-gphi} and proceeding as in \cite{jal_tah}, we can use the Grillakis condition \eqref{eq:grillakis-cond-first} to conclude that for 
$a$ sufficiently close to $1$ there is a small constant $c_a > 0$ such that 
$$
\frac{B_1}{r^2} + B_2 \geq c_a \left ( e^{-2\beta} \frac{\phi^2}{r^2} + e^{-2\alpha} \phi_t^2 \right ) 
$$
Now, if we go back to Stokes' theorem \eqref{stokes_px4} and use the estimates \eqref{px4_s}, \eqref{px4_tau} and 
\eqref{ptot_flux}, we get
\begin{align*}
 &c_a \left [ \int_{K(\uptau,s)} e^{-2\alpha} \phi^2_t r^{a-1} d\bar{\mu}_g  +  \int_{K(\uptau,s)} e^{-2\beta} \frac{\phi^2}{r^2} r^{a-1} \,\bar{\mu}_g \right ] \\
&\leq c \, r^a_2(s) + c \, \int_{\Sigma^O_{\uptau}} \mathbf{e}\,r^a\,  \bar{\mu}_q 
+ c\left(\int_{\Sigma^O_{\uptau}} \mathbf{e}\,r^{2a}\,  \bar{\mu}_q\right)^{\halb}  -c \, r^a_2(\uptau) \text{Flux}(P_T)(\uptau,s).
\end{align*}
As $s \to 0$ we get,
\begin{align}\label{simpler_stokes_ptot}
  &c_a \left [  \int_{K(\uptau)} e^{-2\alpha} \phi^2_t r^{a-1} \,\bar{\mu}_g   +  \int_{K(\uptau)} e^{-2\beta} \frac{\phi^2}{r^2} r^{a-1} \,\bar{\mu}_g  \right ] \nonumber \\
&\leq c \, \int_{\Sigma^O_{\uptau}} \mathbf{e}\,r^a \, \bar{\mu}_q + c \left(\int_{\Sigma^O_{\uptau}} \mathbf{e}\,r^{2a} \, \bar{\mu}_q\right)^{\halb} -c\,r^a_2(\uptau)\text{Flux}(P_T)(\uptau).
\end{align}
In \eqref{simpler_stokes_ptot}, we can estimate
\begin{align} \label{energy_ra}
r^{-a}_2(\uptau)\int_{\Sigma^O_{\uptau}} \mathbf{e}\,r^a\,  \bar{\mu}_q 
&= r^{-a}_2(\uptau)\left( \int_{B_{r_1(\uptau)}} \mathbf{e}\,r^a\,\bar{\mu}_q + \int_{B_{r_2(\uptau)} \setminus B_{r_1(\uptau)}} 
\mathbf{e}\,r^a\,\bar{\mu}_q  \right) \nonumber\\
 &\leq r^{-a}_2(\uptau)\left(r^a_1(\uptau) \int_{B_{r_1(\uptau)}} \mathbf{e}\,\bar{\mu}_q + r^a_2(\uptau) \int_{B_{r_2(\uptau)} \setminus B_{r_1(\uptau)}} 
\mathbf{e}\,\bar{\mu}_q  \right) \nonumber\\
 & \leq \lambda^a \, E_0 +  \,E^O_{\text{ext}} (\uptau)
\end{align}
and
\begin{align}\label{energy_ra_sqrt}
 r^{-a}_2(\uptau)\left(\int_{\Sigma^O_{\uptau}} \mathbf{e}\,r^{2a} \, \bar{\mu}_q\right)^{\halb} 
 &= r_2^{-a}(\uptau) \left(\int_{B_{r_1(\uptau)}} \mathbf{e}\,r^{2a}\,\bar{\mu}_q + \int_{B_{r_2(\uptau)} \setminus B_{r_1(\uptau)}} \mathbf{e}\, r^{2a}\,\bar{\mu}_q \right)^{\halb} \nonumber\\
 &\leq \left(\left(\frac{r_1(\uptau)}{r_2(\uptau)}\right)^{2a}\int_{B_{r_1(\uptau)}} \mathbf{e}\,\bar{\mu}_q 
 + \int_{B_{r_2(\uptau)}} \mathbf{e}\,\bar{\mu}_q \right)^{\halb} \nonumber\\
 &\leq \left( \lambda^{2a} E_0+ E^O_{\text{ext}}(\uptau) \right)^{\halb}.
\end{align}
Hence, in view of \eqref{energy_ra}, \eqref{energy_ra_sqrt}, Corollary \ref{flux_px1_zero} and Lemma \ref{ann},
we can choose $\lambda$ and $\uptau$ in \eqref{simpler_stokes_ptot} small enough so that
\[ \frac{1}{r^a_2(\uptau)}\int_{K(\uptau)}\frac{\phi^2}{r^2} r^{a-1} \,\bar{\mu}_g < \eps \]
for any $\eps >0$.
In view of \eqref{eq:phi-gphi},
\[  \mathbf{f} \leq \frac{c\phi^2}{r^2}.\]
Therefore it follows that
 \begin{align*}
 \frac{1}{r^a_2(\uptau)}\int_{K_\uptau} \mathbf{f}\,\,r^{a-1}\, \,\bar{\mu}_g \to 0 \,\,\,\text{as}\,\,\, \uptau \to 0.
 \end{align*}
This concludes the proof of Lemma \ref{fterm}.
\end{proof}

The remaining term in the energy is $e^{-2\beta}\phi_r^2$. We control it below.
\begin{cor}\label{pot}
Under the assumptions of Lemma \ref{fterm}, the spacetime integral of radial potential energy in the past null cone of $O$ does not concentrate
\begin{align}
 \frac{1}{r^a_2(\uptau)}\int_{K_{\uptau}} e^{-2\beta}\phi^2_r r^{a-1} \,\bar{\mu}_g \to 0\,\,\, \text{as}\,\,\, \uptau \to 0.
 \end{align}
\end{cor}

\begin{proof}
Let us again apply Stokes' theorem for the $\bar{\mu}_g$-divergence of $P_{{\mathcal{R}_a}}$
\begin{align*}
 \int_{K(\uptau,s)} \nabla_\mu  P_{{\mathcal{R}_a}} ^\mu\, \bar{\mu}_g  = \int_{\Sigma^O_{s}} e^{\alpha} \, P^t_{{\mathcal{R}_a}} \, \bar{\mu}_q
- \int_{\Sigma^O_{\uptau}}e^{\alpha} \,P^t_{{\mathcal{R}_a}} \, \bar{\mu}_q
+ \text{Flux}(P_{{\mathcal{R}_a}}) (\uptau,s) 
\end{align*}
therefore, as $s \to 0$
\begin{align*}
 \int_{K(\uptau)} e^{-2\beta} \phi^2_r \,r^{a-1}\, \bar{\mu}_g &\leq c \int_{K(\uptau)} \left(e^{-2\alpha} \phi^2_t + \frac{g^2(\phi)}{r^2} \right)r^{a-1}\, \bar{\mu}_g 
  + \int_{\Sigma^O_{\uptau}}\mathbf{e}\,r^a  \, \bar{\mu}_q - r^a_2(\uptau)\text{Flux}(P_{T}) (\uptau). 
 \end{align*}
Hence,
\begin{align*}
\frac{1}{r^a_2(\uptau)}\int_{K(\uptau)} e^{-2\beta} \phi^2_r \,r^{a-1}\, \bar{\mu}_g < \eps
\end{align*}
for $\uptau$ small enough. This concludes the proof of the corollary.
\end{proof}

\subsection{Proof of non-concentration of energy}

We are now in position to conclude the proof of Theorem \ref{theorem_ener_nonconc}. 
If we collect the terms from Lemmas \ref{kin}, \ref{fterm} and Corollary \ref{pot}, we get
\[  \frac{1}{r^a_2(\uptau)} \int_{K(\uptau)} \mathbf{e} \,r^{a-1} \, \bar{\mu}_g \to 0 \]
as $\uptau \to 0.$
But then,
\begin{align}\label{integrated_energy}
 \frac{1}{r^a_2(\uptau)} \int_{K(\uptau)} \mathbf{e}\, r^{a-1} \, \bar{\mu}_g
& \geq c\,\frac{1}{r^a_2(\uptau)} \int_{K(\uptau)} \mathbf{e} \, r^{a-1} \, \bar{\mu}_q \,d\,t \nonumber \\
&\geq c \frac{1}{r_2(\uptau)} \int_{K(\uptau)} \mathbf{e} \, \bar{\mu}_q \,d\,t \notag \\
&\to 0
\end{align}
as $\uptau \to 0$. We claim that there exists a sequence $ \{\uptau_i\} _i$ such that 
\begin{align}\label{energy_noncon_claim}
 \int_{\Sigma^O_{\uptau_i}} \mathbf{e}\, \bar{\mu}_q \to 0
\end{align}
as $ \{ \uptau_i\} _i \to 0 $.
Let us prove the claim by contradiction. Suppose there exists no sequence
such that \eqref{energy_noncon_claim} holds true. Then there exists an $\eps > 0$ 
such that 
\[ \int_{\Sigma^O_{\uptau}} \mathbf{e} \,\bar{\mu}_q > \eps \]
for all $\uptau \in (-1,0).$
Consequently,
\[ \frac{1}{|\uptau|}\int_{K(\uptau)} \mathbf{e} \,\bar{\mu}_q \, d\,t > \eps.  \]
This implies,
\begin{align}\label{normalized_spatial}
 \frac{1}{r_2(\uptau)} \int_{K(\uptau)} \mathbf{e} \,  \bar{\mu}_q \,d\,t > \eps
\end{align}
for all $\uptau \in [-1,0)$.
This contradicts \eqref{integrated_energy}.
Hence, there exists a $ \{\uptau_i\} _i$ such that 
\begin{align}
E^O (\uptau_i) =  \int_{\Sigma^O_{\uptau_i}} \mathbf{e}\,\bar{\mu}_q \to 0.
\end{align}
But $ E^O (\uptau)$ is monotonic with respect to $\uptau$, therefore
\[ E^O (\uptau) \to 0\]
for all $\uptau \to 0$ i.e., $E^O_{\text{conc}} =0$. This concludes the proof of Theorem \ref{theorem_ener_nonconc}.

\section{Global regularity of the $2+1$ Einstein-wave map problem}\label{sec:proofmaintheorem}

We now proceed to the proof of our main theorem, i.e. Theorem \ref{thm:main-first}. Let $(M, \gg, \phi)$ be the maximal Cauchy development of an asymptotically flat and regular Cauchy data set for the $2+1$ equivariant Einstein-wave map problem \eqref{eq:ein-equiv-wm-first} with target $(N, h)$. Assume that the metric $h$ has the form 
$$
h = d\rho^2 + g^2(\rho) d\theta^2 
$$
for an odd function $g: \Re \to \Re$ with $g'(0) = 1$, such that $g$ satisfies \eqref{eq:nospherecond-first} and the Grillakis condition \eqref{eq:grillakis-cond-first}. 

Our goal is to prove that $(M, \gg, \phi)$ is regular. Assume by contradiction that this does not hold. Then, there exists a first singularity which in view of Theorem \ref{thm:singax} occurs on the axis of symmetry $\Gamma$. Let us denote by $O$ this first singularity which corresponds to $u=\ub=0$ and $t=r=0$ in the $(u,\ub)$ and $(t,r)$ coordinates systems constructed respectively in Section \ref{sec:prel} and Section \ref{sec:tr-coord}. 

Let $\ep>0$ small to be chosen later. In view of Theorem \ref{theorem_ener_nonconc}, there exists a time $t_0<0$ such that
\begin{equation}\lab{eq:smallenergyassumption}
E^O(t_0)\leq \ep.
\end{equation}

\begin{lemma}\lab{lemma:controloffluxsmall}
Recall that $\tau=(\ub+u)/2$. There exists $\tau_0<0$ such that in the space-time region 
$$\{\tau_0\leq\tau\leq 0\}\cap \{\ub\leq 0\},$$
we have
$$\int_{2\tau_0-\ub}^{\ub} \left((\pr_u\phi)^2+\frac{g(\phi)^2}{r^2}\right)(u',\ub)du'\les\ep$$ 
and
$$\int_{\max(u,2\tau_0-u)}^0 \left((\pr_{\ub}\phi)^2+\frac{g(\phi)^2}{r^2}\right)(u,\ub')d\ub'\les\ep.$$
\end{lemma}

\begin{proof}
We choose $\tau_0<0$ small enough such that
$$\{\tau_0\leq\tau\leq 0\}\cap \{\ub\leq 0\}\subset \{t_0\leq t\leq 0\}\cap J^-(O)$$
which is possible since on the one hand $J^-(O)=\{\ub\leq 0\}$, and on the other hand $t$ and $\tau$ are comparable in view of Corollary \ref{cor:comprttauvarho}. In particular, together with \eqref{eq:smallenergyassumption}, Stokes' theorem, and the fact that the vectorfield  $P_T$ is divergence free, we infer
$$\left|\int_{2\tau_0-\ub}^{\ub}\int_{\theta=0}^{2\pi}e^{-\cal{F}}(\mathbf{e-m})  \bar{\mu}_{\ulin{u}}\right|\leq\ep$$
and 
$$\left|\int_{\max(u,2\tau_0-u)}^0\int_{\theta=0}^{2\pi}e^{-\cal{G}}(\mathbf{e+m})  \bar{\mu}_u\right|\leq\ep$$
where we are relying on notations and computations introduced in Section \ref{sec:J-est}. In view of the definition of $\bar{\mu}_{\ulin{u}}$ and $\bar{\mu}_u$, and the rotation invariance, we infer
\bee
\int_{2\tau_0-\ub}^{\ub}e^{-\cal{F}}(\mathbf{e-m}) r\Omega^2 du&\les&\ep,\\
\int_{\max(u,2\tau_0-u)}^0e^{-\cal{G}}(\mathbf{e+m})  r\Omega^2d\ub&\les&\ep.
\eee
In view of the definition of ${\bf e}$ and ${\bf m}$ and the identity \eqref{relationOmFG}, we deduce
\bee
\int_{2\tau_0-\ub}^{\ub}e^{\cal{G}}\left(e^{-2\mathcal{G}}(\pr_u\phi)^2+\frac{g(\phi)^2}{r^2}\right) r du&\les&\ep,\\
\int_{\max(u,2\tau_0-u)}^0e^{\cal{F}}\left(e^{-2\cal{F}}(\pr_{\ub}\phi)^2+\frac{g(\phi)^2}{r^2}\right)  rd\ub&\les&\ep.
\eee
Together with the estimates of Lemma \ref{null_uniform} for $\cal{F}$ and $\cal{G}$, we obtain
\bee
\int_{2\tau_0-\ub}^{\ub}\left((\pr_u\phi)^2+\frac{g(\phi)^2}{r^2}\right) r du&\les&\ep,\\
\int_{\max(u,2\tau_0-u)}^0\left((\pr_{\ub}\phi)^2+\frac{g(\phi)^2}{r^2}\right)  rd\ub&\les&\ep.
\eee
This concludes the proof of the lemma. 
\end{proof}

For any $a>0$, performing the scaling transformation
$$r\to ar,\,\,\,\, \Omega\to \Omega,\,\,\,\, \phi\to \phi,$$
as well as rescaling the coordinates $u$ and $\ub$ accordingly
$$u\to au,\,\,\,\, \ub\to a\ub$$
leads to another solution of the the Einstein-wave map problem. Using this rescaling with $a=|\tau_0|^{-1}$ implies in view of Lemma \ref{lemma:controloffluxsmall}
\bea
\int_{-2-\ub}^{\ub} \left((\pr_u\phi)^2+\frac{g(\phi)^2}{r^2}\right)r(u',\ub)du'&\leq&\ep, \label{smallboundflux1}\\
\int_{\max(u,-2-u)}^0 \left((\pr_{\ub}\phi)^2+\frac{g(\phi)^2}{r^2}\right)r(u,\ub')d\ub'&\leq&\ep. \label{smallboundflux2}
\eea
over the space-time region $\{-1\leq\tau\leq 0\}\cap \{\ub\leq 0\}$. 

\begin{theorem}[Small energy implies regularity]\label{th:smallenergyglobalex}
Let $(M,\gg,\Phi)$ be a solution of the $2+1$ equivariant Einstein-wave map problem \eqref{eq:ein-equiv-wm-first} which is regular in the space-time region
$$\{-1\leq\tau<0\}\cap \{\ub\leq 0\}.$$
Assume furthermore the smallness condition \eqref{smallboundflux1} \eqref{smallboundflux2} on the energy flux. Then, $(M,\gg,\Phi)$ is regular on the closure of $\{-1\leq\tau<0\}\cap \{\ub\leq 0\}$. In particular, there is no singularity at $O$.
\end{theorem}

In view of Theorem \ref{th:smallenergyglobalex}, we infer that $O$ can not be a first singularity of $(M, \gg, \Phi)$, hence contradicting our assumption. Thus, global regularity holds for $(M, \gg, \Phi)$. This concludes the proof of Theorem \ref{thm:main-first}.\\

The rest of the paper is devoted to the proof of Theorem \ref{th:smallenergyglobalex}. In Section \ref{sec:improved-uniform-phi} we derive a uniform weighted upper bound for $\phi$. In Section \ref{sec:improved-uniform-dphi}, we rely on the upper bound of Section \ref{sec:improved-uniform-phi} to derive a uniform upper bound for $\pr\phi$. Finally, we rely on the upper bound of Section \ref{sec:improved-uniform-dphi} to conclude the proof of Theorem \ref{th:smallenergyglobalex}.

\section{An improved uniform bound for $\phi$}\label{sec:improved-uniform-phi}

From now on, we will only work in the $(u,\ub)$ coordinate system. Recall from \eqref{eq:axis-cond-uub} our choice of normalization on $\Gamma$ for the $(u,\ub)$ coordinates system
$$r=0,\,\, \pr_{\ub} r=\frac{1}{2},\,\, \pr_u r=-\frac{1}{2}\textrm{ and }\Omega=1\text{ on }\Gamma.$$
Also, recall that $\tau$ and $\varrho$ are defined by
$$\tau=\frac{u+\ub}{2},\,\, \varrho=\frac{\ub-u}{2}.$$
We restrict our attention to the space-time region
$$\{-1\leq\tau<0\}\cap \{\ub\leq 0\}$$
where our solution is regular, and we where intend to derive estimates which are uniform up to the origin $O$. We assume throughout the rest of the paper the smallness condition \eqref{smallboundflux1} \eqref{smallboundflux2} on the energy flux. Finally, recall from Section \ref{sec:stress-null} and Section \ref{sec:Einsteineqinnullcoordinates} that the 2+1 dimensional equivariant Einstein-wave map system is given in the $(u,\ub)$ coordinates by
$$
\left\{\ba{rcl}
\ds\pr_u(\Omega^{-2}\pr_ur) &=&\ds -\Omega^{-2}r\kappa (\pr_u\phi)^2,\\[2mm]
\ds\pr_{\ub}(\Omega^{-2}\pr_{\ub}r) &=&\ds -\Omega^{-2}r\kappa (\pr_{\ub}\phi)^2,\\[2mm]
\ds\pr_u\pr_{\ub}r &=&\ds r\kappa \frac{\Omega^2}{4}\frac{g(\phi)^2}{r^2},\\[3mm]
\ds\Omega^{-2}(\pr_u\Omega\pr_{\ub}\Omega-\Omega\pr_u\pr_{\ub}\Omega) &=&\ds \frac{1}{8}\Omega^2\kappa\left(\frac{4}{\Omega^2}\pr_u\phi\pr_{\ub}\phi+\frac{g(\phi)^2}{r^2}\right)\\[3mm]
\ds\frac{2}{r\Omega^2}\left(-\pr_u(r\pr_{\ub}\phi)-\pr_{\ub}(r\pr_u\phi)\right) &=& \ds\frac{f(\phi)}{r^2}
\ea\right.
$$
where $f(\phi)=g(\phi)g'(\phi)$. Since $g$ is odd with $g'(0)=1$, note that there exists a smooth function $\zeta$ such that
$$f(\phi)=\phi+\phi^3\zeta(\phi).$$

\subsection{Preliminary estimates}

We start with simple consequences of the smallness condition \eqref{smallboundflux1} \eqref{smallboundflux2} on the energy flux.
\begin{lemma}\lab{lemma:basicphiunif}
We have
$$|\phi|\les \sqrt{\ep}.$$
\end{lemma}

\begin{proof}
The proof is in the same spirit as Lemma \ref{lem:phiLinfty}. Let 
$$\wp(\phi) := \int^\phi_0 g(s)\,ds.$$
Then, since $\phi$ vanishes on $\Gamma$, we have for $\ub<0$,
\begin{align*}
\wp(\phi(u,\ub)) & = \int_{\ub}^u \ptl_u(\wp(\phi(u',\ub))) \, du'  \\
& = \int_{\ub}^u g(\phi)\pr_u\phi(u',\ub) \, du'.
\intertext{Together with \eqref{smallboundflux1}, we infer}
|\wp(\phi(u,\ub))| & \leq \left(\int_u^{\ub} \frac{g(\phi)^2}{r^2}r(u',\ub) \,  du'\right)^{\frac{1}{2}}\left(\int_u^{\ub} (\pr_u\phi)^2 r(u',\ub) \,  du'\right)^{\frac{1}{2}}\\
& \leq \ep. 
 \end{align*}
Since $\wp(0)=0$, $\wp'(0)=0$ and $\wp''(0)=1$, we have in the neighborhood of 0 the following Taylor expansion
$$\wp(\phi)=\frac{\phi^2}{2}+O(\phi^3)$$
and hence, in view of the estimate for $\wp(\phi)$ above and the smallness of $\ep$, there exists $\underline{\phi}$ such that 
$$\wp\left(\underline{\phi}\right)=\wp(\phi),\,\,\,\, \left|\underline{\phi}\right|\les \sqrt{\ep}.$$ 
  
Next, note that $g(\rho)>0$ for all $\rho>0$ in view of the Grillakis condition \eqref{eq:grillakis-cond-first} and the fact that $g(0)=0$. Since $\wp'=g$, this implies that $\wp$ is strictly increasing. In particular, $\wp$ is one-to-one and hence
$$\underline{\phi}=\phi.$$
We infer from the above estimate for $\underline{\phi}$ that 
$$|\phi|\les \sqrt{\ep}.$$
This concludes the proof of the lemma.
\end{proof}

\begin{lemma}\lab{lemma:basicestimate}
We have
\bee
\left|\pr_{\ub}r-\frac{1}{2}\right| \les\ep, \,\,\,\, \left|\pr_ur+\frac{1}{2}\right| \les\ep,\,\,\,\, |r-\varrho|\les \ep\varrho,\,\,\,\, |\Omega-1| \les \ep.
\eee
\end{lemma}

\begin{proof}
Integrating from the axis of symmetry $\Gamma$ the following equation 
$$\pr_u\pr_{\ub}r=r\kappa \frac{\Omega^2}{4}\frac{g(\phi)^2}{r^2}$$
and in view of the smallness condition \eqref{smallboundflux1} \eqref{smallboundflux2} on the energy flux, and the initialization on $\Gamma$, we deduce
$$\left|\pr_{\ub}r-\frac{1}{2}\right| \les\ep, \left|\pr_ur+\frac{1}{2}\right| \les\ep.$$
Moreover, in view of the definition of $\varrho$ and the initialization on $\Gamma$, we have
$$r-\varrho=\pr_u(r-\varrho)=\pr_{\ub}(r-\varrho)=0\textrm{ on }\Gamma$$
which together with the smallness condition \eqref{smallboundflux1} \eqref{smallboundflux2} on the energy flux and the fact that 
$$\pr_u\pr_{\ub}\varrho=0,$$
yields
$$|r-\varrho|\les \ep\varrho.$$
Finally, the control of $\pr_u$ together with the integration from the axis of symmetry $\Gamma$ of the equation 
$$\pr_u(\Omega^{-2}\pr_ur)=-\Omega^{-2}r\kappa (\pr_u\phi)^2,$$
the initialization on $\Gamma$ and the smallness condition \eqref{smallboundflux1} on the energy flux yields
$$|\Omega-1| \les \ep.$$
This concludes the proof of the lemma.
\end{proof}

\subsection{Reduction to a semilinear wave equation}

\begin{lemma}\lab{lemma:waveopinnullcoord}
Let $\phi$ a function depending only on $u$ and $\ub$. Then, we have
\bee
\square_\gg(\phi) &=& \frac{1}{\Omega^2}\Bigg(-4\pr_u\pr_{\ub}\phi+\frac{1}{\varrho}(\pr_{\ub}\phi-\pr_u\phi)\\
&&+\frac{\varrho-r}{r\varrho}(\pr_{\ub}\phi-\pr_u\phi)-\frac{2\pr_ur+1}{r}\pr_{\ub}\phi-\frac{2\pr_{\ub}r-1}{r}\pr_u\phi\Bigg).
\eee
\end{lemma}

\begin{proof}
Recall that
$$\square_\gg(\phi) = \frac{2}{r\Omega^2}\left(-\pr_u(r\pr_{\ub}\phi)-\pr_{\ub}(r\pr_u\phi)\right).$$
We infer
\bee
\square_\gg(\phi) &=& \frac{1}{\Omega^2}\left(-4\pr_u\pr_{\ub}\phi-\frac{2\pr_ur}{r}\pr_{\ub}\phi-\frac{2\pr_{\ub}r}{r}\pr_u\phi\right)\\
&=& \frac{1}{\Omega^2}\left(-4\pr_u\pr_{\ub}\phi+\frac{1}{r}(\pr_{\ub}\phi-\pr_u\phi)-\frac{2\pr_ur+1}{r}\pr_{\ub}\phi-\frac{2\pr_{\ub}r-1}{r}\pr_u\phi\right)\\
&=& \frac{1}{\Omega^2}\Bigg(-4\pr_u\pr_{\ub}\phi+\frac{1}{\varrho}(\pr_{\ub}\phi-\pr_u\phi)\\
&&+\frac{\varrho-r}{r\varrho}(\pr_{\ub}\phi-\pr_u\phi)-\frac{2\pr_ur+1}{r}\pr_{\ub}\phi-\frac{2\pr_{\ub}r-1}{r}\pr_u\phi\Bigg).
\eee
This concludes the proof of the lemma.
\end{proof}

We deduce the following corollary.
\begin{corollary}
We have
\bee
&&-4\pr_u\pr_{\ub}\phi+\frac{1}{\varrho}(\pr_{\ub}\phi-\pr_u\phi)-\frac{\phi}{\varrho^2}\\ 
&=& -\frac{\varrho-r}{r\varrho}(\pr_{\ub}\phi-\pr_u\phi)+\frac{2\pr_ur+1}{r}\pr_{\ub}\phi+\frac{2\pr_{\ub}r-1}{r}\pr_u\phi+\frac{\varrho^2-r^2}{r^2\varrho^2}\phi+\frac{\phi^3\zeta(\phi)}{r^2}\\
&&+(\Omega^2-1)\frac{f(\phi)}{r^2}.
\eee
\end{corollary}

\begin{proof}
In view of the previous lemma, we have
\bee
&&\frac{1}{\Omega^2}\Bigg(-4\pr_u\pr_{\ub}\phi+\frac{1}{\varrho}(\pr_{\ub}\phi-\pr_u\phi)\\
&&+\frac{\varrho-r}{r\varrho}(\pr_{\ub}\phi-\pr_u\phi)-\frac{2\pr_ur+1}{r}\pr_{\ub}\phi-\frac{2\pr_{\ub}r-1}{r}\pr_u\phi\Bigg)= \frac{f(\phi)}{r^2},
\eee
which we rewrite
\bee
&&-4\pr_u\pr_{\ub}\phi+\frac{1}{\varrho}(\pr_{\ub}\phi-\pr_u\phi)\\ 
&=& -\frac{\varrho-r}{r\varrho}(\pr_{\ub}\phi-\pr_u\phi)+\frac{2\pr_ur+1}{r}\pr_{\ub}\phi+\frac{2\pr_{\ub}r-1}{r}\pr_u\phi+\Omega^2 \frac{f(\phi)}{r^2}.
\eee
We have
\bee
\Omega^2 \frac{f(\phi)}{r^2} &=& \frac{f(\phi)}{r^2}+(\Omega^2-1)\frac{f(\phi)}{r^2}\\
&=& \frac{\phi}{r^2}+\frac{\phi^3\zeta(\phi)}{r^2}+(\Omega^2-1)\frac{f(\phi)}{r^2}\\
&=& \frac{\phi}{\varrho^2}+\frac{\varrho^2-r^2}{r^2\varrho^2}\phi+\frac{\phi^3\zeta(\phi)}{r^2}+(\Omega^2-1)\frac{f(\phi)}{r^2}.
\eee
We infer
\bee
&&-4\pr_u\pr_{\ub}\phi+\frac{1}{\varrho}(\pr_{\ub}\phi-\pr_u\phi)-\frac{\phi}{\varrho^2}\\ 
&=& -\frac{\varrho-r}{r\varrho}(\pr_{\ub}\phi-\pr_u\phi)+\frac{2\pr_ur+1}{r}\pr_{\ub}\phi+\frac{2\pr_{\ub}r-1}{r}\pr_u\phi+\frac{\varrho^2-r^2}{r^2\varrho^2}\phi+\frac{\phi^3\zeta(\phi)}{r^2}\\
&&+(\Omega^2-1)\frac{f(\phi)}{r^2}.
\eee
This concludes the proof of the corollary.
\end{proof}

\begin{corollary}
We have
$$\left(-\pr_\tau^2+\pr_{\varrho}^2+\frac{1}{\varrho}\pr_\varrho-\frac{1}{\varrho^2}\right)\phi=\frac{F}{\varrho^2}$$
where
\bee
F &=& -\frac{\varrho(\varrho-r)}{r}(\pr_{\ub}\phi-\pr_u\phi)+\frac{\varrho^2(2\pr_ur+1)}{r}\pr_{\ub}\phi+\frac{\varrho^2(2\pr_{\ub}r-1)}{r}\pr_u\phi+\frac{\varrho^2-r^2}{r^2}\phi+\frac{\varrho^2\phi^3\zeta(\phi)}{r^2}\\
&&+\varrho^2(\Omega^2-1)\frac{f(\phi)}{r^2}.
\eee
\end{corollary}

\begin{proof}
$$\pr_u=\frac{1}{2}(\pr_\tau-\pr_\varrho)\textrm{ and }\pr_{\ub}=\frac{1}{2}(\pr_\tau+\pr_\varrho)$$
and hence
$$\pr_u\pr_{\ub}=\frac{1}{4}(\pr_\tau^2-\pr_{\varrho}^2)\textrm{ and }\frac{1}{\varrho}(\pr_{\ub}-\pr_u)=\frac{1}{\varrho}\pr_\varrho.$$
Thus, we have
\bee
-4\pr_u\pr_{\ub}+\frac{1}{\varrho}(\pr_{\ub}-\pr_u)-\frac{1}{\varrho^2} = -\pr_\tau^2+\pr_{\varrho}^2+\frac{1}{\varrho}\pr_\varrho-\frac{1}{\varrho^2}.
\eee
In view of the previous corollary, this concludes the proof of this corollary.
\end{proof}

\subsection{Set up of the bootstrap procedure}

Let the space-time domain 
$$I_0:=\{\ub\leq 0,\,\, u\leq -1,\,\, \tau\geq -1\}.$$
Let 
$$-1\leq\ubb<0$$
and the space-time domain
$$Q_{\ubb}:=\{-1\leq \ub<\ubb,\,\, -1\leq u<0\},$$
see Figure \ref{fig:Qbregion}. Let 
$$0<\delta<\frac{1}{2}.$$
We make the following bootstrap assumption on $Q_{\ubb}$:
\be\lab{eq:boot1}
\sup_{Q_{\ubb}}r^{1-\delta}|\pr_{\ub}\phi|\leq C.
\ee

\begin{figure}
\centering
\psfrag{ubar=0}{$\ub=0$}
\psfrag{ubar=ubarb}{$\ub=\ub_b$}
\psfrag{Qbar}{$Q_{\ubb}$}
\psfrag{u=-1}{$u=-1$}
\psfrag{tau=-1}{$\tau=-1$}
\psfrag{Io}{$I_0$}
\psfrag{O}{$O$}
\includegraphics[width=2.5in]{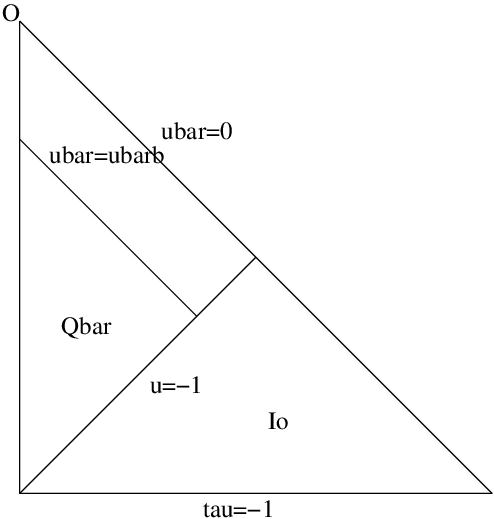}
\caption{The bootstrap region $Q_{\ubb}$}
\label{fig:Qbregion}
\end{figure}

The goal of this section will be to prove that we can improve this bootstrap assumption.
\begin{proposition}\label{prop:boot1}
Assume that 
$$0<\delta<\frac{1}{2}.$$
Then, there exists a universal constant $\underline{C}$ and a constant $C_0$ only depending on the values of the solution in $I_0$ such that for any $-1\leq \ubb<0$, we have
$$\sup_{Q_{\ubb}}r^{1-\delta}|\pr_{\ub}\phi|\leq  \underline{C}(C_0+\ep C).$$
\end{proposition}

\begin{remark}
The constant $\underline{C}$ in Proposition \ref{prop:boot1} depends on $\delta$ such that $0<\delta<1/2$. In particular, it degenerates as $\delta\to 0$ of $\delta\to 1/2$. We will use the improved estimate of Proposition \ref{prop:boot1} at two places in the proof of Theorem \ref{th:smallenergyglobalex}
\begin{itemize}
\item In Lemma \ref{lemma:refinedboundsphietal}, where we apply Proposition \ref{prop:boot1} with any fixed $\delta$ such that 
$$\frac{1}{6}<\delta<\frac{1}{2}.$$

\item In Proposition \ref{prop:unifupperboundv}, where we apply Proposition \ref{prop:boot1} with any fixed $\delta$ such that 
$$\frac{1}{3}<\delta<\frac{1}{2}.$$
\end{itemize}
Thus, we could replace $\delta$ for instance with $5/12$ in the statement of Proposition \ref{prop:boot1}. To make the proof easier to follow, we choose to do it with a general $\delta$, but do not mention the dependence of various constants on $\delta$ since one should think  of a $\delta$ fixed once and for all, e.g. $5/12$. 
\end{remark}

\begin{remark}
The constant $C_0$ appearing in Proposition \ref{prop:boot1} denotes the supremum of the (finitely many) norms on $I_0$ of the solution\footnote{Recall from the assumptions of Theorem \ref{th:smallenergyglobalex} that the solution $(\gg,\Phi)$ is regular in the space-time region
$$\{-1\leq\tau<0\}\cap \{\ub\leq 0\}$$
and hence on the compact region $I_0$. Thus, $C_0$ is a finite constant.} appearing in the proof of Proposition \ref{prop:boot1} below. These norms are not controlled by the energy, and could be in particular arbitrary large compared to $\ep^{-1}$. It is thus crucial that the constant in front of $C$ in the statement of Proposition \ref{prop:boot1}, i.e. $\underline{C}\ep$, does not depend on $C_0$. This will allow us to improve on our bootstrap assumption \eqref{eq:boot1} in Corollary \ref{cor:consequenceboot1} by choosing $\ep$ sufficiently small compared to the universal constant $\underline{C}$.
\end{remark}

\begin{remark}\lab{rem:nocontrolpruphi}
In order to prove Proposition \ref{prop:boot1}, we will first obtain an improved bound for $\phi$ using a representation formula for the wave equation (see Lemma \ref{lemma:improvedunifboundphi}). Then, we infer an improved bound for $\pr_{\ub}\phi$ using Lemma \ref{lemma:eqthetaandxi}. Note that we can not infer a improved bound for $\pr_u\phi$ in this way (see Remark \ref{rem:howtouselemmathetaxi}). This explains why we only have a bootstrap assumption for $\pr_{\ub}\phi$ in Proposition \ref{prop:boot1}, while the terms $\pr_u\phi$ will have to be integrated by parts (see Remark \ref{rem:integrationbypartsuexplained}).
\end{remark}

\begin{remark}
The non-concentration of energy is used in two crucial places in the proof of Theorem \ref{th:smallenergyglobalex}. One chooses $\ep>0$ small enough 
\begin{itemize}
\item In Corollary \ref{cor:consequenceboot1} to improve the bootstrap assumption \eqref{eq:boot1} thanks to Proposition \ref{prop:boot1}.

\item In Proposition \ref{prop:unifupperboundv} in order to exploit the estimate of Corollary \ref{cor:wilallowtoconclude}.
\end{itemize}
\end{remark}

\subsection{First consequences of the bootstrap assumptions}

\begin{lemma}\lab{lemma:consequencebootass}
We have
$$\sup_{Q_{\ubb}}r^{-\delta}|\phi|\les C.$$
and
$$\sup_{Q_{\ubb}}r^{-2\delta}\left(\frac{|r-\varrho|}{\varrho}+\left|\pr_ur+\frac{1}{2}\right|+\left|\pr_{\ub}r-\frac{1}{2}\right|+|\Omega-1|\right)\les C^2.$$
\end{lemma}

\begin{proof}
We start with $\phi$. We have
$$\phi(u,\ub)=\int_u^{\ub}\pr_{\ub}\phi(u,\sigma)d\sigma,$$
and hence using the bootstrap assumption \eqref{eq:boot1}
\bee
|\phi(u,\ub)|&\leq&\int_u^{\ub}|\pr_{\ub}\phi(u,\sigma)|d\sigma\\
&\leq&  C\int_u^{\ub}r(u,\sigma)^{\delta-1}d\sigma\\
&\les&  C\int_u^{\ub}(\sigma-u)^{\delta-1}d\sigma\\
&\les& C(\ub-u)^\delta\\
&\les & Cr(u,\ub)^\delta,
\eee
where we used the fact that $\delta>0$.

Next, recall that 
$$\pr_u\pr_{\ub}r=r\kappa \frac{\Omega^2}{4}\frac{g(\phi)^2}{r^2}.$$
We infer
\bee
\left|\pr_ur+\frac{1}{2}\right| &\leq & \int_u^{\ub}|\pr_{\ub}\pr_ur|(u,\sigma)d\sigma\\
&\les& \int_u^{\ub}\frac{g(\phi)^2}{r}(u,\sigma)d\sigma\\
&\les& C^2\int_u^{\ub}\frac{\phi^2}{r}(u,\sigma)d\sigma\\
&\les& C^2\int_u^{\ub}r(u,\sigma)^{2\delta-1}d\sigma\\
&\les& C^2\int_u^{\ub}(\sigma-u)^{2\delta-1}d\sigma\\
&\les& C^2(\ub-u)^{2\delta}\\
&\les& C^2r(u,\ub)^{2\delta},
\eee
where we used the fact that $\delta>0$ and the previous bound on $\phi$.

Similarly, we have
\bee
\left|\pr_{\ub}r-\frac{1}{2}\right| &\leq & \int_u^{\ub}|\pr_{\ub}\pr_ur|(\sigma,\ub)d\sigma\\
&\les& C^2r(u,\ub)^{2\delta}.
\eee

Next, we consider the bound for $r-\varrho$. We have
\bee
|r-\varrho| &\leq& \int_u^{\ub}\left|\pr_{\ub}r-\frac{1}{2}\right|(u,\sigma)d\sigma\\
&\les& C^2 \int_u^{\ub}r(u,\ub)^{2\delta}(u,\sigma)\sigma\\
&\les& C^2\int_u^{\ub}(\sigma-u)^{2\delta}\sigma\\
&\les& C^2(\ub-u)^{2\delta+1}\\
&\les& C^2r(u, \ub)^{2\delta+1}.
\eee

Finally, we consider $\Omega$. We have
$$\pr_{\ub}(\Omega^{-2}\pr_{\ub}r)=-\Omega^{-2}r\kappa (\pr_{\ub}\phi)^2.$$
This yields
\bee
\left|\Omega^{-2}\pr_{\ub}r-\frac{1}{2}\right| &\leq & \int_u^{\ub}|\pr_{\ub}(\Omega^{-2}\pr_{\ub}r)|(u,\sigma)d\sigma\\
&\les & \int_u^{\ub}r(\pr_{\ub}\phi)^2(u,\sigma)d\sigma\\
&\les & C^2\int_u^{\ub}r(u,\sigma)^{2\delta-1}d\sigma\\
&\les& C^2r(u,\ub)^{2\delta}.
\eee
We infer
\bee
|\Omega^{-2}-1| &\les & \left|\Omega^{-2}\pr_{\ub}r-\pr_{\ub}r\right|\\
&\les& \left|\Omega^{-2}\pr_{\ub}r-\frac{1}{2}\right|+\left|\pr_{\ub}r-\frac{1}{2}\right|\\
&\les& C^2r(u,\ub)^{2\delta}.
\eee
This concludes the proof of the lemma.
\end{proof}

\subsection{An improved uniform bound for $\phi$}

Here we derive an improved uniform bound for $r^{-\delta}\phi$ relying on an explicit representation formula for the flat wave equation. Our approach is inspired by \cite{chris_tah1} (see also \cite{jal_tah1} for a similar approach).

\begin{lemma}\lab{lemma:structurewaveeqphi}
We have
$$\left(-\pr_\tau^2+\pr_{\varrho}^2+\frac{1}{\varrho}\pr_\varrho-\frac{1}{\varrho^2}\right)\phi=\pr_u\left(\frac{F_1}{\varrho}\right)+\frac{F_2}{\varrho^2},$$
where
\bee
F_1 &=& \frac{\varrho-r+\varrho(2\pr_{\ub}r-1)}{r}\phi,
\eee
and
\bee
F_2 &=&  -\frac{\varrho(\varrho-r+\varrho(2\pr_ur+1))}{r}\pr_{\ub}\phi+\frac{\frac{1}{2}(\varrho+r)(\varrho-r)+\frac{\varrho^2}{2}(2\pr_ur+1)+\varrho^2(2\pr_{\ub}r-1)\pr_ur}{r^2}\phi\\
&&- \frac{\kappa\varrho^2\Omega^2g(\phi)^2\phi}{2r^2}
+\frac{\varrho^2\phi^3\zeta(\phi)}{r^2}+\varrho^2(\Omega^2-1)\frac{f(\phi)}{r^2}.
\eee
\end{lemma}

\begin{remark}\lab{rem:integrationbypartsuexplained}
Since we have no control over $\pr_u\phi$ (see Remark \ref{rem:nocontrolpruphi}), we need to integrate the terms involving $\pr_u\phi$ by parts. This results in the term $\pr_u(F_1/\varrho)$ in the statement of Lemma \ref{lemma:structurewaveeqphi}. The fact that this integration by parts is possible is a consequence of the following two observations
\begin{itemize}
\item We are able to estimate $F_2$ (see Lemma \ref{lemma:estF1F2}), which itself is a consequence of the null structure of the problem. 

\item We are able to control the $u$ derivative of the kernel $K$ of the representation formula for the wave equation of Lemma \ref{lemma:repformulaphiwave}. To achieve this, the crucial estimate is the one of Lemma \ref{lemma:upperboundprumu}.
\end{itemize}
\end{remark}

\begin{proof}
Recall that 
$$\left(-\pr_\tau^2+\pr_{\varrho}^2+\frac{1}{\varrho}\pr_\varrho-\frac{1}{\varrho^2}\right)\phi=\frac{F}{\varrho^2}$$
where
\bee
F &=& -\frac{\varrho(\varrho-r)}{r}\pr_{\varrho}\phi +\frac{\varrho^2(2\pr_ur+1)}{r}\pr_{\ub}\phi+\frac{\varrho^2(2\pr_{\ub}r-1)}{r}\pr_u\phi+\frac{\varrho^2-r^2}{r^2}\phi+\frac{\varrho^2\phi^3\zeta(\phi)}{r^2}\\
&&+\varrho^2(\Omega^2-1)\frac{f(\phi)}{r^2}.
\eee
We rewrite $F$ as
\bee
F &=& -\frac{\varrho(\varrho-r)}{r}\pr_{\ub}\phi+\pr_u\left(\frac{\varrho(\varrho-r)}{r}\phi\right)-\pr_u\left(\frac{\varrho(\varrho-r)}{r}\right)\phi
+\frac{\varrho^2(2\pr_ur+1)}{r}\pr_{\ub}\phi\\
&&+\pr_u\left(\frac{\varrho^2(2\pr_{\ub}r-1)}{r}\phi\right)-\pr_u\left(\frac{\varrho^2(2\pr_{\ub}r-1)}{r}\right)\phi
+\frac{\varrho^2-r^2}{r^2}\phi+\frac{\varrho^2\phi^3\zeta(\phi)}{r^2}\\
&&+\varrho^2(\Omega^2-1)\frac{f(\phi)}{r^2}\\
&=&\pr_u\left(\frac{\varrho(\varrho-r)}{r}\phi+\frac{\varrho^2(2\pr_{\ub}r-1)}{r}\phi\right)-\frac{\varrho(\varrho-r)}{r}\pr_{\ub}\phi-\frac{(r-\varrho)^2-\varrho^2(2\pr_ur+1)}{2r^2}\phi\\
&&+\frac{\varrho^2(2\pr_ur+1)}{r}\pr_{\ub}\phi+\frac{\varrho(2\pr_{\ub}r-1)(\varrho\pr_ur+r)-2\varrho^2r\pr_u\pr_{\ub}r}{r^2}\phi+\frac{\varrho^2-r^2}{r^2}\phi+\frac{\varrho^2\phi^3\zeta(\phi)}{r^2}\\
&&+\varrho^2(\Omega^2-1)\frac{f(\phi)}{r^2}\\
&=&\pr_u\left(\frac{\varrho(\varrho-r+\varrho(2\pr_{\ub}r-1))}{r}\phi\right)-\frac{\varrho(\varrho-r+\varrho(2\pr_ur+1))}{r}\pr_{\ub}\phi\\
&&+\frac{\frac{1}{2}(\varrho+3r)(\varrho-r)+\frac{\varrho^2}{2}(2\pr_ur+1)+\varrho(2\pr_{\ub}r-1)(\varrho\pr_ur+r)}{r^2}\phi- \frac{\kappa\varrho^2\Omega^2g(\phi)^2\phi}{2r^2}
+\frac{\varrho^2\phi^3\zeta(\phi)}{r^2}\\
&&+\varrho^2(\Omega^2-1)\frac{f(\phi)}{r^2}.
\eee
This yields
\bee
\frac{F}{\varrho^2} &=& \pr_u\left(\frac{F_1}{\varrho}\right)+\frac{F_2}{\varrho^2},
\eee
where
\bee
F_1 &=& \frac{\varrho-r+\varrho(2\pr_{\ub}r-1)}{r}\phi,
\eee
and
\bee
F_2 &=&  -\frac{\varrho(\varrho-r+\varrho(2\pr_ur+1))}{r}\pr_{\ub}\phi+\frac{\frac{1}{2}(\varrho+r)(\varrho-r)+\frac{\varrho^2}{2}(2\pr_ur+1)+\varrho^2(2\pr_{\ub}r-1)\pr_ur}{r^2}\phi\\
&&- \frac{\kappa\varrho^2\Omega^2g(\phi)^2\phi}{2r^2}
+\frac{\varrho^2\phi^3\zeta(\phi)}{r^2}+\varrho^2(\Omega^2-1)\frac{f(\phi)}{r^2}.
\eee
This concludes the proof of the lemma.
\end{proof}

\begin{lemma}\lab{lemma:estF1F2}
We have
$$\sup_{Q_{\ubb}}r^{-\delta}(|F_1|+|F_2|)\les C\ep.$$
\end{lemma}

\begin{proof}
This is an immediate consequence of Lemma \ref{lemma:basicphiunif}, Lemma \ref{lemma:basicestimate}, the bootstrap assumption  \eqref{eq:boot1} and Lemma \ref{lemma:consequencebootass}, as well as the definition of $F_1$ and $F_2$. 
\end{proof}

\begin{lemma}\lab{lemma:repformulaphiwave}
For $(\tau, \varrho)$ such that $\tau+\varrho\leq 0$ and $-1\leq \tau<0$, we have
\bee
\phi(\tau, \varrho) &=& \phi_0(\tau,\varrho)+\frac{K(-1)}{\sqrt{\varrho}}\int_{\tau-\varrho}^{\tau+\varrho}\frac{F_1\left(\frac{\tau-\varrho+\ub'}{2},\frac{-\tau+\varrho+\ub'}{2}\right)}{\sqrt{\frac{-\tau+\varrho+\ub'}{2}}}  d\ub' -\frac{1}{2}\sqrt{\varrho}\int_0^{+\infty}\frac{F_1(-1, \la)\sqrt{\la}}{\mu\varrho+\la}K(\mu) d\mu\\
&&-\frac{1}{2\sqrt{\varrho}}\int_{\sqrt{(\tau+1)^2-\varrho^2}}^{\tau+\varrho+1}K(\mu)\frac{F_1(-1, \la)}{\sqrt{\la}} d\la\\
&&+\int_{-1}^{+\infty}\int_0^{\la^*}\frac{\sqrt{\varrho}}{\sqrt{\la}\sqrt{\varrho^2+\la^2+2\varrho\la\mu}}K(\mu)\left(\frac{1}{4}F_1(\sigma, \la)+F_2(\sigma, \la)\right)d\la d\mu\\
&&-\int_{-1}^{+\infty}\int_0^{\la^*}\frac{\sqrt{\varrho}}{\sqrt{\la}\sqrt{\varrho^2+\la^2+2\varrho\la\mu}}\la\pr_u\mu\, K'(\mu)F_1(\sigma, \la)d\la d\mu,
\eee
where $\phi_0$ denotes the solution to the homogeneous equation with the same initial conditions as $\phi$ at $\tau=-1$, $\mu$ is given by
$$\mu=\frac{(\tau-\sigma)^2-\varrho^2-\la^2}{2\varrho\la},$$
$\la^*$ is given by
$$\la^*=\sqrt{(1+\tau)^2+(\mu^2-1)\varrho^2}-\mu\varrho,$$
and $K$ is given by
$$K(\mu)=\int_{\max(-\mu,-1)}^1\frac{xdx}{\sqrt{1-x^2}\sqrt{\mu+x}}.$$
\end{lemma}

\begin{proof}
We recall the representation formula derived in \cite{jal_tah1} for the solution $\phi$ of 
$$\left(-\pr_\tau^2+\pr_{\varrho}^2+\frac{1}{\varrho}\pr_\varrho-\frac{1}{\varrho^2}\right)\phi=h.$$
$\phi$ is given by (see \cite{jal_tah} p. 960/961)
$$\phi(\tau, \varrho)=\phi_0(\tau,\varrho)+\int_{R_{\tau, \varrho}}\frac{\sqrt{\la}}{\sqrt{\varrho}}K(\mu)h(\sigma, \la)d\la d\sigma,$$
where 
$$R_{\tau, \varrho}=\{(\sigma, \la)\,/\,\, -1\leq\sigma\leq \tau,\,\,\max(0,\varrho-\tau+\sigma)\leq\la\leq\varrho+\tau-\sigma\},$$
see Figure \ref{fig:Rregion}, $\phi_0$ denotes the solution to the homogeneous equation with the same initial conditions as $\phi$ at $\tau=-1$, $\mu$ is given by
$$\mu=\frac{(\tau-\sigma)^2-\varrho^2-\la^2}{2\varrho\la},$$
and $K$ is given by
$$K(\mu)=\int_{\max(-\mu,-1)}^1\frac{xdx}{\sqrt{1-x^2}\sqrt{\mu+x}}.$$

\begin{figure}
\centering
\psfrag{(tau, rho)}{$(\tau, \varrho)$}
\psfrag{O}{$O$}
\psfrag{muinf}{$\mu=+\infty$}
\psfrag{mu1}{$\mu=1$}
\psfrag{mum1}{$\mu=-1$}
\includegraphics[width=2.5in]{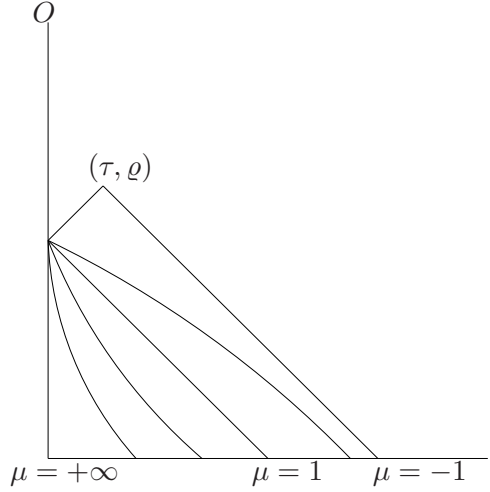}
\caption{The level sets of $\mu$}
\label{fig:mulevelsets}
\end{figure}

\begin{figure}
\centering
\psfrag{ubar=0}{$\ub=0$}
\psfrag{tau =-1}{$\tau =-1$}
\psfrag{(tau, rho)}{$(\tau, \varrho)$}
\psfrag{R}{$R_{\tau,\varrho}$}
\psfrag{O}{$O$}
\includegraphics[width=2.5in]{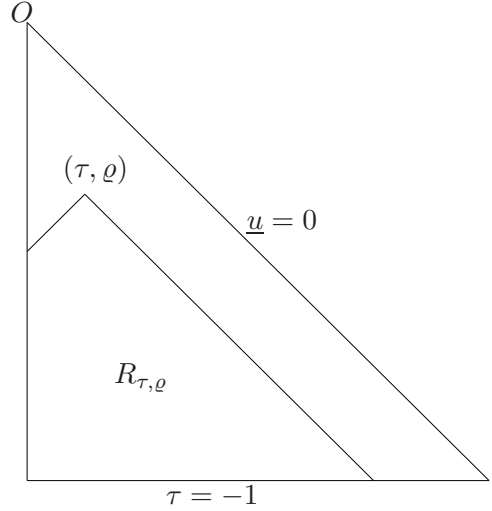}
\caption{The spacetime region $R_{\tau, \varrho}$}
\label{fig:Rregion}
\end{figure}

In our case, note that this formula is well defined since for $(\tau, \varrho)$ such that $\tau+\varrho\leq 0$ and $-1\leq \tau<0$, $R_{\tau, \varrho}$ is included in the region where the solution is assumed to be smooth. Also, we have  
$$h=\pr_u\left(\frac{F_1}{\varrho}\right)+\frac{F_2}{\varrho^2}.$$
Hence, we have 
\bee
\phi(\tau, \varrho) &=& \phi_0(\tau,\varrho)+\int_{R_{\tau, \varrho}}\frac{\sqrt{\la}}{\sqrt{\varrho}}K(\mu)\pr_u\left(\frac{F_1(\sigma, \la)}{\la}\right)d\la d\sigma+\int_{R_{\tau, \varrho}}\frac{\sqrt{\la}}{\sqrt{\varrho}}K(\mu)\frac{F_2(\sigma, \la)}{\la^2}d\la d\sigma\\
&=& \phi_0(\tau,\varrho)+\int_{\pr R_{\tau, \varrho}}\frac{\sqrt{\la}}{\sqrt{\varrho}}K(\mu)\frac{F_1(\sigma, \la)}{\la}\gg(\pr_u,\nu_R) -\int_{R_{\tau, \varrho}}\pr_u\left(\frac{\sqrt{\la}}{\sqrt{\varrho}}K(\mu)\right)\frac{F_1(\sigma, \la)}{\la}d\la d\sigma\\
&&+\int_{R_{\tau, \varrho}}\frac{\sqrt{\la}}{\sqrt{\varrho}}K(\mu)\frac{F_2(\sigma, \la)}{\la^2}d\la d\sigma\\
&=& \phi_0(\tau,\varrho)+\int_{\pr R_{\tau, \varrho}}\frac{\sqrt{\la}}{\sqrt{\varrho}}K(\mu)\frac{F_1(\sigma, \la)}{\la}\gg(\pr_u,\nu_R) +\int_{R_{\tau, \varrho}}\frac{1}{4\sqrt{\la\varrho}}K(\mu)\frac{F_1(\sigma, \la)}{\la}d\la d\sigma\\
&&-\int_{R_{\tau, \varrho}}\frac{\sqrt{\la}}{\sqrt{\varrho}}\pr_u\mu\, K'(\mu)\frac{F_1(\sigma, \la)}{\la}d\la d\sigma+\int_{R_{\tau, \varrho}}\frac{\sqrt{\la}}{\sqrt{\varrho}}K(\mu)\frac{F_2(\sigma, \la)}{\la^2}d\la d\sigma.
\eee

Next, we compute the boundary term. We have
\bee
&&\int_{\pr R_{\tau, \varrho}}f \gg(\pr_u,\nu_R) \\
&=& \int_{\tau-\varrho}^{\tau+\varrho} f\left(\frac{\tau-\varrho+\ub'}{2},\frac{-\tau+\varrho+\ub'}{2}\right)d\ub'+\frac{1}{2}\int_{-1}^{\tau-\varrho}f(\sigma,0)d\sigma-\frac{1}{2}\int_0^{\tau+\varrho+1} f(-1, \la)d\la,
\eee
and
$$\mu=-1\textrm{ on }u=\tau-\varrho.$$
Hence, we have\footnote{Here, we have dropped the boundary term on $\la=0$ as it vanishes. Indeed, we have 
$$\left|\sqrt{\la}K(\mu)\frac{F_1(\sigma, \la)}{\la}\right|\leq C_{\tau, \varrho} \sqrt{\la}|K(\mu)|\leq C_{\tau, \varrho} \sqrt{\la}\to 0\textrm{ as }\la\to 0,$$
where the constant $C_{\tau, \varrho}$ may blow up as $(\tau, \varrho)$ tends to the origin but is finite away from it, and where we used in particular the fact that $F_1/\la$ is bounded in view of the asymptotic for $\rho-r$, $2\pr_{\ub}r-1$ and $\phi$ as $\varrho\to 0$, the fact that $\mu\to +\infty$ when $\la\to 0$ with $\sigma<\tau-\varrho$, and the fact that $K$ is bounded for $\mu\geq 2$ which is immediate from the definition of $K$.}
\bee
\int_{\pr R_{\tau, \varrho}}\frac{\sqrt{\la}}{\sqrt{\varrho}}K(\mu)\frac{F_1(\sigma, \la)}{\la}\gg(\pr_u,\nu_R) &=& \frac{K(-1)}{\sqrt{\varrho}}\int_{\tau-\varrho}^{\tau+\varrho}\frac{F_1\left(\frac{\tau-\varrho+\ub'}{2},\frac{-\tau+\varrho+\ub'}{2}\right)}{\sqrt{\frac{-\tau+\varrho+\ub'}{2}}}  d\ub'\\
&& -\frac{1}{2}\int_0^{\tau+\varrho+1}\frac{\sqrt{\la}}{\sqrt{\varrho}}K(\mu)\frac{F_1(-1, \la)}{\la} d\la.
\eee
We compute
\bee
\pr_\la\mu &=& \frac{\varrho^2-\la^2-(\tau-\sigma)^2}{2\varrho\la^2}\\
&=& \frac{-2\mu\la\varrho-2\la^2}{2\varrho\la^2}\\
&=& -\frac{\mu\varrho+\la}{\varrho\la}.
\eee 
We decompose and perform a change of variable
\bee
&&\int_0^{\tau+\varrho+1}\frac{\sqrt{\la}}{\sqrt{\varrho}}K(\mu)\frac{F_1(-1, \la)}{\la} d\la \\
&=& \int_0^{+\infty}\frac{\sqrt{\la}}{\sqrt{\varrho}}K(\mu)\frac{F_1(-1, \la)}{\la} \frac{\varrho\la}{\mu\varrho+\la}d\mu+\int_{\sqrt{(\tau+1)^2-\varrho^2}}^{\tau+\varrho+1}\frac{\sqrt{\la}}{\sqrt{\varrho}}K(\mu)\frac{F_1(-1, \la)}{\la} d\la\\
&=& \sqrt{\varrho}\int_0^{+\infty}\frac{F_1(-1, \la)\sqrt{\la}}{\mu\varrho+\la}K(\mu) d\mu+\frac{1}{\sqrt{\varrho}}\int_{\sqrt{(\tau+1)^2-\varrho^2}}^{\tau+\varrho+1}K(\mu)\frac{F_1(-1, \la)}{\sqrt{\la}} d\la
\eee
which yields
\bee
&&\int_{\pr R_{\tau, \varrho}}\frac{\sqrt{\la}}{\sqrt{\varrho}}K(\mu)\frac{F_1(\sigma, \la)}{\la}\gg(\pr_u,\nu_R) \\
&=& \frac{K(-1)}{\sqrt{\varrho}}\int_{\tau-\varrho}^{\tau+\varrho}\frac{F_1\left(\frac{\tau-\varrho+\ub'}{2},\frac{-\tau+\varrho+\ub'}{2}\right)}{\sqrt{\frac{-\tau+\varrho+\ub'}{2}}}  d\ub' -\frac{1}{2}\sqrt{\varrho}\int_0^{+\infty}\frac{F_1(-1, \la)\sqrt{\la}}{\mu\varrho+\la}K(\mu) d\mu\\
&&-\frac{1}{2\sqrt{\varrho}}\int_{\sqrt{(\tau+1)^2-\varrho^2}}^{\tau+\varrho+1}K(\mu)\frac{F_1(-1, \la)}{\sqrt{\la}} d\la.
\eee

We deduce
\bee
\phi(\tau, \varrho) &=& \phi_0(\tau,\varrho)+\frac{K(-1)}{\sqrt{\varrho}}\int_{\tau-\varrho}^{\tau+\varrho}\frac{F_1\left(\frac{\tau-\varrho+\ub'}{2},\frac{-\tau+\varrho+\ub'}{2}\right)}{\sqrt{\frac{-\tau+\varrho+\ub'}{2}}}  d\ub' -\frac{1}{2}\sqrt{\varrho}\int_0^{+\infty}\frac{F_1(-1, \la)\sqrt{\la}}{\mu\varrho+\la}K(\mu) d\mu\\
&&-\frac{1}{2\sqrt{\varrho}}\int_{\sqrt{(\tau+1)^2-\varrho^2}}^{\tau+\varrho+1}K(\mu)\frac{F_1(-1, \la)}{\sqrt{\la}} d\la+\int_{R_{\tau, \varrho}}\frac{1}{4\sqrt{\la\varrho}}K(\mu)\frac{F_1(\sigma, \la)}{\la}d\la d\sigma\\
&&-\int_{R_{\tau, \varrho}}\frac{\sqrt{\la}}{\sqrt{\varrho}}\pr_u\mu\, K'(\mu)\frac{F_1(\sigma, \la)}{\la}d\la d\sigma+\int_{R_{\tau, \varrho}}\frac{\sqrt{\la}}{\sqrt{\varrho}}K(\mu)\frac{F_2(\sigma, \la)}{\la^2}d\la d\sigma.
\eee
In the space-time integral, we perform the change of variable $(\sigma, \la)\to (\mu,\la)$ which yields
\bee
\phi(\tau, \varrho) &=& \phi_0(\tau,\varrho)+\frac{K(-1)}{\sqrt{\varrho}}\int_{\tau-\varrho}^{\tau+\varrho}\frac{F_1\left(\frac{\tau-\varrho+\ub'}{2},\frac{-\tau+\varrho+\ub'}{2}\right)}{\sqrt{\frac{-\tau+\varrho+\ub'}{2}}}  d\ub' -\frac{1}{2}\sqrt{\varrho}\int_0^{+\infty}\frac{F_1(-1, \la)\sqrt{\la}}{\mu\varrho+\la}K(\mu) d\mu\\
&&-\frac{1}{2\sqrt{\varrho}}\int_{\sqrt{(\tau+1)^2-\varrho^2}}^{\tau+\varrho+1}K(\mu)\frac{F_1(-1, \la)}{\sqrt{\la}} d\la+\int_{-1}^{+\infty}\int_0^{\la^*}\frac{1}{4\sqrt{\la\varrho}}K(\mu)\frac{F_1(\sigma, \la)}{\la}\frac{1}{|\pr_\sigma\mu|}d\la d\mu\\
&&-\int_{-1}^{+\infty}\int_0^{\la^*}\frac{\sqrt{\la}}{\sqrt{\varrho}}\pr_u\mu\, K'(\mu)\frac{F_1(\sigma, \la)}{\la}\frac{1}{|\pr_\sigma\mu|}d\la d\mu+\int_{-1}^{+\infty}\int_0^{\la^*}\frac{\sqrt{\la}}{\sqrt{\varrho}}K(\mu)\frac{F_2(\sigma, \la)}{\la^2}\frac{1}{|\pr_\sigma\mu|}d\la d\mu,
\eee
where $\la^*$ is given by
$$\la^*=\sqrt{(1+\tau)^2+(\mu^2-1)\varrho^2}-\mu\varrho.$$
We compute
\bee
\pr_\sigma\mu &=& \frac{\sigma-\tau}{\varrho\la}\\
&=& -\frac{\sqrt{\varrho^2+\la^2+2\varrho\la\mu}}{\varrho\la}.
\eee
We infer
\bee
\phi(\tau, \varrho) &=& \phi_0(\tau,\varrho)+\frac{K(-1)}{\sqrt{\varrho}}\int_{\tau-\varrho}^{\tau+\varrho}\frac{F_1\left(\frac{\tau-\varrho+\ub'}{2},\frac{-\tau+\varrho+\ub'}{2}\right)}{\sqrt{\frac{-\tau+\varrho+\ub'}{2}}}  d\ub' -\frac{1}{2}\sqrt{\varrho}\int_0^{+\infty}\frac{F_1(-1, \la)\sqrt{\la}}{\mu\varrho+\la}K(\mu) d\mu\\
&&-\frac{1}{2\sqrt{\varrho}}\int_{\sqrt{(\tau+1)^2-\varrho^2}}^{\tau+\varrho+1}K(\mu)\frac{F_1(-1, \la)}{\sqrt{\la}} d\la\\
&&+\int_{-1}^{+\infty}\int_0^{\la^*}\frac{\sqrt{\varrho}}{\sqrt{\la}\sqrt{\varrho^2+\la^2+2\varrho\la\mu}}K(\mu)\left(\frac{1}{4}F_1(\sigma, \la)+F_2(\sigma, \la)\right)d\la d\mu\\
&&-\int_{-1}^{+\infty}\int_0^{\la^*}\frac{\sqrt{\varrho}}{\sqrt{\la}\sqrt{\varrho^2+\la^2+2\varrho\la\mu}}\la\pr_u\mu\, K'(\mu)F_1(\sigma, \la)d\la d\mu.
\eee
This concludes the proof of the lemma.
\end{proof}

\begin{lemma}\lab{lemma:upperboundprumu}
We have
$$|\la\pr_u\mu|\les |\mu-1|,\,\,\forall -1\leq\mu<+\infty.$$
\end{lemma}

\begin{proof}
We compute
\bee
\pr_u\mu &=& \frac{\la((\sigma-\tau)+\la)+\frac{1}{2}((\tau-\sigma)^2-\varrho^2-\la^2)}{2\varrho\la^2}\\
&=&  \frac{\la(\sigma-\tau)+\frac{\la^2}{2}+\frac{1}{2}(\tau-\sigma)^2-\frac{1}{2}\varrho^2}{2\varrho\la^2}\\
&=&  \frac{(\la+\sigma-\tau)^2-\varrho^2}{4\varrho\la^2}\\
&=& \frac{(\la-\sqrt{\varrho^2+\la^2+2\varrho\la\mu})^2-\varrho^2}{4\varrho\la^2}\\
&=& \frac{\la^2-2\la\sqrt{\varrho^2+\la^2+2\varrho\la\mu}+\varrho^2+\la^2+2\varrho\la\mu-\varrho^2}{4\varrho\la^2}\\
&=& \frac{\la+\varrho\mu-\sqrt{\varrho^2+\la^2+2\varrho\la\mu}}{2\varrho\la}.
\eee
We infer
\bee
\la\pr_u\mu &=& \frac{\la+\varrho\mu-\sqrt{\varrho^2+\la^2+2\varrho\la\mu}}{2\varrho}.
\eee

We consider two cases. If $\mu\geq 0$, we have
\bee
\la\pr_u\mu &=& \frac{\la^2+\varrho^2\mu^2+2\la\varrho\mu-\varrho^2-\la^2-2\varrho\la\mu}{2\varrho(\la+\varrho\mu+\sqrt{\varrho^2+\la^2+2\varrho\la\mu})}\\
&=& \frac{\varrho^2(\mu^2-1)}{2\varrho(\la+\varrho\mu+\sqrt{\varrho^2+\la^2+2\varrho\la\mu})}\\
&=& \frac{\varrho(\mu+1)(\mu-1)}{2(\la+\varrho\mu+\sqrt{\varrho^2+\la^2+2\varrho\la\mu})}.
\eee
Since $\mu\geq 0$, $\la\geq 0$ and $\varrho\geq 0$, we have
$$\la+\varrho\mu+\sqrt{\varrho^2+\la^2+2\varrho\la\mu}\geq \varrho(1+\mu).$$
We infer
\bee
|\la\pr_u\mu| &\leq & \frac{|\mu-1|}{2}.
\eee

If $-1\leq\mu<0$, we have
\bee
\la\pr_u\mu &=& \frac{\la^2-\varrho^2\mu^2+2\varrho\mu\sqrt{\varrho^2+\la^2+2\varrho\la\mu}-\varrho^2-\la^2-2\varrho\la\mu}{2\varrho(\la-\varrho\mu+\sqrt{\varrho^2+\la^2+2\varrho\la\mu})}\\
&=& \frac{-\varrho(\mu^2+1)+2\mu(\sqrt{\varrho^2+\la^2+2\varrho\la\mu}-\la)}{2(\la-\varrho\mu+\sqrt{\varrho^2+\la^2+2\varrho\la\mu})}.
\eee
Since $-1\leq\mu<0$, this yields
\bee
\left|\la\pr_u\mu\right| &=& \frac{\left|-\varrho(\mu^2+1)+2|\mu|(\la-\sqrt{\varrho^2+\la^2-2|\mu|\varrho\la})\right|}{2(\la+|\mu|\varrho+\sqrt{\varrho^2+\la^2-2|\mu|\varrho\la})}\\
&\les& \frac{\varrho+\la+\sqrt{\varrho^2+\la^2}}{\la+\varrho}\\
&\les& 1.
\eee
Together with the case $\mu\geq 0$, we obtain for all $\mu\geq -1$
\bee
\left|\la\pr_u\mu\right| &\les& |\mu-1|.
\eee
This concludes the proof of the lemma.
\end{proof}

\begin{lemma}
For $(\tau, \varrho)\in Q_{\ubb}$, we have
\bee
|\phi(\tau, \varrho)| &\les& C_0\sqrt{\varrho}+(\ep C+C_0)\varrho^{\delta}\\
&&+(\ep C+C_0)\sqrt{\varrho}\int_{-1}^{+\infty}\int_0^{\la^*}\frac{\la^{\delta-\frac{1}{2}}}{\sqrt{\varrho^2+\la^2+2\varrho\la\mu}}(|K(\mu)|+|\mu-1||K'(\mu)|) d\la d\mu,
\eee
where the constant $C_0$ only depends on the values of the solution in $I_0$.
\end{lemma}

\begin{proof}
Recall that
\bee
\phi(\tau, \varrho) &=& \phi_0(\tau,\varrho)+\frac{K(-1)}{\sqrt{\varrho}}\int_{\tau-\varrho}^{\tau+\varrho}\frac{F_1\left(\frac{\tau-\varrho+\ub'}{2},\frac{-\tau+\varrho+\ub'}{2}\right)}{\sqrt{\frac{-\tau+\varrho+\ub'}{2}}}  d\ub' -\frac{1}{2}\sqrt{\varrho}\int_0^{+\infty}\frac{F_1(-1, \la)\sqrt{\la}}{\mu\varrho+\la}K(\mu) d\mu\\
&&-\frac{1}{2\sqrt{\varrho}}\int_{\sqrt{(\tau+1)^2-\varrho^2}}^{\tau+\varrho+1}K(\mu)\frac{F_1(-1, \la)}{\sqrt{\la}} d\la\\
&&+\int_{-1}^{+\infty}\int_0^{\la^*}\frac{\sqrt{\varrho}}{\sqrt{\la}\sqrt{\varrho^2+\la^2+2\varrho\la\mu}}K(\mu)\left(\frac{1}{4}F_1(\sigma, \la)+F_2(\sigma, \la)\right)d\la d\mu\\
&&-\int_{-1}^{+\infty}\int_0^{\la^*}\frac{\sqrt{\varrho}}{\sqrt{\la}\sqrt{\varrho^2+\la^2+2\varrho\la\mu}}\la\pr_u\mu\, K'(\mu)F_1(\sigma, \la)d\la d\mu.
\eee
We have the following properties for $K$ (see for example \cite{chris_tah1} p. 1061):
$$K(-1)=\frac{\pi}{\sqrt{2}},\,\, \sup_{-1\leq \mu\leq 0}|K|\les 1,\,\, K\in L^1(-1,+\infty).$$
Also, we have 
$$\{\sqrt{(\tau+1)^2-\varrho^2}\leq\la\leq \tau+\varrho+1\}\cap\{\sigma=-1\}=\{-1\leq\mu\leq 0\}\cap\{\sigma=-1\}.$$
We deduce
\bee
|\phi(\tau, \varrho)| &\les& |\phi_0(\tau,\varrho)|+\frac{|K(-1)|}{\sqrt{\varrho}}\int_0^{\varrho}\frac{\left|F_1\left(\tau-\varrho+\la,\la\right)\right|}{\sqrt{\la}}  d\la \\
&&+\sqrt{\varrho}\left(\sup_{\la\geq 0}\frac{|F_1(-1, \la)|}{\sqrt{\la}}\right) \int_0^{+\infty}|K(\mu)| d\mu\\
&&+\frac{1}{\sqrt{\varrho}}\left(\sup_{-1\leq \mu\leq 0}|K|\right)\left(\sup_{\la\geq 0}\frac{|F_1(-1, \la)|}{\sqrt{\la}}\right)(\tau+\varrho+1-\sqrt{(\tau+1)^2-\varrho^2})\\
&&+\int_{-1}^{+\infty}\int_0^{\la^*}\frac{\sqrt{\varrho}}{\sqrt{\la}\sqrt{\varrho^2+\la^2+2\varrho\la\mu}}|K(\mu)|(|F_1(\sigma, \la)|+|F_2(\sigma, \la)|)d\la d\mu\\
&&+\int_{-1}^{+\infty}\int_0^{\la^*}\frac{\sqrt{\varrho}}{\sqrt{\la}\sqrt{\varrho^2+\la^2+2\varrho\la\mu}}|\la\pr_u\mu||K'(\mu)| |F_1(\sigma, \la)| d\la d\mu\\
&\les& |\phi_0(\tau,\varrho)|+\sup_{0\leq\la\leq\varrho}|F_1(\tau-\varrho+\la,\la)| +\sqrt{\varrho}\left(\sup_{\la\geq 0}\frac{|F_1(-1, \la)|}{\sqrt{\la}}\right) \\
&&+\sqrt{\varrho}\left(\sup_{\la\geq 0}\frac{|F_1(-1, \la)|}{\sqrt{\la}}\right)\frac{\varrho+\tau+1}{\tau+\varrho+1+\sqrt{(\tau+1)^2-\varrho^2}}\\
&&+\int_{-1}^{+\infty}\int_0^{\la^*}\frac{\sqrt{\varrho}}{\sqrt{\la}\sqrt{\varrho^2+\la^2+2\varrho\la\mu}}|K(\mu)|(|F_1(\sigma, \la)|+|F_2(\sigma, \la)|)d\la d\mu\\
&&+\int_{-1}^{+\infty}\int_0^{\la^*}\frac{\sqrt{\varrho}}{\sqrt{\la}\sqrt{\varrho^2+\la^2+2\varrho\la\mu}}|\la\pr_u\mu||K'(\mu)| |F_1(\sigma, \la)| d\la d\mu.
\eee
Assuming enough regularity on the initial data, we have
$$\sup_{\varrho>0}\varrho^{-\frac{1}{2}}|\phi_0(\tau,\varrho)|+\sup_{\la\geq 0}\frac{|F_1(-1, \la)|}{\sqrt{\la}}\leq C_0$$
where the constant $C_0$ only depend on initial data. Hence, together with the previous lemma, we deduce\footnote{Observe that $v_1(\tau, \varrho, \theta)=\phi_0(\tau, \varrho)\cos(\theta)$ and $v_2(\tau, \varrho, \theta)=\phi_0(\tau, \varrho)\sin(\theta)$ both satisfy the standard wave equation on $\mathbb{R}^{2+1}$ which allows to infer estimates for $\phi_0$ from corresponding estimates for the standard wave equation.}
\bee
|\phi(\tau, \varrho)| &\les& C_0\sqrt{\varrho}+\sup_{0\leq\la\leq\varrho}|F_1(\tau-\varrho+\la,\la)|\\
&&+\int_{-1}^{+\infty}\int_0^{\la^*}\frac{\sqrt{\varrho}}{\sqrt{\la}\sqrt{\varrho^2+\la^2+2\varrho\la\mu}}|K(\mu)|(|F_1(\sigma, \la)|+|F_2(\sigma, \la)|)d\la d\mu\\
&&+\int_{-1}^{+\infty}\int_0^{\la^*}\frac{\sqrt{\varrho}}{\sqrt{\la}\sqrt{\varrho^2+\la^2+2\varrho\la\mu}}|\mu-1||K'(\mu)| |F_1(\sigma, \la)| d\la d\mu.
\eee
Finally, recall that we have
$$\sup_{Q_{\ubb}}r^{-\delta}(|F_1|+|F_2|)\les C\ep.$$
Since we have
$$R_{\tau, \varrho}\subset Q_{\ubb}\cup I_0\textrm{ for any }(\tau, \varrho)\in Q_{\ubb},$$
and 
$$\sup_{I_0}r^{-\delta}(|F_1|+|F_2|)\les C_0,$$
we infer
$$\sup_{(\tau, \varrho)\in Q_{\ubb}}\sup_{R_{\tau, \varrho}}r^{-\delta}(|F_1|+|F_2|)\les C\ep+C_0.$$
Hence, we deduce for $(\tau, \varrho)\in Q_{\ubb}$
\bee
|\phi(\tau, \varrho)| &\les& C_0\sqrt{\varrho}+(\ep C+C_0)\varrho^{\delta}\\
&&+(\ep C+C_0)\sqrt{\varrho}\int_{-1}^{+\infty}\int_0^{\la^*}\frac{\la^{\delta-\frac{1}{2}}}{\sqrt{\varrho^2+\la^2+2\varrho\la\mu}}(|K(\mu)|+|\mu-1||K'(\mu)|) d\la d\mu.
\eee
This concludes the proof of the lemma.
\end{proof}

\begin{lemma}\lab{lemma:improvedunifboundphi}
For $(\tau, \varrho)\in Q_{\ubb}$, we have
\bee
|\phi(\tau, \varrho)| &\les& C_0\sqrt{\varrho}+(\ep C+C_0)\varrho^\delta
\eee
where the constant $C_0$ only depends on the values of the solution in $I_0$.
\end{lemma}

\begin{proof}
Recall that
\bee
|\phi(\tau, \varrho)| &\les& C_0\sqrt{\varrho}+(\ep C+C_0)\varrho^{\delta}\\
&&+(\ep C+C_0)\sqrt{\varrho}\int_{-1}^{+\infty}\int_0^{\la^*}\frac{\la^{\delta-\frac{1}{2}}}{\sqrt{\varrho^2+\la^2+2\varrho\la\mu}}(|K(\mu)|+|\mu-1||K'(\mu)|) d\la d\mu.
\eee
We evaluate the integral in the right-hand side. We have
\bee
&& \int_{-1}^{+\infty}\int_0^{\la^*}\frac{\la^{\delta-\frac{1}{2}}}{\sqrt{\varrho^2+\la^2+2\varrho\la\mu}}(|K(\mu)|+|\mu-1||K'(\mu)|) d\la d\mu\\
&\les& \int_{-1}^0\int_0^{\la^*}\frac{\la^{\delta-\frac{1}{2}}}{\sqrt{\varrho^2+\la^2+2\varrho\la\mu}}(|K(\mu)|+|K'(\mu)|) d\la d\mu\\
&&+\int_0^{+\infty}\int_0^{\la^*}\frac{\la^{\delta-\frac{1}{2}}}{\sqrt{\varrho^2+\la^2+2\varrho\la\mu}}(|K(\mu)|+|\mu-1||K'(\mu)|) d\la d\mu.
\eee

We estimate the two integral on the right-hand side starting with the first one. We have
\bee
&& \int_{-1}^0\int_0^{\la^*}\frac{\la^{\delta-\frac{1}{2}}}{\sqrt{\varrho^2+\la^2+2\varrho\la\mu}}(|K(\mu)|+|K'(\mu)|) d\la d\mu\\
&\les& \sup_{-1\leq \mu\leq 0}(|K(\mu)|+|K'(\mu)|)\int_{-1}^0\int_0^{\la^*}\frac{\la^{\delta-\frac{1}{2}}}{\sqrt{\varrho^2+\la^2+2\varrho\la\mu}} d\la d\mu\\
&\les& \int_0^{+\infty}\int_{\tau-\sqrt{\la^2+\varrho^2}}^{\tau-|\la-\varrho|}\frac{\la^{\delta-\frac{1}{2}}}{\varrho\la} d\sigma d\la, 
\eee
where we used the fact that 
$$\sup_{-1\mu\leq 0}(|K(\mu)|+|K'(\mu)|)<+\infty,$$
which is a consequence of the estimates for $K$ and $K'$ on p. 1061 \cite{chris_tah1},
$$\{-1\leq\mu\leq 0\}=\{|\la-\varrho|\leq \tau-\sigma\leq \sqrt{\varrho^2+\la^2}\},$$
and
$$\pr_{\sigma}\mu=-\frac{\sqrt{\varrho^2+\la^2+2\varrho\la\mu}}{\varrho\la}.$$
We infer
\bee
&& \int_{-1}^0\int_0^{\la^*}\frac{\la^{\delta-\frac{1}{2}}}{\sqrt{\varrho^2+\la^2+2\varrho\la\mu}}(|K(\mu)|+|K'(\mu)|) d\la d\mu\\
&\les& \int_0^{+\infty}(\sqrt{\la^2+\varrho^2}-|\la-\varrho|)\frac{\la^{\delta-\frac{1}{2}}}{\varrho\la}  d\la\\
&\les& \int_0^{+\infty}\frac{\la^{\delta-\frac{1}{2}}}{\sqrt{\la^2+\varrho^2}+|\la-\varrho|}  d\la\\
&\les& \varrho^{\delta-\frac{1}{2}}\int_0^{+\infty}\frac{\la^{\delta-\frac{1}{2}}}{\sqrt{\la^2+1}+|\la-1|}  d\la\\
&\les&  \varrho^{\delta-\frac{1}{2}}, 
\eee
where we used in the last inequality the fact that $0<\delta<1/2$.

Next, we estimate the second integral on the right-hand side. We have
\bee
&&\int_0^{+\infty}\int_0^{\la^*}\frac{\la^{\delta-\frac{1}{2}}}{\sqrt{\varrho^2+\la^2+2\varrho\la\mu}}(|K(\mu)|+|\mu-1||K'(\mu)|) d\la d\mu\\
&\les& \left(\int_0^{+\infty}(|K(\mu)|+|\mu-1||K'(\mu)|)d\mu\right)\int_0^{+\infty}\frac{\la^{\delta-\frac{1}{2}}}{\varrho+\la} d\la\\
&\les& \varrho^{\delta-\frac{1}{2}}\int_0^{+\infty}\frac{\la^{\delta-\frac{1}{2}}}{1+\la} d\la\\
&\les& \varrho^{\delta-\frac{1}{2}},
\eee
where we used the fact that $0<\delta<1/2$ and 
$$\int_0^{+\infty}(|K(\mu)|+|\mu-1||K'(\mu)|)d\mu<+\infty$$
which is a consequence of the estimates for $K$ and $K'$ on p. 1061 \cite{chris_tah1}. We deduce
\bea\lab{eq:estimusefullagainater}
 \int_{-1}^{+\infty}\int_0^{\la^*}\frac{\la^{\delta-\frac{1}{2}}}{\sqrt{\varrho^2+\la^2+2\varrho\la\mu}}(|K(\mu)|+|\mu-1||K'(\mu)|) d\la d\mu &\les& \varrho^{\delta-\frac{1}{2}}
\eea
which yields
\bee
|\phi(\tau, \varrho)| &\les& C_0\sqrt{\varrho}+(\ep C+C_0)\varrho^\delta.
\eee
This concludes the proof of the lemma.
\end{proof}

\subsection{Proof of Proposition \ref{prop:boot1}}

\begin{lemma}\lab{lemma:eqthetaandxi}
Let
$$\Theta=r\pr_{\ub}\phi\textrm{ and }\Xi=r\pr_u\phi.$$
We have
\bee
\pr_u\left(\frac{\Theta}{\sqrt{r}}\right)=-\frac{1}{2\sqrt{r}}\pr_{\ub}r \pr_u\phi-\frac{\Omega^2f(\phi)}{4r^{\frac{3}{2}}},
\eee
and
\bee
\pr_{\ub}\left(\frac{\Xi}{\sqrt{r}}\right)=-\frac{1}{2\sqrt{r}}\pr_ur \pr_{\ub}\phi-\frac{\Omega^2f(\phi)}{4r^{\frac{3}{2}}}.
\eee
\end{lemma} 

\begin{proof}
From 
$$\pr_u(r\pr_{\ub}\phi)+\pr_{\ub}(r\pr_u\phi)=-\frac{\Omega^2}{2}\frac{f(\phi)}{r},$$
we infer 
\bee
\pr_u\Theta=\frac{1}{2}\frac{\pr_ur}{r}\Theta-\frac{1}{2}\pr_{\ub}r \pr_u\phi-\frac{\Omega^2f(\phi)}{4r},
\eee
and
\bee
\pr_{\ub}\Xi=\frac{1}{2}\frac{\pr_{\ub}r}{r}\Xi-\frac{1}{2}\pr_ur \pr_{\ub}\phi-\frac{\Omega^2f(\phi)}{4r}.
\eee
We rewrite the system as
\bee
\pr_u\left(\frac{\Theta}{\sqrt{r}}\right)=-\frac{1}{2\sqrt{r}}\pr_{\ub}r \pr_u\phi-\frac{\Omega^2f(\phi)}{4r^{\frac{3}{2}}},
\eee
and
\bee
\pr_{\ub}\left(\frac{\Xi}{\sqrt{r}}\right)=-\frac{1}{2\sqrt{r}}\pr_ur \pr_{\ub}\phi-\frac{\Omega^2f(\phi)}{4r^{\frac{3}{2}}}.
\eee
This concludes the proof of the lemma.
\end{proof}

We are now in position to prove Proposition \ref{prop:boot1}. Recall that
\bee
\pr_u\left(\frac{\Theta}{\sqrt{r}}\right)=-\frac{1}{2\sqrt{r}}\pr_{\ub}r \pr_u\phi-\frac{\Omega^2f(\phi)}{4r^{\frac{3}{2}}}.
\eee
We integrate between $(u,\ub)$ and $(-\ub-2, \ub)$, where $(-\ub-2, \ub)$ is on $\tau=-1$. We deduce
\bee
\frac{\Theta(u,\ub)}{\sqrt{r(u,\ub)}} &=& \frac{\Theta(-\ub-2,\ub)}{\sqrt{r(-\ub-2,\ub)}}-\int_{-\ub-2}^u\frac{1}{2\sqrt{r}}\pr_{\ub}r \pr_u\phi(\sigma,\ub)d\sigma -\int_{-\ub-2}^u\frac{\Omega^2f(\phi)}{4r^{\frac{3}{2}}}(\sigma,\ub)d\sigma\\
&=& \frac{\Theta(-\ub-2,\ub)}{\sqrt{r(-\ub-2,\ub)}}-\left[\frac{1}{2\sqrt{r}}\pr_{\ub}r \phi(\sigma,\ub)\right]_{-\ub-2}^u+\int_{-\ub-2}^u\frac{1}{2\sqrt{r}}\pr_u\pr_{\ub}r \phi(\sigma,\ub)d\sigma\\
&&-\int_{-\ub-2}^u\frac{1}{4r^{\frac{3}{2}}}\pr_ur\pr_{\ub}r \phi(\sigma,\ub)d\sigma-\int_{-\ub-2}^u\frac{\Omega^2f(\phi)}{4r^{\frac{3}{2}}}(\sigma,\ub)d\sigma.
\eee
We infer
\bee
\sqrt{r(u,\ub)}\pr_{\ub}\phi(u,\ub) &=& \sqrt{r(-\ub-2,\ub)}\pr_{\ub}\phi(-\ub-2,\ub)-\left[\frac{1}{2\sqrt{r}}\pr_{\ub}r \phi(\sigma,\ub)\right]_{-\ub-2}^u\\
&&+\int_{-\ub-2}^u\frac{\Omega^2}{8\sqrt{r}}\kappa \frac{g(\phi)^2}{r}\phi(\sigma,\ub)d\sigma - \int_{-\ub-2}^u\frac{1}{4r^{\frac{3}{2}}}\pr_ur\pr_{\ub}r \phi(\sigma,\ub)d\sigma\\
&& -\int_{-\ub-2}^u\frac{\Omega^2f(\phi)}{4r^{\frac{3}{2}}}(\sigma,\ub)d\sigma.
\eee
We deduce
\bee
\sqrt{r(u,\ub)}|\pr_{\ub}\phi(u,\ub)| &\les& \sqrt{r(-\ub-2,\ub)}|\pr_{\ub}\phi(-\ub-2,\ub)|+\frac{|\phi(-\ub-2,\ub)|}{ \sqrt{r(-\ub-2,\ub)}}+\frac{|\phi(u,\ub)|}{\sqrt{r(u,\ub)}}\\
&&+\int_{-\ub-2}^u\frac{|\phi(\sigma,\ub)|}{r^{\frac{3}{2}}} d\sigma
\eee
and hence
\bee
r(u,\ub)^{1-\delta}|\pr_{\ub}\phi(u,\ub)| &\les& r(u,\ub)^{\frac{1}{2}-\delta}\sqrt{r(-\ub-2,\ub)}|\pr_{\ub}\phi(-\ub-2,\ub)|+r(u,\ub)^{\frac{1}{2}-\delta}\frac{|\phi(-\ub-2,\ub)|}{ \sqrt{r(-\ub-2,\ub)}}\\
&&+r(u,\ub)^{-\delta}|\phi(u,\ub)|+r(u,\ub)^{\frac{1}{2}-\delta}\int_{-\ub-2}^u\frac{|\phi(\sigma,\ub)|}{{r(\sigma,\ub)^{\frac{3}{2}}}} d\sigma\\
&\les& C_0r(u,\ub)^{\frac{1}{2}-\delta}+r(u,\ub)^{-\delta}|\phi(u,\ub)|+r(u,\ub)^{\frac{1}{2}-\delta}\int_{-\ub-2}^u\frac{|\phi(\sigma,\ub)|}{{r(\sigma,\ub)^{\frac{3}{2}}}} d\sigma\\
&\les& C_0+r(u,\ub)^{-\delta}|\phi(u,\ub)|+r(u,\ub)^{\frac{1}{2}-\delta}\int_{-\ub-2}^u\frac{|\phi(\sigma,\ub)|}{{r(\sigma,\ub)^{\frac{3}{2}}}} d\sigma
\eee
where $C_0$ only depends on initial data.

Using the improved uniform bound for $\phi$ of lemma \ref{lemma:improvedunifboundphi}, we infer on $Q_{\ubb}$
\bee
r(u,\ub)^{1-\delta}|\pr_{\ub}\phi(u,\ub)| &\les& C_0+r(u,\ub)^{-\delta}(C_0\sqrt{\varrho(u,\ub)}+(\ep C+C_0)\varrho(u,\ub)^{\delta})\\
&&+r(u,\ub)^{\frac{1}{2}-\delta}\int_{-\ub-2}^u\frac{C_0\sqrt{\varrho(\sigma, \ub)}+(\ep C+C_0)\varrho(\sigma, \ub)^{\delta}}{{r(\sigma,\ub)^{\frac{3}{2}}}} d\sigma\\
&\les& C_0+\ep C
\eee
where we used the fact that $r\sim\varrho$ and $\delta<1/2$. Finally, we have obtained the existence of a universal constant $0<\underline{C}<+\infty$ such that we have on $Q_{\ubb}$
\bee
r(u,\ub)^{1-\delta}|\pr_{\ub}\phi(u,\ub)| &\leq& \underline{C}(C_0+\ep C).
\eee
This concludes the proof of Proposition \ref{prop:boot1}. 

\begin{remark}\lab{rem:howtouselemmathetaxi}
Lemma \ref{lemma:eqthetaandxi} is used as follows. The equation for $\Theta$ is always integrated from the initial data, while the equation for $\Xi$ is always integrated from the axis of symmetry $\Gamma$. We have the following three cases
\begin{itemize}
\item If $|\phi|\leq Cr^\delta$ with $\delta<1/2$, we deduce $|\pr_{\ub}\phi|\les Cr^{\delta-1}$, but no estimate for $\pr_u\phi$. We used this case above for the proof of Proposition \ref{prop:boot1}.

\item If $|\phi|\leq Cr^\delta$ with $\delta>1/2$, then we deduce $|\pr_u\phi|\les Cr^{\delta-1}$, and only $\pr_{\ub}\phi\les C\sqrt{r}$. This case will be used in the proof Lemma \ref{lemma:utlimaterefinedbounds}.

\item If $|\phi|\leq C\sqrt{r}$, then we have the log loss estimate $|\pr_{\ub}\phi|\les C|\log(r)|\sqrt{r}$, and no estimate for $\pr_u\phi$. Due to the log loss for $\pr_{\ub}\phi$, this case is never used.  
\end{itemize}
\end{remark}

\subsection{A more refined bound for $\phi$}

\begin{corollary}\label{cor:consequenceboot1}
For any
$$0<\delta<\frac{1}{2},$$
there exists a constant $C_0$ only depending on the values of the solution in $I_0$ such that we have for all $(\tau, \varrho)$ with $\tau\geq -1$ and $\ub\leq 0$
$$r^{1-\delta}|\pr_{\ub}\phi|+r^{-\delta}|\phi|\leq C_0$$
and
$$r^{-2\delta}\left(\frac{|r-\varrho|}{\varrho}+\left|\pr_ur+\frac{1}{2}\right|+\left|\pr_{\ub}r-\frac{1}{2}\right|+|\Omega-1|\right)\leq C_0.$$
\end{corollary}

\begin{proof}
Choosing $\ep>0$ sufficiently small in Proposition \ref{prop:boot1} yields 
\bea\lab{eq:consequencepropositionimproveboot}
\sup_{Q_{\ubb}}r^{1-\delta}|\pr_{\ub}\phi|\leq  \underline{C}C_0
\eea
which is an improvement of the bootstrap assumption \eqref{eq:boot1}. This implies that \eqref{eq:consequencepropositionimproveboot}  holds for all $-1\leq \ub_b<0$. Hence, as
$$\{\tau\geq -1, \,\,\,\, \ub\leq 0\}=\overline{\cup_{-1\leq \ub_b<0}Q_{\ubb}}\cup I_0,$$
we have for all $(\tau, \varrho)$ with $\tau\geq -1$ and $\ub\leq 0$
$$r^{1-\delta}|\pr_{\ub}\phi|\leq C_0$$
for a constant $C_0$ only depending on the values of the solution in $I_0$. Arguing as in Lemma \ref{lemma:consequencebootass}, we infer
$$r^{-\delta}|\phi|\leq C_0$$
and
$$r^{-2\delta}\left(\frac{|r-\varrho|}{\varrho}+\left|\pr_ur+\frac{1}{2}\right|+\left|\pr_{\ub}r-\frac{1}{2}\right|+|\Omega-1|\right)\leq C_0.$$
This concludes the proof of the corollary.
\end{proof}

In this section, we would like to prove the following refined bounds.
\begin{lemma}\label{lemma:refinedboundsphietal}
There exists a constant $C_0$ only depending on the values of the solution in $I_0$ such that we have for all $(\tau, \varrho)$ with $\tau\geq -1$ and $\ub\leq 0$
$$r^{-\frac{1}{2}}|\phi|\leq C_0$$
and
$$r^{-1}\left(\frac{|r-\varrho|}{\varrho}+\left|\pr_ur+\frac{1}{2}\right|+\left|\pr_{\ub}r-\frac{1}{2}\right|\right)\leq C_0.$$
\end{lemma}

\begin{remark}\lab{rem:problemlogloss}
The above improvement can not be true for $r\pr_{\ub}\phi$ due to a log loss when integrating the improved estimate for $\phi$ using the equation for $\Theta$ (see Remark \ref{rem:howtouselemmathetaxi}). In turn, this improvement can also not be true for $\Omega-1$.
\end{remark}

\begin{proof}
Recall that 
\bee
|\phi(\tau, \varrho)| &\les& C_0\sqrt{\varrho}+\sup_{0\leq\la\leq\varrho}|F_1(\tau-\varrho+\la,\la)|\\
&&+\int_{-1}^{+\infty}\int_0^{\la^*}\frac{\sqrt{\varrho}}{\sqrt{\la}\sqrt{\varrho^2+\la^2+2\varrho\la\mu}}|K(\mu)|(|F_1(\sigma, \la)|+|F_2(\sigma, \la)|)d\la d\mu\\
&&+\int_{-1}^{+\infty}\int_0^{\la^*}\frac{\sqrt{\varrho}}{\sqrt{\la}\sqrt{\varrho^2+\la^2+2\varrho\la\mu}}|\mu-1||K'(\mu)| |F_1(\sigma, \la)| d\la d\mu.
\eee
Also, we have in view of the definition of $F_1$ and $F_2$ together with the estimates of  Corollary \ref{cor:consequenceboot1}
$$r^{-3\delta}(|F_1|+|F_2|)\les C_0,$$
where we choose from now on $\delta$ such that
$$\frac{1}{6}<\delta<\frac{1}{2}.$$
We infer 
\bee
|\phi(\tau, \varrho)| &\les& C_0\sqrt{\varrho}+C_0\varrho^{3\delta}\\
&&+C_0\sqrt{\varrho}\int_{-1}^{+\infty}\int_0^{\la^*}\frac{\la^{3\delta-\frac{1}{2}}}{\sqrt{\varrho^2+\la^2+2\varrho\la\mu}}(|K(\mu)|+|\mu-1||K'(\mu)|) d\la d\mu.
\eee
We have
\bee
&&\int_0^{\la^*}\frac{\la^{3\delta-\frac{1}{2}}}{\sqrt{\varrho^2+\la^2+2\varrho\la\mu}} d\la \\
&=& \int_0^{2\varrho}\frac{\la^{3\delta-\frac{1}{2}}}{\sqrt{\varrho^2+\la^2+2\varrho\la\mu}} d\la+\int_{2\varrho}^{\la^*}\frac{\la^{3\delta-\frac{1}{2}}}{\sqrt{\varrho^2+\la^2+2\varrho\la\mu}} d\la\\
&\les&  \varrho^{2\delta}\int_0^{2\varrho}\frac{\la^{\delta-\frac{1}{2}}}{\sqrt{\varrho^2+\la^2+2\varrho\la\mu}} d\la+\int_{2\varrho}^{\la^*}\frac{\la^{3\delta-\frac{1}{2}}}{\sqrt{(\varrho-\la)^2+2\varrho\la(1+\mu)}} d\la.
\eee
Since $\mu\geq -1$, we infer
\bee
&&\int_0^{\la^*}\frac{\la^{3\delta-\frac{1}{2}}}{\sqrt{\varrho^2+\la^2+2\varrho\la\mu}} d\la \\
&\les&  \varrho^{2\delta}\int_0^{2\varrho}\frac{\la^{\delta-\frac{1}{2}}}{\sqrt{\varrho^2+\la^2+2\varrho\la\mu}} d\la+\int_{2\varrho}^{\la^*}\frac{\la^{3\delta-\frac{1}{2}}}{\la} d\la\\
&\les&  \varrho^{2\delta}\int_0^{\la^*}\frac{\la^{\delta-\frac{1}{2}}}{\sqrt{\varrho^2+\la^2+2\varrho\la\mu}} d\la+\left[\la^{3\delta-\frac{1}{2}}\right]_{2\varrho}^{\la^*}\\
&\les&  1+\varrho^{2\delta}\int_0^{\la^*}\frac{\la^{\delta-\frac{1}{2}}}{\sqrt{\varrho^2+\la^2+2\varrho\la\mu}} d\la
\eee
where used the fact that we chose $\delta>1/6$. 

We have obtained
\bee
|\phi(\tau, \varrho)| &\les& C_0\sqrt{\varrho}\\
&&+C_0\varrho^{2\delta}\sqrt{\varrho}\int_{-1}^{+\infty}\int_0^{\la^*}\frac{\la^{\delta-\frac{1}{2}}}{\sqrt{\varrho^2+\la^2+2\varrho\la\mu}}(|K(\mu)|+|\mu-1||K'(\mu)|) d\la d\mu,
\eee
where we used the fact that we chose $\delta$ such that $\delta>1/6$. Recall estimate \eqref{eq:estimusefullagainater}:
$$\sqrt{\varrho}\int_{-1}^{+\infty}\int_0^{\la^*}\frac{\la^{\delta-\frac{1}{2}}}{\sqrt{\varrho^2+\la^2+2\varrho\la\mu}}(|K(\mu)|+|\mu-1||K'(\mu)|) d\la d\mu\les \varrho^{\delta}.$$
We infer
\bee
|\phi(\tau, \varrho)| &\les& C_0\sqrt{\varrho}+C_0\varrho^{3\delta}\\
&\les& C_0\sqrt{\varrho},
\eee
where we used the fact that we chose $\delta$ such that $\delta>1/6$. This concludes the improvement for $\phi$. \\

Then, one obtains the improvements\footnote{Note that the proof of these improvements only relies on the equation $\pr_u\pr_{\ub}r=\kappa \Omega^2g(\phi)^2/(4r)$. Therefore, the proof requires the improved estimate for $\phi$. The important point is that it does not require the corresponding improvement for $\pr_{\ub}\phi$ which does not hold due to a log loss (see Remark \ref{rem:problemlogloss}).} for $\pr_{\ub}r-1/2$, $\pr_ur+1/2$ and $r-\varrho$ as in the proof of Lemma \ref{lemma:consequencebootass}. This concludes the proof of the lemma.
\end{proof}

\section{An improved uniform bound for $\pr\phi$} \label{sec:improved-uniform-dphi}

Here, we differentiate equation for $\phi$ once and we obtain a uniform bound for $\pr\phi$ and $\phi/r$ again relying on an explicit representation formula for the flat wave equation by adapting the approach in \cite{jal_tah1}.

\subsection{Upper bounds for higher order derivatives}\lab{sec:higherderest}

We start by deriving an upper bound for $\pr_{\ub}\Omega$. 
\begin{lemma}\lab{lemma:upperboundprubOmega}
There exists a constant $C_0$ only depending on the values of the solution in $I_0$ such that we have for any $0<\delta<1/2$ and  for all $(\tau, \varrho)$ with $\tau\geq -1$ and $\ub\leq 0$
\bee
r^{1-2\delta}|\pr_{\ub}\Omega| &\les& C_0.
\eee 
\end{lemma}

\begin{proof}
Recall that 
$$\Omega^{-2}(\pr_u\Omega\pr_{\ub}\Omega-\Omega\pr_u\pr_{\ub}\Omega)=\frac{1}{8}\Omega^2\kappa\left(\frac{4}{\Omega^2}\pr_u\phi\pr_{\ub}\phi+\frac{g(\phi)^2}{r^2}\right).$$
We deduce
\bee
\pr_u(\pr_{\ub}\log(\Omega)) &=& \pr_u\left(\frac{\pr_{\ub}\Omega}{\Omega}\right)\\
&=& \frac{\pr_u\pr_{\ub}\Omega}{\Omega}-\frac{\pr_u\Omega\pr_{\ub}\Omega}{\Omega^2}\\
&=& -\frac{\kappa}{2}\pr_u\phi\pr_{\ub}\phi-\frac{\kappa\Omega^2}{8}\frac{g(\phi)^2}{r^2}.
\eee
Also, recall that
\bee
\square_\gg(\phi) &=& \frac{1}{\Omega^2}\left(-4\pr_u\pr_{\ub}\phi-\frac{2\pr_ur}{r}\pr_{\ub}\phi-\frac{2\pr_{\ub}r}{r}\pr_u\phi\right)
\eee
and
$$\square_\gg\phi=\frac{f(\phi)}{r^2}.$$
We infer
\bee
\pr_u\pr_{\ub}\phi=-\frac{\pr_ur}{2r}\pr_{\ub}\phi-\frac{\pr_{\ub}r}{2r}\pr_u\phi-\frac{\Omega^2f(\phi)}{4r^2}.
\eee
Hence, we deduce
\bee
\pr_u(\pr_{\ub}\log(\Omega)) &=& -\frac{\kappa}{2}\pr_u(\phi\pr_{\ub}\phi)+\frac{\kappa}{2}\phi\pr_u\pr_{\ub}\phi-\frac{\kappa\Omega^2}{8}\frac{g(\phi)^2}{r^2}\\
&=&  -\frac{\kappa}{2}\pr_u(\phi\pr_{\ub}\phi)+\frac{\kappa}{2}\phi\left(-\frac{\pr_ur}{2r}\pr_{\ub}\phi-\frac{\pr_{\ub}r}{2r}\pr_u\phi-\frac{\Omega^2f(\phi)}{4r^2}\right)-\frac{\kappa\Omega^2}{8}\frac{g(\phi)^2}{r^2}\\
&=&  -\frac{\kappa}{2}\pr_u(\phi\pr_{\ub}\phi)-\frac{\kappa}{8}\pr_u\left(\frac{\pr_{\ub}r}{r}\phi^2\right)+\frac{\kappa\pr_u\pr_{\ub}r}{8r}\phi^2-\frac{\kappa\pr_ur\pr_{\ub}r}{8r^2}\phi^2-\frac{\kappa\pr_ur}{4r}\phi\pr_{\ub}\phi\\
&&-\frac{\kappa\Omega^2\phi f(\phi)}{8r^2}-\frac{\kappa\Omega^2}{8}\frac{g(\phi)^2}{r^2}\\
&=&  -\frac{\kappa}{2}\pr_u(\phi\pr_{\ub}\phi)-\frac{\kappa}{8}\pr_u\left(\frac{\pr_{\ub}r}{r}\phi^2\right)+\frac{\kappa^2\Omega^2}{32r^2}g(\phi)^2\phi^2-\frac{\kappa\pr_ur\pr_{\ub}r}{8r^2}\phi^2-\frac{\kappa\pr_ur}{4r}\phi\pr_{\ub}\phi\\
&&-\frac{\kappa\Omega^2\phi f(\phi)}{8r^2}-\frac{\kappa\Omega^2}{8}\frac{g(\phi)^2}{r^2}.
\eee
We integrate between $(u,\ub)$ and $(-\ub-2, \ub)$, where $(-\ub-2, \ub)$ is on $\tau=-1$. We deduce 
\bee
\left[\pr_{\ub}\log(\Omega)(\sigma,\ub)\right]_{-\ub-2}^u &=&  -\frac{\kappa}{2}\left[\phi\pr_{\ub}\phi(\sigma,\ub)\right]_{-\ub-2}^u-\frac{\kappa}{8}\left[\frac{\pr_{\ub}r}{r}\phi^2(\sigma,\ub)\right]_{-\ub-2}^u+\int_{-\ub-2}^u\frac{\kappa^2\Omega^2}{32r^2}g(\phi)^2\phi^2(\sigma,\ub)d\sigma\\
&&-\int_{-\ub-2}^u\frac{\kappa\pr_ur\pr_{\ub}r}{8r^2}\phi^2(\sigma,\ub)d\sigma-\int_{-\ub-2}^u\frac{\kappa\pr_ur}{4r}\phi\pr_{\ub}\phi(\sigma,\ub)d\sigma\\
&&-\int_{-\ub-2}^u\frac{\kappa\Omega^2\phi f(\phi)}{8r^2}(\sigma,\ub)d\sigma-\int_{-\ub-2}^u\frac{\kappa\Omega^2}{8}\frac{g(\phi)^2}{r^2}(\sigma,\ub)d\sigma.
\eee 
This yields
\bee
|\pr_{\ub}\log(\Omega)(u,\ub)| &\les& C_0+ |\phi|\left(|\pr_{\ub}\phi|+\frac{|\pr_{\ub}r|}{r}|\phi|\right)(u,\ub)\\
&&+\int_{-\ub-2}^u\frac{|\pr_ur|}{r}|\phi|\left(|\pr_{\ub}\phi|+\frac{|\pr_{\ub}r|}{r}|\phi|\right)(\sigma,\ub)d\sigma+\int_{-\ub-2}^u\frac{\Omega^2|\phi|}{r^2}(\sigma,\ub)d\sigma\\
&\les& C_0+r^{2\delta-1}C_0\\
&\les& r^{2\delta-1}C_0,
\eee 
where we used the fact that $0<\delta<1/2$, the upper bounds on $\phi$, $\pr_{\ub}\phi$ of Corollary \ref{cor:consequenceboot1}, the uniform bounds for $\pr_ur$, $\pr_{\ub}r$ and $\Omega$, and the fact that $\varrho\sim r$. In view of the fact that $|\Omega|\sim 1$, we finally obtain
\bee
r^{1-2\delta}|\pr_{\ub}\Omega| &\les& C_0.
\eee 
This concludes the proof of the lemma.
\end{proof}

Next, we derive an upper bound for $\pr_{\ub}^2r$. 
\begin{lemma}\lab{lemma:upperboundforpr2ubr}
There exists a constant $C_0$ only depending on the values of the solution in $I_0$ such that we have for any $0<\delta<1/2$ and for all $(\tau, \varrho)$ with $\tau\geq -1$ and $\ub\leq 0$
\bee
|\pr_{\ub}^2r(u,\ub)| &\les&  r^{2\delta-1}C_0.
\eee
\end{lemma}

\begin{proof}
Recall that
$$\pr_u\pr_{\ub}r=\kappa \frac{\Omega^2g(\phi)^2}{4r}.$$
We deduce
\bee
\pr_u\pr_{\ub}^2r &=& \kappa \frac{\Omega(\pr_{\ub}\Omega)g(\phi)^2}{2r}+ \kappa \frac{\Omega^2g(\phi)g'(\phi)\pr_{\ub}\phi}{2r}- \kappa \frac{\Omega^2g(\phi)^2\pr_{\ub}r}{4r^2}.
\eee
We integrate between $(u,\ub)$ and $(-\ub-2, \ub)$, where $(-\ub-2, \ub)$ is on $\tau=-1$. We deduce
\bee
[\pr_{\ub}^2r(\sigma,\ub)]_{-\ub-2}^u &=& \kappa\int_{-\ub-2}^u\frac{\Omega(\pr_{\ub}\Omega)g(\phi)^2}{2r}(\sigma, \ub)d\sigma+ \kappa\int_{-\ub-2}^u \frac{\Omega^2g(\phi)g'(\phi)\pr_{\ub}\phi}{2r}(\sigma, \ub)d\sigma \\
&&- \kappa \int_{-\ub-2}^u\frac{\Omega^2g(\phi)^2\pr_{\ub}r}{4r^2}(\sigma, \ub)d\sigma
\eee
and hence
\bee
|\pr_{\ub}^2r(u,\ub)| &\les& C_0+\int_{-\ub-2}^u\left(\frac{|\pr_{\ub}\Omega||\phi|^2}{r}+\frac{|\phi||\pr_{\ub}\phi|}{r}+\frac{|\phi|^2|\pr_{\ub}r|}{r^2}\right)(\sigma, \ub)d\sigma\\
&\les & C_0+ r^{2\delta-1}C_0\\
&\les&  r^{2\delta-1}C_0,
\eee
where we used the fact that $0<\delta<1/2$, the upper bounds on $\phi$, $\pr_{\ub}\phi$ of Corollary \ref{cor:consequenceboot1}, the uniform bounds for $\pr_{\ub}r$ and $\Omega$, the upper bounds on $\pr_{\ub}\Omega$ of Lemma \ref{lemma:upperboundprubOmega}, the fact that $\pr^2_{\ub}r$ is bounded on $\tau=-1$ by $C_0$ since $C_0$ controls in particular two derivatives of the solution on $I_0$, and the fact that $\varrho\sim r$. This concludes the proof of the lemma.
\end{proof}

Next, we derive an upper bound for $\pr_{\ub}^2\phi$. 
\begin{lemma}\lab{lemma:upperboundpr2ubphi}
There exists a constant $C_0$ only depending on the values of the solution in $I_0$ such that we have for any $0<\delta<1/2$ and for all $(\tau, \varrho)$ with $\tau\geq -1$ and $\ub\leq 0$
\bee
|\pr_{\ub}^2\phi| &\les&  r^{\delta-2}C_0.
\eee
\end{lemma}

\begin{proof}
Recall that
$$\Theta=r\pr_{\ub}\phi$$
satisfies
\bee
\pr_u\left(\frac{\Theta}{\sqrt{r}}\right)=-\frac{1}{2\sqrt{r}}\pr_{\ub}r \pr_u\phi-\frac{\Omega^2f(\phi)}{4r^{\frac{3}{2}}}.
\eee
Differentiating with respect to $\ub$, we obtain
\bee
\pr_u\pr_{\ub}\left(\frac{\Theta}{\sqrt{r}}\right) &=& -\frac{1}{2\sqrt{r}}\pr_{\ub}r \pr_u\pr_{\ub}\phi-\frac{1}{2\sqrt{r}}\pr^2_{\ub}r \pr_u\phi+\frac{1}{4r^{\frac{3}{2}}}(\pr_{\ub}r)^2 \pr_u\phi-\frac{\Omega^2f'(\phi)\pr_{\ub}\phi}{4r^{\frac{3}{2}}}\\
&&-\frac{2\Omega (\pr_{\ub}\Omega)f(\phi)}{4r^{\frac{3}{2}}}+\frac{3\Omega^2f(\phi)\pr_{\ub}r}{8r^{\frac{5}{2}}}.
\eee
In view of
\bee
\pr_u\pr_{\ub}\phi=-\frac{\pr_ur}{2r}\pr_{\ub}\phi-\frac{\pr_{\ub}r}{2r}\pr_u\phi-\frac{\Omega^2f(\phi)}{4r^2}.
\eee
we deduce
\bee
\pr_u\pr_{\ub}\left(\frac{\Theta}{\sqrt{r}}\right) &=& -\frac{1}{2\sqrt{r}}\pr_{\ub}r \left(-\frac{\pr_ur\pr_{\ub}\phi}{2r}-\frac{\pr_{\ub}r\pr_u\phi}{2r}-\frac{\Omega^2f(\phi)}{4r^2}\right)-\frac{1}{2\sqrt{r}}\pr^2_{\ub}r \pr_u\phi+\frac{1}{4r^{\frac{3}{2}}}(\pr_{\ub}r)^2 \pr_u\phi\\
&&-\frac{\Omega^2f'(\phi)\pr_{\ub}\phi}{4r^{\frac{3}{2}}}-\frac{2\Omega (\pr_{\ub}\Omega)f(\phi)}{4r^{\frac{3}{2}}}+\frac{3\Omega^2f(\phi)\pr_{\ub}r}{8r^{\frac{5}{2}}}\\
&=& \left(-\frac{1}{2\sqrt{r}}\pr^2_{\ub}r +\frac{1}{2r^{\frac{3}{2}}}(\pr_{\ub}r)^2\right) \pr_u\phi+\frac{ \pr_{\ub}r\pr_ur\pr_{\ub}\phi}{4r^{\frac{3}{2}}}\\
&&-\frac{\Omega^2f'(\phi)\pr_{\ub}\phi}{4r^{\frac{3}{2}}}-\frac{2\Omega (\pr_{\ub}\Omega)f(\phi)}{4r^{\frac{3}{2}}}+\frac{\Omega^2f(\phi)\pr_{\ub}r}{2r^{\frac{5}{2}}}\\
&=& \pr_u\left[-\frac{1}{2\sqrt{r}}\pr^2_{\ub}r\phi +\frac{1}{2r^{\frac{3}{2}}}(\pr_{\ub}r)^2\phi\right]+\frac{1}{2\sqrt{r}}\pr_u\pr^2_{\ub}r\phi -\frac{\pr_ur}{4r^{\frac{3}{2}}}\pr^2_{\ub}r\phi -\frac{1}{r^{\frac{3}{2}}}\pr_{\ub}r\pr_u\pr_{\ub}r\phi\\
&&+\frac{3}{4r^{\frac{5}{2}}}\pr_ur(\pr_{\ub}r)^2\phi+\frac{ \pr_{\ub}r\pr_ur\pr_{\ub}\phi}{4r^{\frac{3}{2}}}-\frac{\Omega^2f'(\phi)\pr_{\ub}\phi}{4r^{\frac{3}{2}}}-\frac{2\Omega (\pr_{\ub}\Omega)f(\phi)}{4r^{\frac{3}{2}}}+\frac{\Omega^2f(\phi)\pr_{\ub}r}{2r^{\frac{5}{2}}}.
\eee
Recall that
$$\pr_u\pr_{\ub}r=\kappa \frac{\Omega^2g(\phi)^2}{4r},$$
and
\bee
\pr_u\pr_{\ub}^2r &=& \kappa \frac{\Omega(\pr_{\ub}\Omega)g(\phi)^2}{2r}+ \kappa \frac{\Omega^2g(\phi)g'(\phi)\pr_{\ub}\phi}{2r}- \kappa \frac{\Omega^2g(\phi)^2\pr_{\ub}r}{4r^2}.
\eee
We deduce
\bee
\pr_u\pr_{\ub}\left(\frac{\Theta}{\sqrt{r}}\right) &=& \pr_u\left[-\frac{1}{2\sqrt{r}}\pr^2_{\ub}r\phi +\frac{1}{2r^{\frac{3}{2}}}(\pr_{\ub}r)^2\phi\right]\\
&&+ \frac{\kappa\Omega(\pr_{\ub}\Omega)g(\phi)^2}{4r^{\frac{3}{2}}}\phi + \frac{\kappa\Omega^2g(\phi)g'(\phi)\pr_{\ub}\phi}{4r^{\frac{3}{2}}}\phi -  \frac{3\kappa\Omega^2g(\phi)^2\pr_{\ub}r}{8r^{\frac{5}{2}}}\phi \\
&&-\frac{\pr_ur}{4r^{\frac{3}{2}}}\pr^2_{\ub}r\phi +\frac{3}{4r^{\frac{5}{2}}}\pr_ur(\pr_{\ub}r)^2\phi+\frac{ \pr_{\ub}r\pr_ur\pr_{\ub}\phi}{4r^{\frac{3}{2}}}-\frac{\Omega^2f'(\phi)\pr_{\ub}\phi}{4r^{\frac{3}{2}}}\\
&&-\frac{\Omega (\pr_{\ub}\Omega)f(\phi)}{2r^{\frac{3}{2}}}+\frac{\Omega^2f(\phi)\pr_{\ub}r}{2r^{\frac{5}{2}}}.
\eee
We integrate between $(u,\ub)$ and $(-\ub-2, \ub)$, where $(-\ub-2, \ub)$ is on $\tau=-1$. We deduce
\bee
\left[\pr_{\ub}\left(\frac{\Theta}{\sqrt{r}}\right)(\sigma,\ub)\right]_{-\ub-2}^u &=& \left[-\frac{1}{2\sqrt{r}}\pr^2_{\ub}r\phi +\frac{1}{2r^{\frac{3}{2}}}(\pr_{\ub}r)^2\phi(\sigma,\ub)\right]_{-\ub-2}^u\\
&&+ \int_{-\ub-2}^u\Bigg\{\frac{\kappa\Omega(\pr_{\ub}\Omega)g(\phi)^2}{4r^{\frac{3}{2}}}\phi + \frac{\kappa\Omega^2g(\phi)g'(\phi)\pr_{\ub}\phi}{4r^{\frac{3}{2}}}\phi -  \frac{3\kappa\Omega^2g(\phi)^2\pr_{\ub}r}{8r^{\frac{5}{2}}}\phi \\
&&-\frac{\pr_ur}{4r^{\frac{3}{2}}}\pr^2_{\ub}r\phi +\frac{3}{4r^{\frac{5}{2}}}\pr_ur(\pr_{\ub}r)^2\phi+\frac{ \pr_{\ub}r\pr_ur\pr_{\ub}\phi}{4r^{\frac{3}{2}}}-\frac{\Omega^2f'(\phi)\pr_{\ub}\phi}{4r^{\frac{3}{2}}}\\
&&-\frac{\Omega (\pr_{\ub}\Omega)f(\phi)}{2r^{\frac{3}{2}}}+\frac{\Omega^2f(\phi)\pr_{\ub}r}{2r^{\frac{5}{2}}}\Bigg\}(\sigma,\ub)d\sigma
\eee
and hence
\bee
\left|\pr_{\ub}\left(\frac{\Theta}{\sqrt{r}}\right)(u,\ub)\right| &\les& \frac{C_0}{\sqrt{r(-\ub-2, \ub)}}+\left(\frac{|\phi||\pr^2_{\ub}r|}{\sqrt{r}}+\frac{(\pr_{\ub}r)^2|\phi|}{r^{\frac{3}{2}}}\right)(u,\ub)\\
&&+\int_{-\ub-2}^u\frac{|\pr_{\ub}\phi||\phi|+|\pr^2_{\ub}r||\phi|+|\phi|+|\pr_{\ub}\phi|+|\pr_{\ub}\Omega||\phi|}{r^{\frac{3}{2}}}(\sigma,\ub)d\sigma\\
&& +\int_{-\ub-2}^u\frac{|\phi|}{r^{\frac{5}{2}}}(\sigma,\ub)d\sigma\\
&\les &  \frac{C_0}{\sqrt{r(-\ub-2, \ub)}}+ r^{\delta-\frac{3}{2}}C_0\\
&\les&  r^{\delta-\frac{3}{2}}C_0,
\eee
where we used the fact that $0<\delta<1/2$, the upper bounds on $\phi$, $\pr_{\ub}\phi$ of Corollary \ref{cor:consequenceboot1}, the uniform bounds for $\pr_ur$, $\pr_{\ub}r$ and $\Omega$, the upper bounds on $\pr_{\ub}\Omega$ of Lemma \ref{lemma:upperboundprubOmega}, the upper bounds on $\pr^2_{\ub}r$ of Lemma \ref{lemma:upperboundforpr2ubr}, the fact that $\varrho\sim r$ and the fact that $r(-\ub-2, \ub)\geq r$. Since
$$\Theta=r\pr_{\ub}\phi,$$
we infer
$$\pr_{\ub}\left(\frac{\Theta}{\sqrt{r}}\right)=\sqrt{r}\pr_{\ub}^2\phi+\frac{1}{2\sqrt{r}}\pr_{\ub}\phi$$
and hence
\bee
|\pr_{\ub}^2\phi| &\les&  r^{\delta-2}C_0.
\eee
This concludes the proof of the lemma.
\end{proof}

\subsection{A wave equation satisfied by $v$}

\begin{lemma}\lab{lemma:structurewaveeqv}
Let 
$$v=D\phi=\left(\pr_{\varrho}+\frac{1}{\varrho}\right)\phi.$$
We have
\bee
\left(-\pr_\tau^2+\pr_{\varrho}^2+\frac{1}{\varrho}\pr_\varrho\right)v &=& \pr_u\left(\frac{B_1}{\varrho}v+\frac{B_2}{\varrho^2}\right)+\frac{B_3}{\varrho^2}v+\frac{B_4}{\varrho^3},
\eee
where
\bee
B_1 &=&  \frac{\varrho-r+\varrho(2\pr_{\ub}r-1)}{r},
\eee
\bee
B_2 &=&\frac{2\varrho^2(\pr_{\varrho}r-1)}{r}\pr_{\ub}\phi -\frac{\varrho-r+\varrho(2\pr_{\ub}r-1)-2\varrho^2\pr_{\ub}^2r}{r}\phi-\frac{\varrho^2(\Omega^2-1)f(\phi)}{r^2},
\eee
\bee
B_3 &=& \frac{\varrho-r}{r}\left(\frac{1}{2}+\frac{3\varrho}{2r}\pr_{\ub}r+\frac{\varrho}{r}\right)+\frac{\varrho(2\pr_{\ub}r-1)}{r}\left(-1+\frac{\varrho}{r}\pr_ru+\frac{\varrho}{r}\pr_{\varrho}r\right)+\frac{\varrho(2\pr_ur+1)}{2r}\left(1-\frac{\varrho}{r}\pr_{\ub}r\right)\\
&& \frac{\varrho(\pr_{\varrho}r-1)}{r}\left(1-\frac{\varrho}{r}\pr_{\ub}r\right)+\frac{\varrho^2(3\phi^2\zeta(\phi)+\phi^3\zeta'(\phi))}{r^2},
\eee
and
\bee
B_4 &=& \Bigg[\left(-\frac{3}{2}-\frac{\varrho}{r}\pr_ur-\frac{2\varrho^2+3\varrho r}{r^2}\pr_{\varrho}r-\frac{\varrho}{2r}\pr_{\ub}r\right)\frac{\varrho-r}{r}+\left(-\frac{\varrho\pr_{\ub}r}{r}+1\right)\frac{\varrho(2\pr_{\ub}r-1)}{r}\\
&&+\left(-3+\frac{\varrho\pr_{\ub}r}{r}\right)\frac{\varrho(\pr_{\varrho}r-1)}{r}+\left(-\frac{1}{2}+\frac{\varrho\pr_{\ub}r}{2r}\right)\frac{\varrho\left(2\pr_ur+1\right)}{r}+\frac{\kappa\varrho^3\Omega^2g(\phi)^2\pr_{\ub}r}{2r^3}\\
&&+\frac{2\varrho^3\pr_ur\pr_{\ub}^2r}{r^2}-\frac{\varrho^2(2\phi^2\zeta(\phi)+\phi^3\zeta'(\phi))}{r^2}-\frac{2\varrho^3\phi^2\zeta(\phi)\pr_{\varrho}r}{r^3}\Bigg]\phi\\
&&+\Bigg[\frac{\varrho(3\pr_ur+\pr_{\ub}r)}{r} \frac{\varrho^2(\pr_{\varrho}r-1)}{r}\\
&&+\left(-\frac{5\varrho^2(\varrho-r)}{2r}+\frac{\varrho^3(2\pr_ur+1)}{2r}-\frac{2\varrho^3(\pr_{\varrho}r-1)}{r}\right)\frac{\pr_ur+\pr_{\ub}r}{r}\\
&&-\frac{\kappa\varrho^3\Omega^2g(\phi)g'(\phi)\phi}{r^2}-\frac{\kappa\varrho^3\Omega^2g(\phi)^2}{2r^2}\Bigg]\pr_{\ub}\phi+\Bigg[-\frac{\varrho^2(\varrho-r)}{r}+\frac{\varrho^3(2\pr_ur+1)}{r}\Bigg]\pr_{\ub}^2\phi\\
&&+\Bigg[\left(-\frac{2\varrho\pr_{\ub}r}{r}+1\right)\frac{\varrho^2f(\phi)}{r^2}+\frac{\varrho^2f'(\phi)\varrho\pr_{\ub}\phi}{r^2}\Bigg](\Omega^2-1)\\
&&+\Bigg[-\frac{\varrho(\varrho-r)}{r}\frac{\varrho\Omega^2}{4r^2}+\frac{\varrho^3(2\pr_{\ub}r-1)\Omega^2}{4r^3}\Bigg]f(\phi)+(2f(\phi)-\kappa g(\phi)^2\phi)\frac{\varrho^3\Omega\pr_{\ub}\Omega}{r^2}.
\eee
\end{lemma}

\begin{remark}
Again, we need to integrate by parts the terms involving derivatives with respect to $u$ due to a better behavior with respect to $\ub$ derivatives (see for example the estimates for $\pr^2_{\ub}r$, $\pr_{\ub}\Omega$ and $\pr_{\ub}^2\phi$ of section \ref{sec:higherderest} which do not have a corresponding $u$ counterpart). This results in the term $\pr_u(B_1v/\varrho+B_2\varrho^2)$ in the statement of Lemma \ref{lemma:structurewaveeqv}. As emphasized in Remark \ref{rem:integrationbypartsuexplained}, the fact that this integration by parts is possible is a consequence on the one hand of the null structure of the problem, and on the other hand of the nice behavior of the kernel of the representation formula for the wave equation with respect to $u$ derivatives. 
\end{remark}

\begin{remark}
The crucial point of the decomposition of the right-hand side of the wave equation for $v$ in the statement of Lemma \ref{lemma:structurewaveeqv} is the fact that both $B_1$ and $B_3$ include neither $\pr_{\ub}\phi$ nor $\pr_{\ub}\Omega$ in their definition, and hence will satisfy better estimates than $B_2$ and $B_4$ (see Lemma \ref{lemma:upperboundB1234} and Remark \ref{rem:B13betterthanB24}).
\end{remark}

\begin{proof}
We have
$$D\circ\left[-\pr_\tau^2+\pr_{\varrho}^2+\frac{1}{\varrho}\pr_\varrho-\frac{1}{\varrho^2}\right]=\left[-\pr_\tau^2+\pr_{\varrho}^2+\frac{1}{\varrho}\pr_\varrho\right]\circ D.$$
Recall that
$$\left(-\pr_\tau^2+\pr_{\varrho}^2+\frac{1}{\varrho}\pr_\varrho-\frac{1}{\varrho^2}\right)\phi=\frac{F}{\varrho^2}$$
where
\bee
F &=& -\frac{\varrho(\varrho-r)}{r}(\pr_{\ub}\phi-\pr_u\phi)+\frac{\varrho^2(2\pr_ur+1)}{r}\pr_{\ub}\phi+\frac{\varrho^2(2\pr_{\ub}r-1)}{r}\pr_u\phi+\frac{\varrho^2-r^2}{r^2}\phi+\frac{\varrho^2\phi^3\zeta(\phi)}{r^2}\\
&&+\varrho^2(\Omega^2-1)\frac{f(\phi)}{r^2}.
\eee
We infer
$$\left(-\pr_\tau^2+\pr_{\varrho}^2+\frac{1}{\varrho}\pr_\varrho\right)v=\frac{\pr_{\varrho}F}{\varrho^2}-\frac{F}{\varrho^3}.$$

Next, we compute $\pr_{\varrho}F$. We have
\bee
\pr_{\varrho}F &=& -\frac{\varrho(\varrho-r)}{r}(\pr_{\ub}\pr_{\varrho}\phi-\pr_u\pr_{\varrho}\phi)-\frac{r(2\varrho-r-\varrho\pr_{\varrho}r)-\varrho(\varrho-r)\pr_{\varrho}r}{r^2}(\pr_{\ub}\phi-\pr_u\phi)\\
&&+\frac{\varrho^2(2\pr_ur+1)}{r}\pr_{\ub}\pr_{\varrho}\phi+\frac{r(2\varrho(2\pr_ur+1)+2\varrho^2\pr_u\pr_{\varrho}r)-\varrho^2(2\pr_ur+1)\pr_{\varrho}r}{r^2}\pr_{\ub}\phi\\
&&+\frac{\varrho^2(2\pr_{\ub}r-1)}{r}\pr_u\pr_{\varrho}\phi+\frac{r(2\varrho(2\pr_{\ub}r-1)+2\varrho^2\pr_{\ub}\pr_{\varrho}r)-\varrho^2(2\pr_{\ub}r-1)\pr_{\varrho}r}{r^2}\pr_u\phi\\
&&+\frac{\varrho^2-r^2}{r^2}\pr_{\varrho}\phi+\frac{r^2(2\varrho-2r\pr_{\varrho}r)-2(\varrho^2-r^2)r\pr_{\varrho}r}{r^4}\phi\\
&&+\frac{r^2(2\varrho\phi^3\zeta(\phi)+\varrho^2(3\phi^2\zeta(\phi)+\phi^3\zeta'(\phi))\pr_{\varrho}\phi)-2\varrho^2\phi^3\zeta(\phi)r\pr_{\varrho}r}{r^4}\\
&&+\frac{r^2(2\varrho(\Omega^2-1)f(\phi)+2\varrho^2\Omega\pr_{\varrho}(\Omega)f(\phi)+\varrho^2(\Omega^2-1)f'(\phi)\pr_{\varrho}\phi)-2\varrho^2(\Omega^2-1)f(\phi)r\pr_{\varrho}r}{r^4}\\
&=& \pr_u\left(\frac{\varrho(\varrho-r)}{r}\pr_{\varrho}\phi+\frac{2\varrho^2(\pr_{\varrho}r-1)}{r}\pr_{\ub}\phi+\frac{\varrho^2(2\pr_{\ub}r-1)}{r}\pr_{\varrho}\phi-\frac{\varrho^2(\Omega^2-1)f(\phi)}{r^2}\right)\\
&& -\pr_u\left(\frac{\varrho(\varrho-r)}{r}\right)\pr_{\varrho}\phi-\pr_u\left(\frac{2\varrho^2}{r}\pr_{\ub}\phi\right)(\pr_{\varrho}r-1)-\pr_u\left(\frac{\varrho^2(2\pr_{\ub}r-1)}{r}\right)\pr_{\varrho}\phi\\
&&+\pr_u\left(\frac{\varrho^2f(\phi)}{r^2}\right)(\Omega^2-1)-\frac{\varrho(\varrho-r)}{r}\pr_{\ub}\pr_{\varrho}\phi-\frac{r(2\varrho-r-\varrho\pr_{\varrho}r)-\varrho(\varrho-r)\pr_{\varrho}r}{r^2}\pr_{\varrho}\phi\\
&&+\frac{\varrho^2(2\pr_ur+1)}{r}\pr_{\ub}\pr_{\varrho}\phi+\frac{2r\varrho(2\pr_ur+1)-\varrho^2(2\pr_ur+1)\pr_{\varrho}r}{r^2}\pr_{\ub}\phi\\
&&+\frac{r(2\varrho(2\pr_{\ub}r-1)-2\varrho^2\pr_{\ub}\pr_ur)-\varrho^2(2\pr_{\ub}r-1)\pr_{\varrho}r}{r^2}(\pr_{\ub}\phi-\pr_{\varrho}\phi)+\frac{2\varrho^2\pr_{\ub}^2r}{r}\pr_u\phi\\
&&+\frac{\varrho^2-r^2}{r^2}\pr_{\varrho}\phi+\frac{r^2(2\varrho-2r\pr_{\varrho}r)-2(\varrho^2-r^2)r\pr_{\varrho}r}{r^4}\phi\\
&&+\frac{r^2(2\varrho\phi^3\zeta(\phi)+\varrho^2(3\phi^2\zeta(\phi)+\phi^3\zeta'(\phi))\pr_{\varrho}\phi)-2\varrho^2\phi^3\zeta(\phi)r\pr_{\varrho}r}{r^4}\\
&&+\frac{r^2(2\varrho(\Omega^2-1)f(\phi)+2\varrho^2\Omega\pr_{\ub}(\Omega)f(\phi)+\varrho^2(\Omega^2-1)f'(\phi)\pr_{\varrho}\phi)-2\varrho^2(\Omega^2-1)f(\phi)r\pr_{\varrho}r}{r^4}.
\eee
We rewrite the term $\varrho^2\pr_{\ub}^2r\pr_u\phi/r$ as\footnote{This term needs to be integrated by parts as it would otherwise lead to a dangerous term of the type $\pr_{\ub}^2r v$.}
\bee
\frac{\varrho^2\pr_{\ub}^2r}{r}\pr_u\phi &=& \pr_u\left(\frac{\varrho^2\pr_{\ub}^2r}{r}\phi\right)-\frac{\varrho^2\pr_u\pr_{\ub}^2r}{r}\phi+\frac{(\varrho r+\pr_ur\varrho^2)\pr_{\ub}^2r \phi}{r^2}.
\eee
Recall that
\bee
\pr_u\pr_{\ub}^2r &=& \kappa \frac{\Omega(\pr_{\ub}\Omega)g(\phi)^2}{2r}+ \kappa \frac{\Omega^2g(\phi)g'(\phi)\pr_{\ub}\phi}{2r}- \kappa \frac{\Omega^2g(\phi)^2\pr_{\ub}r}{4r^2}.
\eee
We infer
\bee
\frac{\varrho^2\pr_{\ub}^2r}{r}\pr_u\phi &=& \pr_u\left(\frac{\varrho^2\pr_{\ub}^2r}{r}\phi\right)-\frac{\kappa\varrho^2\Omega(\pr_{\ub}\Omega)g(\phi)^2\phi}{2r^2}-\frac{\kappa\varrho^2\Omega^2g(\phi)g'(\phi)\phi\pr_{\ub}\phi}{2r^2}\\
&&+\frac{\kappa\varrho^2\Omega^2g(\phi)^2\phi\pr_{\ub}r}{4r^3}+\frac{(\varrho r+\pr_ur\varrho^2)\pr_{\ub}^2r \phi}{r^2}.
\eee
We obtain
\bee
\pr_{\varrho}F &=& \pr_u\Bigg(\frac{\varrho(\varrho-r)}{r}\left(v-\frac{\phi}{\varrho}\right)+\frac{2r\varrho^2(\pr_{\varrho}r-1)}{r^2}\pr_{\ub}\phi+\frac{\varrho^2(2\pr_{\ub}r-1)}{r}\left(v-\frac{\phi}{\varrho}\right)-\frac{\varrho^2(\Omega^2-1)f(\phi)}{r^2}\\
&&+\frac{2\varrho^2\pr_{\ub}^2r}{r}\phi\Bigg)\\
&& -\frac{r(-\varrho+\frac{1}{2}r-\varrho\pr_ur)-\varrho(\varrho-r)\pr_ur}{r^2}\left(v-\frac{\phi}{\varrho}\right)-2\frac{r(-\varrho\pr_{\ub}\phi+\varrho^2\pr_u\pr_{\ub}\phi)-\varrho^2\pr_{\ub}\phi\pr_ur}{r^2}(\pr_{\varrho}r-1)\\
&&-\frac{r(-\varrho(2\pr_{\ub}r-1)+2\varrho^2\pr_u\pr_{\ub}r)-\varrho^2(2\pr_{\ub}r-1)\pr_ur}{r^2}\left(v-\frac{\phi}{\varrho}\right)\\
&&+\frac{r^2\left(-\varrho f(\phi)+\varrho^2f'(\phi)\left(\pr_{\ub}\phi-v+\frac{\phi}{\varrho}\right)\right)-2\varrho^2f(\phi)r\pr_ur}{r^4}(\Omega^2-1)\\
&&-\frac{\varrho(\varrho-r)}{r}(\pr_{\ub}^2\phi-\pr_{\ub}\pr_u\phi)-\frac{r(2\varrho-r-\varrho\pr_{\varrho}r)-\varrho(\varrho-r)\pr_{\varrho}r}{r^2}\left(v-\frac{\phi}{\varrho}\right)\\
&&+\frac{\varrho^2(2\pr_ur+1)}{r}(\pr_{\ub}^2\phi-\pr_{\ub}\pr_u\phi)+\frac{2r\varrho(2\pr_ur+1)-\varrho^2(2\pr_ur+1)\pr_{\varrho}r}{r^2}\pr_{\ub}\phi\\
&&+\frac{r(2\varrho(2\pr_{\ub}r-1)-2\varrho^2\pr_{\ub}\pr_ur)-\varrho^2(2\pr_{\ub}r-1)\pr_{\varrho}r}{r^2}\left(\pr_{\ub}\phi-v+\frac{\phi}{\varrho}\right)\\
&& -\frac{\kappa\varrho^2\Omega(\pr_{\ub}\Omega)g(\phi)^2\phi}{r^2}-\frac{\kappa\varrho^2\Omega^2g(\phi)g'(\phi)\phi\pr_{\ub}\phi}{r^2}+\frac{\kappa\varrho^2\Omega^2g(\phi)^2\phi\pr_{\ub}r}{2r^3}+\frac{2(\varrho r+\pr_ur\varrho^2)\pr_{\ub}^2r \phi}{r^2}\\
&&+\frac{\varrho^2-r^2}{r^2}\left(v-\frac{\phi}{\varrho}\right)+\frac{r^2(2\varrho-2r\pr_{\varrho}r)-2(\varrho^2-r^2)r\pr_{\varrho}r}{r^4}\phi\\
&&+\frac{r^2\left(2\varrho\phi^3\zeta(\phi)+\varrho^2(3\phi^2\zeta(\phi)+\phi^3\zeta'(\phi))\left(v-\frac{\phi}{\varrho}\right)\right)-2\varrho^2\phi^3\zeta(\phi)r\pr_{\varrho}r}{r^4}\\
&&+\frac{r^2\left(2\varrho(\Omega^2-1)f(\phi)+2\varrho^2\Omega\pr_{\ub}(\Omega)f(\phi)+\varrho^2(\Omega^2-1)f'(\phi)\left(v-\frac{\phi}{\varrho}\right)\right)-2\varrho^2(\Omega^2-1)f(\phi)r\pr_{\varrho}r}{r^4}.
\eee

Recall that
$$\pr_u\pr_{\ub}r=r\kappa \frac{\Omega^2}{4}\frac{g(\phi)^2}{r^2}.$$
Also, recall that
\bee
&&-4\pr_u\pr_{\ub}\phi+\frac{1}{\varrho}(\pr_{\ub}\phi-\pr_u\phi)-\frac{\phi}{\varrho^2}\\ 
&=& -\frac{\varrho-r}{r\varrho}(\pr_{\ub}\phi-\pr_u\phi)+\frac{2\pr_ur+1}{r}\pr_{\ub}\phi+\frac{2\pr_{\ub}r-1}{r}\pr_u\phi+\frac{\varrho^2-r^2}{r^2\varrho^2}\phi+\frac{\phi^3\zeta(\phi)}{r^2}\\
&&+(\Omega^2-1)\frac{f(\phi)}{r^2},
\eee
which yields
\bee
-\pr_u\pr_{\ub}\phi &=& -\frac{\pr_{\ub}r}{2r}\left(v-\frac{\phi}{\varrho}\right)+\frac{\pr_ur+\pr_{\ub}r}{2r}\pr_{\ub}\phi+\frac{\Omega^2f(\phi)}{4r^2}.
\eee
This allows us to rewrite $\pr_{\varrho}F$ as
\bee
\pr_{\varrho}F &=& \pr_u\Bigg(\frac{\varrho(\varrho-r)}{r}\left(v-\frac{\phi}{\varrho}\right)+\frac{2r\varrho^2(\pr_{\varrho}r-1)}{r^2}\pr_{\ub}\phi+\frac{\varrho^2(2\pr_{\ub}r-1)}{r}\left(v-\frac{\phi}{\varrho}\right)-\frac{\varrho^2(\Omega^2-1)f(\phi)}{r^2}\\
&&+\frac{2\varrho^2\pr_{\ub}^2r}{r}\phi\Bigg)-\frac{r(-\varrho+\frac{1}{2}r-\varrho\pr_ur)-\varrho(\varrho-r)\pr_ur}{r^2}\left(v-\frac{\phi}{\varrho}\right)\\
&&-2\frac{r\left(-\varrho\pr_{\ub}\phi+\varrho^2\left(\frac{\pr_{\ub}r}{2r}\left(v-\frac{\phi}{\varrho}\right)-\frac{\pr_ur+\pr_{\ub}r}{2r}\pr_{\ub}\phi-\frac{\Omega^2f(\phi)}{4r^2}\right)\right)-\varrho^2\pr_{\ub}\phi\pr_ur}{r^2}(\pr_{\varrho}r-1)\\
&&-\frac{r\left(-\varrho(2\pr_{\ub}r-1)+2\varrho^2\kappa \frac{\Omega^2}{4}\frac{g(\phi)^2}{r}\right)-\varrho^2(2\pr_{\ub}r-1)\pr_ur}{r^2}\left(v-\frac{\phi}{\varrho}\right)\\
&&+\frac{r^2\left(-\varrho f(\phi)+\varrho^2f'(\phi)\left(\pr_{\ub}\phi-v+\frac{\phi}{\varrho}\right)\right)-2\varrho^2f(\phi)r\pr_ur}{r^4}(\Omega^2-1)\\
&&-\frac{\varrho(\varrho-r)}{r}\left(\pr_{\ub}^2\phi-\frac{\pr_{\ub}r}{2r}\left(v-\frac{\phi}{\varrho}\right)+\frac{\pr_ur+\pr_{\ub}r}{2r}\pr_{\ub}\phi+\frac{\Omega^2f(\phi)}{4r^2}\right)\\
&&-\frac{r(2\varrho-r-\varrho\pr_{\varrho}r)-\varrho(\varrho-r)\pr_{\varrho}r}{r^2}\left(v-\frac{\phi}{\varrho}\right)\\
&&+\frac{\varrho^2(2\pr_ur+1)}{r}\left(\pr_{\ub}^2\phi-\frac{\pr_{\ub}r}{2r}\left(v-\frac{\phi}{\varrho}\right)+\frac{\pr_ur+\pr_{\ub}r}{2r}\pr_{\ub}\phi+\frac{\Omega^2f(\phi)}{4r^2}\right)\\
&&+\frac{2r\varrho(2\pr_ur+1)-\varrho^2(2\pr_ur+1)\pr_{\varrho}r}{r^2}\pr_{\ub}\phi\\
&&+\frac{r\left(2\varrho(2\pr_{\ub}r-1)-\varrho^2\kappa \frac{\Omega^2}{2}\frac{g(\phi)^2}{r}\right)-\varrho^2(2\pr_{\ub}r-1)\pr_{\varrho}r}{r^2}\left(\pr_{\ub}\phi-v+\frac{\phi}{\varrho}\right)\\
&& -\frac{\kappa\varrho^2\Omega(\pr_{\ub}\Omega)g(\phi)^2\phi}{r^2}-\frac{\kappa\varrho^2\Omega^2g(\phi)g'(\phi)\phi\pr_{\ub}\phi}{r^2}+\frac{\kappa\varrho^2\Omega^2g(\phi)^2\phi\pr_{\ub}r}{2r^3}+\frac{2(\varrho r+\pr_ur\varrho^2)\pr_{\ub}^2r \phi}{r^2}\\
&&+\frac{\varrho^2-r^2}{r^2}\left(v-\frac{\phi}{\varrho}\right)+\frac{r^2(2\varrho-2r\pr_{\varrho}r)-2(\varrho^2-r^2)r\pr_{\varrho}r}{r^4}\phi\\
&&+\frac{r^2\left(2\varrho\phi^3\zeta(\phi)+\varrho^2(3\phi^2\zeta(\phi)+\phi^3\zeta'(\phi))\left(v-\frac{\phi}{\varrho}\right)\right)-2\varrho^2\phi^3\zeta(\phi)r\pr_{\varrho}r}{r^4}\\
&&+\frac{r^2\left(2\varrho(\Omega^2-1)f(\phi)+2\varrho^2\Omega\pr_{\ub}(\Omega)f(\phi)+\varrho^2(\Omega^2-1)f'(\phi)\left(v-\frac{\phi}{\varrho}\right)\right)-2\varrho^2(\Omega^2-1)f(\phi)r\pr_{\varrho}r}{r^4}.
\eee
We infer
\bee
\pr_{\varrho}F &=& \pr_u(A_1v+A_2)+A_3v+A_4,
\eee
where
\bee
A_1 &=& \frac{\varrho(\varrho-r)}{r}+\frac{\varrho^2(2\pr_{\ub}r-1)}{r},
\eee
\bee
A_2 &=& -\frac{\varrho(\varrho-r)}{r}\frac{\phi}{\varrho}+\frac{2r\varrho^2(\pr_{\varrho}r-1)}{r^2}\pr_{\ub}\phi-\frac{\varrho^2(2\pr_{\ub}r-1)}{r}\frac{\phi}{\varrho}-\frac{\varrho^2(\Omega^2-1)f(\phi)}{r^2}+\frac{2\varrho^2\pr_{\ub}^2r}{r}\phi,
\eee
\bee
A_3 &=& -\frac{r(-\varrho+\frac{1}{2}r-\varrho\pr_ur)-\varrho(\varrho-r)\pr_ur}{r^2}-2\frac{r\varrho^2\frac{\pr_{\ub}r}{2r}}{r^2}(\pr_{\varrho}r-1)\\
&&-\frac{r\left(-\varrho(2\pr_{\ub}r-1)+2\varrho^2\kappa \frac{\Omega^2}{4}\frac{g(\phi)^2}{r}\right)-\varrho^2(2\pr_{\ub}r-1)\pr_ur}{r^2}\\
&&-\frac{r^2\varrho^2f'(\phi)}{r^4}(\Omega^2-1)+\frac{\varrho(\varrho-r)}{r}\frac{\pr_{\ub}r}{2r}-\frac{r(2\varrho-r-\varrho\pr_{\varrho}r)-\varrho(\varrho-r)\pr_{\varrho}r}{r^2}\\
&&-\frac{\varrho^2(2\pr_ur+1)}{r}\frac{\pr_{\ub}r}{2r}\\
&&-\frac{r\left(2\varrho(2\pr_{\ub}r-1)-\varrho^2\kappa \frac{\Omega^2}{2}\frac{g(\phi)^2}{r}\right)-\varrho^2(2\pr_{\ub}r-1)\pr_{\varrho}r}{r^2}\\
&&+\frac{\varrho^2-r^2}{r^2}+\frac{r^2\varrho^2(3\phi^2\zeta(\phi)+\phi^3\zeta'(\phi))}{r^4}+\frac{r^2\varrho^2(\Omega^2-1)f'(\phi)}{r^4},
\eee
and 
\bee
A_4 &=& \frac{r(-\varrho+\frac{1}{2}r-\varrho\pr_ur)-\varrho(\varrho-r)\pr_ur}{r^2}\frac{\phi}{\varrho}\\
&&-2\frac{r\left(-\varrho\pr_{\ub}\phi+\varrho^2\left(-\frac{\pr_{\ub}r}{2r}\frac{\phi}{\varrho}-\frac{\pr_ur+\pr_{\ub}r}{2r}\pr_{\ub}\phi-\frac{\Omega^2f(\phi)}{4r^2}\right)\right)-\varrho^2\pr_{\ub}\phi\pr_ur}{r^2}(\pr_{\varrho}r-1)\\
&&+\frac{r\left(-\varrho(2\pr_{\ub}r-1)+2\varrho^2\kappa \frac{\Omega^2}{4}\frac{g(\phi)^2}{r}\right)-\varrho^2(2\pr_{\ub}r-1)\pr_ur}{r^2}\frac{\phi}{\varrho}\\
&&+\frac{r^2\left(-\varrho f(\phi)+\varrho^2f'(\phi)\left(\pr_{\ub}\phi+\frac{\phi}{\varrho}\right)\right)-2\varrho^2f(\phi)r\pr_ur}{r^4}(\Omega^2-1)\\
&&-\frac{\varrho(\varrho-r)}{r}\left(\pr_{\ub}^2\phi+\frac{\pr_{\ub}r}{2r}\frac{\phi}{\varrho}+\frac{\pr_ur+\pr_{\ub}r}{2r}\pr_{\ub}\phi+\frac{\Omega^2f(\phi)}{4r^2}\right)\\
&&+\frac{r(2\varrho-r-\varrho\pr_{\varrho}r)-\varrho(\varrho-r)\pr_{\varrho}r}{r^2}\frac{\phi}{\varrho}\\
&&+\frac{\varrho^2(2\pr_ur+1)}{r}\left(\pr_{\ub}^2\phi+\frac{\pr_{\ub}r}{2r}\frac{\phi}{\varrho}+\frac{\pr_ur+\pr_{\ub}r}{2r}\pr_{\ub}\phi+\frac{\Omega^2f(\phi)}{4r^2}\right)\\
&&+\frac{2r\varrho(2\pr_ur+1)-\varrho^2(2\pr_ur+1)\pr_{\varrho}r}{r^2}\pr_{\ub}\phi\\
&&+\frac{r\left(2\varrho(2\pr_{\ub}r-1)-\varrho^2\kappa \frac{\Omega^2}{2}\frac{g(\phi)^2}{r}\right)-\varrho^2(2\pr_{\ub}r-1)\pr_{\varrho}r}{r^2}\left(\pr_{\ub}\phi+\frac{\phi}{\varrho}\right)\\
&& -\frac{\kappa\varrho^2\Omega(\pr_{\ub}\Omega)g(\phi)^2\phi}{r^2}-\frac{\kappa\varrho^2\Omega^2g(\phi)g'(\phi)\phi\pr_{\ub}\phi}{r^2}+\frac{\kappa\varrho^2\Omega^2g(\phi)^2\phi\pr_{\ub}r}{2r^3}+\frac{2(\varrho r+\pr_ur\varrho^2)\pr_{\ub}^2r \phi}{r^2}\\
&&-\frac{\varrho^2-r^2}{r^2}\frac{\phi}{\varrho}+\frac{r^2(2\varrho-2r\pr_{\varrho}r)-2(\varrho^2-r^2)r\pr_{\varrho}r}{r^4}\phi\\
&&+\frac{r^2\left(2\varrho\phi^3\zeta(\phi)-\varrho^2(3\phi^2\zeta(\phi)+\phi^3\zeta'(\phi))\frac{\phi}{\varrho}\right)-2\varrho^2\phi^3\zeta(\phi)r\pr_{\varrho}r}{r^4}\\
&&+\frac{r^2\left(2\varrho(\Omega^2-1)f(\phi)+2\varrho^2\Omega\pr_{\ub}(\Omega)f(\phi)-\varrho^2(\Omega^2-1)f'(\phi)\frac{\phi}{\varrho}\right)-2\varrho^2(\Omega^2-1)f(\phi)r\pr_{\varrho}r}{r^4}.
\eee

We also rewrite $F$. We have
\bee
F &=& -\frac{\varrho(\varrho-r)}{r}\left(v-\frac{\phi}{\varrho}\right)+\frac{\varrho^2(2\pr_ur+1)}{r}\pr_{\ub}\phi+\frac{\varrho^2(2\pr_{\ub}r-1)}{r}\left(\pr_{\ub}\phi-v+\frac{\phi}{\varrho}\right)\\
&&+\frac{\varrho^2-r^2}{r^2}\phi+\frac{\varrho^2\phi^3\zeta(\phi)}{r^2}+\varrho^2(\Omega^2-1)\frac{f(\phi)}{r^2}\\
&=& A_5v+A_6,
\eee
where
\bee
A_5 &=& -\frac{\varrho(\varrho-r)}{r}-\frac{\varrho^2(2\pr_{\ub}r-1)}{r},
\eee
and
\bee
A_6 &=& \frac{\varrho(\varrho-r)}{r}\frac{\phi}{\varrho}+\frac{\varrho^2(2\pr_ur+1)}{r}\pr_{\ub}\phi+\frac{\varrho^2(2\pr_{\ub}r-1)}{r}\left(\pr_{\ub}\phi+\frac{\phi}{\varrho}\right)+\frac{\varrho^2-r^2}{r^2}\phi\\
&&+\frac{\varrho^2\phi^3\zeta(\phi)}{r^2}+\varrho^2(\Omega^2-1)\frac{f(\phi)}{r^2}.
\eee

Finally, we have obtained
\bee
\frac{\pr_{\varrho}F}{\varrho^2}-\frac{F}{\varrho^3} &=& \frac{\pr_u(A_1v+A_2)+A_3v+A_4}{\varrho^2}-\frac{A_5v+A_6}{\varrho^3}\\
&=& \pr_u\left(\frac{A_1v+A_2}{\varrho^2}\right)-\frac{A_1v+A_2}{\varrho^3}+\frac{A_3v+A_4}{\varrho^2}-\frac{A_5v+A_6}{\varrho^3}\\
&=& \pr_u\left(\frac{B_1}{\varrho}v+\frac{B_2}{\varrho^2}\right)+\frac{B_3}{\varrho^2}v+\frac{B_4}{\varrho^3},
\eee
where
\bee
B_1 &=& \frac{A_1}{\varrho}\\
&=&  \frac{\varrho-r+\varrho(2\pr_{\ub}r-1)}{r},
\eee
\bee
B_2 &=& A_2\\
&=&\frac{2\varrho^2(\pr_{\varrho}r-1)}{r}\pr_{\ub}\phi -\frac{\varrho-r+\varrho(2\pr_{\ub}r-1)-2\varrho^2\pr_{\ub}^2r}{r}\phi-\frac{\varrho^2(\Omega^2-1)f(\phi)}{r^2},
\eee
\bee
B_3 &=& -\frac{A_1}{\varrho}+A_3-\frac{A_5}{\varrho}\\
&=& \frac{\varrho-r+\varrho(2\pr_ur+1)}{2r}+\frac{\varrho(\varrho-r)\pr_ur}{r^2}-\frac{\varrho^2\pr_{\ub}r}{r^2}(\pr_{\varrho}r-1)\\
&&+\frac{\varrho(2\pr_{\ub}r-1)}{r}-\frac{\kappa\varrho^2\Omega^2g(\phi)^2}{2r^2}+\frac{\varrho^2(2\pr_{\ub}r-1)\pr_ur}{r^2}-\frac{\varrho^2f'(\phi)}{r^2}(\Omega^2-1)\\
&&+\frac{\varrho(\varrho-r)}{2r^2}\pr_{\ub}r-\frac{\varrho-r+\varrho(1-\pr_{\varrho}r)}{r}+\frac{\varrho(\varrho-r)\pr_{\varrho}r}{r^2}-\frac{\varrho^2(2\pr_ur+1)}{2r^2}\pr_{\ub}r\\
&& -\frac{2\varrho(2\pr_{\ub}r-1)}{r}+\frac{\kappa\varrho^2\Omega^2 g(\phi)^2}{2r^2}+\frac{\varrho^2(2\pr_{\ub}r-1)\pr_{\varrho}r}{r^2}\\
&&+\frac{\varrho^2-r^2}{r^2}+\frac{\varrho^2(3\phi^2\zeta(\phi)+\phi^3\zeta'(\phi))}{r^2}+\frac{\varrho^2(\Omega^2-1)f'(\phi)}{r^2}\\
&=& \frac{\varrho-r}{r}\left(\frac{1}{2}+\frac{3\varrho}{2r}\pr_{\ub}r+\frac{\varrho}{r}\right)+\frac{\varrho(2\pr_{\ub}r-1)}{r}\left(-1+\frac{\varrho}{r}\pr_ru+\frac{\varrho}{r}\pr_{\varrho}r\right)+\frac{\varrho(2\pr_ur+1)}{2r}\left(1-\frac{\varrho}{r}\pr_{\ub}r\right)\\
&& \frac{\varrho(\pr_{\varrho}r-1)}{r}\left(1-\frac{\varrho}{r}\pr_{\ub}r\right)+\frac{\varrho^2(3\phi^2\zeta(\phi)+\phi^3\zeta'(\phi))}{r^2},
\eee
and
\bee
B_4 &=& -A_2+\varrho A_4 -A_6\\
&=& \frac{\varrho-r}{r}\phi-\frac{2\varrho^2(\pr_{\varrho}r-1)}{r}\pr_{\ub}\phi+\frac{\varrho(2\pr_{\ub}r-1)}{r}\phi+\frac{\varrho^2(\Omega^2-1)f(\phi)}{r^2} -\frac{2\varrho^2\pr_{\ub}^2r}{r}\phi\\
&& +\left(-\frac{\varrho-r+\varrho\left(2\pr_ur+1\right)}{2r}-\frac{\varrho(\varrho-r)\pr_ur}{r^2}\right)\phi\\
&&+2\left(\frac{\varrho^2\pr_{\ub}\phi}{r}+\frac{\varrho^2\pr_{\ub}r\phi}{2r^2}+\frac{\varrho^3(\pr_ur+\pr_{\ub}r)\pr_{\ub}\phi}{2r^2}+\frac{\varrho^3\Omega^2f(\phi)}{4r^3}+\frac{\varrho^3\pr_{\ub}\phi\pr_ur}{r^2}\right)(\pr_{\varrho}r-1)\\
&&-\frac{\varrho(2\pr_{\ub}r-1)}{r}\phi+\frac{\kappa\varrho^2 \Omega^2g(\phi)^2}{2r^2}\phi-\frac{\varrho^2(2\pr_{\ub}r-1)\pr_ur}{r^2}\phi\\
&&+\left(\frac{-\varrho^2 f(\phi)+\varrho^2f'(\phi)\left(\varrho\pr_{\ub}\phi+\phi\right)}{r^2}-\frac{2\varrho^3f(\phi)\pr_ur}{r^3}\right)(\Omega^2-1)\\
&&-\frac{\varrho(\varrho-r)}{r}\left(\varrho\pr_{\ub}^2\phi+\frac{\pr_{\ub}r}{2r}\phi+\frac{\pr_ur+\pr_{\ub}r}{2r}\varrho\pr_{\ub}\phi+\frac{\varrho\Omega^2f(\phi)}{4r^2}\right)\\
&&+\left(\frac{\varrho-r-\varrho(\pr_{\varrho}r-1)}{r}-\frac{\varrho(\varrho-r)\pr_{\varrho}r}{r^2}\right)\phi\\
&&+\frac{\varrho^2(2\pr_ur+1)}{r}\left(\varrho\pr_{\ub}^2\phi+\frac{\pr_{\ub}r}{2r}\phi+\frac{\pr_ur+\pr_{\ub}r}{2r}\varrho\pr_{\ub}\phi+\frac{\varrho\Omega^2f(\phi)}{4r^2}\right)\\
&&+\left(\frac{2\varrho^2(2\pr_ur+1)}{r}-\frac{\varrho^3(2\pr_ur+1)\pr_{\varrho}r}{r^2}\right)\pr_{\ub}\phi\\
&&+\left(\frac{2\varrho(2\pr_{\ub}r-1)}{r}-\frac{\kappa\varrho^2\Omega^2g(\phi)^2}{2r^2}-\frac{\varrho^2(2\pr_{\ub}r-1)\pr_{\varrho}r}{r^2}\right)\left(\varrho\pr_{\ub}\phi+\phi\right)\\
&& -\frac{\kappa\varrho^3\Omega(\pr_{\ub}\Omega)g(\phi)^2\phi}{r^2}-\frac{\kappa\varrho^3\Omega^2g(\phi)g'(\phi)\phi\pr_{\ub}\phi}{r^2}+\frac{\kappa\varrho^3\Omega^2g(\phi)^2\phi\pr_{\ub}r}{2r^3}+\frac{2\varrho(\varrho r+\pr_ur\varrho^2)\pr_{\ub}^2r \phi}{r^2}\\
&&-\frac{\varrho^2-r^2}{r^2}\phi+\frac{\varrho(2\varrho-2r\pr_{\varrho}r)}{r^2}\phi-\frac{2\varrho(\varrho^2-r^2)\pr_{\varrho}r}{r^3}\phi\\
&&-\frac{\varrho^2(\phi^2\zeta(\phi)+\phi^3\zeta'(\phi))\phi}{r^2}-\frac{2\varrho^3\phi^3\zeta(\phi)\pr_{\varrho}r}{r^3}\\
&&+\frac{2\varrho^2(\Omega^2-1)f(\phi)+2\varrho^3\Omega\pr_{\ub}(\Omega)f(\phi)-\varrho^2(\Omega^2-1)f'(\phi)\phi}{r^2}-\frac{2\varrho^3(\Omega^2-1)f(\phi)\pr_{\varrho}r}{r^3}\\
&& -\frac{\varrho-r}{r}\phi-\frac{\varrho^2(2\pr_ur+1)}{r}\pr_{\ub}\phi-\frac{\varrho(2\pr_{\ub}r-1)}{r}\left(\varrho\pr_{\ub}\phi+\phi\right)-\frac{\varrho^2-r^2}{r^2}\phi\\
&&-\frac{\varrho^2\phi^3\zeta(\phi)}{r^2}-\varrho^2(\Omega^2-1)\frac{f(\phi)}{r^2}.
\eee
We rewrite $B_4$ as follows
\bee
B_4 &=& \Bigg[\frac{\varrho-r}{r}+\frac{\varrho(2\pr_{\ub}r-1)}{r}-\frac{2\varrho^2\pr_{\ub}^2r}{r}-\frac{\varrho-r+\varrho\left(2\pr_ur+1\right)}{2r}-\frac{\varrho(\varrho-r)\pr_ur}{r^2}\\
&&-\frac{\varrho(2\pr_{\ub}r-1)}{r}+\frac{\kappa\varrho^2 \Omega^2g(\phi)^2}{2r^2}-\frac{\varrho^2(2\pr_{\ub}r-1)\pr_ur}{r^2}+\frac{\varrho-r-\varrho(\pr_{\varrho}r-1)}{r}\\
&&-\frac{\varrho(\varrho-r)\pr_{\varrho}r}{r^2}+\frac{2\varrho(2\pr_{\ub}r-1)}{r}-\frac{\kappa\varrho^2\Omega^2g(\phi)^2}{2r^2}-\frac{\varrho^2(2\pr_{\ub}r-1)\pr_{\varrho}r}{r^2}\\
&&+\frac{\kappa\varrho^3\Omega^2g(\phi)^2\pr_{\ub}r}{2r^3}+\frac{2\varrho(\varrho r+\pr_ur\varrho^2)\pr_{\ub}^2r}{r^2}\\
&&-\frac{\varrho^2-r^2}{r^2}+\frac{\varrho(2\varrho-2r\pr_{\varrho}r)}{r^2}-\frac{2\varrho(\varrho^2-r^2)\pr_{\varrho}r}{r^3}-\frac{\varrho-r}{r}-\frac{\varrho(2\pr_{\ub}r-1)}{r}-\frac{\varrho^2-r^2}{r^2}\\
&&+2\frac{\varrho^2\pr_{\ub}r}{2r^2}(\pr_{\varrho}r-1)-\frac{\varrho(\varrho-r)}{r}\frac{\pr_{\ub}r}{2r}
+\frac{\varrho^2(2\pr_ur+1)}{r}\frac{\pr_{\ub}r}{2r}-\frac{\varrho^2(\phi^2\zeta(\phi)+\phi^3\zeta'(\phi))}{r^2}\\
&&-\frac{2\varrho^3\phi^2\zeta(\phi)\pr_{\varrho}r}{r^3}-\frac{\varrho^2\phi^2\zeta(\phi)}{r^2}\Bigg]\phi\\
&&+\Bigg[-\frac{2\varrho^2(\pr_{\varrho}r-1)}{r}+2\left(\frac{\varrho^2}{r}+\frac{\varrho^3(\pr_ur+\pr_{\ub}r)}{2r^2}+\frac{\varrho^3\pr_ur}{r^2}\right)(\pr_{\varrho}r-1)\\
&&-\frac{\varrho(\varrho-r)}{r}\frac{\pr_ur+\pr_{\ub}r}{2r}\varrho+\frac{\varrho^2(2\pr_ur+1)}{r}\frac{\pr_ur+\pr_{\ub}r}{2r}\varrho+\frac{2\varrho^2(2\pr_ur+1)}{r}-\frac{\varrho^3(2\pr_ur+1)\pr_{\varrho}r}{r^2}\\
&&+\left(\frac{2\varrho(2\pr_{\ub}r-1)}{r}-\frac{\kappa\varrho^2\Omega^2g(\phi)^2}{2r^2}-\frac{\varrho^2(2\pr_{\ub}r-1)\pr_{\varrho}r}{r^2}\right)\varrho
-\frac{\kappa\varrho^3\Omega^2g(\phi)g'(\phi)\phi}{r^2}-\frac{\varrho^2(2\pr_ur+1)}{r}\\
&&-\frac{\varrho(2\pr_{\ub}r-1)}{r}\varrho\Bigg]\pr_{\ub}\phi+\Bigg[-\frac{\varrho(\varrho-r)}{r}\varrho+\frac{\varrho^2(2\pr_ur+1)}{r}\varrho\Bigg]\pr_{\ub}^2\phi\\
&&+\Bigg[\frac{\varrho^2f(\phi)}{r^2}+\frac{-\varrho^2 f(\phi)+\varrho^2f'(\phi)\left(\varrho\pr_{\ub}\phi+\phi\right)}{r^2}-\frac{2\varrho^3f(\phi)\pr_ur}{r^3}\\
&&+\frac{2\varrho^2f(\phi)-\varrho^2f'(\phi)\phi}{r^2}-\frac{2\varrho^3f(\phi)\pr_{\varrho}r}{r^3}-\varrho^2\frac{f(\phi)}{r^2}\Bigg](\Omega^2-1)\\
&&+\Bigg[2\frac{\varrho^3\Omega^2}{4r^3}(\pr_{\varrho}r-1)-\frac{\varrho(\varrho-r)}{r}\frac{\varrho\Omega^2}{4r^2}+\frac{\varrho^2(2\pr_ur+1)}{r}\frac{\varrho\Omega^2}{4r^2}\Bigg]f(\phi)+(2f(\phi)-\kappa g(\phi)^2\phi)\frac{\varrho^3\Omega\pr_{\ub}\Omega}{r^2},
\eee
and we finally obtain
\bee
B_4 &=& \Bigg[\left(-\frac{3}{2}-\frac{\varrho}{r}\pr_ur-\frac{2\varrho^2+3\varrho r}{r^2}\pr_{\varrho}r-\frac{\varrho}{2r}\pr_{\ub}r\right)\frac{\varrho-r}{r}+\left(-\frac{\varrho\pr_{\ub}r}{r}+1\right)\frac{\varrho(2\pr_{\ub}r-1)}{r}\\
&&+\left(-3+\frac{\varrho\pr_{\ub}r}{r}\right)\frac{\varrho(\pr_{\varrho}r-1)}{r}+\left(-\frac{1}{2}+\frac{\varrho\pr_{\ub}r}{2r}\right)\frac{\varrho\left(2\pr_ur+1\right)}{r}+\frac{\kappa\varrho^3\Omega^2g(\phi)^2\pr_{\ub}r}{2r^3}\\
&&+\frac{2\varrho^3\pr_ur\pr_{\ub}^2r}{r^2}-\frac{\varrho^2(2\phi^2\zeta(\phi)+\phi^3\zeta'(\phi))}{r^2}-\frac{2\varrho^3\phi^2\zeta(\phi)\pr_{\varrho}r}{r^3}\Bigg]\phi\\
&&+\Bigg[\frac{\varrho(3\pr_ur+\pr_{\ub}r)}{r} \frac{\varrho^2(\pr_{\varrho}r-1)}{r}\\
&&+\left(-\frac{5\varrho^2(\varrho-r)}{2r}+\frac{\varrho^3(2\pr_ur+1)}{2r}-\frac{2\varrho^3(\pr_{\varrho}r-1)}{r}\right)\frac{\pr_ur+\pr_{\ub}r}{r}\\
&&-\frac{\kappa\varrho^3\Omega^2g(\phi)g'(\phi)\phi}{r^2}-\frac{\kappa\varrho^3\Omega^2g(\phi)^2}{2r^2}\Bigg]\pr_{\ub}\phi+\Bigg[-\frac{\varrho^2(\varrho-r)}{r}+\frac{\varrho^3(2\pr_ur+1)}{r}\Bigg]\pr_{\ub}^2\phi\\
&&+\Bigg[\left(-\frac{2\varrho\pr_{\ub}r}{r}+1\right)\frac{\varrho^2f(\phi)}{r^2}+\frac{\varrho^2f'(\phi)\varrho\pr_{\ub}\phi}{r^2}\Bigg](\Omega^2-1)\\
&&+\Bigg[-\frac{\varrho(\varrho-r)}{r}\frac{\varrho\Omega^2}{4r^2}+\frac{\varrho^3(2\pr_{\ub}r-1)\Omega^2}{4r^3}\Bigg]f(\phi)+(2f(\phi)-\kappa g(\phi)^2\phi)\frac{\varrho^3\Omega\pr_{\ub}\Omega}{r^2}.
\eee
This concludes the proof of the lemma.
\end{proof}

Next, we derive upper bounds for $B_1$, $B_2$, $B_3$ and $B_4$.
\begin{lemma}\lab{lemma:upperboundB1234}
We have for all $(\tau, \varrho)$ with $\tau\geq -1$ and $\ub\leq 0$
\bee
|B_1| \les \ep\textrm{ and }|B_1|\les C_0r,\,\,\,\, |B_2| \les C_0r^{3\delta},\,\,\,\, |B_3| \les C_0r,\textrm{ and }|B_4| \les C_0r^{3\delta},
\eee
where the constant $C_0$ only depends on the values of the solution in $I_0$.
\end{lemma}

\begin{remark}\lab{rem:B13betterthanB24}
$B_1$ and $B_3$ behave better\footnote{The estimates for $B_1$ and $B_3$ correspond to the case $\delta=1/2$, while we have $\delta<1/2$ for  $B_2$ and $B_4$.} than $B_2$ and $B_4$. This is due to the fact that both $B_1$ and $B_3$ include neither $\pr_{\ub}\phi$ nor $\pr_{\ub}\Omega$ in their definition, so that we can estimate them using Lemma \ref{lemma:refinedboundsphietal} which has a $1/2-\delta$ gain with respect to Corollary \ref{cor:consequenceboot1}. 
\end{remark}

\begin{proof}
In view of the definition of $B_1$ and the Lemma \ref{lemma:basicestimate}, we have
\bee
|B_1| &\les& \ep.
\eee
Also, in view of Lemma \ref{lemma:refinedboundsphietal}, we have
\bee
|B_1| &\les & \frac{|\varrho-r|}{r}+\left|2\pr_{\ub}r-1\right|\\
&\les& C_0r.
\eee

In view of the definition of $B_2$, Corollary \ref{cor:consequenceboot1} and Lemma \ref{lemma:upperboundforpr2ubr}, we have
\bee
|B_2| &\les & \varrho\left|\pr_{\varrho}r-1\right||\pr_{\ub}\phi| +\frac{|\varrho-r|}{r}|\phi|+\left|2\pr_{\ub}r-1\right||\phi|+|\Omega^2-1||f(\phi)|+\varrho|\pr^2_{\ub}r||\phi|\\
&\les& C_0r^{3\delta}.
\eee

In view of the definition of $B_3$ and Lemma \ref{lemma:refinedboundsphietal}, we have
\bee
|B_3| &\les& \frac{|\varrho-r|}{r}+\left|\pr_{\ub}r-\frac{1}{2}\right|+\left|\pr_ur+\frac{1}{2}\right|+|\phi|^2\\
&\les& C_0r.
\eee

Finally, in view of the definition of $B_4$, Corollary \ref{cor:consequenceboot1}, Lemma \ref{lemma:upperboundprubOmega}, Lemma \ref{lemma:upperboundforpr2ubr} and Lemma \ref{lemma:upperboundpr2ubphi}, we have
\bee
|B_4| &\les& \left(\frac{|\varrho-r|}{r}+\left|\pr_{\ub}r-\frac{1}{2}\right|+\left|\pr_ur+\frac{1}{2}\right|+\varrho|\pr^2_{\ub}r|+|\phi|^2+|\Omega^2-1|+\varrho|\pr_{\ub}\Omega|\right)\\
&&\times (|\phi|+\varrho|\pr_{\ub}\phi|+\varrho^2|\pr_{\ub}^2\phi|)\\
&\les& C_0r^{3\delta}.
\eee
This concludes the proof of the lemma.
\end{proof}

\subsection{An upper bound for $v$}

\begin{lemma}
We have for all $(\tau, \varrho)$ with $\tau\geq -1$ and $\ub\leq 0$
\bee
&&v(\tau, \varrho) \\
&=& v_0(\tau,\varrho)\\
&+&\frac{J(-1)}{\sqrt{\varrho}}\int_{\tau-\varrho}^{\tau+\varrho}\left(\frac{B_1\left(\frac{\tau-\varrho+\ub'}{2},\frac{-\tau+\varrho+\ub'}{2}\right)}{\sqrt{\frac{-\tau+\varrho+\ub'}{2}}}v\left(\frac{\tau-\varrho+\ub'}{2},\frac{-\tau+\varrho+\ub'}{2}\right)+\frac{B_2\left(\frac{\tau-\varrho+\ub'}{2},\frac{-\tau+\varrho+\ub'}{2}\right)}{\left(\frac{-\tau+\varrho+\ub'}{2}\right)^{\frac{3}{2}}}\right) d\ub'\\
&& -\frac{1}{2}\sqrt{\varrho}\int_0^{+\infty}J(\mu)\left(B_1(-1, \la)v(-1, \la)+\frac{B_2(-1, \la)}{\la}\right)\frac{\sqrt{\la}}{\mu\varrho+\la}d\mu\\
&&-\frac{1}{2\sqrt{\varrho}}\int_{\sqrt{(\tau+1)^2-\varrho^2}}^{\tau+\varrho+1}J(\mu)\left(\frac{B_1(-1, \la)}{\sqrt{\la}}v(-1, \la)+\frac{B_2(-1, \la)}{\la^{\frac{3}{2}}}\right) d\la\\
&&+\int_{R_{\tau, \varrho}}\frac{1}{4\sqrt{\la\varrho}}J(\mu)\left(\frac{B_1(\sigma, \la)}{\la}v(\sigma, \la)+\frac{B_2(\sigma, \la)}{\la^2}\right)d\la d\sigma\\
&&-\int_{R_{\tau, \varrho}}\frac{\sqrt{\la}}{\sqrt{\varrho}}\pr_u\mu\, J'(\mu)\left(\frac{B_1(\sigma, \la)}{\la}v(\sigma, \la)+\frac{B_2(\sigma, \la)}{\la^2}\right)d\la d\sigma\\
&&+\int_{R_{\tau, \varrho}}\frac{\sqrt{\la}}{\sqrt{\varrho}}J(\mu)\left(\frac{B_3(\sigma, \la)}{\la^2}v(\sigma, \la)+\frac{B_4(\sigma, \la)}{\la^3}\right)d\la d\sigma,
\eee
where $v_0$ denotes the solution to the homogeneous equation with the same initial conditions as $v$ at $\tau=-1$, $\mu$ is given by
$$\mu=\frac{(\tau-\sigma)^2-\varrho^2-\la^2}{2\varrho\la},$$
$R_{\tau, \varrho}$ is the space-time region given by
$$R_{\tau, \varrho}=\{(\sigma, \la)\,/\,\, -1\leq\sigma\leq \tau,\,\,\max(0,\varrho-\tau+\sigma)\leq\la\leq\varrho+\tau-\sigma\},$$
and $J$ is given by
$$J(\mu)=\int_{\max(-\mu,-1)}^1\frac{dx}{\sqrt{1-x^2}\sqrt{\mu+x}}.$$
\end{lemma}

\begin{proof}
We recall the representation formula derived in \cite{chris_tah1} for the solution $v$ of 
$$\left(-\pr_\tau^2+\pr_{\varrho}^2+\frac{1}{\varrho}\pr_\varrho\right)v=h.$$
$v$ is given by (see \cite{chris_tah1} p. 1060)
$$v(\tau, \varrho)=v_0(\tau,\varrho)+\int_{R_{\tau, \varrho}}\frac{\sqrt{\la}}{\sqrt{\varrho}}J(\mu)h(\sigma, \la)d\la d\sigma,$$
where 
$$R_{\tau, \varrho}=\{(\sigma, \la)\,/\,\, -1\leq\sigma\leq \tau,\,\,\max(0,\varrho-\tau+\sigma)\leq\la\leq\varrho+\tau-\sigma\},$$
$v_0$ denotes the solution to the homogeneous equation
$$\left(-\pr_\tau^2+\pr_{\varrho}^2+\frac{1}{\varrho}\pr_\varrho\right)v_0=0$$
with the same initial conditions as $v$ at $\tau=-1$, $\mu$ is given as before by
$$\mu=\frac{(\tau-\sigma)^2-\varrho^2-\la^2}{2\varrho\la},$$
and $J$ is given by
$$J(\mu)=\int_{\max(-\mu,-1)}^1\frac{dx}{\sqrt{1-x^2}\sqrt{\mu+x}}.$$
In our case, we have  
\bee
h= \pr_u\left(\frac{B_1}{\varrho}v+\frac{B_2}{\varrho^2}\right)+\frac{B_3}{\varrho^2}v+\frac{B_4}{\varrho^3}.
\eee
Hence, we have 
\bee
v(\tau, \varrho) &=& v_0(\tau,\varrho)+\int_{R_{\tau, \varrho}}\frac{\sqrt{\la}}{\sqrt{\varrho}}J(\mu)\pr_u\left(\frac{B_1(\sigma, \la)}{\la}v(\sigma, \la)+\frac{B_2(\sigma, \la)}{\la^2}\right)d\la d\sigma\\
&&+\int_{R_{\tau, \varrho}}\frac{\sqrt{\la}}{\sqrt{\varrho}}J(\mu)\left(\frac{B_3(\sigma, \la)}{\la^2}v(\sigma, \la)+\frac{B_4(\sigma, \la)}{\la^3}\right)d\la d\sigma\\
&=& v_0(\tau,\varrho)+\int_{\pr R_{\tau, \varrho}}\frac{\sqrt{\la}}{\sqrt{\varrho}}J(\mu)\left(\frac{B_1(\sigma, \la)}{\la}v(\sigma, \la)+\frac{B_2(\sigma, \la)}{\la^2}\right)\gg(\pr_u,\nu_R)\\
&& -\int_{R_{\tau, \varrho}}\pr_u\left(\frac{\sqrt{\la}}{\sqrt{\varrho}}J(\mu)\right)\left(\frac{B_1(\sigma, \la)}{\la}v(\sigma, \la)+\frac{B_2(\sigma, \la)}{\la^2}\right)d\la d\sigma\\
&&+\int_{R_{\tau, \varrho}}\frac{\sqrt{\la}}{\sqrt{\varrho}}J(\mu)\left(\frac{B_3(\sigma, \la)}{\la^2}v(\sigma, \la)+\frac{B_4(\sigma, \la)}{\la^3}\right)d\la d\sigma\\
&=& v_0(\tau,\varrho)+\int_{\pr R_{\tau, \varrho}}\frac{\sqrt{\la}}{\sqrt{\varrho}}J(\mu)\left(\frac{B_1(\sigma, \la)}{\la}v(\sigma, \la)+\frac{B_2(\sigma, \la)}{\la^2}\right)\gg(\pr_u,\nu_R) \\
&&+\int_{R_{\tau, \varrho}}\frac{1}{4\sqrt{\la\varrho}}J(\mu)\left(\frac{B_1(\sigma, \la)}{\la}v(\sigma, \la)+\frac{B_2(\sigma, \la)}{\la^2}\right)d\la d\sigma\\
&&-\int_{R_{\tau, \varrho}}\frac{\sqrt{\la}}{\sqrt{\varrho}}\pr_u\mu\, J'(\mu)\left(\frac{B_1(\sigma, \la)}{\la}v(\sigma, \la)+\frac{B_2(\sigma, \la)}{\la^2}\right)d\la d\sigma\\
&&+\int_{R_{\tau, \varrho}}\frac{\sqrt{\la}}{\sqrt{\varrho}}J(\mu)\left(\frac{B_3(\sigma, \la)}{\la^2}v(\sigma, \la)+\frac{B_4(\sigma, \la)}{\la^3}\right)d\la d\sigma.
\eee

Next, we compute the boundary term. Recall that we have
\bee
&&\int_{\pr R_{\tau, \varrho}}f \gg(\pr_u,\nu_R) \\
&=& \int_{\tau-\varrho}^{\tau+\varrho} f\left(\frac{\tau-\varrho+\ub'}{2},\frac{-\tau+\varrho+\ub'}{2}\right)d\ub'+\frac{1}{2}\int_{-1}^{\tau-\varrho}f(\sigma,0)d\sigma-\frac{1}{2}\int_0^{\tau+\varrho+1} f(-1, \la)d\la,
\eee
and
$$\mu=-1\textrm{ on }u=\tau-\varrho.$$
Hence, we have\footnote{Here, we have dropped the boundary term on $\la=0$ as it vanishes. Indeed, we have 
$$\left|\sqrt{\la}J(\mu)\left(\frac{B_1(\sigma, \la)}{\la}v(\sigma, \la)+\frac{B_2(\sigma, \la)}{\la^2}\right)\right|\leq C_{\tau, \varrho} \sqrt{\la}(1+|v(\sigma, \la)|)|J(\mu)|\leq C_{\tau, \varrho} \sqrt{\la}\to 0\textrm{ as }\la\to 0,$$
where the constant $C_{\tau, \varrho}$ may blow up as $(\tau, \varrho)$ tends to the origin but is finite away from it, and where we used in particular the fact that $B_1/\la$ and $B_2/\la^2$ are bounded, the fact that $\mu\to +\infty$ when $\la\to 0$ with $\sigma<\tau-\varrho$, and the fact that $J$ is bounded for $\mu\geq 2$.}
\bee
&&\int_{\pr R_{\tau, \varrho}}\frac{\sqrt{\la}}{\sqrt{\varrho}}J(\mu)\left(\frac{B_1(\sigma, \la)}{\la}v(\sigma, \la)+\frac{B_2(\sigma, \la)}{\la^2}\right)\gg(\pr_u,\nu_R)\\
&=&  \frac{J(-1)}{\sqrt{\varrho}}\int_{\tau-\varrho}^{\tau+\varrho}\left(\frac{B_1\left(\frac{\tau-\varrho+\ub'}{2},\frac{-\tau+\varrho+\ub'}{2}\right)}{\sqrt{\frac{-\tau+\varrho+\ub'}{2}}}v\left(\frac{\tau-\varrho+\ub'}{2},\frac{-\tau+\varrho+\ub'}{2}\right)+\frac{B_2\left(\frac{\tau-\varrho+\ub'}{2},\frac{-\tau+\varrho+\ub'}{2}\right)}{\left(\frac{-\tau+\varrho+\ub'}{2}\right)^{\frac{3}{2}}}\right) d\ub'\\
&& -\frac{1}{2}\int_0^{\tau+\varrho+1}\frac{\sqrt{\la}}{\sqrt{\varrho}}J(\mu)\left(\frac{B_1(-1, \la)}{\la}v(-1, \la)+\frac{B_2(-1, \la)}{\la^2}\right) d\la
\eee

Recall that
\bee
\pr_\la\mu &=& -\frac{\mu\varrho+\la}{\varrho\la}.
\eee 
We decompose and perform a change of variable
\bee
&&\int_0^{\tau+\varrho+1}\frac{\sqrt{\la}}{\sqrt{\varrho}}J(\mu)\left(\frac{B_1(-1, \la)}{\la}v(-1, \la)+\frac{B_2(-1, \la)}{\la^2}\right) d\la\\
&=& \int_0^{+\infty}\frac{\sqrt{\la}}{\sqrt{\varrho}}J(\mu)\left(\frac{B_1(-1, \la)}{\la}v(-1, \la)+\frac{B_2(-1, \la)}{\la^2}\right)\frac{\varrho\la}{\mu\varrho+\la}d\mu\\
&&+\int_{\sqrt{(\tau+1)^2-\varrho^2}}^{\tau+\varrho+1}\frac{\sqrt{\la}}{\sqrt{\varrho}}J(\mu)\left(\frac{B_1(-1, \la)}{\la}v(-1, \la)+\frac{B_2(-1, \la)}{\la^2}\right) d\la\\
&=& \sqrt{\varrho}\int_0^{+\infty}J(\mu)\left(B_1(-1, \la) v(-1, \la)+\frac{B_2(-1, \la)}{\la}\right)\frac{\sqrt{\la}}{\mu\varrho+\la}d\mu\\
&&+\frac{1}{\sqrt{\varrho}}\int_{\sqrt{(\tau+1)^2-\varrho^2}}^{\tau+\varrho+1}J(\mu)\left(\frac{B_1(-1, \la)}{\sqrt{\la}}v(-1, \la)+\frac{B_2(-1, \la)}{\la^{\frac{3}{2}}}\right) d\la
\eee
which yields
\bee
&&\int_{\pr R_{\tau, \varrho}}\frac{\sqrt{\la}}{\sqrt{\varrho}}J(\mu)\left(\frac{B_1(\sigma, \la)}{\la}v(\sigma, \la)+\frac{B_2(\sigma, \la)}{\la^2}\right)\gg(\pr_u,\nu_R)\\
&=&  \frac{J(-1)}{\sqrt{\varrho}}\int_{\tau-\varrho}^{\tau+\varrho}\left(\frac{B_1\left(\frac{\tau-\varrho+\ub'}{2},\frac{-\tau+\varrho+\ub'}{2}\right)}{\sqrt{\frac{-\tau+\varrho+\ub'}{2}}}v\left(\frac{\tau-\varrho+\ub'}{2},\frac{-\tau+\varrho+\ub'}{2}\right)+\frac{B_2\left(\frac{\tau-\varrho+\ub'}{2},\frac{-\tau+\varrho+\ub'}{2}\right)}{\left(\frac{-\tau+\varrho+\ub'}{2}\right)^{\frac{3}{2}}}\right) d\ub'\\
&& -\frac{1}{2}\sqrt{\varrho}\int_0^{+\infty}J(\mu)\left(B_1(-1, \la)v(-1, \la)+\frac{B_2(-1, \la)}{\la}\right)\frac{\sqrt{\la}}{\mu\varrho+\la}d\mu\\
&&-\frac{1}{2\sqrt{\varrho}}\int_{\sqrt{(\tau+1)^2-\varrho^2}}^{\tau+\varrho+1}J(\mu)\left(\frac{B_1(-1, \la)}{\sqrt{\la}}v(-1, \la)+\frac{B_2(-1, \la)}{\la^{\frac{3}{2}}}\right) d\la.
\eee

Finally, we deduce
\bee
&&v(\tau, \varrho) \\
&=& v_0(\tau,\varrho)\\
&+&\frac{J(-1)}{\sqrt{\varrho}}\int_{\tau-\varrho}^{\tau+\varrho}\left(\frac{B_1\left(\frac{\tau-\varrho+\ub'}{2},\frac{-\tau+\varrho+\ub'}{2}\right)}{\sqrt{\frac{-\tau+\varrho+\ub'}{2}}}v\left(\frac{\tau-\varrho+\ub'}{2},\frac{-\tau+\varrho+\ub'}{2}\right)+\frac{B_2\left(\frac{\tau-\varrho+\ub'}{2},\frac{-\tau+\varrho+\ub'}{2}\right)}{\left(\frac{-\tau+\varrho+\ub'}{2}\right)^{\frac{3}{2}}}\right) d\ub'\\
&& -\frac{1}{2}\sqrt{\varrho}\int_0^{+\infty}J(\mu)\left(B_1(-1, \la)v(-1, \la)+\frac{B_2(-1, \la)}{\la}\right)\frac{\sqrt{\la}}{\mu\varrho+\la}d\mu\\
&&-\frac{1}{2\sqrt{\varrho}}\int_{\sqrt{(\tau+1)^2-\varrho^2}}^{\tau+\varrho+1}J(\mu)\left(\frac{B_1(-1, \la)}{\sqrt{\la}}v(-1, \la)+\frac{B_2(-1, \la)}{\la^{\frac{3}{2}}}\right) d\la\\
&&+\int_{R_{\tau, \varrho}}\frac{1}{4\sqrt{\la\varrho}}J(\mu)\left(\frac{B_1(\sigma, \la)}{\la}v(\sigma, \la)+\frac{B_2(\sigma, \la)}{\la^2}\right)d\la d\sigma\\
&&-\int_{R_{\tau, \varrho}}\frac{\sqrt{\la}}{\sqrt{\varrho}}\pr_u\mu\, J'(\mu)\left(\frac{B_1(\sigma, \la)}{\la}v(\sigma, \la)+\frac{B_2(\sigma, \la)}{\la^2}\right)d\la d\sigma\\
&&+\int_{R_{\tau, \varrho}}\frac{\sqrt{\la}}{\sqrt{\varrho}}J(\mu)\left(\frac{B_3(\sigma, \la)}{\la^2}v(\sigma, \la)+\frac{B_4(\sigma, \la)}{\la^3}\right)d\la d\sigma.
\eee
This concludes the proof of the lemma.
\end{proof}

\begin{lemma}
We have for all $(\tau, \varrho)$ with $\tau\geq -1$ and $\ub\leq 0$
\bee
|v(\tau, \varrho)| &\les& C_0+C_0r^{3\delta-1}+\ep\sup_{0\leq\la\leq\varrho}|v(\tau-\varrho+\la,\la)|\\
&&+C_0\int_{R_{\tau, \varrho}}\frac{\sqrt{\la}}{\sqrt{\varrho}}(J(\mu)+|\mu-1||J'(\mu)|)\left(\frac{|v(\sigma, \la)|}{\la}+\frac{\la^{3\delta}}{\la^3}\right)d\la d\sigma.
\eee
where the constant $C_0$ only depends on the values of the solution in $I_0$. 
\end{lemma}

\begin{proof}
Recall that
\bee
&&v(\tau, \varrho) \\
&=& v_0(\tau,\varrho)\\
&+&\frac{J(-1)}{\sqrt{\varrho}}\int_{\tau-\varrho}^{\tau+\varrho}\left(\frac{B_1\left(\frac{\tau-\varrho+\ub'}{2},\frac{-\tau+\varrho+\ub'}{2}\right)}{\sqrt{\frac{-\tau+\varrho+\ub'}{2}}}v\left(\frac{\tau-\varrho+\ub'}{2},\frac{-\tau+\varrho+\ub'}{2}\right)+\frac{B_2\left(\frac{\tau-\varrho+\ub'}{2},\frac{-\tau+\varrho+\ub'}{2}\right)}{\left(\frac{-\tau+\varrho+\ub'}{2}\right)^{\frac{3}{2}}}\right) d\ub'\\
&& -\frac{1}{2}\sqrt{\varrho}\int_0^{+\infty}J(\mu)\left(B_1(-1, \la)v(-1, \la)+\frac{B_2(-1, \la)}{\la}\right)\frac{\sqrt{\la}}{\mu\varrho+\la}d\mu\\
&&-\frac{1}{2\sqrt{\varrho}}\int_{\sqrt{(\tau+1)^2-\varrho^2}}^{\tau+\varrho+1}J(\mu)\left(\frac{B_1(-1, \la)}{\sqrt{\la}}v(-1, \la)+\frac{B_2(-1, \la)}{\la^{\frac{3}{2}}}\right) d\la\\
&&+\int_{R_{\tau, \varrho}}\frac{1}{4\sqrt{\la\varrho}}J(\mu)\left(\frac{B_1(\sigma, \la)}{\la}v(\sigma, \la)+\frac{B_2(\sigma, \la)}{\la^2}\right)d\la d\sigma\\
&&-\int_{R_{\tau, \varrho}}\frac{\sqrt{\la}}{\sqrt{\varrho}}\pr_u\mu\, J'(\mu)\left(\frac{B_1(\sigma, \la)}{\la}v(\sigma, \la)+\frac{B_2(\sigma, \la)}{\la^2}\right)d\la d\sigma\\
&&+\int_{R_{\tau, \varrho}}\frac{\sqrt{\la}}{\sqrt{\varrho}}J(\mu)\left(\frac{B_3(\sigma, \la)}{\la^2}v(\sigma, \la)+\frac{B_4(\sigma, \la)}{\la^3}\right)d\la d\sigma.
\eee
Noticing that $J(\mu)\geq 0$ for all $\mu\geq -1$ and that 
$$\{\sqrt{(\tau+1)^2-\varrho^2}\leq\la\leq\tau+\varrho+1\}\cap\{\sigma=-1\}=\{-1\leq\mu\leq 0\}\cap\{\sigma=-1\},$$
this yields
\bee
&&|v(\tau, \varrho)| \\
&\les& |v_0(\tau,\varrho)|+\frac{J(-1)}{\sqrt{\varrho}}\int_0^{\varrho}\left(\frac{|B_1(\tau-\varrho+\la,\la)|}{\sqrt{\la}}|v(\tau-\varrho+\la,\la)|+\frac{|B_2(\tau-\varrho+\la,\la)|}{\la^{\frac{3}{2}}}\right) d\la\\
&& +\sqrt{\varrho}\left(\sup_{\la\geq 0}\left(\frac{|B_1(-1, \la)|}{\sqrt{\la}}|v(-1, \la)|+\frac{|B_2(-1, \la)|}{\la^{\frac{3}{2}}}\right)\right)\int_0^2J(\mu)d\mu\\
&& +\sqrt{\varrho}\left(\sup_{\la\geq 0}\left(\frac{|B_1(-1, \la)|}{\la}|v(-1, \la)|+\frac{|B_2(-1, \la)|}{\la^2}\right)\right)\int_2^{+\infty}\frac{\la^{\frac{3}{2}}}{\mu\varrho+\la}J(\mu)d\mu\\
&&+\frac{1}{\sqrt{\varrho}}\left( \sup_{-1\leq\mu\leq 0}J(\mu)\right)\left(\sup_{\la\geq 0}\left(\frac{|B_1(-1, \la)|}{\sqrt{\la}}|v(-1, \la)|+\frac{|B_2(-1, \la)|}{\la^{\frac{3}{2}}}\right)\right)(\tau+\varrho+1-\sqrt{(\tau+1)^2-\varrho^2})\\
&&+\int_{R_{\tau, \varrho}}\frac{1}{4\sqrt{\la\varrho}}J(\mu)\left(\frac{|B_1(\sigma, \la)|}{\la}v(\sigma, \la)+\frac{|B_2(\sigma, \la)|}{\la^2}\right)d\la d\sigma\\
&&+\int_{R_{\tau, \varrho}}\frac{\sqrt{\la}}{\sqrt{\varrho}}|\pr_u\mu|\, |J'(\mu)|\left(\frac{|B_1(\sigma, \la)|}{\la}|v(\sigma, \la)|+\frac{|B_2(\sigma, \la)|}{\la^2}\right)d\la d\sigma\\
&&+\int_{R_{\tau, \varrho}}\frac{\sqrt{\la}}{\sqrt{\varrho}}J(\mu)\left(\frac{|B_3(\sigma, \la)|}{\la^2}|v(\sigma, \la)|+\frac{|B_4(\sigma, \la)|}{\la^3}\right)d\la d\sigma.
\eee
We have the following properties for $J$ (see for example \cite{chris_tah1} p. 1061):
$$J(-1)=\frac{\pi}{\sqrt{2}},\,\, \sup_{-1\leq\mu\leq 0}J(\mu)\les 1,\,\, J\in L^1(0,2).$$
Also, we have
$$\sup_{2\leq\mu<+\infty}\sqrt{\mu}J(\mu)\les 1$$
and hence, we infer
\bee
\int_2^{+\infty}\frac{\la^{\frac{3}{2}}}{\mu\varrho+\la}J(\mu)d\mu &\les& \int_2^{+\infty}\frac{\la^{\frac{3}{2}}}{\sqrt{\mu}(\mu\varrho+\la)}d\mu\\
&\les& \frac{1}{\sqrt{\varrho}}\int_0^{+\infty}\frac{\la^{\frac{3}{2}}}{\sqrt{s}(s+\la)}ds\\
&\les& \frac{1}{\sqrt{\varrho}}\int_0^1\frac{\sqrt{\la}}{\sqrt{s}}ds+\frac{1}{\sqrt{\varrho}}\int_1^{+\infty}\frac{\la^{\frac{3}{2}}}{s^{\frac{3}{2}}}ds\\
&\les&  \frac{1}{\sqrt{\varrho}}\left(\int_0^1\frac{ds}{\sqrt{s}}+\int_1^{+\infty}\frac{ds}{s^{\frac{3}{2}}}\right)\\
&\les&  \frac{1}{\sqrt{\varrho}}. 
\eee
We deduce
\bee
&&|v(\tau, \varrho)| \\
&\les& |v_0(\tau,\varrho)|+\sup_{0\leq\la\leq\varrho}\left(|B_1(\tau-\varrho+\la,\la)||v(\tau-\varrho+\la,\la)|+\frac{|B_2(\tau-\varrho+\la,\la)|}{\la}\right)\\
&& +\sqrt{\varrho}\left(\sup_{\la\geq 0}\left(\frac{|B_1(-1, \la)|}{\sqrt{\la}}|v(-1, \la)|+\frac{|B_2(-1, \la)|}{\la^{\frac{3}{2}}}\right)\right)\\
&&+\sup_{\la\geq 0}\left(\frac{|B_1(-1, \la)|}{\la}|v(-1, \la)|+\frac{|B_2(-1, \la)|}{\la^2}\right)\\
&&+\sqrt{\varrho}\left(\sup_{\la\geq 0}\left(\frac{|B_1(-1, \la)|}{\sqrt{\la}}|v(-1, \la)|+\frac{|B_2(-1, \la)|}{\la^{\frac{3}{2}}}\right)\right)\frac{\tau+\varrho+1}{\tau+\varrho+1+\sqrt{(\tau+1)^2-\varrho^2}}\\
&&+\int_{R_{\tau, \varrho}}\frac{\sqrt{\la}}{\sqrt{\varrho}}J(\mu)\left(\frac{|B_1(\sigma, \la)|+|B_3(\sigma, \la)|}{\la^2}|v(\sigma, \la)|+\frac{|B_2(\sigma, \la)|+|B_4(\sigma, \la)|}{\la^3}\right)d\la d\sigma\\
&&+\int_{R_{\tau, \varrho}}\frac{\sqrt{\la}}{\sqrt{\varrho}}|\pr_u\mu|\, |J'(\mu)|\left(\frac{|B_1(\sigma, \la)|}{\la}|v(\sigma, \la)|+\frac{|B_2(\sigma, \la)|}{\la^2}\right)d\la d\sigma.
\eee
Assuming enough regularity on the initial data, we have
\bee
&&\sup_{\varrho\geq 0} |v_0(\tau,\varrho)|+\sup_{\la\geq 0}\left(\frac{|B_1(-1, \la)|}{\sqrt{\la}}|v(-1, \la)|+\frac{|B_2(-1, \la)|}{\la^{\frac{3}{2}}}\right)\\
&&+\sup_{\la\geq 0}\left(\frac{|B_1(-1, \la)|}{\la}|v(-1, \la)|+\frac{|B_2(-1, \la)|}{\la^2}\right)\\
&\leq& C_0,
\eee
where the constant $C_0$ only depends on initial data. Hence, together with Lemma \ref{lemma:upperboundprumu}, we deduce
\bee
&&|v(\tau, \varrho)| \\
&\les& C_0+\sup_{0\leq\la\leq\varrho}\left(|B_1(\tau-\varrho+\la,\la)||v(\tau-\varrho+\la,\la)|+\frac{|B_2(\tau-\varrho+\la,\la)|}{\la}\right)\\
&+&\int_{R_{\tau, \varrho}}\frac{\sqrt{\la}}{\sqrt{\varrho}}(J(\mu)+|\mu-1||J'(\mu)|)\left(\frac{|B_1(\sigma, \la)|+|B_3(\sigma, \la)|}{\la^2}|v(\sigma, \la)|+\frac{|B_2(\sigma, \la)|+|B_4(\sigma, \la)|}{\la^3}\right)d\la d\sigma.
\eee
Together with Lemma \ref{lemma:upperboundB1234}, we infer
\bee
|v(\tau, \varrho)| &\les& C_0+C_0r^{3\delta-1}+\ep\sup_{0\leq\la\leq\varrho}|v(\tau-\varrho+\la,\la)|\\
&&+C_0\int_{R_{\tau, \varrho}}\frac{\sqrt{\la}}{\sqrt{\varrho}}(J(\mu)+|\mu-1||J'(\mu)|)\left(\frac{|v(\sigma, \la)|}{\la}+\frac{\la^{3\delta}}{\la^3}\right)d\la d\sigma,
\eee
where the constant $C_0$ only depends on the values of the solution in $I_0$. This concludes the proof of the lemma.
\end{proof}

\begin{lemma}
We have for all $\mu\geq -1$
$$|\mu-1||J'(\mu)|\les J(\mu).$$
\end{lemma}

\begin{proof}
For $-1\leq \mu< 1$, we have (see for example \cite{chris_tah1} p. 1087)
$$J'(\mu)=\frac{1}{4(1+\mu)}\int_{-\mu}^1\frac{\sqrt{1-x}dx}{(1+x)^{\frac{3}{2}}\sqrt{\mu+x}}.$$ 
We see in particular that $J'(\mu)\geq 0$ for $-1\leq \mu< 1$, and hence
$$J(\mu)\geq J(-1)\textrm{ on }-1\leq \mu< 1.$$
Since we have
$$J(-1)=\frac{\pi}{\sqrt{2}},$$
we deduce
$$J(\mu)\gtrsim 1\textrm{ on }-1\leq \mu< 1.$$
Also, we have for all $\mu\geq -1$ (see for example \cite{chris_tah1} p. 1061)
$$|\mu-1||J'(\mu)|\les 1.$$
We infer
$$|\mu-1||J'(\mu)|\les J(\mu)\textrm{ on }-1\leq \mu< 1.$$

Next, we consider the case $\mu\geq 1$, Then, we have
$$J'(\mu)=-\frac{1}{2}\int_{-1}^1\frac{dx}{\sqrt{1-x^2}(\mu+x)^{\frac{3}{2}}}.$$
We infer
\bee
|\mu-1||J'(\mu)| &\les& \int_{-1}^1\frac{|\mu-1|dx}{\sqrt{1-x^2}(\mu+x)^{\frac{3}{2}}}\\
&\les &  \int_{-1}^1\frac{dx}{\sqrt{1-x^2}\sqrt{\mu+x}}\\
&\les& J(\mu).
\eee
This concludes the proof of the lemma.
\end{proof}

We deduce the following corollary.
\begin{corollary}\lab{cor:wilallowtoconclude}
We have for all $(\tau, \varrho)$ with $\tau\geq -1$ and $\ub\leq 0$
\bee
|v(\tau, \varrho)| &\les& C_0+C_0r^{3\delta-1}+\ep\sup_{0\leq\la\leq\varrho}|v(\tau-\varrho+\la,\la)|+C_0\int_{R_{\tau, \varrho}}\frac{\sqrt{\la}}{\sqrt{\varrho}}J(\mu)\left(\frac{|v(\sigma, \la)|}{\la}+\frac{\la^{3\delta}}{\la^3}\right)d\la d\sigma.
\eee
where the constant $C_0$ only depends on the values of the solution in $I_0$.
\end{corollary}

The following proposition is the core of this section.
\begin{proposition}\lab{prop:unifupperboundv}
We have for all $(\tau, \varrho)$ with $\tau\geq -1$ and $\ub\leq 0$
$$|v|\leq C_0,$$
where the constant $C_0$ only depends on the values of the solution in $I_0$.
\end{proposition}

\begin{proof}
We choose $\delta$ such that
$$\frac{1}{3}<\delta<\frac{1}{2}.$$
Let 
$$\varpi=\varrho^{3\delta-1}.$$ 
Then, $\varpi$ satisfies 
$$\left(-\pr^2_{\tau}+\pr^2_{\varrho}+\frac{1}{\varrho}\pr_{\varrho}\right)\varpi=(3\delta-1)^2\varrho^{3\delta-3},$$
and hence
$$\varpi=\varpi_0(\tau,\varrho)+(3\delta-1)^2\int_{R_{\tau, \varrho}}\frac{\sqrt{\la}}{\sqrt{\varrho}}J(\mu)\frac{\la^{3\delta}}{\la^3} d\la d\sigma,$$
where $\varpi_0$ denotes the solution to the homogeneous wave equation with the same initial conditions as $\varpi$. This yields
$$\left|\int_{R_{\tau, \varrho}}\frac{\sqrt{\la}}{\sqrt{\varrho}}J(\mu)\frac{\la^{3\delta}}{\la^3} d\la d\sigma\right|\les \varrho^{3\delta-1}+|\varpi_0(\tau,\varrho)|\les 1,$$
where we used the fact that $\delta>1/3$ and $0\leq\varrho\leq 1$. 

Recall that 
\bee
|v(\tau, \varrho)| &\les& C_0+C_0r^{3\delta-1}+\ep\sup_{0\leq\la\leq\varrho}|v(\tau-\varrho+\la,\la)|+C_0\int_{R_{\tau, \varrho}}\frac{\sqrt{\la}}{\sqrt{\varrho}}J(\mu)\left(\frac{|v(\sigma, \la)|}{\la}+\frac{\la^{3\delta}}{\la^3}\right)d\la d\sigma.
\eee
We infer
\bee
|v(\tau, \varrho)| &\les& C_0+\ep\sup_{0\leq\la\leq\varrho}|v(\tau-\varrho+\la,\la)|+C_0\int_{R_{\tau, \varrho}}\frac{\sqrt{\la}}{\sqrt{\varrho}}J(\mu)\frac{|v(\sigma, \la)|}{\la} d\la d\sigma,
\eee
where we used again the fact that $\delta>1/3$ and $0\leq\varrho\leq 1$. 

We introduce the notation 
$$\b(u)=\sup_{u\leq\ub\leq 0}|v(u,\ub)|.$$
We obtain
\bee
|v(u, \ub)| &\les& C_0+\ep\b(u)+C_0\int_{-2}^u\Phi(u,\ub,u')\b(u')du',
\eee
where $\Phi$ is given by
$$\Phi(u,\ub,u')=\int_{u'}^{\ub}\frac{\sqrt{\la}}{\sqrt{\varrho}}J(\mu)\frac{1}{\la}d\ub'.$$

Next, we compute $\Phi$. We have
\bee
\Phi(u,\ub,u') &=& \int_{-1}^{+\infty}\frac{\sqrt{\la}}{\sqrt{\varrho}}J(\mu)\frac{1}{\la}\frac{1}{\pr_{\ub}\mu}d\mu
\eee
Since
\bee
\pr_{\ub}\mu &=& \frac{\la(-(\tau-\sigma)-\la)-\frac{1}{2}((\tau-\sigma)^2-\varrho^2-\la^2)}{2\varrho\la^2}\\
&=& \frac{\varrho^2-(\tau-\sigma+\la)^2}{4\varrho\la^2}\\
&=& \frac{\varrho^2-(\tau-u')^2}{4\varrho\la^2},
\eee
we infer
\bee
\Phi(u,\ub,u') &=&\frac{4\sqrt{\varrho}}{(\tau-u')^2-\varrho^2}\int_{-1}^{+\infty}J(\mu)\la^{\frac{3}{2}}d\mu.
\eee
Also, we have
$$\sigma=u'+\la$$
and hence
$$\mu=\frac{(\tau-u'-\la)^2-\varrho^2-\la^2}{2\varrho\la}=\frac{(\tau-u')^2-2\la(\tau-u')-\varrho^2}{2\varrho\la},$$
which yields
$$\la=\frac{(\tau-u')^2-\varrho^2}{2(\varrho\mu+\tau-u')}.$$
Hence, we obtain
\bee
\Phi(u,\ub,u') &=& \sqrt{2}\sqrt{\varrho}\sqrt{(\tau-u')^2-\varrho^2}\int_{-1}^{+\infty}\frac{J(\mu)}{(\varrho\mu+\tau-u')^{\frac{3}{2}}}d\mu.
\eee
Using the formula for $J(\mu)$ and Fubini, we infer
\bee
\Phi(u,\ub,u') &=& \sqrt{2}\sqrt{\varrho}\sqrt{(\tau-u')^2-\varrho^2}\int_{-1}^1\frac{dx}{\sqrt{1-x^2}}\int_{-x}^{+\infty}\frac{d\mu}{\sqrt{\mu+x}(\varrho\mu+\tau-u')^{\frac{3}{2}}}\\
&=& \sqrt{2}\sqrt{\varrho}\sqrt{(\tau-u')^2-\varrho^2}\int_{-1}^1\frac{dx}{\sqrt{1-x^2}}\int_0^{+\infty}\frac{d\mu'}{\sqrt{\mu'}(\varrho\mu'-\varrho x+\tau-u')^{\frac{3}{2}}}.
\eee
We have
\bee
\int_0^{+\infty}\frac{d\mu'}{\sqrt{\mu'}(\varrho\mu'-\varrho x+\tau-u')^{\frac{3}{2}}} &=& \frac{1}{\sqrt{\varrho}(-\varrho x+\tau-u')}\int_0^{+\infty}\frac{d\mu''}{\sqrt{\mu''}(\mu''+1)^{\frac{3}{2}}}\\
&\les&  \frac{1}{\sqrt{\varrho}(-\varrho x+\tau-u')}.
\eee
We deduce
\bee
\Phi(u,\ub,u') &\les& \sqrt{(\tau-u')^2-\varrho^2}\int_{-1}^1\frac{dx}{\sqrt{1-x^2}(-\varrho x+\tau-u')}.
\eee

We have for $x\geq 0$
\bee
-\varrho x+\tau-u' &\leq & \varrho x+\tau-u' 
\eee
and hence
\bee
\int_{-1}^1\frac{dx}{\sqrt{1-x^2}(-\varrho x+\tau-u')} &=& \int_{-1}^0\frac{dx}{\sqrt{1-x^2}(-\varrho x+\tau-u')}+\int_0^1\frac{dx}{\sqrt{1-x^2}(-\varrho x+\tau-u')}\\
&=& \int_0^1\frac{dx}{\sqrt{1-x^2}(\varrho x+\tau-u')}+\int_0^1\frac{dx}{\sqrt{1-x^2}(-\varrho x+\tau-u')}\\
&\leq& 2\int_0^1\frac{dx}{\sqrt{1-x^2}(-\varrho x+\tau-u')}.
\eee
This yields
\bee
\Phi(u,\ub,u') &\les& \sqrt{(\tau-u')^2-\varrho^2}\int_0^1\frac{dx}{\sqrt{1-x^2}(-\varrho x+\tau-u')}\\
&\les& \sqrt{(\tau-u')^2-\varrho^2}\int_0^1\frac{dx}{\sqrt{1-x}(-\varrho x+\tau-u')}.
\eee
Changing variables, we obtain
\bee
\Phi(u,\ub,u') &\les& \sqrt{(\tau-u')^2-\varrho^2}\int_0^1\frac{dy}{\sqrt{y}(\varrho y-\varrho+\tau-u')}\\
&\les& \frac{\sqrt{(\tau-u')^2-\varrho^2}}{\sqrt{\varrho}\sqrt{-\varrho+\tau-u'}}\int_0^{\frac{\varrho}{-\varrho+\tau-u'}}\frac{dz}{\sqrt{z}(z+1)}\\
&\les& \sqrt{\frac{-\varrho+\tau-u'}{\varrho}}\int_0^{\frac{\varrho}{-\varrho+\tau-u'}}\frac{dz}{\sqrt{z}}\\
&\les& 1.
\eee

Coming back to $v$, we obtain
\bee
|v(u, \ub)| &\les& C_0+\ep\b(u)+C_0\int_{-2}^u\b(u')du'.
\eee
Taking the supremum in $\ub$ for $u\leq\ub\leq 0$ yields
\bee
\b(u) &\les& C_0+\ep\b(u)+C_0\int_{-2}^u\b(u')du'.
\eee
For $\ep$ small enough, this implies
\bee
\b(u) &\les& C_0+C_0\int_{-2}^u\b(u')du'.
\eee
Using Gronwall's Lemma, we infer
\bee
\b(u) &\les& C_0,
\eee
and hence
$$|v|\les C_0.$$
This concludes the proof of the proposition.
\end{proof}

\subsection{Consequences of Proposition \ref{prop:unifupperboundv}}

\begin{lemma}\lab{lemma:utlimaterefinedbounds}
We have for all $(\tau, \varrho)$ with $\tau\geq -1$ and $\ub\leq 0$
$$|\phi|\leq rC_0\textrm{ and }|\pr_u\phi|+|\pr_{\ub}\phi|\leq C_0$$
where the constant $C_0$ only depends on the values of the solution in $I_0$.
\end{lemma}

\begin{proof}
We first derive a refined upper bound for $\phi$. Recall that
$$v=\left(\pr_{\varrho}+\frac{1}{\varrho}\right)\phi.$$
This yields
$$\pr_{\varrho}(\varrho\phi)=\varrho\pr_{\varrho}\phi+\phi=\varrho v$$
and hence
$$\varrho\phi(\tau, \varrho)=\int_0^{\varrho}\la v(\tau, \la)d\la.$$
We infer
\bee
\varrho|\phi(\tau, \varrho)| &\les& \int_0^{\varrho}\la |v(\tau, \la)|d\la\\
&\les& C_0\int_0^{\varrho}\la d\la\\
&\les& C_0\varrho^2
\eee
and hence 
$$|\phi|\les C_0\varrho.$$

The above upper bound for $\phi$ together with the upper bound for $v$ and the definition of $v$ implies
$$|\pr_{\varrho}\phi|\leq |v|+\frac{1}{\varrho}|\phi|\les C_0.$$

Next, we derive an upper bound for $\pr_u\phi$. Recall that $\Xi=r\pr_u\phi$ satisfies
\bee
\pr_{\ub}\left(\frac{\Xi}{\sqrt{r}}\right)=-\frac{1}{2\sqrt{r}}\pr_ur \pr_{\ub}\phi-\frac{\Omega^2f(\phi)}{4r^{\frac{3}{2}}}.
\eee
We integrate between $(u,\ub)$ and $(u, u)$. We deduce
\bee
\frac{\Xi(u,\ub)}{\sqrt{r(u,\ub)}} &=& \frac{\Xi(u,u)}{\sqrt{r(u,u)}}-\int_{u}^{\ub}\frac{1}{2\sqrt{r}}\pr_ur \pr_{\ub}\phi(u,\sigma)d\sigma -\int_u^{\ub}\frac{\Omega^2f(\phi)}{4r^{\frac{3}{2}}}(u,\sigma)d\sigma\\
&=& \frac{\Xi(u,u)}{\sqrt{r(u,u)}}-\left[\frac{1}{2\sqrt{r}}\pr_ur \phi(u,\sigma)\right]_u^{\ub}+\int_{u}^{\ub}\frac{1}{2\sqrt{r}}\pr_{\ub}\pr_ur \phi(u,\sigma)d\sigma\\
&&-\int_{u}^{\ub}\frac{1}{4r^{\frac{3}{2}}}\pr_{\ub}r \pr_ur \phi(u,\sigma)d\sigma -\int_u^{\ub}\frac{\Omega^2f(\phi)}{4r^{\frac{3}{2}}}(u,\sigma)d\sigma.
\eee
In view of the definition of $\Xi$ and since $(u,u)$ is on the axis of symmetry, we infer
\bee
\sqrt{r(u,\ub)}\pr_u\phi(u,\ub) &=& -\left[\frac{1}{2\sqrt{r}}\pr_ur \phi(u,\sigma)\right]_u^{\ub}+\int_{u}^{\ub}\frac{1}{2\sqrt{r}}\kappa \frac{f(\phi)^2}{r} \phi(u,\sigma)d\sigma\\
&&-\int_{u}^{\ub}\frac{1}{4r^{\frac{3}{2}}}\pr_{\ub}r \pr_ur \phi(u,\sigma)d\sigma -\int_u^{\ub}\frac{\Omega^2f(\phi)}{4r^{\frac{3}{2}}}(u,\sigma)d\sigma.
\eee
Using again the fact that $(u,u)$ is on the axis of symmetry, we deduce
\bee
\sqrt{r(u,\ub)}|\pr_u\phi(u,\ub)| &\les& \frac{|\phi(u,\ub)|}{\sqrt{r}}+\int_{u}^{\ub} \frac{|\phi(u,\sigma)|}{r^{\frac{3}{2}}} d\sigma.
\eee
Using the above upper bound for $\phi$, we infer 
\bee
\sqrt{r(u,\ub)}|\pr_u\phi(u,\ub)| &\les& C_0\sqrt{r(u,\ub)}+C_0\int_{u}^{\ub} \frac{1}{r^{\frac{1}{2}}} d\sigma\\
&\les& C_0\sqrt{r(u,\ub)}.
\eee
Thus, we obtain
\bee
|\pr_u\phi(u,\ub)| &\les& C_0.
\eee
Together with the upper bound for $\pr_{\varrho}\phi$, we also obtain
$$|\pr_{\ub}\phi|\leq |\pr_{\varrho}\phi|+|\pr_u\phi|\les C_0.$$
This concludes the proof of the lemma.
\end{proof}

\begin{lemma}\lab{lemma:utlimaterefinedboundsbis} 
We have for all $(\tau, \varrho)$ with $\tau\geq -1$ and $\ub\leq 0$
$$\frac{|\varrho-r|}{\varrho}+\left|\pr_u r+\frac{1}{2}\right|+\left|\pr_{\ub}r-\frac{1}{2}\right|+|\Omega-1|\leq C_0\varrho^2,$$
where the constant $C_0$ only depends on the values of the solution in $I_0$.
\end{lemma}

\begin{proof}
Integrating 
$$\pr_u\pr_{\ub}r=r\kappa\frac{\Omega^2}{4}\frac{g(\phi)^2}{r^2}$$
from the axis along $u$ or $\ub$ together with the upper bound on $\phi$ of Lemma  \ref{lemma:utlimaterefinedbounds} yields
$$\left|\pr_u r+\frac{1}{2}\right|+\left|\pr_{\ub}r-\frac{1}{2}\right|\leq C_0\varrho^2.$$
Together with 
$$\pr_{\ub}(\Omega^{-2}\pr_{\ub}r)=-\Omega^{-2}r\kappa(\pr_{\ub}\phi)^2$$
which we integrate from the axis  along $u$ yields in view of Lemma  \ref{lemma:utlimaterefinedbounds} 
$$\left|\Omega^{-2}\pr_{\ub}r-\frac{1}{2}\right|\leq C_0\varrho^2$$
and hence
$$|\Omega-1|\leq C_0\varrho^2.$$
Also, we have
$$\pr_u(r-\varrho)=\pr_ur+\frac{1}{2},$$
and integrating from the axis where $r=\varrho$, and using the estimate proved above for $\pr_ur$, we infer
$$|\varrho-r|\leq C_0\varrho^3.$$ 
This concludes the proof of the lemma.
\end{proof}

\begin{lemma}\lab{lemma:utlimaterefinedboundster} 
We have for all $(\tau, \varrho)$ with $\tau\geq -1$ and $\ub\leq 0$
$$|\pr^2_u r|+|\pr^2_{\ub}r|+|\pr_u\pr_{\ub}r|+|\pr_u\Omega|+|\pr_{\ub}\Omega|\leq C_0\varrho,$$
where the constant $C_0$ only depends on the values of the solution in $I_0$.
\end{lemma}

\begin{proof}
First, note that
$$\pr_u\pr_{\ub}r=r\kappa\frac{\Omega^2}{4}\frac{g(\phi)^2}{r^2}$$
together with Lemma \ref{lemma:utlimaterefinedbounds} immediately yields 
$$|\pr_u\pr_{\ub}r|\leq C_0\varrho.$$

Next, recall that $\Omega$ satisfies
$$\Omega^{-2}(\pr_u\Omega\pr_{\ub}\Omega-\Omega\pr_u\pr_{\ub}\Omega)=\frac{1}{8}\Omega^2\kappa\left(\frac{4}{\Omega^2}\pr_u\phi\pr_{\ub}\phi+\frac{g(\phi)^2}{r^2}\right)$$
and hence
$$\pr_u\pr_{\ub}\log(\Omega)=-\frac{1}{2}\pr_u\phi\pr_{\ub}\phi - \frac{\Omega^2}{8}\frac{g(\phi)^2}{r^2}.$$
Integrating from the axis along $u$ or $\ub$ yields in view of Lemma  \ref{lemma:utlimaterefinedbounds} 
$$|\pr_u\Omega |+|\pr_{\ub}\Omega|\leq C_0\varrho,$$
where we also used the fact that $\pr_u\Omega$ and $\pr_{\ub}\Omega$ vanish on the axis as a consequence of Lemma \ref{lemma:utlimaterefinedboundsbis}.

Finally, in view of 
$$\pr_{\ub}(\Omega^{-2}\pr_{\ub}r)=-\Omega^{-2}r\kappa(\pr_{\ub}\phi)^2$$
and 
$$\pr_u(\Omega^{-2}\pr_ur)=-\Omega^{-2}r\kappa(\pr_u\phi)^2,$$
and using  Lemma  \ref{lemma:utlimaterefinedbounds}  and the above bounds for $\pr_u\Omega$ and $\pr_{\ub}\Omega$, we infer
$$|\pr^2_u r|+|\pr^2_{\ub}r|\leq C_0\varrho.$$
This concludes the proof of the lemma.
\end{proof}

\section{Small energy implies global existence} \label{sec:small-energy-global}

In this section, we conclude the proof of Theorem \ref{th:smallenergyglobalex}.

\subsection{A wave equation for $\phi/\varrho$}

We first derive a wave equation for $\phi/\varrho$.
\begin{lemma}\label{lemma:waveequationforw}
We introduce 
$$w=\frac{\phi}{\varrho}.$$ 
Then, we have
\bee
 -4\pr_u\pr_{\ub}w+\frac{3}{\varrho}(\pr_{\ub}w-\pr_uw) &=&  -\frac{\varrho-r}{r\varrho}(\pr_{\ub}w-\pr_uw)+\frac{2\pr_ur+1}{r}\pr_{\ub}w+\frac{2\pr_{\ub}r-1}{r}\pr_uw\\ 
&& +\left(\frac{\Omega^2-1}{r^2}+\frac{\varrho - r}{r^2\varrho}+\frac{\pr_u r+\frac{1}{2}}{r\varrho}-\frac{\pr_{\ub}r-\frac{1}{2}}{r\varrho}\right)w\\
&&+\Omega^2 \frac{\varrho^2}{r^2}w^3\zeta(\varrho w).
\eee
\end{lemma}

\begin{proof}
We have
\bee
\square_\gg w &=& \frac{\square_\gg\phi}{\varrho}-2\frac{\D^\a\phi\D_\a \varrho}{\varrho^2}+\phi\left(-\frac{\square_\gg \varrho}{\varrho^2}+2\frac{\D^\a \varrho\D_\a \varrho}{\varrho^3}\right)\\
&=& \frac{f(\phi)}{r^2\varrho}-2\frac{\D_\a \varrho}{\varrho}\left(\D^\a w+\frac{w}{\varrho}\D^\a \varrho\right)+w\left(-\frac{\square_\gg \varrho}{\varrho}+2\frac{\D^\a \varrho\D_\a \varrho}{\varrho^2}\right)\\
&=& \frac{w}{r^2}+\frac{\varrho^2}{r^2}w^3\zeta(\varrho w)-2\frac{\D_\a \varrho}{\varrho}\left(\D^\a w+\frac{w}{\varrho}\D^\a\varrho\right)+w\left(-\frac{\square_\gg \varrho}{\varrho}+2\frac{\D^\a \varrho\D_\a \varrho}{\varrho^2}\right)\\
&=& -\frac{2}{\varrho}\D_\a \varrho\D^\a w+\left(\frac{1}{r^2}-\frac{\square_\gg \varrho}{\varrho}\right)w+\frac{\varrho^2}{r^2}w^3\zeta(\varrho w)\\
&=& \frac{2}{\varrho\Omega^2}(-\pr_{\ub}w+\pr_uw)+\left(\frac{1}{r^2}+\frac{1}{\Omega^2r\varrho}(\pr_u r-\pr_{\ub}r)\right)w+\frac{\varrho^2}{r^2}w^3\zeta(\varrho w).
\eee
In view of Lemma \ref{lemma:waveopinnullcoord}, we infer
\bee
 -4\pr_u\pr_{\ub}w+\frac{3}{\varrho}(\pr_{\ub}w-\pr_uw) &=&  -\frac{\varrho-r}{r\varrho}(\pr_{\ub}w-\pr_uw)+\frac{2\pr_ur+1}{r}\pr_{\ub}w+\frac{2\pr_{\ub}r-1}{r}\pr_uw\\ 
&& +\left(\frac{\Omega^2-1}{r^2}+\frac{\varrho - r}{r^2\varrho}+\frac{\pr_u r+\frac{1}{2}}{r\varrho}-\frac{\pr_{\ub}r-\frac{1}{2}}{r\varrho}\right)w\\
&&+\Omega^2 \frac{\varrho^2}{r^2}w^3\zeta(\varrho w).
\eee
This concludes the proof of the lemma.
\end{proof}

\begin{remark}\lab{rem:strichnotpossible}
In this paper, we first obtain improved uniform bounds for $\phi$ and then use the wave equation for $w$ to obtain regularity, following the approach of \cite{jal_tah1} and \cite{chris_tah1} for the 2+1 wave map problem. Note that we can not use a more direct approach based on Strichartz estimates for the wave equation for $w$ as in \cite{ShSt93} for the 2+1 wave map problem. Indeed, this approach does not allow to deal with the following terms in the right-hand side of the equation for $w$
$$\frac{\varrho-r}{r\varrho}\pr_{\varrho}w, \frac{2\pr_ur+1}{r}\pr_{\ub}w, \frac{2\pr_{\ub}r-1}{r}\pr_uw, \frac{\Omega^2-1}{r^2}w, \frac{\varrho - r}{r^2\varrho}w, \frac{\pr_u r+\frac{1}{2}}{r\varrho}w, \frac{\pr_{\ub}r-\frac{1}{2}}{r\varrho}w.$$
\end{remark}

\subsection{Embeddings for radial functions on $\R^4$}

Note that we have
\bee
 -4\pr_u\pr_{\ub}w+\frac{3}{\varrho}(\pr_{\ub}w-\pr_uw) &=& -\pr_\tau^2w+\pr_{\varrho}^2w+\frac{3}{\varrho}\pr_{\varrho}w.
\eee
Thus, the left-hand side of the wave equation for $w$ in Lemma \ref{lemma:waveequationforw} is the 4-dimensional radial wave operator. Therefore, we introduce the following spaces on backward null cones centered in the axis $\Gamma$
$$L^p_\ub:=\left\{\psi\,\,/\,\,\int_{-1}^{\ub} \psi(u)^p \varrho^3du\right\},\,\,\, 1\leq p<+\infty,\,\,\, -\frac{1}{2}\leq\ub<0,$$
and
$$H^\ell_\ub:=\left\{\psi\,\,/\,\,\sum_{j=0}^\ell\sum_{|\a|=j}\|\pr^\a_u\psi\|_{L^2_\ub}\right\},\,\,\, \ell\geq 0,\,\,\, -\frac{1}{2}\leq\ub<0.$$
We have the following Hardy estimates
\begin{equation}\label{eq:sec9:hardy}
\int_{-1}^{\ub}\left(\frac{\psi(u)}{\varrho^{\frac{1}{2}}}\right)^2 \varrho^3du+\int_{-1}^{\ub}\left(\frac{\psi(u)}{\varrho}\right)^2 \varrho^3du\les \|\psi\|^2_{H^1_\ub},\,\,\,\, -\frac{1}{2}\leq\ub<0,
\end{equation}
and
\begin{equation}\label{eq:sec9:hardybis}
\int_{-1}^{\ub}\left(\frac{\psi(u)}{\varrho^{\frac{3}{2}}}\right)^2 \varrho^3du\les \|\psi\|^2_{H^2_\ub},\,\,\,\, -\frac{1}{2}\leq\ub<0,
\end{equation}
as well as the following Gagliardo-Nirenberg inequalities  
\begin{equation}\label{eq:sec9:sobolev}
\|\psi\|_{L^p_\ub}\les \|\psi\|_{L^2_\ub}^{\frac{4}{p}-1}\|\psi\|_{H^1_\ub}^{2-\frac{4}{p}},\,\,\, 2\leq p\leq 4,\,\,\,\, -\frac{1}{2}\leq\ub<0,
\end{equation}
and
\begin{equation}\label{eq:sec9:sobolevbis}
\|\psi\|_{L^p_\ub}\les \|\psi\|_{L^2_\ub}^{\frac{2}{p}}\|\psi\|_{H^2_\ub}^{1-\frac{2}{p}},\,\,\, 2\leq p<+\infty,\,\,\,\, -\frac{1}{2}\leq\ub<0.
\end{equation}
Finally, we will also use the following weighted estimates
\begin{equation}\label{eq:sec9:weightedestimate}
\sup_{-1\leq u\leq \ub}\varrho|\psi|\les \|\psi\|_{H^1_\ub},\,\,\,\, -\frac{1}{2}\leq\ub<0,
\end{equation}
\begin{equation}\label{eq:sec9:weightedestimatebis}
\sup_{-1\leq u\leq \ub}\sqrt{\varrho}|\psi|\les \|\psi\|^{\frac{1}{2}}_{H^2_\ub}\|\psi\|^{\frac{1}{2}}_{H^1_\ub},\,\,\,\, -\frac{1}{2}\leq\ub<0,
\end{equation}
as well as the following non sharp Sobolev embedding 
\begin{equation}\label{eq:sec9:sobolevter}
\sup_{-1\leq u\leq \ub}|\psi|\les \|\psi\|_{H^3_\ub},\,\,\,\, -\frac{1}{2}\leq\ub<0
\end{equation}
and the following consequence of \eqref{eq:sec9:hardy} \eqref{eq:sec9:hardybis} \eqref{eq:sec9:sobolev} 
\begin{equation}\label{eq:sec9:hardy-sobolev}
\left\|\frac{\psi}{\varrho^{\frac{1}{2}}}\right\|_{L^4_{\ub}} \les\|\psi\|_{H^2_{\ub}},\,\,\,\, -\frac{1}{2}\leq\ub<0.
\end{equation}

\begin{remark}
Note that the constants in the above estimates \eqref{eq:sec9:hardy}-\eqref{eq:sec9:hardy-sobolev} are uniform in $\ub$ for $-1/2\leq\ub<0$.
\end{remark}

\subsection{Proof of Theorem \ref{th:smallenergyglobalex}}

We are now ready to prove Theorem \ref{th:smallenergyglobalex}. Recall that 
$$w=\frac{\phi}{\varrho}$$
satisfies the following wave equation
\bea\lab{eq:waveeqwsimplied}
 -4\pr_u\pr_{\ub}w+\frac{3}{\varrho}(\pr_{\ub}w-\pr_uw) &=& h,
\eea
where $h$ is given by
\bee
h &=&  -\frac{\varrho-r}{r\varrho}(\pr_{\ub}w-\pr_uw)+\frac{2\pr_ur+1}{r}\pr_{\ub}w+\frac{2\pr_{\ub}r-1}{r}\pr_uw\\ 
&& +\left(\frac{\Omega^2-1}{r^2}+\frac{\varrho - r}{r^2\varrho}+\frac{\pr_u r+\frac{1}{2}}{r\varrho}-\frac{\pr_{\ub}r-\frac{1}{2}}{r\varrho}\right)w+\Omega^2 \frac{\varrho^2}{r^2}w^3\zeta(\varrho w).
\eee
For convenience, we rewrite $h$ as
\bea
\nn h &=& \left( -\frac{\varrho-r}{r\varrho^2} -\frac{\pr_\varrho r-1}{r\varrho}\right)\varrho\pr_{\varrho}w+\frac{\pr_{\tau}r}{r}\pr_{\tau}w\\ 
&& +\left(\frac{\Omega^2-1}{r^2}+\frac{\varrho - r}{r^2\varrho}-\frac{\pr_\varrho r-1}{r\varrho}\right)w+\Omega^2 \frac{\varrho^2}{r^2}w^3\zeta(\varrho w).
\eea

To define our energies, it will be convenient to differentiate with respect to the cartesian frame of $\mathbb{R}\times\mathbb{R}^4$ based on the coordinates $(\tau, \varrho)$, i.e. 
\bee
\left(\tau, x^1, x^2, x^3, x^4\right)\in \mathbb{R}\times\mathbb{R}^4,\,\,\,\, \varrho=\sqrt{(x^1)^2+(x^2)^2+(x^3)^2+(x^4)^2},
\eee
as the partial derivatives with respect to such a frame commute with the wave operator ${}^{(1+4)}\square$ of $\mathbb{R}^{1+4}$ unlike $\pr_\varrho$. We also denote $\pr_{\tau, x}$ a generic partial derivative in this cartesian coordinates system. 

For $\ell\geq 0$, we introduce the following notations
\bee
D_\ell(\ub):= \max_{|\a| = \ell}\sup_{-1\leq u\leq\ub}&& \Bigg(\varrho^{\frac{3}{2}}\left(\left|\pr_{\tau,x}^\a\left(\frac{\varrho - r}{\varrho^3}\right)\right|+\left|\pr_{\tau,x}^\a\left(\frac{\pr_\varrho r-1}{\varrho^2}\right)\right|+\left|\pr_{\tau,x}^\a\left(\frac{\Omega^2-1}{\varrho^2}\right)\right|\right)\\
&&+\varrho^{\frac{1}{2}}\left(\left|\pr_{\tau,x}^\a\left(\frac{\pr_\tau r}{\varrho}\right)\right|\right)+\left|\pr_{\tau,x}^\a\left(\frac{\varrho - r}{\varrho}\right)\right|
+\left|\pr_{\tau,x}^\a\left(\pr_{\ub}r-\frac{1}{2}\right)\right|\\
&&+\left|\pr_{\tau,x}^\a\left(\pr_ur+\frac{1}{2}\right)\right|+\left|\pr_{\tau,x}^\a\left(\Omega^2-1\right)\right|\Bigg)(u,\ub),\,\,\,  -\frac{1}{2}\leq\ub<0,
\eee
$$D_{\leq \ell}(\ub):=\max_{0\leq j\leq\ell}D_j(\ub),$$
and for $\ell\geq 1$
$$E_{\ell}(\ub):= \max_{|\a|=\ell-1}\|\pr_u\pr_{\tau,x}^\alpha w(.,\ub)\|_{L^2_{\ub}},\,\,\,\, E_{\leq \ell}(\ub):=\max_{1\leq j\leq\ell}E_j(\ub).$$
 
\begin{lemma}\label{lemma:estforrighthandisdeprah}
Let $-1/2\leq\ub<0$. Then, we have the following estimates for $h$
\bee
\|h(\c, \ub)\|_{L^2_\ub} &\les& C_0 E_1(\ub)+C_0\|\pr_{\ub} w(\c, \ub)\|_{L^2_{\ub}}+C_0 
\eee
and if $\ell\geq 1$ and $|\a|=\ell$
\bee
&&\|\pr_{\tau,x}^\a h(\c, \ub)\|_{L^2_\ub}\\
 &\les& C_0\|\pr_{\ub}\pr_{\tau,x}^\a w(\c, \ub)\|_{L^2_{\ub}}+C_0\Big(1+E_{\leq \ell}^{\ell-1}(\ub)\Big)\Big(1+D^{\ell-1}_{\leq\ell -1}(\ub)\Big)\\
&&+\left(C_0+E_{\leq \ell}^{\ell-1}(\ub)+D_{\leq\ell -1}^{\ell-1}(\ub)\right)\Big(D_\ell(\ub)+E_{\ell+1}(\ub)\Big)
\eee
where the constant $C_0$ only depends on the values of the solution in $I_0$. 
\end{lemma}

\begin{lemma}\label{lemma:estimateforprubderivativeofw}
Let $-1/2\leq\ub<0$. Then, we have the following estimates for $w$
\bee
\|\pr_{\ub}w(\c, \ub)\|_{L^2_{\ub}} &\les& C_0+C_0E_1(\ub),
\eee
and for $\ell\geq 1$ and $|\a|=\ell$
\bee
\|\pr_{\ub}\pr_{\tau,x}^\a w(\c, \ub)\|_{L^2_{\ub}} &\les& C_0\Big(1+E_{\leq \ell}^{\ell-1}(\ub)\Big)\Big(1+D^{\ell-1}_{\leq\ell -1}(\ub)\Big)\\
&&+\left(C_0+E_{\leq \ell}^{\ell-1}(\ub)+D_{\leq\ell -1}^{\ell-1}(\ub)\right)\Big(D_\ell(\ub)+E_{\ell+1}(\ub)\Big)
\eee
where the constant $C_0$ only depends on the values of the solution in $I_0$. 
\end{lemma}

We postpone the proof of Lemma \ref{lemma:estforrighthandisdeprah} to section \ref{sec:lemma:estforrighthandisdeprah} and the proof of Lemma \ref{lemma:estimateforprubderivativeofw} to section \ref{sec:lemma:estimateforprubderivativeofw}. Next, let 
$$\widetilde{D}_{\ell}(\ub):=\sup_{-\frac{1}{2}\leq\ub'\leq\ub}D_{\ell}(\ub'),\,\, \widetilde{E}_{\ell}(\ub):=\sup_{-\frac{1}{2}\leq\ub'\leq\ub}E_{\ell}(\ub'),$$
and
$$\widetilde{D}_{\leq \ell}(\ub):=\sup_{-\frac{1}{2}\leq\ub'\leq\ub}D_{\leq\ell}(\ub'),\,\, \widetilde{E}_{\leq \ell}(\ub):=\sup_{-\frac{1}{2}\leq\ub'\leq\ub}E_{\leq\ell}(\ub').$$
Lemma \ref{lemma:estforrighthandisdeprah} and Lemma \ref{lemma:estimateforprubderivativeofw} immediately imply the following corollary.
\begin{corollary}\label{corollary:estforrighthandisdeprah}
Let $-\frac{1}{2}\leq\ub<0$. Then, we have the following estimates for $h$
\bee
\sup_{-\frac{1}{2}\leq\ub'\leq\ub}\|h(\c,\ub')\|_{L^2_{\ub'}} &\les& C_0 \widetilde{E}_1(\ub)+C_0, 
\eee
and if $\ell\geq 1$ and $|\a|=\ell$
\bee
\max_{|\a|=\ell}\sup_{-\frac{1}{2}\leq\ub'\leq\ub}\|\pr_{\tau,x}^\a h(\c,\ub')\|_{L^2_{\ub'}} &\les& \Big(C_0+\widetilde{D}^{\ell-1}_{\leq\ell-1}(\ub)+\widetilde{E}^{\ell-1}_{\leq\ell}(\ub)\Big)\Big(\widetilde{D}_\ell(\ub)+\widetilde{E}_{\ell+1}(\ub)\Big)\\
&&+C_0\Big(1+\widetilde{E}^{\ell-1}_{\leq \ell}(\ub)\Big) \Big(1+\widetilde{D}^{\ell-1}_{\leq\ell -1}(\ub)\Big),
\eee
where the constant $C_0$ only depends on the values of the solution in $I_0$. Furthermore, we also have 
\bee
\sup_{-\frac{1}{2}\leq\ub'\leq\ub}\|\pr_{\ub} w(\c, \ub')\|_{L^2_{\ub}} &\les& C_0 \widetilde{E}_1(\ub)+C_0
\eee
and if $\ell\geq 1$ and $|\a|=\ell$
\bee
&&\max_{|\a|=\ell}\sup_{-\frac{1}{2}\leq\ub'\leq\ub}\|\pr_{\ub}\pr_{\tau,x}^\a w(\c, \ub')\|_{L^2_{\ub}} \\
&\les& \Big(C_0+\widetilde{D}_{\leq\ell-1}(\ub)+\widetilde{E}_{\leq\ell}(\ub)\Big)\Big(\widetilde{D}_\ell(\ub)+\widetilde{E}_{\ell+1}(\ub)\Big)+C_0\Big(1+\widetilde{E}^{\ell-1}_{\leq \ell}(\ub)\Big) \Big(1+\widetilde{D}^{\ell-1}_{\leq\ell -1}(\ub)\Big).
\eee
\end{corollary}

Next, we derive an estimate for $\widetilde{D}_{\ell}$.
\begin{lemma}\label{lemma:estimateforDell}
Let $-\frac{1}{2}\leq\ub<0$ and let $\ell\in\N$. We have the following estimate 
\bee
\widetilde{D}_{\ell}(\ub)\les  C_0\Big(1+\widetilde{D}_{\leq\ell-1}(\ub)+\widetilde{E}_{\leq\ell}(\ub)\Big)^\ell\left(1+\widetilde{E}_{\ell+1}(\ub)\right),
\eee
where the constant $C_0$ only depends on the values of the solution in $I_0$. 
\end{lemma}

We postpone the proof of Lemma \ref{lemma:estimateforDell} to section \ref{sec:lemma:estimateforDell}. In view of Lemma \ref{lemma:utlimaterefinedboundsbis}, we have
$$\sup_{-\frac{1}{2}\leq\ub<0}D_0(\ub)\leq C_0.$$
By iteration, we infer from Lemma \ref{lemma:estimateforDell} that for all $\ell\in\N$, we have
\bee
\widetilde{D}_{\ell}(\ub)&\les& C_0\Big(1+\widetilde{E}_{\leq\ell}(\ub)\Big)^{l!}\left(1+\widetilde{E}_{\ell+1}(\ub)\right).
\eee
Together with Corollary \ref{corollary:estforrighthandisdeprah}, we obtain
$$
\max_{|\a|=\ell}\sup_{-\frac{1}{2}\leq\ub'\leq\ub}\|\pr_{\tau,x}^\a h(\c, \ub')\|_{L^2_{\ub'}} \les \Big(1+\widetilde{E}_{\leq\ell}(\ub)\Big)^{(l+1)!}\left(1+ \widetilde{E}_{\ell+1}(\ub)\right).
$$
Now, recall that we have
\bee
 -4\pr_u\pr_{\ub}w+\frac{3}{\varrho}(\pr_{\ub}w-\pr_uw) &=& h.
\eee
We rewrite this as
\bee
{}^{(1+4)}\square w &=& h
\eee
where ${}^{(1+4)}\square$ is the D'Alembertian on the $1+4$ dimensional Minkowski spacetime. Commuting with $\pr_{\tau,x}^\a$, we infer
\bee
{}^{(1+4)}\square\pr_{\tau, x}^\a w &=& \pr_{\tau, x}^\a h.
\eee
The energy estimates for the wave equation \eqref{eq:waveeqwsimplied} yields
\bee
E_{\ell+1}(\ub) &\les& E_{\ell+1}\left(-\frac{1}{2}\right)+\int_{-1}^{\ub}\|\pr_{\tau, x}^\a h(\ub', \cdot)\|_{L^2_{\ub}}d\ub',\,\,\,\,-\frac{1}{2}\leq\ub<0.
\eee
Together with the above estimate for $\pr_{\tau,x}^\a h$, we infer
\bee
E_{\ell+1}(\ub) &\les& E_{\ell+1}\left(-\frac{1}{2}\right)+\int_{-1}^{\ub}\Big(1+\widetilde{E}_{\leq\ell}(\ub')\Big)^{(l+1)!}\left(1+ \widetilde{E}_{\ell+1}(\ub')\right)d\ub'.
\eee
Using Gronwall Lemma together with the fact that the solution is smooth on $\{\ub=-\frac{1}{2}\}$\footnote{Note that the region $\{\ub=-\frac{1}{2}\}$ is compact and does not contain the potential singularity $(0,0)$.}, and together with the above estimate for $\widetilde{D}_{\ell}(\ub)$, we deduce by iteration for all $\ell\in\N$
$$\sup_{-\frac{1}{2}\leq\ub<0}(D_{\ell}(\ub)+E_{\ell+1}(\ub))\leq C_{\ell}<+\infty.$$
We have thus obtained the regularity of $(M, \gg, \phi)$ at the origin $(u, \ub)=(0,0)$. This concludes the proof of Theorem \ref{th:smallenergyglobalex}.


\subsection{Proof of Lemma \ref{lemma:estforrighthandisdeprah}}\label{sec:lemma:estforrighthandisdeprah}


Recall that we have
\bee
\nn h &=& \left( -\frac{\varrho-r}{r\varrho^2}  -\frac{\pr_\varrho r-1}{r\varrho}\right)\varrho\pr_{\varrho}w+\frac{\pr_{\tau}r}{r}\pr_{\tau}w\\ 
&& +\left(\frac{\Omega^2-1}{r^2}+\frac{\varrho - r}{r^2\varrho} -\frac{\pr_\varrho r-1}{r\varrho}\right)w+\Omega^2 \frac{\varrho^2}{r^2}w^3\zeta(\varrho w).
\eee
We rewrite it schematically as
\bee
h &=& \left(\frac{\varrho^2}{r^2}\frac{\Omega^2-1}{\varrho^2}, \frac{\varrho}{r}\frac{\varrho - r}{\varrho^3}, \frac{\varrho^2}{r^2}\frac{\varrho - r}{\varrho^3}, \frac{\varrho}{r}\frac{\pr_\varrho r-1}{\varrho^2}\right)\left(\varrho\pr_\varrho w, w\right)+\frac{\varrho}{r}\frac{\pr_{\tau}r}{\varrho}\pr_{\tau}w\\
&&+\frac{\varrho^2}{r^2}\Omega^2w^3\zeta(\varrho w).
\eee
Now, we have
\bea\label{eq:expansionofrhooverr}
\frac{\varrho}{r} = \frac{1}{1-\frac{\varrho-r}{\varrho}}= \sum_{j\geq 0}\left(\frac{\varrho-r}{\varrho}\right)^j,\,\, \frac{\varrho^2}{r^2} = \frac{1}{\left(1-\frac{\varrho-r}{\varrho}\right)^2}= \sum_{j\geq 0}(j+1)\left(\frac{\varrho-r}{\varrho}\right)^j
\eea
where the convergence follows from the estimate $ |r-\varrho|\les \ep\varrho$ proved in Lemma \ref{lemma:basicestimate}. This yields 
\bea\label{eq:expansionofrhooverrbis}
\nn\left|\pr_{\tau,x}^\beta\left( \frac{\varrho}{r}, \frac{\varrho^2}{r^2}\right)\right| &\les& \left(\sum_{j\geq 0}(j+|\beta|+1)\ep^j\right)\prod_{\sum_j\beta_j=\beta}\left|\pr_{\tau,x}^{\beta_j}\left(\frac{\varrho - r}{\varrho}\right)\right|\\
\nn&\les& \prod_{\sum_j\beta_j=\beta}\left|\pr_{\tau,x}^{\beta_j}\left(\frac{\varrho - r}{\varrho}\right)\right|\\
&\les& D_{|\beta|}(\ub)+D_{\leq |\beta|-1}^{|\beta|}(\ub)
\eea
where we used the definition of $D_j(\ub)$ with $j= |\b|-1$ and $j= |\b|$. Choosing $|\alpha|=\ell$ and together with the definition of $D_\ell(\ub)$, Leibniz formula and the commutator identity $[\pr_{x^i}, \varrho\pr_{\varrho}]=\pr_{x^i}$, we immediately deduce
\bee
|\pr_{\tau,x}^\a h| &\les& \left(\frac{|\pr_\tau w|}{\varrho^{\frac{1}{2}}}+\frac{|\varrho\pr_\varrho w|}{\varrho^{\frac{3}{2}}}+\frac{|w|}{\varrho^{\frac{3}{2}}}\right)D_\ell(\ub)\\
&&+\left(1+D_{\leq\ell -1}^\ell(\ub)\right)\left(\sum_{|\b|\leq \ell-1}\frac{|\varrho\pr_\varrho\pr_{\tau,x}^\b w|+\varrho|\pr_\tau\pr_{\tau,x}^\b w|+|\pr_{\tau,x}^\b w|}{\varrho^{\frac{3}{2}}}\right)\\
&&+C_0|\varrho\pr_\varrho\pr_{\tau,x}^\a w|+C_0\varrho|\pr_\tau\pr_{\tau,x}^\a w|+C_0|\pr_{\tau,x}^\a w|\\
&& +\frac{C_0}{\varrho^{\frac{1}{2}}}\left(D_\ell(\ub)+D_{\leq\ell -1}^\ell(\ub)\right)+\left(1+D_{\leq\ell -1}^\ell(\ub)\right)\left(\sum_{|\b|\leq \ell-1}\frac{\left|\pr_{\tau,x}^\b\left(w^3\zeta(\varrho w)\right)\right|}{\varrho^{\frac{1}{2}}}\right)\\
&&+C_0\left|\pr_{\tau,x}^\a\left(w^3\zeta(\varrho w)\right)\right|,
\eee
where we also used the estimates of Lemma \ref{lemma:utlimaterefinedboundsbis} as well as the fact that
 in view of Lemma \ref{lemma:utlimaterefinedbounds}, $w$ satisfies the following a priori bound
$$|w|\leq C_0$$
where the constant $C_0$ only depends on the values of the solution in $I_0$. We infer
\bea\lab{eq:intermediaryestimateforpartialalphah}
\nn\|\pr_{\tau,x}^\a h(\c, \ub)\|_{L^2_{\ub}} &\les& \left(C_0+\left\|\frac{\pr_u w(\c, \ub)}{\varrho^{\frac{1}{2}}}\right\|_{L^2_{\ub}}+\left\|\frac{\pr_{\ub} w(\c, \ub)}{\varrho^{\frac{1}{2}}}\right\|_{L^2_{\ub}} +\left\|\frac{w(\c, \ub)}{\varrho^{\frac{3}{2}}}\right\|_{L^2_{\ub}} \right)D_\ell(\ub)\\
\nn&&+\sum_{|\b|\leq \ell-1}\Bigg(C_0+\left\|\frac{\pr_u\pr_{\tau,x}^\b w(\c, \ub)}{\varrho^{\frac{1}{2}}}\right\|_{L^2_{\ub}} +\left\|\frac{\pr_{\ub}\pr_{\tau,x}^\b w(\c, \ub)}{\varrho^{\frac{1}{2}}}\right\|_{L^2_{\ub}} \\
\nn&&+\left\|\frac{\pr_{\tau,x}^\b w(\c, \ub)}{\varrho^{\frac{3}{2}}}\right\|_{L^2_{\ub}}+\left\|\frac{\pr_{\tau,x}^\b(w^3\zeta(\rho w))}{\varrho^{\frac{1}{2}}}\right\|_{L^2_{\ub}}\Bigg)\left(1+D_{\leq\ell -1}^\ell(\ub)\right)\\
\nn&& +C_0\|\pr_u\pr_{\tau,x}^\a w(\c, \ub)\|_{L^2_{\ub}}+C_0\|\pr_{\ub}\pr_{\tau,x}^\a w(\c, \ub)\|_{L^2_{\ub}} +C_0\left\|\pr_{\tau,x}^\a w(\c, \ub)\right\|_{L^2_{\ub}}\\
&&+C_0\|\pr_{\tau,x}^\a(w^3\zeta(\rho w))(\c, \ub)\|_{L^2_{\ub}}.
\eea

Note from its definition that $\zeta$ is an even function so that there exists a smooth function $\widetilde{\zeta}:\mathbb{R}\to \mathbb{R}$ such that
$$\zeta(s)=\widetilde{\zeta}(s^2).$$
Hence, we have
\bee
w^3\zeta(\rho w) = w^3\widetilde{\zeta}(\rho^2w^2)
\eee
where we may now exploit the fact that $\varrho^2$ - unlike $\varrho$ - is smooth in the cartesian coordinates system $(\tau, x)$. This yields
\bee
 |\pr^\a_{\tau, x}(w^3\zeta(\rho w))| &\les& C_0|\pr^\a_{\tau, x}w|+C_0|\pr_{\tau, x}w||\pr^{\a-1}_{\tau, x}w|1_{\{|\a|\geq 2\}}+C_0|\pr^{\leq \a-2}_{\tau, x}w|^21_{\{|\a|\geq 4\}}\\
&&+C_0\Big(1+|\pr^{\leq \a-3}_{\tau, x}w|^{|\a|-1}\Big)|\pr^{\leq \a-2}_{\tau, x}w|1_{\{|\a|\geq 5\}}
\eee
where we used the fact that $|w|\leq C_0$ and $0\leq \varrho\leq 1$. We infer
\bee
\|\pr^\a_{\tau, x}(w^3\zeta(\rho w))\|_{L^2_{\ub}} &\les& C_0\|\pr^\a_{\tau, x}w\|_{L^2_{\ub}}+C_0\|\pr_{\tau, x}w\|_{L^4_{\ub}}\|\pr^{\a-1}_{\tau, x}w\|_{L^4_{\ub}}1_{\{|\a|\geq 2\}}\\
&&+C_0\|\pr^{\leq \a-2}_{\tau, x}w\|^2_{L^4_{\ub}}1_{\{|\a|\geq 4\}}\\
&&+C_0\Big(1+\|\pr^{\leq \a-3}_{\tau, x}w\|^{|\a|-1}_{L^\infty_{\ub}}\Big)\|\pr^{\leq \a-2}_{\tau, x}w\|_{L^2_{\ub}}1_{\{|\a|\geq 5\}}\\
&\les& C_0\|\pr^\a_{\tau, x}w\|_{L^2_{\ub}}+C_0\|\pr_{\tau, x}w\|_{H^1_{\ub}}\|\pr^{\a-1}_{\tau, x}w\|_{H^1_{\ub}}1_{\{|\a|\geq 2\}}\\
&&+C_0\|\pr^{\leq \a-2}_{\tau, x}w\|^2_{H^1_{\ub}}1_{\{|\a|\geq 4\}}\\
&&+C_0\Big(1+\|\pr^{\leq \a-3}_{\tau, x}w\|^{|\a|-1}_{H^3_{\ub}}\Big)\|\pr^{\leq \a-2}_{\tau, x}w\|_{L^2_{\ub}}1_{\{|\a|\geq 5\}}\\
&\les& C_0E_{\ell+1}(\ub)+C_0\Big(E^2_{\leq \ell}(\ub)+E_{\leq \ell}^{\ell}(\ub)\Big)
\eee
where we used in particular the Sobolev embeddings \eqref{eq:sec9:sobolev} \eqref{eq:sec9:sobolevter} and the fact that $|\a|=\ell$. Similarly, we have
\bee
 |\pr^\b_{\tau, x}(w^3\zeta(\rho w))| &\les& C_0|\pr^\b_{\tau, x}w|+C_0|\pr_{\tau, x}w||\pr^{\b-1}_{\tau, x}w|1_{\{|\b|\geq 2\}}\\
&&+C_0\Big(1+|\pr^{\leq \b-2}_{\tau, x}w|^{|\b|-1}\Big)|\pr^{\leq \b-2}_{\tau, x}w|1_{\{|\b|\geq 4\}}
\eee
and hence
\bee
\left\|\frac{\pr_{\tau,x}^\b(w^3\zeta(\rho w))}{\varrho^{\frac{1}{2}}}\right\|_{L^2_{\ub}} &\les& C_0\left\|\frac{\pr_{\tau,x}^\b w}{\varrho^{\frac{1}{2}}}\right\|_{L^2_{\ub}}+C_0\|\pr_{\tau, x}w\|_{L^4_{\ub}}\left\|\frac{\pr_{\tau,x}^{\b-1}w}{\varrho^{\frac{1}{2}}}\right\|_{L^4_{\ub}}1_{\{|\b|\geq 2\}}\\
&& +C_0\Big(1+\|\pr^{\leq \b-2}_{\tau, x}w\|^{|\b|-1}_{L^\infty_{\ub}}\Big)\left\|\frac{\pr_{\tau,x}^{\leq \b-2} w}{\varrho^{\frac{1}{2}}}\right\|_{L^2_{\ub}}1_{\{|\b|\geq 4\}}\\
&\les&  C_0\left\|\pr_{\tau,x}^\b w\right\|_{H^1_{\ub}}+C_0\|\pr_{\tau, x}w\|_{H^1_{\ub}}\left\|\pr_{\tau,x}^{\b-1}w\right\|_{H^2_{\ub}}1_{\{|\b|\geq 2\}}\\
&& +C_0\Big(1+\|\pr^{\leq \b-2}_{\tau, x}w\|^{|\b|-1}_{H^3_{\ub}}\Big)\left\|\pr_{\tau,x}^{\leq \b-2} w\right\|_{H^1_{\ub}}1_{\{|\b|\geq 4\}}\\
&\les& C_0\Big(1+E_{\leq \ell}^{\ell-1}(\ub)\Big)
\eee
where we used in particular the Sobolev embeddings \eqref{eq:sec9:sobolev} \eqref{eq:sec9:sobolevter}, the Hardy estimates \eqref{eq:sec9:hardy} \eqref{eq:sec9:hardy-sobolev}, and the fact that $|\b|\leq\ell-1$. Together with \eqref{eq:intermediaryestimateforpartialalphah} and the Hardy estimates  \eqref{eq:sec9:hardy}  \eqref{eq:sec9:hardybis}, we infer
\bee
\|\pr_{\tau,x}^\a h(\c, \ub)\|_{L^2_\ub} &\les& C_0\|\pr_{\ub}\pr_{\tau,x}^\a w(\c, \ub)\|_{L^2_{\ub}}+\Big(C_0+E_{\leq 2}(\ub)+\|\pr_{\ub} w(\c, \ub)\|_{H^1_{\ub}}\Big)D_\ell(\ub)\\
&&+C_0\left(E_{\leq \ell+1}(\ub)+E_{\leq \ell}^{\ell-1}(\ub)+1\right)\left(1+D_{\leq\ell -1}^\ell(\ub)\right)\\
&& +C_0E_{\ell+1}(\ub)+C_0\Big(E^2_{\leq \ell}(\ub)+E_{\leq \ell}^{\ell}(\ub)\Big).
\eee
This yields
\bee
&&\|\pr_{\tau,x}^\a h(\c, \ub)\|_{L^2_\ub}\\
 &\les& C_0\|\pr_{\ub}\pr_{\tau,x}^\a w(\c, \ub)\|_{L^2_{\ub}}+C_0\Big(1+E_{\leq \ell}^{\ell-1}(\ub)\Big)\Big(1+D^\ell_{\leq\ell -1}(\ub)\Big)\\
&&+\left(C_0+E_{\leq 2}(\ub)+\|\pr_{\ub} w(\c, \ub)\|_{H^1_{\ub}}+E_{\leq \ell}^{\ell-1}(\ub)+D_{\leq\ell -1}^\ell(\ub)\right)\Big(D_\ell(\ub)+E_{\ell+1}(\ub)\Big)
\eee
which is the desired estimate for $\ell\geq 2$. 

It remains to consider the cases $\ell=0$ and $\ell=1$. We start with the case $\ell=0$. We have in view of Lemma \ref{lemma:utlimaterefinedboundsbis}
\bee
|h| &\les& C_0|\pr_uw|+C_0|\pr_{\ub}w|+C_0\left|\frac{w}{\varrho}\right|+C_0\\
&\les& C_0|\pr_uw|+C_0|\pr_{\ub}w|+\frac{C_0}{\varrho}+C_0,
\eee
where we used the bound $|w|\leq C_0$.Hence, we infer
\bee
\|h(\c, \ub)\|_{L^2_\ub} &\les& C_0 E_1(\ub)+C_0\|\pr_{\ub} w(\c, \ub)\|_{L^2_{\ub}}+C_0 
\eee
which is the desired estimate for $\ell=0$. 

Finally, we consider the case $\ell=1$. We have 
\bee
|\pr_{\tau, x} h| &\les& \left| -\frac{\varrho-r}{r\varrho^2}  -\frac{\pr_\varrho r-1}{r\varrho}\right||\pr_{\tau, x}(\varrho\pr_{\varrho}w)|+ \left|\pr_{\tau, x}\left( -\frac{\varrho-r}{r\varrho^2}  -\frac{\pr_\varrho r-1}{r\varrho}\right)\right||\varrho\pr_\varrho w|\\
&&+\left|\frac{\pr_{\tau}r}{r}\right||\pr_{\tau, x}\pr_{\tau}w|+\left|\pr_{\tau, x}\left(\frac{\pr_{\tau}r}{r}\right)\right||\pr_{\tau}w|  +\left|\frac{\Omega^2-1}{r^2}+\frac{\varrho - r}{r^2\varrho} -\frac{\pr_\varrho r-1}{r\varrho}\right||\pr_{\tau,x}w|\\
&&+\left|\pr_{\tau, x}\left(\frac{\Omega^2-1}{r^2}+\frac{\varrho - r}{r^2\varrho} -\frac{\pr_\varrho r-1}{r\varrho}\right)\right||w|\\
&&+\left|\Omega^2 \frac{\varrho^2}{r^2}\right||\pr_{\tau, x}(w^3\zeta(\varrho w))|+\left|\pr_{\tau, x}\left(\Omega^2 \frac{\varrho^2}{r^2}\right)\right||w^3\zeta(\varrho w)|.
\eee
Using Lemma \ref{lemma:utlimaterefinedbounds}, Lemma \ref{lemma:utlimaterefinedboundsbis} and Lemma \ref{lemma:utlimaterefinedboundster}, we infer
\bee
|\pr_{\tau,x} h| &\les& C_0+\frac{C_0}{\varrho}+C_0|\pr_{\tau, x}w|+C_0|\pr_u\pr_{\tau, x}w|+C_0|\pr_{\ub}\pr_{\tau, x}w|.
\eee
This yields
\bee
\|\pr_{\tau,x} h(\tau,\c)\|_{L^2_{r,\tau}} &\les& C_0E_2(\tau)+C_0\|\pr_{\ub}\pr_{\tau, x}w(\c, \ub)\|_{L^2_{\ub}}+C_0E_1(\tau)+C_0.
\eee
which concludes the proof of Lemma \ref{lemma:estforrighthandisdeprah}.

\subsection{Proof of Lemma \ref{lemma:estimateforprubderivativeofw}}\label{sec:lemma:estimateforprubderivativeofw}

Recall that we have
\bee
 -4\pr_u\pr_{\ub}w+\frac{3}{\varrho}(\pr_{\ub}w-\pr_uw) &=& h.
\eee
We rewrite this as
\bee
{}^{(1+4)}\square w &=& h
\eee
where ${}^{(1+4)}\square$ is the D'Alembertian on the $1+4$ dimensional Minkowski spacetime which is given in the $(u, \ub, \omega)$ coordinates system, with $\omega\in \mathbb{S}^3$, as
\bee
{}^{(1+4)}\square &=&  -4\pr_u\pr_{\ub}+\frac{3}{\varrho}(\pr_{\ub}-\pr_u) +\frac{1}{\varrho^2}\Delta_{\mathbb{S}^3}
\eee 
where $\Delta_{\mathbb{S}^3}$ is the Laplace-Beltrami operator on $\mathbb{S}^3$. We differentiate and obtain
\bee
{}^{(1+4)}\square(\pr_{\tau,x}^\a w) &=&  \pr_{\tau,x}^\a h.
\eee
Also, note that
\bee
\Delta_{\mathbb{S}^3} &=& \sum_{1\leq i<j\leq 4}\Omega_{i,j}^2
\eee
where $\Omega_{i,j}=x^i\pr_{x^j}-x^j\pr_{x^i}$ denote the angular momentum operators of $\mathbb{R}^4$. Since $w$ is radial, we have, we have $\Omega_{i,j}w=0$ and hence
\bee
\Delta_{\mathbb{S}^3}(\pr_{\tau,x}^\a w) &=& \sum_{1\leq i<j\leq 4}[\Omega_{i,j},[\Omega_{i,j},\pr_{\tau,x}^\a]]w. 
\eee
Furthermore, we have
\bee
[\Omega_{i,j}, \pr_\tau]=0,\,\,\,\, [\Omega_{i,j}, \pr_{x^l}]=-\delta_{i,l}\pr_{x^j}+\delta_{j,l}\pr_{x^i},
\eee
and hence
\bee
\Delta_{\mathbb{S}^3}(\pr_{\tau,x}^\a w) &=& \sum_{|\gamma|=|\a|}n_{\gamma, \a}\pr_{\tau,x}^\gamma w
\eee
where $n_{\gamma, \a}$ are integers. This yields
\bee
 -4\pr_u\pr_{\ub}(\pr_{\tau,x}^\a w)+\frac{3}{\varrho}(\pr_{\ub}\pr_{\tau,x}^\a w-\pr_u\pr_{\tau,x}^\a w) + \frac{1}{\varrho^2}\sum_{|\gamma|=|\a|}n_{\gamma, \a}\pr_{\tau,x}^\gamma w &=&  \pr_{\tau,x}^\a h.
\eee
We infer
\bee
\pr_u(\varrho^{\frac{3}{2}}\pr_{\ub}\pr_{\tau,x}^\a w) &=& -\frac{3}{4}\sqrt{\varrho}\pr_u\pr_{\tau,x}^\a w+\frac{1}{4\sqrt{\varrho}}\sum_{|\gamma|=|\a|}n_{\gamma, \a}\pr_{\tau,x}^\gamma w -\frac{1}{4}\varrho^{\frac{3}{2}}\pr_{\tau,x}^\a h.
\eee
Integrating from $u=-1$ where $w$ is controlled by $C_0$, we infer
\begin{equation}\label{eq:intermediateestimateforpartialubpartialalphaw}
\varrho^{\frac{3}{2}}|\pr_{\ub}\pr_{\tau,x}^\a w| \les C_0+\int_{-1}^u\sqrt{\varrho}|\pr_u\pr_{\tau,x}^\a w|du'+\sum_{|\gamma|=|\a|}\int_{-1}^u\frac{1}{\sqrt{\varrho}}|\pr_{\tau,x}^\gamma w|du'+\int_{-1}^u\varrho^{\frac{3}{2}}|\pr_{\tau,x}^\a h|du'.
\end{equation}

Next, recall from the proof of Lemma \ref{lemma:estforrighthandisdeprah} that 
\bee
|h| &\les& C_0|\pr_uw|+C_0|\pr_{\ub}w|+\frac{C_0}{\varrho}+C_0,
\eee
\bee
|\pr_{\tau,x} h| &\les& C_0+\frac{C_0}{\varrho}+C_0|\pr_{\tau, x}w|+C_0|\pr_u\pr_{\tau, x}w|+C_0|\pr_{\ub}\pr_{\tau, x}w|,
\eee
and for $|\a|=\ell\geq 2$
\bee
|\pr_{\tau,x}^\a h| &\les& \left(\frac{|\pr_\tau w|}{\varrho^{\frac{1}{2}}}+\frac{|\varrho\pr_\varrho w|}{\varrho^{\frac{3}{2}}}+\frac{|w|}{\varrho^{\frac{3}{2}}}\right)D_\ell(\ub)\\
&&+\left(1+D_{\leq\ell -1}^{\ell-1}(\ub)\right)\left(\sum_{|\b|\leq \ell-1}\frac{|\varrho\pr_\varrho\pr_{\tau,x}^\b w|+\varrho|\pr_\tau\pr_{\tau,x}^\b w|+|\pr_{\tau,x}^\b w|}{\varrho^{\frac{3}{2}}}\right)\\
&&+C_0|\varrho\pr_\varrho\pr_{\tau,x}^\a w|+C_0\varrho|\pr_\tau\pr_{\tau,x}^\a w|+C_0|\pr_{\tau,x}^\a w|\\
&& +\frac{C_0D_\ell(\ub)}{\varrho^{\frac{1}{2}}}+\left(1+D_{\leq\ell -1}^{\ell-1}(\ub)\right)\left(\sum_{|\b|\leq \ell-1}\frac{\left|\pr_{\tau,x}^\b\left(w^3\zeta(\varrho w)\right)\right|}{\varrho^{\frac{1}{2}}}\right)+C_0\left|\pr_{\tau,x}^\a\left(w^3\zeta(\varrho w)\right)\right|.
\eee
Using Cauchy-Schwarz, we deduce
\bee
\int_{-1}^u\varrho^{\frac{3}{2}}|h|du' &\les& C_0\int_{-1}^u\varrho^{\frac{3}{2}}|\pr_{\ub} w|du'+C_0\|\pr_u w(\c, \ub)\|_{L^2_{\ub}}+C_0,
\eee
\bee
\int_{-1}^u\varrho^{\frac{3}{2}}|\pr_{\tau,x} h|du' &\les& C_0\int_{-1}^u\varrho^{\frac{3}{2}}|\pr_{\ub}\pr_{\tau,x} w|du'+C_0\|\pr_u\pr_{\tau,x} w(\c, \ub)\|_{L^2_{\ub}}+C_0\|\pr_{\tau,x} w(\c, \ub)\|_{L^2_{\ub}}+C_0,
\eee
and for $|\a|=\ell\geq 2$
\bee
\int_{-1}^u\varrho^{\frac{3}{2}}|\pr_{\tau,x}^\a h|du' &\les& \left(C_0+\left\|\frac{\pr_u w(\c, \ub)}{\varrho^{\frac{1}{2}}}\right\|_{L^2_{\ub}}+\left\|\frac{\pr_{\ub} w(\c, \ub)}{\varrho^{\frac{1}{2}}}\right\|_{L^2_{\ub}} +\left\|\frac{w(\c, \ub)}{\varrho^{\frac{3}{2}}}\right\|_{L^2_{\ub}} \right)D_\ell(\ub)\\
&&+\sum_{|\b|\leq \ell-1}\Bigg(\left\|\frac{\pr_u\pr_{\tau,x}^\b w(\c, \ub)}{\varrho^{\frac{1}{2}}}\right\|_{L^2_{\ub}} +\left\|\frac{\pr_{\ub}\pr_{\tau,x}^\b w(\c, \ub)}{\varrho^{\frac{1}{2}}}\right\|_{L^2_{\ub}} \\
\nn&&+\left\|\frac{\pr_{\tau,x}^\b w(\c, \ub)}{\varrho^{\frac{3}{2}}}\right\|_{L^2_{\ub}}+\left\|\frac{\pr_{\tau,x}^\b(w^3\zeta(\rho w))}{\varrho^{\frac{1}{2}}}\right\|_{L^2_{\ub}}\Bigg)\left(1+D_{\leq\ell -1}^{\ell-1}(\ub)\right)\\
&& +C_0\int_{-1}^u\varrho^{\frac{3}{2}}|\pr_{\ub}\pr_{\tau,x}^\a w|du'+C_0\|\pr_u\pr_{\tau,x}^\a w(\c, \ub)\|_{L^2_{\ub}} +C_0\left\|\pr_{\tau,x}^\a w(\c, \ub)\right\|_{L^2_{\ub}}\\&&+C_0\|\pr_{\tau,x}^\a(w^3\zeta(\rho w))(\c, \ub)\|_{L^2_{\ub}}.
\eee
Coming back to \eqref{eq:intermediateestimateforpartialubpartialalphaw} and arguing as in the proof of Lemma \ref{lemma:estforrighthandisdeprah}, we infer 
\bee
\varrho^{\frac{3}{2}}|\pr_{\ub}w| &\les& C_0+C_0\int_{-1}^u\sqrt{\varrho}|\pr_uw|du'+C_0\int_{-1}^u\varrho^{\frac{3}{2}}|\pr_{\ub}w|du'
\eee
and for $|\alpha|\geq 1$
\bee
\varrho^{\frac{3}{2}}|\pr_{\ub}\pr_{\tau,x}^\a w| &\les& C_0+\int_{-1}^u\sqrt{\varrho}|\pr_u\pr_{\tau,x}^\a w|du'+\sum_{|\gamma|=|\a|}\int_{-1}^u\frac{1}{\sqrt{\varrho}}|\pr_{\tau,x}^\gamma w|du'+C_0\int_{-1}^u\varrho^{\frac{3}{2}}|\pr_{\ub}\pr_{\tau,x}^\a w|du'\\
&& +C_0\Big(1+E_{\leq \ell}^{\ell-1}(\ub)\Big)\Big(1+D^{\ell-1}_{\leq\ell -1}(\ub)\Big)\\
&&+\left(C_0+E_{\leq \ell}^{\ell-1}(\ub)+D_{\leq\ell -1}^{\ell-1}(\ub)\right)\Big(D_\ell(\ub)+E_{\ell+1}(\ub)\Big).
\eee
In view of Gronwall Lemma, we deduce
\bee
\varrho^{\frac{3}{2}}|\pr_{\ub}w| &\les& C_0+C_0\int_{-1}^u\sqrt{\varrho}|\pr_uw|du'
\eee
and for $|\alpha|\geq 1$
\bee
\varrho^{\frac{3}{2}}|\pr_{\ub}\pr_{\tau,x}^\a w| &\les& \int_{-1}^u\sqrt{\varrho}|\pr_u\pr_{\tau,x}^\a w|du'+\sum_{|\gamma|=|\a|}\int_{-1}^u\frac{1}{\sqrt{\varrho}}|\pr_{\tau,x}^\gamma w|du' +C_0\Big(1+E_{\leq \ell}^{\ell-1}(\ub)\Big)\Big(1+D^{\ell-1}_{\leq\ell -1}(\ub)\Big)\\
&&+\left(C_0+E_{\leq \ell}^{\ell-1}(\ub)+D_{\leq\ell -1}^{\ell-1}(\ub)\right)\Big(D_\ell(\ub)+E_{\ell+1}(\ub)\Big).
\eee
Now, we square and integrate, and we use
\bee
 \int_{-1}^\ub\left(\int_{-1}^u\sqrt{\varrho}|\pr_u\pr_{\tau,x}^\a w|du'\right)^2du &\les&  \int_{-1}^\ub\left(\int_{-1}^u\frac{du'}{\varrho^{\frac{3}{2}}}\right)\left(\int_{-1}^u\varrho^{\frac{5}{2}}(\pr_u\pr_{\tau,x}^\a w)^2du'\right)du\\
 &\les&  \int_{-1}^\ub \frac{1}{\sqrt{\varrho}}\left(\int_{-1}^u\varrho^{\frac{5}{2}}(\pr_u\pr_{\tau,x}^\a w)^2du'\right)du\\
  &\les&  \int_{-1}^\ub \varrho^{\frac{5}{2}}(\pr_u\pr_{\tau,x}^\a w)^2\left(\int_{u'}^\ub  \frac{1}{\sqrt{\varrho}}du\right)du'\\
 &\les&  \|\pr_u\pr_{\tau,x}^\a w\|_{L^2_{\ub}}^2
\eee
and
\bee
 \int_{-1}^\ub\left(\int_{-1}^u\frac{1}{\sqrt{\varrho}}|\pr_{\tau,x}^\gamma w|du'\right)^2du &\les&  \int_{-1}^\ub\left(\int_{-1}^u\frac{du'}{\varrho^{\frac{3}{2}}}\right)\left(\int_{-1}^u\varrho^{\frac{1}{2}}(\pr_{\tau,x}^\gamma w)^2du'\right)du\\
 &\les&  \int_{-1}^\ub \frac{1}{\sqrt{\varrho}}\left(\int_{-1}^u\varrho^{\frac{1}{2}}(\pr_{\tau,x}^\gamma w)^2du'\right)du\\
  &\les&  \int_{-1}^\ub \varrho^{\frac{1}{2}}(\pr_{\tau,x}^\gamma w)^2\left(\int_{u'}^\ub  \frac{1}{\sqrt{\varrho}}du\right)du'\\
 &\les&  \left\|\frac{\pr_{\tau,x}^\gamma w}{\varrho}\right\|_{L^2_{\ub}}^2\\
  &\les&  \|\pr_{\tau,x}^\gamma w\|_{H^1_{\ub}}^2.
\eee
We obtain
\bee
\|\pr_{\ub}w(\c, \ub)\|_{L^2_{\ub}} &\les&  C_0+C_0E_1(\ub)
\eee
and for $|\alpha|\geq 1$
\bee
\|\pr_{\ub}\pr_{\tau,x}^\a w(\c, \ub)\|_{L^2_{\ub}} &\les&  \left(\int_{-1}^\ub\left(\int_{-1}^u\sqrt{\varrho}|\pr_u\pr_{\tau,x}^\a w|du'\right)^2du\right)^{\frac{1}{2}}+\sum_{|\gamma|=|\a|}\left(\int_{-1}^\ub\left(\int_{-1}^u\frac{1}{\sqrt{\varrho}}|\pr_{\tau,x}^\gamma w|du'\right)^2du\right)^{\frac{1}{2}}\\
&&+C_0\Big(1+E_{\leq \ell}^{\ell-1}(\ub)\Big)\Big(1+D^{\ell-1}_{\leq\ell -1}(\ub)\Big)\\
&&+\left(C_0+E_{\leq \ell}^{\ell-1}(\ub)+D_{\leq\ell -1}^{\ell-1}(\ub)\right)\Big(D_\ell(\ub)+E_{\ell+1}(\ub)\Big)\\
&\les& C_0\Big(1+E_{\leq \ell}^{\ell-1}(\ub)\Big)\Big(1+D^{\ell-1}_{\leq\ell -1}(\ub)\Big)\\
&&+\left(C_0+E_{\leq \ell}^{\ell-1}(\ub)+D_{\leq\ell -1}^{\ell-1}(\ub)\right)\Big(D_\ell(\ub)+E_{\ell+1}(\ub)\Big).
\eee
This concludes the proof of Lemma \ref{lemma:estimateforprubderivativeofw}.

\subsection{Proof of Lemma \ref{lemma:estimateforDell}}\label{sec:lemma:estimateforDell}

First, note that Lemma \ref{lemma:utlimaterefinedboundsbis} and Lemma \ref{lemma:utlimaterefinedboundster} immediately imply
$$\widetilde{D}_{\leq 1}(\ub)\les C_0$$
which implies the desired estimate for $\ell=0$ and $\ell=1$. Thus, from now on, we focus on the case $\ell\geq 2$. 

Recall that we have
\bee
D_\ell(\ub)= \max_{|\a| = \ell}\sup_{-1\leq u\leq\ub}&& \Bigg(\varrho^{\frac{3}{2}}\left(\left|\pr_{\tau,x}^\a\left(\frac{\varrho - r}{\varrho^3}\right)\right|+\left|\pr_{\tau,x}^\a\left(\frac{\pr_\varrho r-1}{\varrho^2}\right)\right|+\left|\pr_{\tau,x}^\a\left(\frac{\Omega^2-1}{\varrho^2}\right)\right|\right)\\
&&+\varrho^{\frac{1}{2}}\left(\left|\pr_{\tau,x}^\a\left(\frac{\pr_\tau r}{\varrho}\right)\right|\right)+\left|\pr_{\tau,x}^\a\left(\frac{\varrho - r}{\varrho}\right)\right|
+\left|\pr_{\tau,x}^\a\left(\pr_{\ub}r-\frac{1}{2}\right)\right|\\
&&+\left|\pr_{\tau,x}^\a\left(\pr_ur+\frac{1}{2}\right)\right|+\left|\pr_{\tau,x}^\a\left(\Omega^2-1\right)\right|\Bigg)(u,\ub),\,\,\,  -\frac{1}{2}\leq\ub<0,
\eee
Here, it will be more convenient to use the $(u, \ub)$ coordinates system, and hence, we will use the following estimate 
\bee
\nn D_\ell(\ub)\les \max_{|\a| = \ell}\sup_{-1\leq u\leq\ub}&&  \Bigg(\varrho^{\frac{3}{2}}\left(\left|\pr_{u,\ub}^\a\left(\frac{\varrho - r}{\varrho^3}\right)\right|+\left|\pr_{u,\ub}^\a\left(\frac{\pr_\varrho r-1}{\varrho^2}\right)\right|+\left|\pr_{u,\ub}^\a\left(\frac{\Omega^2-1}{\varrho^2}\right)\right|\right)\\
&&+\varrho^{\frac{1}{2}}\left(\left|\pr_{u,\ub}^\a\left(\frac{\pr_\tau r}{\varrho}\right)\right|\right)+\left|\pr_{u,\ub}^\a\left(\frac{\varrho - r}{\varrho}\right)\right|
+\left|\pr_{u,\ub}^\a\left(\pr_{\ub}r-\frac{1}{2}\right)\right|\\
&&+\left|\pr_{u,\ub}^\a\left(\pr_ur+\frac{1}{2}\right)\right|+\left|\pr_{u,\ub}^\a\left(\Omega^2-1\right)\right|\Bigg)(u,\ub)+D_{\leq\ell-1}(\ub).
\eee
Also, it will be convient to estimate $\log(\Omega)$ instead of $\Omega^2-1$. Using
\bee
\Omega^2-1 = e^{2\log(\Omega)}-1 = \sum_{j\geq 1}\frac{2^j\log(\Omega)^j}{j!}
\eee
where the convergence follows from the bound $|\log(\Omega)|\les \ep$ which is a consequence of Lemma \ref{lemma:basicestimate}, we deduce
\bea\label{eq:alternatedefintionDellinuubframe}
\nn D_\ell(\ub)\les \max_{|\a| = \ell}\sup_{-1\leq u\leq\ub}&&  \Bigg(\varrho^{\frac{3}{2}}\left(\left|\pr_{u,\ub}^\a\left(\frac{\varrho - r}{\varrho^3}\right)\right|+\left|\pr_{u,\ub}^\a\left(\frac{\pr_\varrho r-1}{\varrho^2}\right)\right|+\left|\pr_{u,\ub}^\a\left(\frac{\log(\Omega)}{\varrho^2}\right)\right|\right)\\
\nn &&+\varrho^{\frac{1}{2}}\left(\left|\pr_{u,\ub}^\a\left(\frac{\pr_\tau r}{\varrho}\right)\right|\right)+\left|\pr_{u,\ub}^\a\left(\frac{\varrho - r}{\varrho}\right)\right|
+\left|\pr_{u,\ub}^\a\left(\pr_{\ub}r-\frac{1}{2}\right)\right|\\
&&+\left|\pr_{u,\ub}^\a\left(\pr_ur+\frac{1}{2}\right)\right|+\left|\pr_{u,\ub}^\a\left(\Omega^2-1\right)\right|\Bigg)(u,\ub)+D_{\leq\ell-1}^\ell(\ub).
\eea
From now on, we focus on estimating the right-hand side of \eqref{eq:alternatedefintionDellinuubframe} for $\ell\geq 2$. 

We start by estimating the following quantities which we will often encounter in the sequel. We have
\bea\label{eq:intermediaryusefulestimates1}
\nn&&\max_{|\b|\leq \ell-1}\sup_{-1\leq u\leq \ub}\sqrt{\varrho}\left(\left|\pr_{u,\ub}^\b\left(\frac{\Omega^2}{4}\frac{\varrho}{r}\frac{g(\varrho w)^2}{\varrho^2}\right)\right|(u,\ub)+\left|\pr_{u,\ub}^\b\left(\frac{\Omega^2}{8}\frac{\varrho^2}{r^2}\frac{g(\varrho w)^2}{\varrho^2}\right)\right|(u,\ub)\right)\\
&\les& C_0\Big(1+\widetilde{D}_{\leq\ell-1}(\ub)+\widetilde{E}_{\leq\ell}(\ub)\Big)^{\ell-1}\widetilde{E}_{\ell+1}(\ub)+C_0\Big(1+\widetilde{D}_{\leq\ell-1}(\ub)+\widetilde{E}_{\leq\ell}(\ub)\Big)^\ell
\eea
where we used the definition of $\widetilde{D}_{\leq\ell-1}(\ub)$, $\widetilde{E}_{\leq\ell}(\ub)$ and $\widetilde{E}_{\ell+1}(\ub)$, the estimate \eqref{eq:expansionofrhooverrbis} for $\varrho/r$ and $\varrho^2/r^2$, and the estimates \eqref{eq:sec9:weightedestimatebis} \eqref{eq:sec9:sobolevter}. Also, we have
\bea\label{eq:intermediaryusefulestimates1bis}
\max_{|\b|\leq \ell-1}\sup_{-1\leq u\leq \ub}\sqrt{\varrho}\left|\pr_{u,\ub}^\b\left(\pr_u(\varrho w)\pr_{\ub}(\varrho w)\right)\right|(u,\ub)\les \widetilde{E}_{\leq\ell}(\ub)\widetilde{E}_{\ell+1}(\ub)+\widetilde{E}_{\leq\ell}^2(\ub)
\eea
where we used the definition of $\widetilde{E}_{\leq\ell}(\ub)$ and $\widetilde{E}_{\ell+1}(\ub)$, and the estimates \eqref{eq:sec9:weightedestimate} \eqref{eq:sec9:weightedestimatebis} \eqref{eq:sec9:sobolevter}. 

Next, we estimate $\pr_{u,\ub}^\a\log(\Omega)$ in sup norm for $\ell\geq 2$. To this end, recall that we have
\bee
\pr_u\pr_{\ub}\log(\Omega)=-\frac{1}{2}\pr_u(\varrho w)\pr_{\ub}(\varrho w) - \frac{\Omega^2}{8}\frac{\varrho^2}{r^2}\frac{g(\varrho w)^2}{\varrho^2}.
\eee
First notice that we have in view of \eqref{eq:intermediaryusefulestimates1} and \eqref{eq:intermediaryusefulestimates1bis} the following estimate 
\bee
&&\max_{|\b|\leq \ell-1}\sup_{-1\leq u\leq \ub}\sqrt{\varrho}\left|\pr_{u,\ub}^\b\left(-\frac{1}{2}\pr_u(\varrho w)\pr_{\ub}(\varrho w) - \frac{\Omega^2}{8}\frac{\varrho^2}{r^2}\frac{g(\varrho w)^2}{\varrho^2}\right)\right|(u,\ub)\\
&\les& C_0\Big(1+\widetilde{D}_{\leq\ell-1}(\ub)+\widetilde{E}_{\leq\ell}(\ub)\Big)^{\ell-1}\widetilde{E}_{\ell+1}(\ub)+C_0\Big(1+\widetilde{D}_{\leq\ell-1}(\ub)+\widetilde{E}_{\leq\ell}(\ub)\Big)^\ell.
\eee
Commuting the equation for $\log(\Omega)$ with $\pr_{u,\ub}^\b$, integrating in $u$ from $u=-1$ where $\pr_{\ub}\pr^\b_{u,\ub}\log(\Omega)$ is controlled by $C_0$, and using the fact that $\sqrt{\varrho}^{-1}$ is integrable in $u$, we infer
\bee
\max_{|\b|\leq\ell-1}\sup_{-1\leq u\leq\ub}|\pr_{\ub}\pr_{u,\ub}^\b\log(\Omega)| &\les&  C_0\Big(1+\widetilde{D}_{\leq\ell-1}(\ub)+\widetilde{E}_{\leq\ell}(\ub)\Big)^{\ell-1}\widetilde{E}_{\ell+1}(\ub)\\
&&+C_0\Big(1+\widetilde{D}_{\leq\ell-1}(\ub)+\widetilde{E}_{\leq\ell}(\ub)\Big)^\ell.
\eee
Hence, the only missing derivative is $\pr_u^\ell\log(\Omega)$ which is contained for exemple in $\pr_u\pr^{\ell-1}_\tau\log(\Omega)$. This yields
\bee
\max_{|\a| =\ell}\sup_{-1\leq u\leq\ub}|\pr_{u,\ub}^\a\log(\Omega)| &\les& \sup_{-1\leq u\leq\ub}|\pr_u\pr^{\ell-1}_\tau\log(\Omega)|+C_0\Big(1+\widetilde{D}_{\leq\ell-1}(\ub)+\widetilde{E}_{\leq\ell}(\ub)\Big)^\ell\\
&& +C_0\Big(1+\widetilde{D}_{\leq\ell-1}(\ub)+\widetilde{E}_{\leq\ell}(\ub)\Big)^{\ell-1}\widetilde{E}_{\ell+1}(\ub).
\eee
Next, commuting the equation for $\log(\Omega)$ with $\pr^{\ell-1}_\tau$, integrating in $\ub$ from $(u,u)$ which is on the axis so that $\pr_u\pr^{\ell-1}_\tau\log(\Omega)$ vanishes, and using the fact that $\sqrt{\varrho}^{-1}$ is integrable in $\ub$, we infer
\bee
 \sup_{-1\leq u\leq\ub}|\pr_u\pr^{\ell-1}_\tau\log(\Omega)| &\les& C_0\Big(1+\widetilde{D}_{\leq\ell-1}(\ub)+\widetilde{E}_{\leq\ell}(\ub)\Big)^{\ell-1}\widetilde{E}_{\ell+1}(\ub)\\
&&  +C_0\Big(1+\widetilde{D}_{\leq\ell-1}(\ub)+\widetilde{E}_{\leq\ell}(\ub)\Big)^\ell.
\eee
Thus we have finally obtained 
\bee
\max_{|\a| =\ell}\sup_{-1\leq u\leq\ub}|\pr_{u,\ub}^\a\log(\Omega)| &\les& C_0\Big(1+\widetilde{D}_{\leq\ell-1}(\ub)+\widetilde{E}_{\leq\ell}(\ub)\Big)^{\ell-1}\widetilde{E}_{\ell+1}(\ub)\\
&&+C_0\Big(1+\widetilde{D}_{\leq\ell-1}(\ub)+\widetilde{E}_{\leq\ell}(\ub)\Big)^\ell.
\eee

Next, relying on 
\bee
\Omega^2-1 = e^{2\log(\Omega)}-1 = \sum_{j\geq 1}\frac{2^j\log(\Omega)^j}{j!}
\eee
where the convergence follows from the bound $|\log(\Omega)|\les \ep$ which is a consequence of Lemma \ref{lemma:basicestimate}, we deduce
\bea\label{eq:marcronpresident0}
\nn\max_{|\a| =\ell}\sup_{-1\leq u\leq\ub}|\pr_{u,\ub}^\a(\Omega^2-1)| &\les& \max_{|\a| =\ell}\sup_{-1\leq u\leq\ub}|\pr_{u,\ub}^\a\log(\Omega)|+\left(\max_{|\a| \leq\ell-1}\sup_{-1\leq u\leq\ub}|\pr_{u,\ub}^\a(\Omega^2-1)|\right)^\ell\\
\nn&\les& C_0\Big(1+\widetilde{D}_{\leq\ell-1}(\ub)+\widetilde{E}_{\leq\ell}(\ub)\Big)^{\ell-1}\widetilde{E}_{\ell+1}(\ub)\\
&&+C_0\Big(1+\widetilde{D}_{\leq\ell-1}(\ub)+\widetilde{E}_{\leq\ell}(\ub)\Big)^\ell.
\eea

Next, we estimate derivatives of $r$ in sup norm. First we notice that we have the following estimate 
\bee
&&\max_{|\a| =\ell}\sup_{-1\leq u\leq\ub}\left(\varrho\left|\pr_{u,\ub}^\a\left( \frac{\varrho}{r}\frac{\Omega^2}{4}\frac{g(\varrho w)^2}{\varrho^2}\right)\right|(u,\ub)+\left|\pr_{u,\ub}^\a\left( \frac{\varrho^2}{r}\frac{\Omega^2}{4}\frac{g(\varrho w)^2}{\varrho^2}\right)\right|(u,\ub)\right)\\
&\les& C_0\max_{|\a| =\ell}\sup_{-1\leq u\leq\ub}\left|\varrho\pr^\a_{u,\ub}\left(\frac{\varrho-r}{\varrho}\right)\right| +C_0\Big(1+\widetilde{D}_{\leq\ell-1}(\ub)+\widetilde{E}_{\leq\ell}(\ub)\Big)^{\ell-1}\widetilde{E}_{\ell+1}(\ub)\\
&&+C_0\Big(1+\widetilde{D}_{\leq\ell-1}(\ub)+\widetilde{E}_{\leq\ell}(\ub)\Big)^\ell
\eee
where we used the definition of $\widetilde{D}_{\leq\ell-1}(\ub)$, $\widetilde{E}_{\leq\ell}(\ub)$ and $\widetilde{E}_{\ell+1}(\ub)$, the expansion \eqref{eq:expansionofrhooverr} for $\varrho/r$, the estimates \eqref{eq:sec9:weightedestimate} \eqref{eq:sec9:sobolevter}, and the estimate \eqref{eq:marcronpresident0}. Since we have on the other hand\footnote{Here, we use in particular
$$\pr_u(r-\varrho) =  \pr_ur +\frac{1}{2},\,\,\,\, \pr_{\ub}(r-\varrho) =  \pr_{\ub}r -\frac{1}{2}.$$} 
\bee
 \max_{|\a| =\ell}\left|\varrho\pr^\a_{u,\ub}\left(\frac{\varrho-r}{\varrho}\right)\right|&\les&  \max_{|\a| =\ell}\left|\pr^\a_{u,\ub}(\varrho-r)\right|+ \max_{|\a| =\ell}\left|[\varrho,\pr^\a_{u,\ub}]\left(\frac{\varrho-r}{\varrho}\right)\right|\\
 &\les& \sum_{|\b|= \ell-1}\left|\pr^\b_{u,\ub}\left(\pr_ur+\frac{1}{2}\right)\right|+ \sum_{|\b|= \ell-1}\left|\pr^\b_{u,\ub}\left(\pr_{\ub}r-\frac{1}{2}\right)\right|\\
 &&+\sum_{|\b|= \ell-1}\left|\pr^\b_{u,\ub}\left(\frac{\varrho-r}{\varrho}\right)\right|\\
 &\les& \widetilde{D}_{\ell-1}(\ub),
\eee
we infer
\bee
\nn&&\max_{|\a| =\ell}\sup_{-1\leq u\leq\ub}\left(\left|\pr_{u,\ub}^\a\left( \frac{\varrho^2}{r}\frac{\Omega^2}{4}\frac{g(\varrho w)^2}{\varrho^2}\right)\right|(u,\ub)+\varrho\left|\pr_{u,\ub}^\a\left( \frac{\varrho}{r}\frac{\Omega^2}{4}\frac{g(\varrho w)^2}{\varrho^2}\right)\right|(u,\ub)\right)\,\,\,\,\,\,\,\,\\
&\les& C_0\Big(1+\widetilde{D}_{\leq\ell-1}(\ub)+\widetilde{E}_{\leq\ell}(\ub)\Big)^{\ell-1}\widetilde{E}_{\ell+1}(\ub)+C_0\Big(1+\widetilde{D}_{\leq\ell-1}(\ub)+\widetilde{E}_{\leq\ell}(\ub)\Big)^\ell.
\eee
Together with \eqref{eq:intermediaryusefulestimates1}, this yields
\bea\label{eq:intermediaryusefulestimates2}
\nn&&\max_{|\a| \leq\ell}\sup_{-1\leq u\leq\ub}\left(\varrho\left|\pr_{u,\ub}^\a\left( \frac{\varrho}{r}\frac{\Omega^2}{4}\frac{g(\varrho w)^2}{\varrho^2}\right)\right|(u,\ub)+\left|\pr_{u,\ub}^\a\left( \frac{\varrho^2}{r}\frac{\Omega^2}{4}\frac{g(\varrho w)^2}{\varrho^2}\right)\right|(u,\ub)\right)\,\,\,\,\,\,\,\,\\
&\les& C_0\Big(1+\widetilde{D}_{\leq\ell-1}(\ub)+\widetilde{E}_{\leq\ell}(\ub)\Big)^{\ell-1}\widetilde{E}_{\ell+1}(\ub)+C_0\Big(1+\widetilde{D}_{\leq\ell-1}(\ub)+\widetilde{E}_{\leq\ell}(\ub)\Big)^\ell.
\eea
Commuting the equation for $r$ with $\pr_{u,\ub}^\a$, and integrating in $u$ from $u=-1$ where $\pr^\a_{u,\ub}(\pr_{\ub}r-1/2)$ is controlled by $C_0$, we infer
\bee
\sup_{-1\leq u\leq\ub}\left|\pr^\a_{u,\ub}\left(\pr_{\ub}r-\frac{1}{2}\right)\right| &\les&  C_0\Big(1+\widetilde{D}_{\leq\ell-1}(\ub)+\widetilde{E}_{\leq\ell}(\ub)\Big)^{\ell-1}\widetilde{E}_{\ell+1}(\ub)\\
&&+C_0\Big(1+\widetilde{D}_{\leq\ell-1}(\ub)+\widetilde{E}_{\leq\ell}(\ub)\Big)^\ell.
\eee
Hence, the only missing derivative of order $\ell+1$ for $r$ is $\pr_u^{\ell+1}r$ which is contained for example in $\pr_u\pr^\ell_\tau r$. This yields
\bee
&&\max_{|\a| =\ell}\sup_{-1\leq u\leq\ub}\left|\pr^\a_{u,\ub}\left(\pr_{\ub}r-\frac{1}{2}\right)\right|  +\max_{|\a| =\ell}\sup_{-1\leq u\leq\ub}\left|\pr^\a_{u,\ub}\left(\pr_ur+\frac{1}{2}\right)\right| \\
&\les& \sup_{-1\leq u\leq\ub}\left|\pr_u\pr^\ell_\tau r\right|+C_0\Big(1+\widetilde{D}_{\leq\ell-1}(\ub)+\widetilde{E}_{\leq\ell}(\ub)\Big)^{\ell-1}\widetilde{E}_{\ell+1}(\ub)\\
&&+C_0\Big(1+\widetilde{D}_{\leq\ell-1}(\ub)+\widetilde{E}_{\leq\ell}(\ub)\Big)^\ell.
\eee
Next, commuting the equation for $r$ with $\pr^\ell_\tau$, and integrating in $\ub$ from $(u,u)$ which is on the axis so that $\pr_u\pr^\ell_\tau r$ vanishes, we infer
\bee
 \sup_{-1\leq u\leq\ub}|\pr_u\pr^\ell_\tau r| &\les& C_0\Big(1+\widetilde{D}_{\leq\ell-1}(\ub)+\widetilde{E}_{\leq\ell}(\ub)\Big)^{\ell-1}\widetilde{E}_{\ell+1}(\ub)+ C_0\Big(1+\widetilde{D}_{\leq\ell-1}(\ub)+\widetilde{E}_{\leq\ell}(\ub)\Big)^\ell.
\eee
Thus we have finally obtained 
\bea\label{eq:macronpresident}
&&\max_{|\a| =\ell} \sup_{-1\leq u\leq\ub}\left|\pr^\a_{u,\ub}\left(\pr_{\ub}r-\frac{1}{2}\right)\right|  +\max_{|\a| =\ell}\sup_{-1\leq u\leq\ub}\left|\pr^\a_{u,\ub}\left(\pr_ur+\frac{1}{2}\right)\right|\\
\nn &\les& C_0\Big(1+\widetilde{D}_{\leq\ell-1}(\ub)+\widetilde{E}_{\leq\ell}(\ub)\Big)^{\ell-1}\widetilde{E}_{\ell+1}(\ub)+ C_0\Big(1+\widetilde{D}_{\leq\ell-1}(\ub)+\widetilde{E}_{\leq\ell}(\ub)\Big)^\ell.
\eea

Next, we estimate the other terms in the RHS of \eqref{eq:alternatedefintionDellinuubframe}. For a regular scalar function $f$, we have the following Taylor expansions  
\bee
f(u, \ub) &=& f(u, u) + 2\varrho\int_0^1\pr_{\ub}f(u, u+2\sigma\varrho)d\sigma,\\
f(u, \ub) &=& f(\ub, \ub) - 2\varrho\int_0^1\pr_uf(\ub-2\sigma\varrho,\ub)d\sigma,\\
f(u,\ub) &=& f(u, u) + 2\varrho\int_0^1\pr_{\ub}f(u+2\sigma\varrho, u+2\sigma\varrho)d\sigma\\
&&  -4\varrho^2\int_0^1\int_0^1\sigma\pr_u\pr_{\ub}f(u+2\sigma\varrho-2s\sigma\varrho, u+2\sigma\varrho)dsd\sigma,
\eee
where we used the fact that $\ub-u=2\varrho$. In particular, since $(u,u)$, $(\ub, \ub)$ and $(u+2\sigma\varrho, u+2\sigma\varrho)$ are on the axis, we deduce
\bee
\frac{f(u, \ub)}{\varrho} &=& 2\int_0^1\pr_{\ub}f(u, u+2\sigma\varrho)d\sigma\textrm{ if }f=0\textrm{ on }\Gamma,\\
\frac{f(u, \ub)}{\varrho} &=&  - 2\int_0^1\pr_uf(\ub-2\sigma\varrho,\ub)d\sigma\textrm{ if }f=0\textrm{ on }\Gamma,\\
\frac{f(u, \ub)}{\varrho^2} &=& -4\int_0^1\int_0^1\sigma\pr_u\pr_{\ub}f(u+2\sigma\varrho-2s\sigma\varrho, u+2\sigma\varrho)dsd\sigma\textrm{ if }f=\pr_{\ub}f=0\textrm{ on }\Gamma.
\eee
In view of\footnote{Recall that we have
$$\pr_u(r-\varrho) =  \pr_ur +\frac{1}{2},\,\,\,\, \pr_{\ub}(r-\varrho) =  \pr_{\ub}r -\frac{1}{2}.$$} 
$$\varrho=r=0,\,\,\, \pr_{\ub}r=\frac{1}{2},\,\,\, \pr_ur=-\frac{1}{2},\,\,\, \Omega=1,\,\,\,r-\varrho=\pr_u(r-\varrho) =  \pr_{\ub}(r-\varrho)=\pr_u\Omega=\pr_{\ub}\Omega=0\textrm{ on }\Gamma,$$
we infer
\bee
\frac{r-\varrho}{\varrho} &=& 2\int_0^1\left(\pr_{\ub}r-\frac{1}{2}\right)(u, u+2\sigma\varrho)d\sigma,\\
\frac{\varrho - r}{\varrho^3} &=& \frac{4}{\varrho}\int_0^1\int_0^1\sigma\pr_u\pr_{\ub}r(u+2\sigma\varrho-2s\sigma\varrho, u+2\sigma\varrho)dsd\sigma,\\
\frac{\pr_\varrho r-1}{\varrho^2} &=& -\frac{\pr_u r+\frac{1}{2}}{\varrho^2} +\frac{\pr_{\ub}r-\frac{1}{2}}{\varrho^2}\\
&=& -\frac{2}{\varrho}\int_0^1\pr_{\ub}\pr_ur(u, u+2\sigma\varrho)d\sigma- \frac{2}{\varrho}\int_0^1\pr_u\pr_{\ub}r(\ub-2\sigma\varrho,\ub)d\sigma,\\
\frac{\pr_\tau r}{\varrho} &=& -4\varrho\int_0^1\int_0^1\sigma\pr_{\tau}\pr_u\pr_{\ub}r(u+2\sigma\varrho-2s\sigma\varrho, u+2\sigma\varrho)dsd\sigma,\\
\frac{\log(\Omega)}{\varrho^2} &=& -4\int_0^1\int_0^1\sigma\pr_u\pr_{\ub}\log(\Omega)(u+2\sigma\varrho-2s\sigma\varrho, u+2\sigma\varrho)dsd\sigma.
\eee

Next, recall that $r$ satisfies
\bee
\pr_{\ub}\pr_ur &=& r\kappa \frac{\Omega^2}{4}\frac{g(\phi)^2}{r^2}\\
&=& \varrho\frac{\varrho}{r}\kappa \frac{\Omega^2}{4}\frac{g(\varrho w)^2}{\varrho^2}.
\eee
Together with the identities 
\bee
\varrho(u+2\sigma\varrho-2s\sigma\varrho, u+2\sigma\varrho) &=& s\sigma\varrho(u, \ub),\\
\varrho(u, u+2\sigma\varrho) &=& \sigma\varrho(u, \ub),\\
\varrho(\ub-2\sigma\varrho,\ub) &=& \sigma\varrho(u,\ub),
\eee
we infer
\bee
\frac{\varrho - r}{\varrho^3} &=& 4\int_0^1\int_0^1s\sigma^2\left[\frac{\varrho}{r}\kappa \frac{\Omega^2}{4}\frac{g(\varrho w)^2}{\varrho^2}\right](u+2\sigma\varrho-2s\sigma\varrho, u+2\sigma\varrho)dsd\sigma,\\
\frac{\pr_\varrho r-1}{\varrho^2} &=& - 2\int_0^1\sigma\left[\frac{\varrho}{r}\kappa \frac{\Omega^2}{4}\frac{g(\varrho w)^2}{\varrho^2}\right](u, u+2\sigma\varrho)d\sigma- 2\int_0^1\sigma\left[ \frac{\varrho}{r}\kappa \frac{\Omega^2}{4}\frac{g(\varrho w)^2}{\varrho^2}\right](\ub-2\sigma\varrho,\ub)d\sigma,\\
\frac{\pr_\tau r}{\varrho} &=& -4\varrho^2\int_0^1\int_0^1s\sigma^2\left[\pr_{\tau}\left(\frac{\varrho}{r}\kappa \frac{\Omega^2}{4}\frac{g(\varrho w)^2}{\varrho^2}\right)\right](u+2\sigma\varrho-2s\sigma\varrho, u+2\sigma\varrho)dsd\sigma.
\eee
Also, recall that $\Omega$ satisfies
$$\Omega^{-2}(\pr_u\Omega\pr_{\ub}\Omega-\Omega\pr_u\pr_{\ub}\Omega)=\frac{1}{8}\Omega^2\kappa\left(\frac{4}{\Omega^2}\pr_u\phi\pr_{\ub}\phi+\frac{g(\phi)^2}{r^2}\right)$$
and hence
$$\pr_u\pr_{\ub}\log(\Omega)=-\frac{1}{2}\pr_u(\varrho w)\pr_{\ub}(\varrho w) - \frac{\Omega^2}{8}\frac{\varrho^2}{r^2}\frac{g(\varrho w)^2}{\varrho^2}.$$
We infer
\bee
\frac{\log(\Omega)}{\varrho^2} &=& -4\int_0^1\int_0^1\sigma\left[-\frac{1}{2}\pr_u(\varrho w)\pr_{\ub}(\varrho w) - \frac{\Omega^2}{8}\frac{\varrho^2}{r^2}\frac{g(\varrho w)^2}{\varrho^2}\right](u+2\sigma\varrho-2s\sigma\varrho, u+2\sigma\varrho)dsd\sigma.
\eee

Next, we commute with $\pr^\a_{u, \ub}$. Note that we have for $\a=(\a_1, \a_2)$
\bee
\pr^\a_{u, \ub}\Big(f(u, u+2\sigma\varrho)\Big) &=& \Big[\sigma^{\a_2}(\pr_u+(1-\sigma)\pr_{\ub})^{\a_1}\pr_{\ub}^{\a_2}f\Big](u, u+2\sigma\varrho),\\
\pr^\a_{u,\ub}\Big(f(\ub-2\sigma\varrho,\ub)\Big) &=& \Big[\sigma^{\a_1}\pr_u^{\a_1}(\pr_{\ub}+(1-\sigma)\pr_u)^{\a_2}f\Big](\ub-2\sigma\varrho,\ub),
\eee
and
\bee
&&\pr^\a_{u,\ub}\Big(f(u+2\sigma\varrho-2s\sigma\varrho, u+2\sigma\varrho)\Big)\\
 &=& \Big[((1-\sigma(1-s))\pr_u+(1-\sigma)\pr_{\ub})^{\a_1}(\sigma\pr_{\ub}+\sigma(1-s)\pr_u)^{\a_2}f\Big](u+2\sigma\varrho-2s\sigma\varrho, u+2\sigma\varrho).
\eee
We deduce
\bee
\pr^\a_{u, \ub}\left(\frac{r-\varrho}{\varrho}\right) &=& 2\int_0^1\left[\sigma^{\a_2}(\pr_u+(1-\sigma)\pr_{\ub})^{\a_1}\pr_{\ub}^{\a_2}\left(\pr_{\ub}r-\frac{1}{2}\right)\right](u, u+2\sigma\varrho)d\sigma,\\
\pr^\a_{u, \ub}\left(\frac{\varrho - r}{\varrho^3}\right) &=& 4\int_0^1\int_0^1s\sigma^2\Bigg[((1-\sigma(1-s))\pr_u+(1-\sigma)\pr_{\ub})^{\a_1}\\
&& (\sigma\pr_{\ub}+\sigma(1-s)\pr_u)^{\a_2}\left(\frac{\varrho}{r}\kappa \frac{\Omega^2}{4}\frac{g(\varrho w)^2}{\varrho^2}\right)\Bigg](u+2\sigma\varrho-2s\sigma\varrho, u+2\sigma\varrho)dsd\sigma,\\
\pr^\a_{u, \ub}\left(\frac{\pr_\varrho r-1}{\varrho^2}\right) &=& - 2\int_0^1\sigma\left[\sigma^{\a_2}(\pr_u+(1-\sigma)\pr_{\ub})^{\a_1}\pr_{\ub}^{\a_2}\left( \frac{\varrho}{r}\kappa \frac{\Omega^2}{4}\frac{g(\varrho w)^2}{\varrho^2}\right)\right](u, u+2\sigma\varrho)d\sigma\\
&&- 2\int_0^1\sigma\left[\sigma^{\a_1}\pr_u^{\a_1}(\pr_{\ub}+(1-\sigma)\pr_u)^{\a_2}\left( \frac{\varrho}{r}\kappa \frac{\Omega^2}{4}\frac{g(\varrho w)^2}{\varrho^2}\right)\right](\ub-2\sigma\varrho,\ub)d\sigma,\\
\pr^\a_{u, \ub}\left(\frac{\pr_\tau r}{\varrho}\right) &=& -4\sum_{|\beta|\leq 2}\pr^\beta_{u,\ub}(\varrho^2)\int_0^1\int_0^1s\sigma^2\Bigg[((1-\sigma(1-s))\pr_u+(1-\sigma)\pr_{\ub})^{\a_1-\b_1}\\
&&(\sigma\pr_{\ub}+\sigma(1-s)\pr_u)^{\a_2-\b_2}\pr_{\tau}\left(\frac{\varrho}{r}\kappa \frac{\Omega^2}{4}\frac{g(\varrho w)^2}{\varrho^2}\right)\Bigg](u+2\sigma\varrho-2s\sigma\varrho, u+2\sigma\varrho)dsd\sigma,\\
\pr^\a_{u, \ub}\left(\frac{\log(\Omega)}{\varrho^2}\right) &=& -4\int_0^1\int_0^1\sigma\Bigg[((1-\sigma(1-s))\pr_u+(1-\sigma)\pr_{\ub})^{\a_1}\\
&&  (\sigma\pr_{\ub}+\sigma(1-s)\pr_u)^{\a_2}\left(-\frac{1}{2}\pr_u(\varrho w)\pr_{\ub}(\varrho w) - \frac{\Omega^2}{8}\frac{\varrho^2}{r^2}\frac{g(\varrho w)^2}{\varrho^2}\right)\Bigg]\\
&&(u+2\sigma\varrho-2s\sigma\varrho, u+2\sigma\varrho)dsd\sigma.
\eee

In view of the first identity, we have
\bee
\left|\pr^\a_{u, \ub}\left(\frac{r-\varrho}{\varrho}\right)\right| &\les& \int_0^1\left|\left[\sigma^{\a_2}(\pr_u+(1-\sigma)\pr_{\ub})^{\a_1}\pr_{\ub}^{\a_2}\left(\pr_{\ub}r-\frac{1}{2}\right)\right]\right|(u, u+2\sigma\varrho)d\sigma\\
&\les& \sup_{-2\leq \ub'\leq \ub}\sup_{-1\leq u\leq\ub}\left|\pr^\a_{u,\ub}\left(\pr_{\ub}r-\frac{1}{2}\right)\right|
\eee
and hence, in view of \eqref{eq:macronpresident}, we deduce
\bea\label{eq:macronpresident1}
\nn\max_{|\a| =\ell} \sup_{-1\leq u\leq\ub}\left|\pr^\a_{u, \ub}\left(\frac{r-\varrho}{\varrho}\right)\right|  &\les& C_0\Big(1+\widetilde{D}_{\leq\ell-1}(\ub)+\widetilde{E}_{\leq\ell}(\ub)\Big)^{\ell-1}\widetilde{E}_{\ell+1}(\ub)\\
&&+C_0\Big(1+\widetilde{D}_{\leq\ell-1}(\ub)+\widetilde{E}_{\leq\ell}(\ub)\Big)^\ell.
\eea

Next, recall that 
\bee
\varrho(u+2\sigma\varrho-2s\sigma\varrho, u+2\sigma\varrho) &=& s\sigma\varrho(u, \ub),\\
\varrho(u, u+2\sigma\varrho) &=& \sigma\varrho(u, \ub),\\
\varrho(\ub-2\sigma\varrho,\ub) &=& \sigma\varrho(u,\ub),
\eee
which allows us to rewrite the second and the third identities above as
\bee
\varrho\pr^\a_{u, \ub}\left(\frac{\varrho - r}{\varrho^3}\right) &=& 4\int_0^1\int_0^1\sigma\Bigg[\varrho((1-\sigma(1-s))\pr_u+(1-\sigma)\pr_{\ub})^{\a_1}\\
&&  (\sigma\pr_{\ub}+\sigma(1-s)\pr_u)^{\a_2}\left(\frac{\varrho}{r}\kappa \frac{\Omega^2}{4}\frac{g(\varrho w)^2}{\varrho^2}\right)\Bigg](u+2\sigma\varrho-2s\sigma\varrho, u+2\sigma\varrho)dsd\sigma,\\
\varrho\pr^\a_{u, \ub}\left(\frac{\pr_\varrho r-1}{\varrho^2}\right) &=& - 2\int_0^1\left[\varrho\sigma^{\a_2}(\pr_u+(1-\sigma)\pr_{\ub})^{\a_1}\pr_{\ub}^{\a_2}\left( \frac{\varrho}{r}\kappa \frac{\Omega^2}{4}\frac{g(\varrho w)^2}{\varrho^2}\right)\right](u, u+2\sigma\varrho)d\sigma\\
&&- 2\int_0^1\left[\varrho\sigma^{\a_1}\pr_u^{\a_1}(\pr_{\ub}+(1-\sigma)\pr_u)^{\a_2}\left( \frac{\varrho}{r}\kappa \frac{\Omega^2}{4}\frac{g(\varrho w)^2}{\varrho^2}\right)\right](\ub-2\sigma\varrho,\ub)d\sigma.
\eee
We infer
\bee
&&\max_{|\a| = \ell}\sup_{-1\leq u\leq\ub} \Bigg(\varrho\Bigg(\left|\pr_{u,\ub}^\a\left(\frac{\varrho - r}{\varrho^3}\right)\right|+\left|\pr_{u,\ub}^\a\left(\frac{\pr_\varrho r-1}{\varrho^2}\right)\right|\Bigg)\Bigg)(u,\ub)\\
&\les& \max_{|\a| = \ell}\sup_{-2\leq \ub'\leq \ub}\sup_{-1\leq u\leq\ub}\varrho\left|\pr_{u,\ub}^\a\left(\frac{\varrho}{r}\Omega^2\frac{g(\varrho w)^2}{\varrho^2}\right)\right|
\eee
which together with \eqref{eq:intermediaryusefulestimates2} yields
\bea\label{eq:macronpresident2}
\nn&&\max_{|\a| = \ell}\sup_{-1\leq u\leq\ub} \Bigg(\varrho\Bigg(\left|\pr_{u,\ub}^\a\left(\frac{\varrho - r}{\varrho^3}\right)\right|+\left|\pr_{u,\ub}^\a\left(\frac{\pr_\varrho r-1}{\varrho^2}\right)\right|\Bigg)\Bigg)(u,\ub)\\
&\les& C_0\Big(1+\widetilde{D}_{\leq\ell-1}(\ub)+\widetilde{E}_{\leq\ell}(\ub)\Big)^\ell(1+\widetilde{E}_{\ell+1}(\ub)).
\eea

Next, recall that we have
\bee
\pr^\a_{u, \ub}\left(\frac{\pr_\tau r}{\varrho}\right) &=& -4\sum_{|\beta|\leq 2}\pr^\beta_{u,\ub}(\varrho^2)\int_0^1\int_0^1s\sigma^2\Bigg[((1-\sigma(1-s))\pr_u+(1-\sigma)\pr_{\ub})^{\a_1-\b_1}\\
&&(\sigma\pr_{\ub}+\sigma(1-s)\pr_u)^{\a_2-\b_2}\pr_{\tau}\left(\frac{\varrho}{r}\kappa \frac{\Omega^2}{4}\frac{g(\varrho w)^2}{\varrho^2}\right)\Bigg](u+2\sigma\varrho-2s\sigma\varrho, u+2\sigma\varrho)dsd\sigma.
\eee
We infer
\bee
\varrho^{\frac{1}{2}}\left|\pr^\a_{u, \ub}\left(\frac{\pr_\tau r}{\varrho}\right)\right| &\les& \varrho^\frac{5}{2}\Bigg|\int_0^1\int_0^1s\sigma^2\Bigg[((1-\sigma(1-s))\pr_u+(1-\sigma)\pr_{\ub})^{\a_1}\\
&&(\sigma\pr_{\ub}+\sigma(1-s)\pr_u)^{\a_2}\pr_{\tau}\left(\frac{\varrho}{r}\kappa \frac{\Omega^2}{4}\frac{g(\varrho w)^2}{\varrho^2}\right)\Bigg](u+2\sigma\varrho-2s\sigma\varrho, u+2\sigma\varrho)dsd\sigma\Bigg|\\
&& +\sum_{|\beta|=1}\varrho^{\frac{3}{2}}\Bigg|\int_0^1\int_0^1s\sigma^2\Bigg[((1-\sigma(1-s))\pr_u+(1-\sigma)\pr_{\ub})^{\a_1-\b_1}\\
&&(\sigma\pr_{\ub}+\sigma(1-s)\pr_u)^{\a_2-\b_2}\pr_{\tau}\left(\frac{\varrho}{r}\kappa \frac{\Omega^2}{4}\frac{g(\varrho w)^2}{\varrho^2}\right)\Bigg](u+2\sigma\varrho-2s\sigma\varrho, u+2\sigma\varrho)dsd\sigma\Bigg|\\
&&+\sum_{|\beta|\leq 2}\varrho^{\frac{1}{2}}\Bigg|\int_0^1\int_0^1s\sigma^2\Bigg[((1-\sigma(1-s))\pr_u+(1-\sigma)\pr_{\ub})^{\a_1-\b_1}\\
&&(\sigma\pr_{\ub}+\sigma(1-s)\pr_u)^{\a_2-\b_2}\pr_{\tau}\left(\frac{\varrho}{r}\kappa \frac{\Omega^2}{4}\frac{g(\varrho w)^2}{\varrho^2}\right)\Bigg](u+2\sigma\varrho-2s\sigma\varrho, u+2\sigma\varrho)dsd\sigma\Bigg|
\eee
and hence
\bee
\varrho^{\frac{1}{2}}\left|\pr^\a_{u, \ub}\left(\frac{\pr_\tau r}{\varrho}\right)\right| &\les& \varrho^\frac{3}{2}\Bigg|\int_0^1\int_0^1\sigma\Bigg[\varrho\pr_\tau((1-\sigma(1-s))\pr_u+(1-\sigma)\pr_{\ub})^{\a_1}\\
&&(\sigma\pr_{\ub}+\sigma(1-s)\pr_u)^{\a_2}\left(\frac{\varrho}{r}\kappa \frac{\Omega^2}{4}\frac{g(\varrho w)^2}{\varrho^2}\right)\Bigg](u+2\sigma\varrho-2s\sigma\varrho, u+2\sigma\varrho)dsd\sigma\Bigg|\\
&& +\sum_{|\beta|=1}\varrho^{\frac{1}{2}}\Bigg|\int_0^1\int_0^1\sigma\Bigg[\varrho((1-\sigma(1-s))\pr_u+(1-\sigma)\pr_{\ub})^{\a_1-\b_1}\\
&&(\sigma\pr_{\ub}+\sigma(1-s)\pr_u)^{\a_2-\b_2}\pr_{\tau}\left(\frac{\varrho}{r}\kappa \frac{\Omega^2}{4}\frac{g(\varrho w)^2}{\varrho^2}\right)\Bigg](u+2\sigma\varrho-2s\sigma\varrho, u+2\sigma\varrho)dsd\sigma\Bigg|\\
&&+\sum_{|\beta|\leq 2}\Bigg|\int_0^1\int_0^1s^{\frac{1}{2}}\sigma^{\frac{3}{2}}\Bigg[\varrho^{\frac{1}{2}}((1-\sigma(1-s))\pr_u+(1-\sigma)\pr_{\ub})^{\a_1-\b_1}\\
&&(\sigma\pr_{\ub}+\sigma(1-s)\pr_u)^{\a_2-\b_2}\pr_{\tau}\left(\frac{\varrho}{r}\kappa \frac{\Omega^2}{4}\frac{g(\varrho w)^2}{\varrho^2}\right)\Bigg](u+2\sigma\varrho-2s\sigma\varrho, u+2\sigma\varrho)dsd\sigma\Bigg|
\eee
which yields
\bee
\varrho^{\frac{1}{2}}\left|\pr^\a_{u, \ub}\left(\frac{\pr_\tau r}{\varrho}\right)\right| &\les& \varrho^\frac{3}{2}\Bigg|\int_0^1\int_0^1\sigma\Bigg[\varrho\pr_\tau((1-\sigma(1-s))\pr_u+(1-\sigma)\pr_{\ub})^{\a_1}\\
&&(\sigma\pr_{\ub}+\sigma(1-s)\pr_u)^{\a_2}\left(\frac{\varrho}{r}\kappa \frac{\Omega^2}{4}\frac{g(\varrho w)^2}{\varrho^2}\right)\Bigg](u+2\sigma\varrho-2s\sigma\varrho, u+2\sigma\varrho)dsd\sigma\Bigg|\\
&&+\max_{|\a| \leq \ell}\sup_{-2\leq \ub'\leq \ub}\sup_{-1\leq u\leq\ub}\varrho\left|\pr_{u,\ub}^\a\left(\frac{\varrho}{r}\Omega^2\frac{g(\varrho w)^2}{\varrho^2}\right)\right|\\
&&+\max_{|\a| \leq \ell-1}\sup_{-2\leq \ub'\leq \ub}\sup_{-1\leq u\leq\ub}\varrho^{\frac{1}{2}}\left|\pr_{u,\ub}^\a\left(\frac{\varrho}{r}\Omega^2\frac{g(\varrho w)^2}{\varrho^2}\right)\right|
\eee
and hence, we infer in view of the estimates \eqref{eq:intermediaryusefulestimates1} \eqref{eq:intermediaryusefulestimates2}
\bee
\varrho^{\frac{1}{2}}\left|\pr^\a_{u, \ub}\left(\frac{\pr_\tau r}{\varrho}\right)\right| &\les& \varrho^\frac{3}{2}\Bigg|\int_0^1\int_0^1\sigma\Bigg[\varrho\pr_\tau((1-\sigma(1-s))\pr_u+(1-\sigma)\pr_{\ub})^{\a_1}\\
&&(\sigma\pr_{\ub}+\sigma(1-s)\pr_u)^{\a_2}\left(\frac{\varrho}{r}\kappa \frac{\Omega^2}{4}\frac{g(\varrho w)^2}{\varrho^2}\right)\Bigg](u+2\sigma\varrho-2s\sigma\varrho, u+2\sigma\varrho)dsd\sigma\Bigg|\\
&&+C_0\Big(1+\widetilde{D}_{\leq\ell-1}(\ub)+\widetilde{E}_{\leq\ell}(\ub)\Big)^\ell(1+\widetilde{E}_{\ell+1}(\ub)).
\eee
Now, relying on the identity
\bee
 &&\Big[\varrho\pr_\tau ((1-\sigma(1-s))\pr_u+(1-\sigma)\pr_{\ub})^{\a_1}(\sigma\pr_{\ub}+\sigma(1-s)\pr_u)^{\a_2}f\Big](u+2\sigma\varrho-2s\sigma\varrho, u+2\sigma\varrho)\\
   &=& \frac{1}{2}(-s^2\pr_s+s\sigma\pr_\sigma)\Big[\Big\{((1-\sigma(1-s))\pr_u+(1-\sigma)\pr_{\ub})^{\a_1}\\
   &&(\sigma\pr_{\ub}+\sigma(1-s)\pr_u)^{\a_2}f\Big\}(u+2\sigma\varrho-2s\sigma\varrho, u+2\sigma\varrho)\Big]\\
   && - \frac{1}{2}\Big[(-s^2\pr_s+s\sigma\pr_\sigma)((1-\sigma(1-s))\pr_u+(1-\sigma)\pr_{\ub})^{\a_1}\\
   &&(\sigma\pr_{\ub}+\sigma(1-s)\pr_u)^{\a_2}f\Big](u+2\sigma\varrho-2s\sigma\varrho, u+2\sigma\varrho)\\
   &=& \frac{1}{2}(-s^2\pr_s+s\sigma\pr_\sigma)\Big[\Big\{((1-\sigma(1-s))\pr_u+(1-\sigma)\pr_{\ub})^{\a_1}\\
   &&(\sigma\pr_{\ub}+\sigma(1-s)\pr_u)^{\a_2}f\Big\}(u+2\sigma\varrho-2s\sigma\varrho, u+2\sigma\varrho)\Big]\\
   && - \frac{s}{2}\Big\{\Big[-s\pr_s+\sigma\pr_\sigma, ((1-\sigma(1-s))\pr_u+(1-\sigma)\pr_{\ub})^{\a_1}\\
   &&(\sigma\pr_{\ub}+\sigma(1-s)\pr_u)^{\a_2}\Big]f\Big\}(u+2\sigma\varrho-2s\sigma\varrho, u+2\sigma\varrho),
\eee
we have
\bee
&& \int_0^1\int_0^1\sigma\Bigg[\varrho\pr_\tau((1-\sigma(1-s))\pr_u+(1-\sigma)\pr_{\ub})^{\a_1}\\
&&(\sigma\pr_{\ub}+\sigma(1-s)\pr_u)^{\a_2}\left(\frac{\varrho}{r}\kappa \frac{\Omega^2}{4}\frac{g(\varrho w)^2}{\varrho^2}\right)\Bigg](u+2\sigma\varrho-2s\sigma\varrho, u+2\sigma\varrho)dsd\sigma\\
&=& \frac{1}{2}\int_0^1\int_0^1\sigma(-s^2\pr_s+s\sigma\pr_\sigma)\Bigg\{\Bigg[((1-\sigma(1-s))\pr_u+(1-\sigma)\pr_{\ub})^{\a_1}\\
&&(\sigma\pr_{\ub}+\sigma(1-s)\pr_u)^{\a_2}\left(\frac{\varrho}{r}\kappa \frac{\Omega^2}{4}\frac{g(\varrho w)^2}{\varrho^2}\right)\Bigg](u+2\sigma\varrho-2s\sigma\varrho, u+2\sigma\varrho)\Bigg\}dsd\sigma\\
&& - \frac{1}{2}\int_0^1\int_0^1\sigma s\Bigg\{\Bigg[-s\pr_s+\sigma\pr_\sigma, ((1-\sigma(1-s))\pr_u+(1-\sigma)\pr_{\ub})^{\a_1}\\
&&(\sigma\pr_{\ub}+\sigma(1-s)\pr_u)^{\a_2}\Bigg]\left(\frac{\varrho}{r}\kappa \frac{\Omega^2}{4}\frac{g(\varrho w)^2}{\varrho^2}\right)\Bigg\}(u+2\sigma\varrho-2s\sigma\varrho, u+2\sigma\varrho)dsd\sigma
\eee
and hence, we integrate by parts the first term
\bee
&& \int_0^1\int_0^1\sigma\Bigg[\varrho\pr_\tau((1-\sigma(1-s))\pr_u+(1-\sigma)\pr_{\ub})^{\a_1}\\
&&(\sigma\pr_{\ub}+\sigma(1-s)\pr_u)^{\a_2}\left(\frac{\varrho}{r}\kappa \frac{\Omega^2}{4}\frac{g(\varrho w)^2}{\varrho^2}\right)\Bigg](u+2\sigma\varrho-2s\sigma\varrho, u+2\sigma\varrho)dsd\sigma\\
&=& -\frac{1}{2}\int_0^1\sigma\Bigg[(\pr_u+(1-\sigma)\pr_{\ub})^{\a_1}(\sigma\pr_{\ub})^{\a_2}\left(\frac{\varrho}{r}\kappa \frac{\Omega^2}{4}\frac{g(\varrho w)^2}{\varrho^2}\right)\Bigg](u, u+2\sigma\varrho)d\sigma\\
&& +\int_0^1\int_0^1s\sigma\Bigg[((1-\sigma(1-s))\pr_u+(1-\sigma)\pr_{\ub})^{\a_1}\\
&&(\sigma\pr_{\ub}+\sigma(1-s)\pr_u)^{\a_2}\left(\frac{\varrho}{r}\kappa \frac{\Omega^2}{4}\frac{g(\varrho w)^2}{\varrho^2}\right)\Bigg](u+2\sigma\varrho-2s\sigma\varrho, u+2\sigma\varrho)dsd\sigma\\
&& +\frac{1}{2}\int_0^1s\Bigg[(s\pr_u)^{\a_1}(\pr_{\ub}+(1-s)\pr_u)^{\a_2}\left(\frac{\varrho}{r}\kappa \frac{\Omega^2}{4}\frac{g(\varrho w)^2}{\varrho^2}\right)\Bigg](\ub -2s\varrho, \ub)ds\\
&& - \int_0^1\int_0^1s\sigma\Bigg[((1-\sigma(1-s))\pr_u+(1-\sigma)\pr_{\ub})^{\a_1}\\
&&(\sigma\pr_{\ub}+\sigma(1-s)\pr_u)^{\a_2}\left(\frac{\varrho}{r}\kappa \frac{\Omega^2}{4}\frac{g(\varrho w)^2}{\varrho^2}\right)\Bigg](u+2\sigma\varrho-2s\sigma\varrho, u+2\sigma\varrho)dsd\sigma\\
&& - \frac{1}{2}\int_0^1\int_0^1\sigma s\Bigg\{\Bigg[-s\pr_s+\sigma\pr_\sigma, ((1-\sigma(1-s))\pr_u+(1-\sigma)\pr_{\ub})^{\a_1}\\
&&(\sigma\pr_{\ub}+\sigma(1-s)\pr_u)^{\a_2}\Bigg]\left(\frac{\varrho}{r}\kappa \frac{\Omega^2}{4}\frac{g(\varrho w)^2}{\varrho^2}\right)\Bigg\}(u+2\sigma\varrho-2s\sigma\varrho, u+2\sigma\varrho)dsd\sigma.
\eee
This yields 
\bee
&& \varrho\int_0^1\int_0^1\sigma\Bigg[\varrho\pr_\tau((1-\sigma(1-s))\pr_u+(1-\sigma)\pr_{\ub})^{\a_1}\\
&&(\sigma\pr_{\ub}+\sigma(1-s)\pr_u)^{\a_2}\left(\frac{\varrho}{r}\kappa \frac{\Omega^2}{4}\frac{g(\varrho w)^2}{\varrho^2}\right)\Bigg](u+2\sigma\varrho-2s\sigma\varrho, u+2\sigma\varrho)dsd\sigma\\
&=& -\frac{1}{2}\int_0^1\Bigg[\varrho(\pr_u+(1-\sigma)\pr_{\ub})^{\a_1}(\sigma\pr_{\ub})^{\a_2}\left(\frac{\varrho}{r}\kappa \frac{\Omega^2}{4}\frac{g(\varrho w)^2}{\varrho^2}\right)\Bigg](u, u+2\sigma\varrho)d\sigma\\
&& +\int_0^1\int_0^1\Bigg[\varrho((1-\sigma(1-s))\pr_u+(1-\sigma)\pr_{\ub})^{\a_1}\\
&&(\sigma\pr_{\ub}+\sigma(1-s)\pr_u)^{\a_2}\left(\frac{\varrho}{r}\kappa \frac{\Omega^2}{4}\frac{g(\varrho w)^2}{\varrho^2}\right)\Bigg](u+2\sigma\varrho-2s\sigma\varrho, u+2\sigma\varrho)dsd\sigma\\
&& +\frac{1}{2}\int_0^1\Bigg[\varrho(s\pr_u)^{\a_1}(\pr_{\ub}+(1-s)\pr_u)^{\a_2}\left(\frac{\varrho}{r}\kappa \frac{\Omega^2}{4}\frac{g(\varrho w)^2}{\varrho^2}\right)\Bigg](\ub -2s\varrho, \ub)ds\\
&& - \int_0^1\int_0^1\Bigg[\varrho((1-\sigma(1-s))\pr_u+(1-\sigma)\pr_{\ub})^{\a_1}\\
&&(\sigma\pr_{\ub}+\sigma(1-s)\pr_u)^{\a_2}\left(\frac{\varrho}{r}\kappa \frac{\Omega^2}{4}\frac{g(\varrho w)^2}{\varrho^2}\right)\Bigg](u+2\sigma\varrho-2s\sigma\varrho, u+2\sigma\varrho)dsd\sigma\\
&& - \frac{1}{2}\int_0^1\int_0^1\Bigg\{\varrho\Bigg[-s\pr_s+\sigma\pr_\sigma, ((1-\sigma(1-s))\pr_u+(1-\sigma)\pr_{\ub})^{\a_1}\\
&&(\sigma\pr_{\ub}+\sigma(1-s)\pr_u)^{\a_2}\Bigg]\left(\frac{\varrho}{r}\kappa \frac{\Omega^2}{4}\frac{g(\varrho w)^2}{\varrho^2}\right)\Bigg\}(u+2\sigma\varrho-2s\sigma\varrho, u+2\sigma\varrho)dsd\sigma.
\eee
and hence
\bee
&&  \varrho\Bigg|\int_0^1\int_0^1\sigma\Bigg[\varrho\pr_\tau((1-\sigma(1-s))\pr_u+(1-\sigma)\pr_{\ub})^{\a_1}\\
&&(\sigma\pr_{\ub}+\sigma(1-s)\pr_u)^{\a_2}\left(\frac{\varrho}{r}\kappa \frac{\Omega^2}{4}\frac{g(\varrho w)^2}{\varrho^2}\right)\Bigg](u+2\sigma\varrho-2s\sigma\varrho, u+2\sigma\varrho)dsd\sigma\Bigg|\\
&\les&  \max_{|\a| \leq \ell}\sup_{-2\leq \ub'\leq \ub}\sup_{-1\leq u\leq\ub}\varrho\left|\pr_{u,\ub}^\a\left(\frac{\varrho}{r}\Omega^2\frac{g(\varrho w)^2}{\varrho^2}\right)\right|\\
&\les& C_0\Big(1+\widetilde{D}_{\leq\ell-1}(\ub)+\widetilde{E}_{\leq\ell}(\ub)\Big)^\ell(1+\widetilde{E}_{\ell+1}(\ub)) 
\eee
where we have used \eqref{eq:intermediaryusefulestimates2} in the last inequality. Since on the other hand we have obtained previously 
\bee
\varrho^{\frac{1}{2}}\left|\pr^\a_{u, \ub}\left(\frac{\pr_\tau r}{\varrho}\right)\right| &\les& \varrho^\frac{3}{2}\Bigg|\int_0^1\int_0^1\sigma\Bigg[\varrho\pr_\tau((1-\sigma(1-s))\pr_u+(1-\sigma)\pr_{\ub})^{\a_1}\\
&&(\sigma\pr_{\ub}+\sigma(1-s)\pr_u)^{\a_2}\left(\frac{\varrho}{r}\kappa \frac{\Omega^2}{4}\frac{g(\varrho w)^2}{\varrho^2}\right)\Bigg](u+2\sigma\varrho-2s\sigma\varrho, u+2\sigma\varrho)dsd\sigma\Bigg|\\
&&+C_0\Big(1+\widetilde{D}_{\leq\ell-1}(\ub)+\widetilde{E}_{\leq\ell}(\ub)\Big)^\ell(1+\widetilde{E}_{\ell+1}(\ub)),
\eee
we deduce
\bea\label{eq:macronpresident3}
\max_{|\a| = \ell}\sup_{-1\leq u\leq\ub}\varrho^{\frac{1}{2}}\left|\pr^\a_{u, \ub}\left(\frac{\pr_\tau r}{\varrho}\right)\right| \les C_0\Big(1+\widetilde{D}_{\leq\ell-1}(\ub)+\widetilde{E}_{\leq\ell}(\ub)\Big)^\ell(1+\widetilde{E}_{\ell+1}(\ub)).
\eea

Next, recall that we have
\bee
\pr^\a_{u, \ub}\left(\frac{\log(\Omega)}{\varrho^2}\right) &=& -4\int_0^1\int_0^1\sigma\Bigg[((1-\sigma(1-s))\pr_u+(1-\sigma)\pr_{\ub})^{\a_1}\\
&& (\sigma\pr_{\ub}+\sigma(1-s)\pr_u)^{\a_2}\left(-\frac{1}{2}\pr_u(\varrho w)\pr_{\ub}(\varrho w) - \frac{\Omega^2}{8}\frac{\varrho^2}{r^2}\frac{g(\varrho w)^2}{\varrho^2}\right)\Bigg]\\
&&(u+2\sigma\varrho-2s\sigma\varrho, u+2\sigma\varrho)dsd\sigma.
\eee
Assume first that $\a_2\geq 1$. Then, relying on the identity
\bee
&& \varrho\Big[(\sigma\pr_{\ub}+\sigma(1-s)\pr_u)f\Big](u+2\sigma\varrho-2s\sigma\varrho, u+2\sigma\varrho)\\
&=& \frac{\sigma}{2}\pr_\sigma\Big[f(u+2\sigma\varrho-2s\sigma\varrho, u+2\sigma\varrho)\Big]\\
\eee
we have
\bee
\varrho\pr^\a_{u, \ub}\left(\frac{\log(\Omega)}{\varrho^2}\right) &=& -2\int_0^1\int_0^1\sigma^2\pr_\sigma\Bigg\{\Bigg[((1-\sigma(1-s))\pr_u+(1-\sigma)\pr_{\ub})^{\a_1}\\
&& (\sigma\pr_{\ub}+\sigma(1-s)\pr_u)^{\a_2-1}\left(-\frac{1}{2}\pr_u(\varrho w)\pr_{\ub}(\varrho w) - \frac{\Omega^2}{8}\frac{\varrho^2}{r^2}\frac{g(\varrho w)^2}{\varrho^2}\right)\Bigg]\\
&&(u+2\sigma\varrho-2s\sigma\varrho, u+2\sigma\varrho)\Bigg\}dsd\sigma\\
&& +2\int_0^1\int_0^1\sigma^2\Bigg\{\Bigg[\pr_\sigma, ((1-\sigma(1-s))\pr_u+(1-\sigma)\pr_{\ub})^{\a_1}\\
&& (\sigma\pr_{\ub}+\sigma(1-s)\pr_u)^{\a_2-1}\Bigg]\left(-\frac{1}{2}\pr_u(\varrho w)\pr_{\ub}(\varrho w) - \frac{\Omega^2}{8}\frac{\varrho^2}{r^2}\frac{g(\varrho w)^2}{\varrho^2}\right)\Bigg\}\\
&&(u+2\sigma\varrho-2s\sigma\varrho, u+2\sigma\varrho)dsd\sigma
\eee
and hence we integrate by parts the first term
\bee
\varrho\pr^\a_{u, \ub}\left(\frac{\log(\Omega)}{\varrho^2}\right) &=& -2\int_0^1\Bigg[(s\pr_u)^{\a_1}(\pr_{\ub}+(1-s)\pr_u)^{\a_2-1}\left(-\frac{1}{2}\pr_u(\varrho w)\pr_{\ub}(\varrho w) - \frac{\Omega^2}{8}\frac{\varrho^2}{r^2}\frac{g(\varrho w)^2}{\varrho^2}\right)\Bigg]\\
&&(\ub-2s\varrho, \ub)ds +4\int_0^1\int_0^1\sigma\Bigg[((1-\sigma(1-s))\pr_u+(1-\sigma)\pr_{\ub})^{\a_1}\\
&& (\sigma\pr_{\ub}+\sigma(1-s)\pr_u)^{\a_2-1}\left(-\frac{1}{2}\pr_u(\varrho w)\pr_{\ub}(\varrho w) - \frac{\Omega^2}{8}\frac{\varrho^2}{r^2}\frac{g(\varrho w)^2}{\varrho^2}\right)\Bigg]\\
&&(u+2\sigma\varrho-2s\sigma\varrho, u+2\sigma\varrho)dsd\sigma\\
&& +2\int_0^1\int_0^1\sigma^2\Bigg\{\Bigg[\pr_\sigma, ((1-\sigma(1-s))\pr_u+(1-\sigma)\pr_{\ub})^{\a_1}\\
&& (\sigma\pr_{\ub}+\sigma(1-s)\pr_u)^{\a_2-1}\Bigg]\left(-\frac{1}{2}\pr_u(\varrho w)\pr_{\ub}(\varrho w) - \frac{\Omega^2}{8}\frac{\varrho^2}{r^2}\frac{g(\varrho w)^2}{\varrho^2}\right)\Bigg\}\\
&&(u+2\sigma\varrho-2s\sigma\varrho, u+2\sigma\varrho)dsd\sigma.
\eee
This yields
\bee
\varrho^{\frac{3}{2}}\pr^\a_{u, \ub}\left(\frac{\log(\Omega)}{\varrho^2}\right) &=& -2\int_0^1\frac{1}{s^{\frac{1}{2}}}\Bigg[\varrho^{\frac{1}{2}}(s\pr_u)^{\a_1}(\pr_{\ub}+(1-s)\pr_u)^{\a_2-1}\Bigg(-\frac{1}{2}\pr_u(\varrho w)\pr_{\ub}(\varrho w)\\
&& - \frac{\Omega^2}{8}\frac{\varrho^2}{r^2}\frac{g(\varrho w)^2}{\varrho^2}\Bigg)\Bigg](\ub-2s\varrho, \ub)ds\\
&& +4\int_0^1\int_0^1\frac{\sigma^{\frac{1}{2}}}{s^{\frac{1}{2}}}\Bigg[\varrho^{\frac{1}{2}}((1-\sigma(1-s))\pr_u+(1-\sigma)\pr_{\ub})^{\a_1}\\
&& (\sigma\pr_{\ub}+\sigma(1-s)\pr_u)^{\a_2-1}\left(-\frac{1}{2}\pr_u(\varrho w)\pr_{\ub}(\varrho w) - \frac{\Omega^2}{8}\frac{\varrho^2}{r^2}\frac{g(\varrho w)^2}{\varrho^2}\right)\Bigg]\\
&&(u+2\sigma\varrho-2s\sigma\varrho, u+2\sigma\varrho)dsd\sigma\\
&& +2\int_0^1\int_0^1\frac{\sigma^{\frac{3}{2}}}{s^{\frac{1}{2}}}\Bigg\{\varrho^{\frac{1}{2}}\Bigg[\pr_\sigma, ((1-\sigma(1-s))\pr_u+(1-\sigma)\pr_{\ub})^{\a_1}\\
&& (\sigma\pr_{\ub}+\sigma(1-s)\pr_u)^{\a_2-1}\Bigg]\left(-\frac{1}{2}\pr_u(\varrho w)\pr_{\ub}(\varrho w) - \frac{\Omega^2}{8}\frac{\varrho^2}{r^2}\frac{g(\varrho w)^2}{\varrho^2}\right)\Bigg\}\\
&&(u+2\sigma\varrho-2s\sigma\varrho, u+2\sigma\varrho)dsd\sigma.
\eee
Hence, we infer
\bee
&&\max_{|\a| \leq \ell}\sup_{-1\leq u\leq\ub}\varrho^{\frac{3}{2}}\left|\pr^\a_{u, \ub}\left(\frac{\log(\Omega)}{\varrho^2}\right)\right|(u,\ub)\\
 &\les& \max_{|\a| \leq \ell-1}\sup_{-2\leq \ub'\leq \ub}\sup_{-1\leq u\leq\ub}\varrho^{\frac{1}{2}}\left|\pr_{u,\ub}^\a\left(\pr_u(\varrho w)\pr_{\ub}(\varrho w)\right)\right|\\
&& + \max_{|\a| \leq \ell-1}\sup_{-2\leq \ub'\leq \ub}\sup_{-1\leq u\leq\ub}\varrho^{\frac{1}{2}}\left|\pr_{u,\ub}^\a\left(\frac{\varrho^2}{r^2}\Omega^2\frac{g(\varrho w)^2}{\varrho^2}\right)\right|
\eee
which in view of \eqref{eq:intermediaryusefulestimates1} and \eqref{eq:intermediaryusefulestimates1bis} implies
\bee
\max_{|\a| \leq \ell}\sup_{-1\leq u\leq\ub}\varrho^{\frac{3}{2}}\left|\pr^\a_{u, \ub}\left(\frac{\log(\Omega)}{\varrho^2}\right)\right|(u,\ub) &\les& C_0\Big(1+\widetilde{D}_{\leq\ell-1}(\ub)+\widetilde{E}_{\leq\ell}(\ub)\Big)^\ell(1+\widetilde{E}_{\ell+1}(\ub)).
\eee
The case where $\a_2=0$ and hence $\a_1\geq 1$ is treated similarly by using the formula 
\bee
\pr^\a_{u, \ub}\left(\frac{\log(\Omega)}{\varrho^2}\right) &=& -4\int_0^1\int_0^1\sigma\Bigg[(\sigma\pr_u+\sigma(1-s)\pr_{\ub})^{\a_1}\\
&&((1-\sigma(1-s))\pr_{\ub}+(1-\sigma)\pr_u)^{\a_2}\left(-\frac{1}{2}\pr_u(\varrho w)\pr_{\ub}(\varrho w) - \frac{\Omega^2}{8}\frac{\varrho^2}{r^2}\frac{g(\varrho w)^2}{\varrho^2}\right)\Bigg]\\
&&(\ub-2\sigma\varrho, \ub-2\sigma\varrho+2s\sigma\varrho)dsd\sigma
\eee
which is obtained by reversing the role of $u$ and $\ub$, and by relying on the following identity to integrate by parts
\bee
&& \varrho\Big[(\sigma\pr_u+\sigma(1-s)\pr_\ub)f\Big](\ub-2\sigma\varrho, \ub-2\sigma\varrho+2s\sigma\varrho)\\
&=& -\frac{\sigma}{2}\pr_\sigma\Big[f(\ub-2\sigma\varrho, \ub-2\sigma\varrho+2s\sigma\varrho)\Big].
\eee
Together with \eqref{eq:marcronpresident0} \eqref{eq:macronpresident} \eqref{eq:macronpresident1} \eqref{eq:macronpresident2} \eqref{eq:macronpresident3}, we infer
\bee
\max_{|\a| \leq \ell}\sup_{-1\leq u\leq\ub} && \Bigg(\varrho^{\frac{3}{2}}\left(\left|\pr_{u,\ub}^\a\left(\frac{\varrho - r}{\varrho^3}\right)\right|+\left|\pr_{u,\ub}^\a\left(\frac{\pr_\varrho r-1}{\varrho^2}\right)\right|+\left|\pr_{u,\ub}^\a\left(\frac{\log(\Omega)}{\varrho^2}\right)\right|\right)\\
\nn &&+\varrho^{\frac{1}{2}}\left(\left|\pr_{u,\ub}^\a\left(\frac{\pr_\tau r}{\varrho}\right)\right|\right)+\left|\pr_{u,\ub}^\a\left(\frac{\varrho - r}{\varrho}\right)\right|
+\left|\pr_{u,\ub}^\a\left(\pr_{\ub}r-\frac{1}{2}\right)\right|\\&&+\left|\pr_{u,\ub}^\a\left(\pr_ur+\frac{1}{2}\right)\right|+\left|\pr_{u,\ub}^\a\left(\Omega^2-1\right)\right|\Bigg)(u,\ub)\\
&\les& C_0\Big(1+\widetilde{D}_{\leq\ell-1}(\ub)+\widetilde{E}_{\leq\ell}(\ub)\Big)^\ell(1+\widetilde{E}_{\ell+1}(\ub)).
\eee
Together with \eqref{eq:alternatedefintionDellinuubframe}, this yields
\bee
\nn D_\ell(\ub) &\les& C_0\Big(1+\widetilde{D}_{\leq\ell-1}(\ub)+\widetilde{E}_{\leq\ell}(\ub)\Big)^\ell(1+\widetilde{E}_{\ell+1}(\ub))
\eee
which concludes the proof of Lemma \ref{lemma:estimateforDell}.

\subsection*{Acknowledgements} We are grateful to an anonymous referee, who has done an outstanding and constructive refereeing job, for pointing out an error in an earlier version and for remarks which helped us substantially improve the exposition.  L.A. and N.G. thank the Erwin Sch\"odinger Institute, Vienna, for hospitality and support during part of this work. The majority of the work of N.G was supported by the International Max Planck Research School graduate scholarship of  Max Planck Society
at Albert Einstein Institute, Golm and the DFG Postdoctoral Fellowship GU1513/1-1 at Yale University. N.G is also grateful to Vincent Moncrief for numerous insightful discussions, in particular the experiences he shared from his earlier work on Gowdy spacetimes were influential in some of the approaches adopted for this problem. Further, the authors thank the Mathematical Sciences Research Institute in Berkeley, California, where part of this work was carried out, for hospitality and support. The work carried out during the semester programme on Mathematical General Relativity at MSRI during the fall of 2013 was supported in part by the National Science Foundation under Grant No. 0932078000. The third author is supported by the project ERC 291214 BLOWDISOL.

\appendix

\bibliography{sgwm}
\bibliographystyle{abbrv}

\end{document}